\documentclass[12pt]{amsart}

\usepackage{graphicx,amssymb}
\usepackage[labelfont=bf, width=\linewidth]{caption}

\usepackage{hyperref}  
\hypersetup{colorlinks=true}
\usepackage{hypcap}   
\usepackage{bookmark} 

\usepackage{dcpic}
\usepackage{pictexwd}
\usepackage{mathrsfs}
\usepackage{bm}
\usepackage{subcaption}
\usepackage{tabularx}
\usepackage{xcolor} 
\usepackage[french, english]{babel}
\usepackage[sort&compress]{natbib}
\bibpunct{[}{]}{,}{n}{,}{,~}

  \hypersetup{
  pdftitle = {Fukaya categories of Lagrangian cobordisms and duality},
  pdfauthor = {Emily Campling},
  pdfsubject = {},
 pdfkeywords = {}
}

\newtheorem{prop}{Proposition}[section]
\newtheorem{lem}[prop]{Lemma}
\newtheorem{cor}[prop]{Corollary}  
\newtheorem{thm}[prop]{Theorem} 
\newtheorem{conj}[prop]{Conjecture}  
\theoremstyle{definition}
\newtheorem{defn}[prop]{Definition}                         
\newtheorem{remk}[prop]{Remark}

\newcommand{\R}{\mathbb{R}}
\newcommand{\C}{\mathbb{C}}
\newcommand{\Z}{\mathbb{Z}}
\newcommand{\N}{\mathbb{N}}

\begin{document}

\title{Fukaya categories of Lagrangian cobordisms and duality}

\author{Emily Campling}

\begin{abstract}

We introduce a new type of duality structure for $A_\infty$-categories called a \emph{relative weak Calabi-Yau pairing} which generalizes Kontsevich and Soibelman's notion of a weak (proper) Calabi-Yau structure. We prove the existence of a relative weak Calabi-Yau pairing on Biran and Cornea's Fukaya category of Lagrangian cobordisms $\mathcal{F}uk_{\mathit{cob}}(\C\times M)$. Here $M$ is a symplectic manifold which is closed or tame at infinity. This duality structure on $\mathcal{F}uk_{\mathit{cob}}(\C\times M)$ extends the relative Poincar\'{e} duality satisfied by Floer complexes for pairs of Lagrangian cobordisms. Moreover, we show that the relative weak Calabi-Yau pairing on $\mathcal{F}uk_{\mathit{cob}}(\C\times M)$ satisfies a compatibility condition with respect to the usual weak Calabi-Yau structure on the monotone Fukaya category of $M$. 

The construction of the relative weak Calabi-Yau pairing on $\mathcal{F}uk_{\mathit{cob}}(\C\times M)$ is based on counts of curves in $\C\times M$ satisfying an inhomogeneous nonlinear Cauchy-Riemann equation. In order to prove the existence of this duality structure and to verify its properties, we extend the methods of Biran and Cornea to establish regularity and compactness results for the relevant moduli spaces. We also consider the implications of the existence of the relative weak Calabi-Yau pairing on $\mathcal{F}uk_{\mathit{cob}}(\C\times M)$ for the cone decomposition in the derived Fukaya category of $M$ associated to a Lagrangian cobordism, and we present an example involving Lagrangian surgery.

\end{abstract}

\maketitle
\tableofcontents

\section{Introduction}

The study of Lagrangian cobordisms, which originated with Arnold \cite{Arn80,Arn80II}, has seen many significant developments in recent years, most notably in the work of Biran and Cornea  \cite{BC13,BC14,BCLefschetz}. The underlying theme of their work has been one of establishing connections between geometric notions and seemingly disparate algebraic ones. Their central result of this nature concerns two types of decomposition of a Lagrangian $L$ in a symplectic manifold $M$ which is closed or tame at infinity: one type of decomposition is a geometric decomposition, via a Lagrangian cobordism in $\C\times M$; the other is an algebraic decomposition, by exact triangles in the derived (monotone) Fukaya category of $M$. A cobordism is viewed as a decomposition of $L$ into Lagrangians $L_1,\ldots,L_s\subset M$ if the cobordism has as ends $L,L_1,\ldots,L_s$. The result of Biran and Cornea associates to such a cobordism a cone decomposition of $L$ as an object of the derived Fukaya category of $M$, with factors $L_1,\ldots,L_s$. The proof of this result relies on constructing a Fukaya category whose objects are Lagrangian cobordisms in $\C\times M$.  

The connections between algebraic and geometric notions stemming from the study of Lagrangian cobordisms can be deepened by considering extra structures on the Fukaya category of cobordisms $\mathcal{F}uk_{\mathit{cob}}(\C\times M)$. For example, by taking into account the filtration of Floer complexes of Lagrangian cobordisms by the action functional, Biran and Cornea together with Shelukhin showed that $\mathcal{F}uk_{\mathit{cob}}(\C\times M)$ has the structure of a ``weakly filtered'' $A_\infty$-category. From this, they were able to derive new results concerning Lagrangian intersections and to introduce new pseudo-metrics and metrics on classes of Lagrangians \cite{biran2018lagrangian}. 

In this thesis, we consider a different kind of extra structure on $\mathcal{F}uk_{\mathit{cob}}(\C\times M)$, namely a duality structure. Duality for abstract $A_\infty$-categories can be formulated in terms of a \emph{weak Calabi-Yau structure}, a notion introduced by Kontsevich and Soibelman \cite{KontsevichSoibelman}. The monotone Fukaya category of a compact symplectic manifold admits a weak Calabi-Yau structure whose geometric description generalizes the Poincar\'{e} duality satisfied by Lagrangian Floer complexes \cite{Sei,SheridanFano}. We introduce an algebraic variant of weak Calabi-Yau structures for $A_\infty$-categories which we call a \emph{relative weak Calabi-Yau pairing}. We then describe a geometric realization of this structure for Fukaya categories of cobordisms which generalizes the Poincar\'{e} duality relative to boundary satisfied by Floer complexes of Lagrangian cobordisms. Our main result establishes the existence of this structure for $\mathcal{F}uk_{\mathit{cob}}(\C\times M)$, as well as its compatibility with the usual weak Calabi-Yau structure on $\mathcal{F}uk(M)$, the monotone Fukaya category of $M$.

\subsection{Context and new ingredients}

This work builds on notions from both algebra and geometry. The algebra involved is the homological algebra of $A_\infty$-categories, a categorical extension of $A_\infty$-algebras. The latter generalize differential graded algebras and have been of interest in algebraic topology (see for example \cite{May72,Adams78,Kad79}) since their introduction by Stasheff in 1963 \cite{Sta63}. Starting in the early nineties, $A_\infty$-categories have attracted considerable attention beyond algebraic topology -- in algebra, geometry, and mathematical physics (see for example \cite{GetzJones90,Sta92,Fuk93,McCl99,PolZas01,Abou10}). Much of this interest has been spurred by the introduction of the Fukaya $A_\infty$-category and its connection with Kontsevich's famous homological mirror symmetry conjecture \cite{Kontsevich95}. 

The conjecture of Kontsevich posits an equivalence between two triangulated categories associated to $A_\infty$-categories: the derived category of coherent sheaves on a Calabi-Yau manifold and the derived Fukaya category of the ``mirror'' Calabi-Yau manifold. The derived category of coherent sheaves on a compact Calabi-Yau manifold is a so-called Calabi-Yau triangulated category. This refers to the presence of a duality structure on the category coming from Serre duality. The existence of this structure relies on the particularly simple form Serre duality takes for Calabi-Yau manifolds, resulting from the triviality of the canonical line bundle. 

Weak Calabi-Yau structures for $A_\infty$-categories generalize the Calabi-Yau property for ordinary categories, and have played an important role in the work of Abouzaid-Smith \cite{AbouzaidSmith}, Ganatra \cite{GanatraOriginalThesis}, Ganatra-Perutz-Sheridan \cite{ganatra2015mirror}, and Sheridan \cite{SheridanHodge,SheridanFano}, among others. As we explain in more detail in Section \ref{SECTION: wCY structures}, drawing from the detailed guide to Calabi-Yau structures in \cite[\S 6.1]{ganatra2015mirror}, there are many different $A_\infty$-versions of Calabi-Yau structures. The particular kind we are interested in, weak (proper) Calabi-Yau structures, is defined in terms of a quasi-isomorphism between two $A_\infty$-modules which can be associated to any $A_\infty$-category: the diagonal bimodule and the Serre bimodule, a dual to the diagonal bimodule. 

The algebraic description of duality we formulate, which to the best of our knowledge is new, generalizes weak Calabi-Yau structures to situations where we are given an $A_\infty$-functor $\mathbf{I}:\mathcal{B}\to\mathcal{A}$, and the $A_\infty$-category $\mathcal{A}$ is equipped with certain auxiliary structures. In particular, we require $\mathcal{A}$ to have an associated bimodule, which we call the \emph{relative Serre bimodule}, that plays the role of the Serre bimodule for $\mathcal{A}$ in our formulation of duality. The duality structure we introduce, a relative weak Calabi-Yau pairing on $\mathcal{A}$,  induces an ordinary weak Calabi-Yau structure on the category $\mathcal{B}$. 

Proving the existence of a relative weak Calabi-Yau pairing on $\mathcal{F}uk_{\mathit{cob}}(\C\times M)$ relies on first establishing the existence of an appropriately defined relative Serre bimodule. The functor $\mathbf{I}$ in this context is an inclusion functor of $\mathcal{F}uk(M)$ into $\mathcal{F}uk_{\mathit{cob}}(\C\times M)$. The existence of such an inclusion functor was established by Biran and Cornea; however we adapt their construction of the functor in such a way that the relative weak Calabi-Yau pairing on $\mathcal{F}uk_{\mathit{cob}}(\C\times M)$ induces the usual weak Calabi-Yau structure on $\mathcal{F}uk(M)$. 

On the geometric side, following the path laid out by Biran and Cornea in \cite{BC13,BC14}, our work combines two widely used methods for studying Lagrangian submanifolds: the ``rigid'' method of Lagrangian intersection theory, and the ``flexible'' method of Lagrangian cobordism. The first method, Lagrangian intersection theory, is the central method in symplectic topology for studying Lagrangians. It has as a main tool Floer homology \cite{Fl88, Fl89, Fl89b}, which builds on Gromov's theory of pseudoholomorphic curves \cite{Gr85}. The techniques of Gromov and Floer are at the heart of a large proportion of developments in modern symplectic topology. Gromov's groundbreaking discovery was that, under appropriate assumptions, moduli spaces of $J$-holomorphic curves in a symplectic manifold $M$ are finite-dimensional smooth manifolds for a generic almost complex structure $J$ on $M$. Moreover, if the curves satisfy a uniform energy bound, these moduli spaces can be compactified by nodal curves. The theory of $J$-holomorphic curves draws on diverse techniques and theorems from partial differential equations, and functional and complex analysis, such as the theory of elliptic partial differential equations, Fredholm theory, and the Riemann-Roch theorem. 

Floer related Gromov's work to Lagrangian intersection theory by defining a complex generated by the intersection points of two Lagrangians (in the transverse case), and whose differential counts pseudoholomorphic strips between the Lagrangians. Floer homology plays a role as fundamental to symplectic topology as the role of singular homology in classical algebraic topology. A key development in the more recent history of Lagrangian intersection theory has been the construction of the Fukaya category of a symplectic manifold $M$, whose objects are Lagrangian submanifolds in $M$ (satisfying some conditions) and whose morphism spaces are Floer chain complexes. In other words, the Fukaya category incorporates all Floer complexes for pairs of Lagrangians in $M$ into a coherent algebraic structure. The construction of the Fukaya category, which is highly technical, is pursued in detail in Seidel's seminal text \cite{Sei} in the exact case.

The second method of studying Lagrangians, Lagrangian cobordism, is a natural notion extending to the Lagrangian setting the topological notion of cobordism which has been prevalent in differential topology since the work of Thom in the fifties. Lagrangian cobordisms have been studied in various contexts by Eliashberg \cite{Eli84}, Audin \cite{Aud85},  Chekanov \cite{Ch97}, and several others. Biran and Cornea have studied monotone cobordisms in $\C\times M$ \cite{BC13,BC14} as well as in Lefschetz fibrations \cite{BCLefschetz} and explored the relationship of cobordisms with algebraic structures in Lagrangian intersection theory, namely Floer complexes, and Fukaya and derived Fukaya categories.

The weak Calabi-Yau structure on the Fukaya category of a compact symplectic manifold is defined by counting curves which can be seen as computing $A_\infty$-comparison maps. At the chain level, the Poincar\'{e} duality quasi-isomorphism for Floer complexes is a composition of two maps: a formal map whose definition does not involve any counts of pseudoholomorphic curves, and a comparison map interpolating between two sets of perturbation data. In the transverse case, neglecting degree considerations, the formal map simply acts as the identity on intersection points viewed as generators of the complex for the pair of Lagrangians $L,L'$ in the domain, and the dual of the complex for the pair $L',L$ in the target. For the Fukaya category, we can similarly interpret duality as a composition of two morphisms, one whose definition is formal and does not involve curve counts, and another counting curves which interpolate between two sets of perturbation data for defining Floer complexes. 

In the Lagrangian cobordism context, the Poincar\'{e} duality quasi-isomorphism for Floer complexes, as well as the relative weak Calabi-Yau pairing, can also be understood as a composition of a formal map and a comparison map. For self-Floer complexes (where the homology computed is the quantum homology of the Lagrangian cobordism), the formal map has as domain a complex computing quantum homology relative to the boundary of the cobordism and as target one computing absolute quantum cohomology. The perturbation data involved in defining Floer complexes for Lagrangian cobordisms satisfy a splitting condition with respect to the product structure of the ambient symplectic manifold $\C\times M$ outside of a compact set. The key difference in the construction of the weak Calabi-Yau pairing on $\mathcal{F}uk_{\mathit{cob}}(\C\times M)$ versus the Calabi-Yau structure on $\mathcal{F}uk(M)$ is that the comparison maps involved only interpolate between data in the direction of the fibre $M$, whereas the data on $\C$ remain unchanged.

In Biran and Cornea's extension of the construction of Fukaya categories to Lagrangian cobordisms \cite{BC14}, the main technical challenge is to prove compactness results for the moduli spaces involved. The difficulty in this arises from the non-compactness of the ambient manifold $\C\times M$. In order to prove our main result, we extend the methods of \cite{BC14} to establish compactness and regularity results for the moduli spaces involved in the definition of the relative weak Calabi-Yau pairing.

\subsection{Organization of the thesis}
We proceed to a brief outline of the contents of the sections in the thesis. The first section introduces the algebraic concepts involving $A_\infty$-categories that we will use: basic concepts such as functors and natural transformations, modules and bimodules, as well as more specialized concepts such as dualization, Yoneda and abstract Serre functors, and Hochschild homology and cohomology. The second section introduces the geometric background we will rely on: Lagrangian cobordisms, Lagrangian Floer homology and Poincar\'{e} duality for Floer complexes, as well as the monotone Fukaya category. 

In the third section we review the algebraic notion of a weak Calabi-Yau structure on an $A_\infty$-category from three different perspectives: in terms of Hochschild homology, in terms of diagonal and Serre bimodules, as well as in terms of Yoneda functors and abstract Serre functors. We also introduce the modified version of this notion that we will use, that of a relative weak Calabi-Yau pairing. Finally we describe the geometric realization of the weak Calabi-Yau structure on the monotone Fukaya category. 

In the fourth section we set up the Fukaya category of cobordisms following \cite{BC14} and describe its duality structure, both from the perspective of Hochschild homology and of natural transformations between Yoneda and relative abstract Serre functors. We state the main theorem (Theorem \ref{THM: Main theorem}) which asserts that the structure we describe is in fact a weak Calabi-Yau pairing compatible with the weak Calabi-Yau structure on the Fukaya category of $M$. The theorem also relates the two alternative descriptions of this structure. Section \ref{SECTION: Proofs of the main results} is devoted to the proof of this result. In the remaining section we consider an application of the main theorem relating to the cone decomposition associated to a cobordism. We present an example coming from Lagrangian surgery and state a conjecture for the cone decomposition associated to an arbitrary cobordism.

\addtocontents{toc}{\protect\setcounter{tocdepth}{0}}
\subsection*{Acknowledgements}
This thesis was completed under the direction of Octav Cornea at the University of Montreal. I am grateful to him for his excellent guidance. I also wish to thank Sheel Ganatra for valuable feedback on this work, and Paul Biran for helpful discussions. This research was supported by an NSERC Postgraduate Scholarship, an ISM Scholarship for Outstanding PhD Candidates, and a University of Montreal final year scholarship.
\addtocontents{toc}{\protect\setcounter{tocdepth}{3}}

\section{Algebraic preliminaries}

In this section, we review some background material relating to $A_\infty$-categories. We refer the reader to \cite{Sei} for an in-depth account of large portions of this material which uses cohomological conventions, and to the appendix of \cite{BC14} for a summary using the homological conventions we follow here. We note however that unlike in \cite{BC14}, the concepts in this section will be presented in the graded context. $A_\infty$-bimodules are covered in \cite{GanatraThesis,SeidelSubalg08,SheridanFano}, and Hochschild homology and cohomology for $A_\infty$-categories using cohomological conventions in \cite{GanatraThesis,SheridanFano}. We assume all vector spaces to be over the field $\Z_2$ and all gradings to be over $\Z$. All of the algebraic constructions we present can also be performed over arbitrary fields, but it requires the insertion of Koszul signs in all formulae (see \cite{Sei} for details of this in the cohomological context). 

\subsection{$A_\infty$-categories}

\begin{defn}
An \textbf{$A_\infty$-category} $\mathcal{A}$ of degree $N_{\mathcal{A}}\in\Z$ consists of a class of objects $\mathrm{Ob}(\mathcal{A})$, a graded vector space $\mathcal{A}(X_0,X_1)$ for every pair of objects $X_0,X_1\in \mathrm{Ob}(\mathcal{A})$, and for every family $X_0,\ldots,X_k$ of $k+1$ objects a linear map
\begin{equation}\label{EQ: mu_k}
\mu^{\mathcal{A}}_k:\mathcal{A}(X_0,X_1)\otimes\cdots \otimes\mathcal{A}(X_{k-1},X_k)\to \mathcal{A}(X_0,X_k)
\end{equation}
of degree $-2+k(1-N_{\mathcal{A}})+N_{\mathcal{A}}$.
These maps are required to satisfy the $A_\infty$ relations:
\begin{equation}\label{A_infty relation}
\sum_{j, s}\mu^{\mathcal{A}}_{k-s+1}(x_1,\ldots,x_j,\mu^{\mathcal{A}}_s(x_{j+1},\ldots,x_{j+s}),x_{j+s+1},\ldots,x_k)=0,
\end{equation}
for all $k>0$. Here the sum is over all $0< s\le k$ and $0\le j \le k-s$, i.e.~over all possible terms.
\end{defn} 

In particular, setting $k=1$ in the equation above, we have $\mu^{\mathcal{A}}_1\circ \mu^{\mathcal{A}}_1=0$ and hence $\mu^{\mathcal{A}}_1$ defines a differential on $\mathcal{A}(X_0,X_1)$ for all pairs of objects $X_0$ and $X_1$ in $\mathrm{Ob}(\mathcal{A})$. Setting $k=2$, we obtain that $\mu^{\mathcal{A}}_2$ defines a product which satisfies the Leibniz rule with respect to $\mu_1^\mathcal{A}$. Setting $k=3$ in Equation \eqref{A_infty relation}, we see that this product is associative up to homotopy, with the chain homotopy given by $\mu^{\mathcal{A}}_3$. These observations imply that $\mu_2^\mathcal{A}$ descends to an associative product on homology and we can therefore make the following definition. 

\begin{defn}
The \textbf{homological category} $H(\mathcal{A})$ associated to $\mathcal{A}$ is the category whose objects are the same as those of $\mathcal{A}$, whose morphism spaces are the homology spaces of the morphism spaces in $\mathcal{A}$, and whose composition maps are the maps induced on homology by $\mu_2^\mathcal{A}$. This is an ordinary linear graded category, although possibly without identity morphisms.
\end{defn}

\begin{remk}
The homological conventions of \cite{BC14} that we use here differ from the cohomological conventions of \cite{Sei} in two respects. The first is in the arbitrary choice to order the objects as $X_0, X_1,\ldots,X_k$ in defining the composition maps $\mu^\mathcal{A}_k$ (and the maps occurring in the definitions of other structures), as opposed to the usual ordering for a cohomological category $\mathcal{A}'$ where the compositions are given as maps
\begin{equation}
\mu^k_{\mathcal{A}'}:\mathcal{A}'(X_{k-1},X_k)\otimes\cdots\otimes\mathcal{A}'(X_0,X_1)\to \mathcal{A}'(X_0,X_k).
\end{equation}
The second difference is in the association of a degree to a homological $A_\infty$-category, a notion which is not necessary in defining cohomological $A_\infty$-categories, but which is needed to specify the degrees of maps when defining homological $A_\infty$-categories. 
\end{remk}

\begin{remk}
All definitions and constructions relating to $A_\infty$-categories that we consider also have ungraded versions. An ungraded $A_\infty$-category has as morphism spaces ungraded vector spaces. All other definitions are identical except there is no requirement for the degree of maps. When considering Fukaya categories, we will work in the ungraded context for simplicity, but we set up the algebra in the graded context for completeness. 
\end{remk}

For ease of notation we introduce the convenient shorthand
\begin{equation}
\mathcal{A}(X_0,\ldots,X_k):=\mathcal{A}(X_0,X_1)\otimes\cdots\otimes\mathcal{A}(X_{k-1},X_k).
\end{equation}

\begin{defn}
The \textbf{opposite category} $\mathcal{A}^{opp}$ of an $A_\infty$-category $\mathcal{A}$ of degree $N_{\mathcal{A}}$ is the $A_\infty$-category of degree $N_{\mathcal{A}}$ whose objects are the same as those of $\mathcal{A}$, whose morphism spaces are given by $\mathcal{A}^{opp}(X_0,X_1)=\mathcal{A}(X_1,X_0)$, and whose higher composition maps $\mu^{\mathcal{A}^{opp}}_k$ are given by reversing the order of morphisms for the maps $\mu^{\mathcal{A}}_k$:
\begin{equation}
\mu_k^{\mathcal{A}^{opp}}(x_1,\ldots,x_k)=\mu_k^{\mathcal{A}}(x_k,\ldots,x_1).
\end{equation}
\end{defn}

\begin{defn}
The \textbf{$j$-fold suspension} $\mathcal{A}[j]$ of $\mathcal{A}$ is an $A_\infty$-category of degree $N_\mathcal{A}-j$ which is defined by shifting by $j$ the degrees in the morphism spaces of $\mathcal{A}$:
\begin{equation}
\mathcal{A}[j](X_0,X_1)_p=\mathcal{A}(X_0,X_1)_{p+j}.
\end{equation}
\end{defn}

\begin{defn}
An \textbf{$A_\infty$-functor} $\mathbf{F}$ between two $A_\infty$-categories $\mathcal{A}$ and $\mathcal{B}$ consists of a map on objects $\mathbf{F}:\mathrm{Ob}(\mathcal{A})\to \mathrm{Ob}(\mathcal{B})$ and, for every family $X_0,\ldots, X_k$ of $k\ge 1$ objects, a linear map
$$\mathbf{F}_k:\mathcal{A}(X_0,\ldots,X_k)\to \mathcal{B}(\mathbf{F}(X_0),\mathbf{F}(X_k))$$
of degree $-1+k(1-N_{\mathcal{A}})+N_{\mathcal{B}}$. The maps $\mathbf{F}_k$ are required to satisfy the relations
\begin{align}\label{A_infty functor relation}
&\sum_{s}\sum_{p_1,\ldots,p_s}\mu^{\mathcal{B}}_{s}(\mathbf{F}_{p_1}(x_1,\ldots,x_{p_1}),\ldots,\mathbf{F}_{p_s}(x_{k-p_s+1},\ldots,x_{k}))\nonumber\\
&\quad=\sum_{j,q}\mathbf{F}_{k-q+1}(x_1,\ldots,x_j,\mu^{\mathcal{A}}_q(x_{j+1},\ldots,x_{j+q}), x_{j+q+1},\ldots, x_k),
\end{align}
where the sum is over all possible terms.
\end{defn}

The relations \eqref{A_infty functor relation} imply that $\mathbf{F}$ induces a functor $H(\mathbf{F}):H(\mathcal{A})\to H(\mathcal{B})$ which is defined on morphisms by $H(\mathbf{F})([x])=[\mathbf{F}_1(x)]$.

Two $A_\infty$-functors $\mathbf{F}:\mathcal{A}\to \mathcal{B}$ and $\mathbf{G}:\mathcal{B}\to\mathcal{C}$ can be composed to give an $A_\infty$-functor $\mathbf{G}\circ\mathbf{F}:\mathcal{A}\to\mathcal{C}$. The map on objects for $\mathbf{G}\circ\mathbf{F}$ is the composition of the object maps for $\mathbf{F}$ and $\mathbf{G}$, and the maps $(\mathbf{G}\circ\mathbf{F})_k$ are defined by
\begin{equation}
\begin{aligned}
&(\mathbf{G}\circ\mathbf{F})_k(x_1,\ldots,x_k)=\\
&\qquad \sum_{s}\sum_{p_1,\ldots,p_s}\mathbf{G}_{s}(\mathbf{F}_{p_1}(x_1,\ldots,x_{p_1}),\mathbf{F}_{p_2}(x_{p_1+1},\ldots,x_{p_1+p_2}),\ldots\\
&\qquad\qquad\qquad\qquad\ldots,\mathbf{F}_{p_s}(x_{k-p_s+1},\ldots,x_k)).
\end{aligned}
\end{equation}
The sum is taken over all possible terms of the appropriate form.

Natural transformations between $A_\infty$-functors are defined as a specific class of more general transformations called pre-natural transformations.
\begin{defn}\label{DEFN: pre-natural transformation}
A \textbf{pre-natural transformation} $T$ of degree $g$ between two $A_\infty$-functors $\mathbf{F}_0$ and $\mathbf{F}_1$ from $\mathcal{A}$ to $\mathcal{B}$ is a sequence $T=(T_0,T_1,\ldots)$, where 
\begin{itemize}
\item $T_0$ consists of a family of elements $(T_0)_X\in \mathcal{B}(\mathbf{F}_0(X),\mathbf{F}_1(X))$ of degree $g$ for each object $X\in\mathrm{Ob}(\mathcal{A})$, and  
\item $T_k$ is a collection of linear maps 
$$T_k:\mathcal{A}(X_0,\ldots,X_k)\to \mathcal{B}(\mathbf{F}_0(X_0),\mathbf{F}_1(X_k))$$
of degree $g+k(1-N_{\mathcal{A}})$ for each family of objects $X_0,\ldots,X_k$ in $\mathrm{Ob}(\mathcal{A})$.
\end{itemize}
\end{defn}

The collection of all $A_\infty$-functors from $\mathcal{A}$ to $\mathcal{B}$ form the objects of an $A_\infty$-category $\mathit{fun}(\mathcal{A},\mathcal{B})$ of degree $N_{\mathcal{B}}$. The morphisms in $\mathit{fun}(\mathcal{A},\mathcal{B})$ are spaces of pre-natural transformations, and the maps $\mu_k^{\mathit{fun}(\mathcal{A},\mathcal{B})}$ are given by Equation (75) and a generalization of Equation (76) in \cite{BC14}. Those pre-natural transformations $\nu$ satisfying $\mu_1^{\mathit{fun}(\mathcal{A},\mathcal{B})}(\nu)=0$ are called \textbf{natural transformations}.

\begin{defn} Fix a functor $\mathbf{F}:\mathcal{A}\to\mathcal{B}$ and an $A_\infty$-category $\mathcal{C}$. The \textbf{left and right composition functors} associated to $\mathbf{F}$ are the functors
\begin{equation}
\mathbf{L}_{\mathbf{F}}:\mathit{fun}(\mathcal{C},\mathcal{A})\to \mathit{fun}(\mathcal{C},\mathcal{B}),\; \mathbf{R}_{\mathbf{F}}:\mathit{fun}(\mathcal{B},\mathcal{C})\to\mathit{fun}(\mathcal{A},\mathcal{C}),
\end{equation}
whose action on objects is given by composition with $\mathbf{F}$ and whose higher maps are defined as follows. The first order map for $\mathbf{L}_{\mathbf{F}}$ applied to a pre-natural transformation $T\in \mathit{fun}(\mathcal{C},\mathcal{A})(\mathbf{G},\mathbf{G}')$ is given by
\begin{equation}
\begin{aligned}
&(\mathbf{L}_{\mathbf{F}})_1(T)(x_1,\ldots,x_m)\\
&\qquad=\sum_{s,i}\sum_{p_1,\ldots,p_s} \mathbf{F}_{s}(\mathbf{G}_{p_1}(x_1,\ldots,x_{p_1}),\ldots,\mathbf{G}_{p_{i-1}}(x_{p_1+\cdots+p_{i-2}+1},\ldots,x_{p_1+\cdots+p_{i-1}}),\\
&\qquad\qquad T_{p_i}(x_{p_1+\cdots+p_{i-1}+1},\ldots,x_{p_1+\cdots+p_i}),
\mathbf{G}'_{p_{i+1}}(x_{p_1+\cdots+p_i+1},\ldots,x_{p_1+\cdots+p_{i+1}}),\\
&\qquad\qquad\qquad\ldots,\mathbf{G}'_{p_s}(x_{m-p_s+1},\ldots,x_{m})).
\end{aligned}
\end{equation}
The higher order maps $(\mathbf{L}_{\mathbf{F}})_k$, $k\ge 2$, are given by an obvious generalization of this formula. The first order map for $\mathbf{R}_{\mathbf{F}}$ applied to a pre-natural transformation $S\in \mathit{fun}(\mathcal{B},\mathcal{C})(\mathbf{H},\mathbf{H}')$ is given by
\begin{equation}
\begin{aligned}
&(\mathbf{R}_{\mathbf{F}})_1(S)(y_1,\ldots,y_m)\\
&\qquad=\sum_t\sum_{p_1,\ldots,p_t}S_{t}(\mathbf{F}_{p_1}(y_1,\ldots,y_{p_1}),\ldots,\mathbf{F}_{p_t}(y_{m-p_t+1},\ldots,y_m)).
\end{aligned}
\end{equation}
The higher order maps $(\mathbf{R}_{\mathbf{F}})_k$, $k\ge 2$, all vanish.
\end{defn}

There are different notions of what it means for an $A_\infty$-category to be equipped with identity morphisms.

\begin{defn}
An $A_\infty$-category $\mathcal{A}$ is \textbf{strictly unital} if for each object $X$ there is a unique element $e_X\in \mathcal{A}(X,X)_{N_{\mathcal{A}}}$ satisfying  
\begin{enumerate}
\item $\mu^{\mathcal{A}}_1(e_X)=0$.
\item For $x\in \mathcal{A}(X_0,X_1)$, $\mu^\mathcal{A}_2(x,e_{X_1})=x=\mu^\mathcal{A}_2(e_{X_0},x)$.
\item $\mu^\mathcal{A}_k(x_1,\ldots,x_i,e_{X_i},x_{i+1},\ldots, x_{k-1})=0$ for $k>2$, $x_j\in \mathcal{A}(X_{j-1},X_j)$, and any $0\le i\le k-1$.
\end{enumerate}
\end{defn}

A weaker notion than strict unitality is that of homological unitality.
\begin{defn}
An $A_\infty$-category $\mathcal{A}$ is \textbf{homologically unital} if for each object $X$ there is a unique element $1_X\in H_{N_{\mathcal{A}}}(\mathcal{A}(X,X))$ which is an identity with respect to composition in the category $H(\mathcal{A})$.
\end{defn}

\subsection{$A_\infty$-modules and bimodules}\label{SECTION: A_infty modules and bimodules}

\begin{defn}
A \textbf{left $\mathcal{A}$-module} $\mathcal{M}$ consists of the following:
\begin{itemize}
\item For every object $X\in \mathrm{Ob}(\mathcal{A})$, a graded $\Z_2$-vector space $\mathcal{M}(X)$.
\item For every $k\ge 0$ and every family of objects $X_0,\ldots,X_k\in \mathrm{Ob}(\mathcal{A})$, \textbf{module structure maps} $\mu^{\mathcal{M}}_{k|1}$. These are linear maps of degree $-1+k(1-N_{\mathcal{A}})$, 
\begin{equation}
\mu^{\mathcal{M}}_{k|1}:\mathcal{A}(X_0,\ldots,X_k)\otimes \mathcal{M}(X_k)\to\mathcal{M}(X_0),
\end{equation} 
which must satisfy the following $A_\infty$-relations:
\begin{align}\label{EQ: Module A_infty relations}
&\sum\mu^\mathcal{M}_{i-1|1}(x_1,\ldots, x_{i-1},\mu^\mathcal{M}_{k-i+1|1}(x_i,\ldots,x_k,\mathbf{w}))\\
&\quad+\sum\mu^\mathcal{M}_{k-j'+j|1}(x_1,\ldots,\mu^{\mathcal{A}}_{j'-j+1}(x_{j},\ldots,x_{j'}),\ldots,x_k,\mathbf{w})=0.\nonumber
\end{align}
\end{itemize}
The first $k$ inputs of $\mu^{\mathcal{M}}_{k|1}$ are referred to as \textbf{category inputs} and the $(k+1)$th input is referred to as the \textbf{module input}. We adopt the convention of writing module inputs in bold. 
\end{defn}

The relations \eqref{EQ: Module A_infty relations} imply that $\mu_{0|1}^\mathcal{M}$ is a differential and that the map $\mu^\mathcal{M}_{1|1}$ induces an operation on homology. 

Right modules over $\mathcal{A}$ are defined similarly, but with the module structure maps defining operations of $\mathcal{A}$ on the right.

\begin{defn}
A \textbf{right $\mathcal{A}$-module} $\mathcal{N}$ consists of the following:
\begin{itemize}
\item For every object $X\in \mathrm{Ob}(\mathcal{A})$, a graded $\Z_2$-vector space $\mathcal{N}(X)$.
\item For every $k\ge 0$ and every family of objects $X_0,\ldots,X_k\in \mathrm{Ob}(\mathcal{A})$, linear maps of degree $-1+k(1-N_{\mathcal{A}})$, 
\begin{equation}
\mu^{\mathcal{N}}_{1|k}: \mathcal{N}(X_k)\otimes\mathcal{A}(X_k,\ldots,X_0)\to\mathcal{N}(X_0),
\end{equation}
satisfying the following $A_\infty$-relations:
\begin{align}
&\sum\mu^\mathcal{N}_{1|i-1}(\mu^{\mathcal{N}}_{1|k-i+1}(\mathbf{w},x_k,\ldots,x_i),x_{i-1},\ldots,x_1)\\
&\quad+\sum\mu^\mathcal{N}_{1|k-j'+j}(\mathbf{w},x_k,\ldots,\mu^{\mathcal{A}}_{j'-j+1}(x_{j'},\ldots,x_{j}),\ldots,x_1)=0.\nonumber
\end{align}
\end{itemize}
\end{defn}

Likewise one can define bimodules by considering module structure maps which define operations by $A_\infty$-categories on both the left and right.

\begin{defn}
An \textbf{$\mathcal{A}\mathit{\mbox{--}}\mathcal{B}$ bimodule} $\mathcal{K}$ consists of the following:
\begin{itemize}
\item For every pair of objects $X\in \mathrm{Ob}(\mathcal{A})$ and $Y\in \mathrm{Ob}(\mathcal{B})$, a graded $\Z_2$-vector space $\mathcal{K}(X,Y)$.
\item For every $k,m\ge 0$ and every family of objects $X_0,\ldots,X_k\in \mathrm{Ob}(\mathcal{A})$ and $Y_0,\ldots,Y_m\in \mathrm{Ob}(\mathcal{B})$, linear maps of degree $-1+k(1-N_{\mathcal{A}})+m(1-N_{\mathcal{B}})$, 
\begin{equation}
\mu^{\mathcal{K}}_{k|1|m}:  \mathcal{A}(X_0,\ldots,X_k)\otimes\mathcal{K}(X_k,Y_m)\otimes\mathcal{B}(Y_m,\ldots,Y_0)\to\mathcal{K}(X_0,Y_0),
\end{equation}
satisfying the following $A_\infty$-relations:
\begin{equation}\label{EQ: bimodules A_infty relations}
\begin{aligned}
&\sum\mu^\mathcal{K}_{i-1|1|i'-1}(x_1,\ldots,\mu^{\mathcal{K}}_{k-i+1|1|m-i'+1}(x_i,\ldots,x_k,\mathbf{z},y_m,\ldots,y_{i'}),\ldots,y_1)\\
&\quad +\sum\mu^{\mathcal{K}}_{k-j'+j|1|m}(x_1,\ldots,\mu_{j'-j+1}^\mathcal{A}(x_{j},\ldots,x_{j'}),\ldots,x_k,\mathbf{z},y_m,\ldots,y_1)\\
&\quad+\sum\mu^\mathcal{K}_{k|1|m-j'+j}(x_1,\ldots,x_k,\mathbf{z},y_m,\ldots,\\
&\qquad\qquad\qquad\qquad\qquad\qquad\mu^{\mathcal{B}}_{j'-j+1}(y_{j'},\ldots,y_{j}),\ldots,y_1)=0.
\end{aligned}
\end{equation}
\end{itemize}
\end{defn}

\begin{defn}
A \textbf{pre-morphism} of left $\mathcal{A}$-modules $\nu:\mathcal{M}\to\mathcal{M}'$ of degree $|\nu|$ consists of maps 
\begin{equation}
\nu_{k|1}: \mathcal{A}(X_0,\ldots,X_k)\otimes \mathcal{M}(X_k)\to \mathcal{M}'(X_0)
\end{equation}
of degree $|\nu|+k(1-N_\mathcal{A})$ for all $k\ge 0$.
\end{defn}

\begin{defn}
A \textbf{pre-morphism} of right $\mathcal{A}$-modules $\eta:\mathcal{N}\to\mathcal{N}'$ of degree $|\eta|$ consists of maps
\begin{equation}
\eta_{1|k}: \mathcal{N}(X_k)\otimes\mathcal{A}(X_k,\ldots,X_0)\to \mathcal{N}'(X_0)
\end{equation}
of degree $|\eta|+k(1-N_\mathcal{A})$ for all $k\ge 0$.
\end{defn}

\begin{defn}
A \textbf{pre-morphism} of $\mathcal{A}\mathit{\mbox{--}}\mathcal{B}$ bimodules $\tau:\mathcal{K}\to\mathcal{K}'$ consists of maps
\begin{equation}
\tau_{k|1|m}: \mathcal{A}(X_0,\ldots,X_k)\otimes\mathcal{K}(X_k,Y_m)\otimes\mathcal{B}(Y_m,\ldots,Y_0)\to \mathcal{K}'(X_0,Y_0)
\end{equation}
of degree $|\tau|+k(1-N_\mathcal{A})+m(1-N_\mathcal{B})$ for all $k,m\ge 0$.
\end{defn}

\begin{remk}\label{REMK: Left and right modules as specialized bimodules}
As noted in \cite{GanatraThesis}, both left and right $\mathcal{A}$-modules are a special case of bimodules. To see this, we view the field $\Z_2$ as an $A_\infty$-category with a single object whose self-morphism space is $\Z_2$. We set $\mu^{\Z_2}_1=0$ and $\mu^{\Z_2}_k=0$ for $k\ge 3$, and we set $\mu_2^{\Z_2}$ to be the $\Z_2$ multiplication map. Then left $\mathcal{A}$-modules correspond to $\mathcal{A}\mathit{\mbox{--}}\Z_2$ bimodules $\mathcal{M}$ satisfying $\mu^{\mathcal{M}}_{k|1|m}=0$ for $m>0$ and right $\mathcal{A}$-modules correspond to $\Z_2\mathit{\mbox{--}}\mathcal{A}$ bimodules $\mathcal{N}$ satisfying $\mu^{\mathcal{N}}_{k|1|m}=0$ for $k>0$. For two $\mathcal{A}\mathit{\mbox{--}}\Z_2$ bimodules $\mathcal{M}$ and $\mathcal{M}'$ satisfying $\mu^{\mathcal{K}}_{k|1|m}=0$ and $\mu^{\mathcal{K}'}_{k|1|m}=0$  for $m>0$, the data of a pre-morphism of bimodules $\nu:\mathcal{M}\to\mathcal{M}'$ with $\nu_{k|1|m}=0$ for $m>0$ is equivalent to the data of a pre-morphism of the corresponding left $\mathcal{A}$-modules. Similarly, for two $\Z_2\mathit{\mbox{--}}\mathcal{A}$ bimodules $\mathcal{N}$ and $\mathcal{N}'$ satisfying $\mu^{\mathcal{N}}_{k|1|m}=0$ and $\mu^{\mathcal{N}'}_{k|1|m}=0$  for $k>0$, the data of a pre-morphism of bimodules $\eta:\mathcal{N}\to\mathcal{N}'$ with $\nu_{k|1|m}=0$ for $k>0$ is equivalent to the data of a pre-morphism of the corresponding right $\mathcal{A}$-modules.
\end{remk}

\begin{remk}\label{REMK: Module suspensions}
Note that an $\mathcal{A}\mathit{\mbox{--}}\mathcal{B}$ bimodule $\mathcal{M}$ can also be viewed as an $\mathcal{A}[j]\mathit{\mbox{--}}\mathcal{B}[j']$ bimodule, where $\mathcal{A}[j]$ and $\mathcal{B}[j']$ are the $j$-fold and $j'$-fold suspensions of $\mathcal{A}$ and $\mathcal{B}$ respectively. Additionally, there is a notion of the suspension of the bimodule $\mathcal{M}$. The \textbf{$j$-fold suspension of $\mathcal{M}$} is the $\mathcal{A}\mathit{\mbox{--}}\mathcal{B}$ bimodule $\mathcal{M}[j]$ with $\mathcal{M}[j](X,Y)_p=\mathcal{M}(X,Y)_{p+j}$, and with module structure maps induced by the $\mu^\mathcal{M}_{k|1|m}$. A pre-morphism of $\mathcal{A}\mathit{\mbox{--}}\mathcal{B}$ bimodules $\nu:\mathcal{M}\to\mathcal{M}'$ of degree $|\nu|$ induces a pre-morphism of $\mathcal{A}\mathit{\mbox{--}}\mathcal{B}$ bimodules $\nu[j]:\mathcal{M}[j]\to\mathcal{M}'[j]$ of degree $|\nu|$. Taking the $j$-fold suspension of $\mathcal{A}\mathit{\mbox{--}}\mathcal{B}$ bimodules and pre-morphisms of $\mathcal{A}\mathit{\mbox{--}}\mathcal{B}$ bimodules defines the \textbf{$j$-fold suspension functor}
\begin{equation}
\bm{\Sigma}^j:\mathcal{A}\mathit{\mbox{--}mod\mbox{--}}\mathcal{B}\to \mathcal{A}\mathit{\mbox{--}mod\mbox{--}}\mathcal{B}.
\end{equation}

Similar definitions and statements hold for left and right modules.
\end{remk}

The class of $\mathcal{A}\mathit{\mbox{--}}\mathcal{B}$ bimodules forms the objects of a strictly unital $A_\infty$-category $\mathcal{A}\mathit{\mbox{--}mod\mbox{--}}\mathcal{B}$ of degree zero. The morphism spaces are pre-morphisms between bimodules. The operation $\mu^{\mathcal{A}\mathit{\mbox{--}mod\mbox{--}}\mathcal{B}}_1$ is defined on $\nu\in \mathcal{A}\mathit{\mbox{--}mod\mbox{--}}\mathcal{B}(\mathcal{K},\mathcal{K}')$ by 
\begin{align}
&(\mu^{\mathcal{A}\mathit{\mbox{--}mod\mbox{--}}\mathcal{B}}_1(\nu))_{k|1|m}(x_1,\ldots,x_k,\mathbf{z},y_m,\ldots,y_1)\nonumber\\
&\quad=\sum\mu^{\mathcal{K}'}_{i-1|1|i'-1}(x_1,\ldots,\nu_{k-i+1|1|m-i'+1}(x_i,\ldots,x_k,\mathbf{z},y_m,\ldots,y_{i'}), \ldots, y_1) \nonumber\\
&\quad+ \sum\nu_{i-1|1|i'-1}(x_1,\ldots,\mu^{\mathcal{K}}_{k-i+1|1|m-i'+1}(x_i,\ldots,x_k,\mathbf{z},y_m,\ldots,y_{i'}),\ldots,y_1) \nonumber\\
& \quad + \sum \nu_{k-j'+j|1|m}(x_1,\ldots,\mu^{\mathcal{A}}_{j'-j+1}(x_{j},\ldots,x_{j'}),\ldots,x_k,\mathbf{z},y_m,\ldots,y_1)\nonumber\\
& \quad + \sum \nu_{k|1|m-j'+j}(x_1,\ldots,x_k,\mathbf{z},y_m,\ldots,\mu^{\mathcal{B}}_{j'-j+1}(y_{j'},\ldots,y_{j}),\ldots,y_1).
\end{align}

The operation $\mu^{\mathcal{A}\mathit{\mbox{--}mod\mbox{--}}\mathcal{B}}_2$ is defined on $\nu\in \mathcal{A}\mathit{\mbox{--}mod\mbox{--}}\mathcal{B}(\mathcal{K},\mathcal{K}')$ and $\nu'\in \mathcal{A}\mathit{\mbox{--}mod\mbox{--}}\mathcal{B}(\mathcal{K}',\mathcal{K}'')$ by
\begin{equation}
\begin{aligned}
&(\mu^{\mathcal{A}\mathit{\mbox{--}mod\mbox{--}}\mathcal{B}}_2(\nu,\nu'))_{k|1|m}(x_1,\ldots,x_k,\mathbf{z},y_m,\ldots,y_1)\\
&\quad=\sum\nu'_{i-1|1|i'-1}(x_1,\ldots,\nu_{k-i+1|1|m-i'+1}(x_{i},\ldots,x_k,\mathbf{z},y_m,\ldots,y_{i'}), \ldots, y_1). 
\end{aligned}
\end{equation}

The operations $\mu^{\mathcal{A}\mathit{\mbox{--}mod\mbox{--}}\mathcal{B}}_k$ for $k\ge 3$ are all zero, meaning $\mathcal{A}\mathit{\mbox{--}mod\mbox{--}}\mathcal{B}$ is in fact a dg-category. Bimodule pre-morphisms $\nu$ satisfying $\mu^{\mathcal{A}\mathit{\mbox{--}mod\mbox{--}}\mathcal{B}}_1(\nu)=0$ are called \textbf{bimodule morphisms}.

The units in the category $\mathcal{A}\mathit{\mbox{--}mod\mbox{--}}\mathcal{B}$ are the module endomorphisms $e_\mathcal{M}$ given by
\begin{equation}
\begin{aligned}
&(e_{\mathcal{M}})_{0|1|0}:\mathcal{M}(X)\to \mathcal{M}(X),\;(e_{\mathcal{M}})_{0|1|0}=\mathrm{id}_{\mathcal{M}(X)},\\
&(e_{\mathcal{M}})_{k|1|m}=0 \text{ for }(k,m)\ne (0,0).
\end{aligned}
\end{equation}
Being strictly unital, $\mathcal{A}\mathit{\mbox{--}mod\mbox{--}}\mathcal{B}$ is also homologically unital and it therefore makes sense to speak of both isomorphisms and quasi-isomorphisms in $\mathcal{A}\mathit{\mbox{--}mod\mbox{--}}\mathcal{B}$.

The class of left $\mathcal{A}$-modules also forms the objects of a strictly unital $A_\infty$-category of degree zero which we denote $\mathcal{A}\mathit{\mbox{--}mod}$. This is similarly a dg-category. The operations $\mu^{\mathcal{A}\mathit{\mbox{--}mod}}_k$ are induced by the operations $\mu^{\mathcal{A}\mathit{\mbox{--}mod\mbox{--}}\Z_2}_k$ in $\mathcal{A}\mathit{\mbox{--}mod\mbox{--}}\Z_2$ by viewing left $\mathcal{A}$-modules as $\mathcal{A}\mathit{\mbox{--}}\Z_2$ bimodules (see Remark \ref{REMK: Left and right modules as specialized bimodules}). Likewise, right $\mathcal{A}$-modules are the objects of a strictly unital degree-zero $A_\infty$-category $\mathit{mod\mbox{--}}\mathcal{A}$ which is also a dg-category. The operations $\mu^{\mathit{mod\mbox{--}}\mathcal{A}}_k$ are induced by the operations $\mu^{\Z_2\mathit{\mbox{--}mod\mbox{--}}\mathcal{A}}_k$ in $\Z_2\mathit{\mbox{--}mod\mbox{--}}\mathcal{A}$.

\begin{defn} Let $\mathcal{A}$ be a homologically unital $A_\infty$-category and for any $X\in \mathrm{Ob}(\mathcal{A})$ denote by $e_X\in \mathcal{A}(X,X)$ a representative of the homology unit for $X$. A left $\mathcal{A}$-module $\mathcal{M}$ is \textbf{homologically unital} if for all $X\in \mathrm{Ob}(\mathcal{A})$,
\begin{equation}
\mu^{\mathcal{M}}_{1|1}([e_X],[\mathbf{w}])=[\mathbf{w}],
\end{equation}
for any $\mathbf{w}\in \mathcal{M}(X)$ with $\mu_{0|1}(\mathbf{w})=0$. Similarly, a right $\mathcal{A}$-module $\mathcal{N}$ is homologically unital if the homology-level multiplication $\mu^{\mathcal{N}}_{1|1}$ with the homology units in $\mathcal{A}$ is the identity. An $\mathcal{A}\mathit{\mbox{--}}\mathcal{B}$ bimodule $\mathcal{K}$ is homologically unital if the homology-level multiplication $\mu^{\mathcal{K}}_{1|1|0}$ with homology units in $\mathcal{A}$ is the identity and the homology-level multiplication $\mu^{\mathcal{K}}_{0|1|1}$ with homology units in $\mathcal{B}$ is also the identity.
\end{defn}

Although we will mainly use the preceding description of module and bimodule categories, there is also an interpretation of these categories as categories of functors into the dg-category of chain complexes. Denote by $\mathit{Ch}$ the dg-category of $\Z$-graded chain complexes over $\Z_2$ where the differential lowers degree by one. We view $\mathit{Ch}$ as an $A_\infty$-category of degree zero by setting $\mu^{\mathit{Ch}}_k=0$ for $k\ge 3$. 

The $A_\infty$-category of left $\mathcal{A}$-modules is given by 
\begin{equation}
\mathcal{A}\mathit{\mbox{--}\mathit{mod}} = \mathit{fun}(\mathcal{A},\mathit{Ch}^{opp})^{opp},
\end{equation}
the $A_\infty$-category of right $\mathcal{A}$-modules is given by 
\begin{equation}
\mathit{mod\mbox{--}}\mathcal{A}=(\mathit{fun}(\mathcal{A}^{\mathit{opp}},\mathit{Ch}^{\mathit{opp}}))^{\mathit{opp}},
\end{equation}
and the $A_\infty$-category of $\mathcal{A}\mathit{\mbox{--}}\mathcal{B}$ bimodules is given by
\begin{equation}
\mathcal{A}\mathit{\mbox{--}mod\mbox{--}}\mathcal{B}=\mathit{fun}(\mathcal{A}\times\mathcal{B}^{\mathit{opp}},\mathit{Ch}^{\mathit{opp}})^{\mathit{opp}}.
\end{equation}
By spelling out the definitions of these functor categories, it is not hard to see that these descriptions of modules and module categories are equivalent to the previous ones. 

There is also another interpretation of categories of $A_\infty$-bimodules as functor categories. This results from the existence of isomorphisms of dg-categories 
\begin{equation}\label{EQ: Category isos between functors into left modules and bimodules}
\begin{aligned}
&\Phi^l:\mathit{fun}(\mathcal{B},\mathcal{A}\mathit{\mbox{--}mod})\xrightarrow{\cong} \mathcal{A}\mathit{\mbox{--}mod\mbox{--}}\mathcal{B},\\
&\Phi^r:\mathit{fun}(\mathcal{B},(\mathit{mod}\mbox{--}\mathcal{A})^{\mathit{opp}})\xrightarrow{\cong} (\mathcal{B}\mathit{\mbox{--}mod\mbox{--}}\mathcal{A})^{\mathit{opp}}.
\end{aligned}
\end{equation}
The isomorphism $\Phi^l$ is defined as follows. For $\mathbf{F}\in \mathrm{Ob}(\mathit{fun}(\mathcal{B},\mathcal{A}\mathit{\mbox{--}mod}))$, $\Phi^l(\mathbf{F})$ is the $\mathcal{A}\mathit{\mbox{--}}\mathcal{B}$ bimodule specified by the following data:
\begin{itemize}
\item For $X\in\mathrm{Ob}(\mathcal{A})$ and $Y\in\mathrm{Ob}(\mathcal{B})$, $\Phi^l(\mathbf{F})(X,Y)=\mathbf{F}(Y)(X)$.
\item $\mu^{\Phi^l(\mathbf{F})}_{k|1|m}:\mathcal{A}(X_0,\ldots,X_k)\otimes\Phi^l(\mathbf{F})(X_k,Y_m)\otimes\mathcal{B}(Y_m,\ldots,Y_0)\to \Phi^l(\mathbf{F})(X_0,Y_0),$
\begin{equation}
\mu^{\Phi^l(\mathbf{F})}_{k|1|m}(x_1,\ldots,x_k,\mathbf{w},y_m,\ldots,y_1)=(\mathbf{F}_m(y_m,\ldots,y_1))_{k|1}(x_1,\ldots,x_k,\mathbf{w}).
\end{equation}
\end{itemize}
The components $\Phi^l_k$ for $k>1$ are all zero, and the component $\Phi^l_1$ is defined as follows. For any pair of functors $\mathbf{F}_0,\mathbf{F}_1\in \mathrm{Ob}(\mathit{fun}(\mathcal{B},\mathcal{A}\mathit{\mbox{--}mod}))$, $\Phi^l_1$ is the map
\begin{equation}
\Phi^l_1:\mathit{fun}(\mathcal{B},\mathcal{A}\mathit{\mbox{--}mod})(\mathbf{F}_0,\mathbf{F}_1)\to\mathcal{A}\mathit{\mbox{--}mod\mbox{--}}\mathcal{B}(\Phi^l(\mathbf{F}_0),\Phi^l(\mathbf{F}_1))
\end{equation}
which is defined on a pre-natural transformation $T\in \mathit{fun}(\mathcal{B},\mathcal{A}\mathit{\mbox{--}mod})(\mathbf{F}_0,\mathbf{F}_1)$ to be the pre-morphism of $\mathcal{A}\mathit{\mbox{--}}\mathcal{B}$ bimodules from $\Phi^l(\mathbf{F}_0)$ to $\Phi^l(\mathbf{F}_1)$ given by 
\begin{equation}
(\Phi^l_1(T))_{k|1|m}(x_1,\ldots,x_k,\mathbf{w},y_m,\ldots,y_1)=(T_m(y_m,\ldots,y_1))_{k|1}(x_1,\ldots,x_k,\mathbf{w}).
\end{equation}
The isomorphism $\Phi^r$ is defined similarly.

We will make use of the following distinguished bimodule associated to an $A_\infty$-category.

\begin{defn}
The \textbf{diagonal bimodule} $\mathcal{A}_\Delta$ is the $\mathcal{A}\mathit{\mbox{--}}\mathcal{A}$ bimodule defined on objects by $\mathcal{A}_\Delta(X,X')=\mathcal{A}(X,X')$ and whose bimodule structure maps are given by $\mu^{\mathcal{A}_\Delta}_{k|1|m}=\mu^\mathcal{A}_{k+m+1}$.
\end{defn}

The following definition appears in \cite{Tra08} for the case of $A_\infty$-algebras.
\begin{defn}
The \textbf{linear dual} of an $\mathcal{A}\mathit{\mbox{--}}\mathcal{B}$ bimodule $\mathcal{M}$ is the $\mathcal{B}\mathit{\mbox{--}}\mathcal{A}$ bimodule $\mathcal{M}^\vee$ defined by 
\begin{align}
&(\mathcal{M}^\vee(Y,X))_q=hom(\mathcal{M}(X,Y)_{-q},\Z_2)\\
&\mu^{\mathcal{M}^\vee}_{k|1|m}:\mathcal{B}(Y_0,\ldots,Y_k)\otimes\mathcal{M}^\vee(Y_k,X_m)\otimes \mathcal{A}(X_m,\ldots,X_0)\to  \mathcal{M}^\vee(Y_0,X_0),\nonumber\\
&\langle\mu_{k|1|m}^{\mathcal{M}^\vee}(y_1,\ldots, y_k,\mathbf{f},x_m,\ldots,x_1),\mathbf{w}\rangle=\langle \mathbf{f},\mu_{m|1|k}^{\mathcal{M}}(x_m,\ldots, x_1,\mathbf{w},y_1,\ldots, y_k)\rangle.\nonumber
\end{align}
One sees easily that the bimodule structure maps $\mu_{k|1|m}^{\mathcal{M}^\vee}$ have the correct degrees and satisfy the $A_\infty$ relations \eqref{EQ: bimodules A_infty relations}.

The linear dual of the diagonal bimodule $\mathcal{A}_\Delta$ is the \textbf{Serre bimodule}, denoted $\mathcal{A}^\vee$.
\end{defn}

By viewing left $\mathcal{A}$-modules as $\mathcal{A}\mathit{\mbox{--}}\Z_2$ bimodules and right $\mathcal{A}$-modules as $\Z_2\mathit{\mbox{--}}\mathcal{A}$ bimodules (see\ Remark \ref{REMK: Left and right modules as specialized bimodules}), we see that the linear dual of a left $\mathcal{A}$-module is a right $\mathcal{A}$-module and vice versa. 

\begin{remk}
There are two other notions of duals for $A_\infty$-modules worth mentioning, although we will not make use of them. The first notion, that of a \emph{module dual}, applies to left or right $A_\infty$-modules. For a fixed $\mathcal{A}\mathit{\mbox{--}}\mathcal{B}$ bimodule $\mathcal{K}$, one can assign to a left $\mathcal{A}$-module $\mathcal{M}$ a right $\mathcal{B}$-module $\mathrm{hom}_{\mathcal{A}\mathit{\mbox{--}mod}}(\mathcal{M},\mathcal{K})$, and to a right $\mathcal{B}$-module $\mathcal{N}$ a left $\mathcal{A}$-module $\mathrm{hom}_{\mathit{mod\mbox{--}}\mathcal{B}}(\mathcal{N},\mathcal{K})$ (see \cite[\S 2.13]{GanatraThesis}). The right $\mathcal{A}$-module $\mathrm{hom}_{\mathcal{A}\mathit{\mbox{--}mod}}(\mathcal{M},\mathcal{A}_\Delta)$ and the left $\mathcal{B}$-module $\mathrm{hom}_{\mathit{mod\mbox{--}}\mathcal{B}}(\mathcal{N},\mathcal{B}_\Delta)$ are called the module dual of $\mathcal{M}$ and of $\mathcal{N}$ respectively. The second notion, that of a \emph{bimodule dual} applies to an $\mathcal{A}\mathit{\mbox{--}}\mathcal{A}$ bimodule. We refer to \cite[Definition 2.40]{GanatraThesis} for the definition of this dual.
\end{remk}

Taking the linear dual of bimodules forms the object map of a dg-functor.
\begin{defn}
The  \textbf{dualization functor} is the functor of dg-categories
\begin{equation}
\mathbf{D}: \mathcal{A}\mathit{\mbox{--}mod\mbox{--}}\mathcal{B}\to (\mathcal{B}\mathit{\mbox{--}mod\mbox{--}}\mathcal{A})^{opp}.
\end{equation}
which on objects is given by the linear dual, and on morphisms is defined by
\begin{align}
&\mathbf{D}:\mathcal{A}\mathit{\mbox{--}mod\mbox{--}}\mathcal{B}(\mathcal{M},\mathcal{N})\to \mathcal{B}\mathit{\mbox{--}mod\mbox{--}}\mathcal{A}(\mathcal{N}^\vee,\mathcal{M}^\vee),\; \nu\mapsto \nu^\vee,\\
&\nu^\vee_{m|1|k}:\mathcal{B}(Y_0,\ldots,Y_m)\otimes\mathcal{N}^\vee(Y_m,X_k)\otimes\mathcal{A}(X_k,\ldots,X_0)\to \mathcal{M}^\vee(Y_0,X_0),\nonumber\\
&\langle\nu^\vee_{m|1|k}(y_1,\ldots,y_m,\mathbf{g},x_k,\ldots,x_1),\mathbf{w}\rangle
=\langle\mathbf{g},\nu_{k|1|m}(x_k,\ldots,x_1,\mathbf{w},y_1,\ldots,y_m)\rangle.\nonumber
\end{align}
\end{defn}

Viewing left $\mathcal{A}$-modules as $\mathcal{A}\mathit{\mbox{--}}\Z_2$ bimodules and right $\mathcal{A}$-modules as $\Z_2\mathit{\mbox{--}}\mathcal{A}$ bimodules we obtain dualization functors
\begin{equation}
\mathcal{A}\mathit{\mbox{--}mod}\to (\mathit{mod\mbox{--}}\mathcal{A})^{opp}, \quad\mathit{mod\mbox{--}}\mathcal{A}\to (\mathcal{A}\mathit{\mbox{--}mod})^{opp}.
\end{equation} 

The following proposition says that the dualization functor is compatible with the isomorphisms \eqref{EQ: Category isos between functors into left modules and bimodules}. It is proved by direct computation.
\begin{prop}\label{PROP: Compatibility dualization and Phi}
The following diagram of dg-functors commutes:
$$\begindc{\commdiag}[25] 
\obj(0,20)[A]{$\mathit{fun}(\mathcal{B},(\mathit{mod\mbox{--}}\mathcal{A})^{opp})$}
\obj(45,20)[C]{$(\mathcal{B}\mathit{\mbox{--}mod}\mbox{--}\mathcal{A})^{\mathit{opp}}$} 
\obj(0,0)[A']{$\mathit{fun}(\mathcal{B},\mathcal{A}\mathit{\mbox{--}mod})$}
\obj(45,0)[C']{$\mathcal{A}\mathit{\mbox{--}mod}\mbox{--}\mathcal{B}$.} 
\mor{A}{A'}{$\mathbf{L}_{\mathbf{D}^{opp}}$} 
\mor{A}{C}{$\Phi^r$}
\mor{A'}{C'}{$\Phi^l$} 
\mor{C}{C'}{$\mathbf{D}^{\mathit{opp}}$}
\enddc$$
\end{prop}

Another operation that can be performed on bimodules is the pullback along a pair of functors.
\begin{defn}
Let $\mathbf{F}_0:\mathcal{A}\to \mathcal{A}'$ and $\mathbf{F}_1:\mathcal{B}\to \mathcal{B'}$ be functors of $A_\infty$-categories. For an $\mathcal{A}'\mathit{\mbox{--}}\mathcal{B}'$ module $\mathcal{M}$, the \textbf{pullback} of $\mathcal{M}$ along $\mathbf{F}_0$ and $\mathbf{F}_1$ is the $\mathcal{A}\mathit{\mbox{--}}\mathcal{B}$ bimodule 
$(\mathbf{F}_0\otimes\mathbf{F}_1)^*(\mathcal{M})$ specified by
\begin{align}
&(\mathbf{F}_0\otimes\mathbf{F}_1)^*\mathcal{M}(X,Y)=\mathcal{M}(\mathbf{F}_0(X),\mathbf{F}_1(Y)),\nonumber\\
&\mu^{(\mathbf{F}_0\otimes\mathbf{F}_1)^*\mathcal{M}}_{k|1|m}(x_1,\ldots,x_k,\mathbf{w},y_m,\ldots,y_1)=\nonumber\\
&\qquad \sum_{s,s'} \sum_{\substack{p_1,\ldots,p_s,\\q_1,\ldots,q_{s'}}}\mu^\mathcal{M}_{s|1|s'}((\mathbf{F}_0)_{p_1}(x_1,\ldots,x_{p_1}),\ldots,(\mathbf{F}_0)_{p_s}(x_{k-p_s+1},\ldots,x_k),\mathbf{w},\nonumber\\
&\qquad\qquad\qquad\qquad(\mathbf{F}_1)_{q_{s'}}(y_m,\ldots,y_{m-q_{s'}+1}),\ldots,(\mathbf{F}_1)_{q_1}(y_{q_1},\ldots,y_1)).
\end{align}
The pullback of modules along $\mathbf{F}_0$ and $\mathbf{F}_1$ is the object map of the \textbf{pullback functor}
\begin{equation}
(\mathbf{F}_0\otimes\mathbf{F}_1)^*:\mathcal{A}'\mathit{\mbox{--}mod\mbox{--}}\mathcal{B}'\to \mathcal{A}\mathit{\mbox{--}mod\mbox{--}}\mathcal{B}.
\end{equation}
The functor $(\mathbf{F}_0\otimes\mathbf{F}_1)^*$ is defined on morphisms in $\mathcal{A}'\mathit{\mbox{--}mod\mbox{--}}\mathcal{B}'$ by
\begin{equation}
\begin{aligned}
&((\mathbf{F}_0\otimes\mathbf{F}_1)^*(\nu))_{k|1|m}(x_1,\ldots,x_k,\mathbf{w},y_m,\ldots,y_1)=\\
&\qquad \sum_{s,s'} \sum_{\substack{p_1,\ldots,p_s,\\q_1,\ldots,q_{s'}}} \nu_{s|1|s'}((\mathbf{F}_0)_{p_1}(x_1,\ldots,x_{p_1}),\ldots,(\mathbf{F}_0)_{p_s}(x_{k-p_s+1},\ldots,x_k),\mathbf{w},\\
&\qquad\qquad\qquad(\mathbf{F}_1)_{q_{s'}}(y_m,\ldots,y_{m-q_{s'}+1}),\ldots,(\mathbf{F}_1)_{q_1}(y_{q_1},\ldots,y_1)).
\end{aligned}
\end{equation}
We will use the shorthand $\mathbf{F}^*=(\mathbf{F}\otimes\mathbf{F})^*$.
\end{defn}
Using the description of left $\mathcal{A}$-modules as $\mathcal{A}\mathit{\mbox{--}}\Z_2$ bimodules and right $\mathcal{B}$-modules as $\Z_2\mathit{\mbox{--}}\mathcal{B}$ bimodules, we obtain pullback functors
\begin{equation}
\mathbf{F}_0^*:=(\mathbf{F}_0\otimes\mathbf{Id}_{\Z_2})^*:\mathcal{A}'\mathit{\mbox{--}mod}\to \mathcal{A}\mathit{\mbox{--}mod}, \; \mathbf{F}_1^*:=(\mathbf{Id}_{\Z_2}\otimes \mathbf{F}_1)^*:\mathit{mod\mbox{--}}\mathcal{B}'\to \mathit{mod\mbox{--}}\mathcal{B}.
\end{equation}

The next proposition follows from a direct check using the definitions of the pullback and dualization functors. 

\begin{prop}\label{PROP: Compatibility pullback and dualization}
Pullback functors are compatible with dualization in the following sense. For functors $\mathbf{F}_0:\mathcal{A}\to\mathcal{A}'$ and $\mathbf{F}_1:\mathcal{B}\to \mathcal{B}'$, the following diagram commutes
$$\begindc{\commdiag}[25] 
\obj(0,20)[A]{$\mathcal{A}'\mathit{\mbox{--}mod}\mbox{--}\mathcal{B}'$}
\obj(65,20)[C]{$\mathcal{A}\mathit{\mbox{--}mod}\mbox{--}\mathcal{B}$} 
\obj(0,0)[A']{$(\mathcal{B}'\mathit{\mbox{--}mod}\mbox{--}\mathcal{A}')^{\mathit{opp}}$}
\obj(65,0)[C']{$(\mathcal{B}\mathit{\mbox{--}mod}\mbox{--}\mathcal{A})^{\mathit{opp}}$} 
\mor{A}{A'}{$\mathbf{D}$} 
\mor{A}{C}{$(\mathbf{F}_0\otimes\mathbf{F}_1)^*$}
\mor{A'}{C'}{$((\mathbf{F}_1\otimes\mathbf{F}_0)^*)^{\mathit{opp}}$} 
\mor{C}{C'}{$\mathbf{D}$}
\enddc$$
\end{prop}

There is also an interpretation of the pullback of bimodules via the isomorphisms \eqref{EQ: Category isos between functors into left modules and bimodules}. Define the following compositions
\begin{align}\label{EQ: Def of G_F0,F1}
&\mathbf{G}^l_{\mathbf{F}_0,\mathbf{F}_1}:\mathit{fun}(\mathcal{B}',\mathcal{A}'\mathit{\mbox{--}mod})\xrightarrow{\mathbf{R}_{\mathbf{F}_1}} \mathit{fun}(\mathcal{B},\mathcal{A}'\mathit{\mbox{--}mod})\xrightarrow{\mathbf{L}_{\mathbf{F}_0^*}}\mathit{fun}(\mathcal{B},\mathcal{A}\mathit{\mbox{--}mod}),\nonumber\\
&\mathbf{G}^r_{\mathbf{F}_0,\mathbf{F}_1}:\mathit{fun}(\mathcal{A}',(\mathit{mod\mbox{--}}\mathcal{B}')^{opp})\xrightarrow{\mathbf{R}_{\mathbf{F}_0}} \mathit{fun}(\mathcal{A},(\mathit{mod\mbox{--}}\mathcal{B}')^{opp})\nonumber\\
&\qquad\qquad\qquad\qquad\qquad\qquad\qquad\xrightarrow{\mathbf{L}_{(\mathbf{F}_1^*)^{opp}}}\mathit{fun}(\mathcal{A},(mod\mathit{\mbox{--}}\mathcal{B})^{opp}).
\end{align}
We will use the shorthand $\mathbf{G}^l_{\mathbf{F}}=\mathbf{G}^l_{\mathbf{F},\mathbf{F}}$
and $\mathbf{G}^r_{\mathbf{F}}=\mathbf{G}^r_{\mathbf{F},\mathbf{F}}$.

By the next proposition, which again is proved by direct computation, the functors $\mathbf{G}^l_{\mathbf{F}_0,\mathbf{F}_1}$ and $\mathbf{G}^r_{\mathbf{F}_0,\mathbf{F}_1}$ correspond to the pullback functor $(\mathbf{F}_0\otimes\mathbf{F}_1)^*$ under the isomorphisms \eqref{EQ: Category isos between functors into left modules and bimodules}.
\begin{prop}\label{PROP:Compatibility G and pullback}
The following diagrams commute
$$\begindc{\commdiag}[25] 
\obj(0,20)[A]{$\mathit{fun}(\mathcal{A}',(\mathit{mod\mbox{--}}\mathcal{B}')^{opp})$}
\obj(50,20)[C]{$(\mathcal{A}'\mathit{\mbox{--}mod}\mbox{--}\mathcal{B}')^{\mathit{opp}}$} 
\obj(0,0)[A']{$\mathit{fun}(\mathcal{A},(\mathit{mod\mbox{--}}\mathcal{B})^{\mathit{opp}})$}
\obj(50,0)[C']{$(\mathcal{A}\mathit{\mbox{--}mod}\mbox{--}\mathcal{B})^{\mathit{opp}}$} 
\mor{A}{A'}{$\mathbf{G}^r_{\mathbf{F}_0,\mathbf{F}_1}$} 
\mor{A}{C}{$\Phi^r$}
\mor{A'}{C'}{$\Phi^r$} 
\mor{C}{C'}{$((\mathbf{F}_0\otimes\mathbf{F}_1)^*)^{\mathit{opp}}$}
\enddc$$
$$\begindc{\commdiag}[25] 
\obj(0,20)[A]{$\mathit{fun}(\mathcal{B}',\mathcal{A}'\mathit{\mbox{--}mod})$}
\obj(50,20)[C]{$\mathcal{A}'\mathit{\mbox{--}mod}\mbox{--}\mathcal{B}'$} 
\obj(0,0)[A']{$\mathit{fun}(\mathcal{B},\mathcal{A}\mathit{\mbox{--}mod})$}
\obj(50,0)[C']{$\mathcal{A}\mathit{\mbox{--}mod}\mbox{--}\mathcal{B}$} 
\mor{A}{A'}{$\mathbf{G}^l_{\mathbf{F}_0,\mathbf{F}_1}$} 
\mor{A}{C}{$\Phi^l$}
\mor{A'}{C'}{$\Phi^l$} 
\mor{C}{C'}{$(\mathbf{F}_0\otimes\mathbf{F}_1)^*$}
\enddc$$
\end{prop}

Given this correspondence of the functors $\mathbf{G}^l_{\mathbf{F}_0,\mathbf{F}_1}$ and $\mathbf{G}^r_{\mathbf{F}_0,\mathbf{F}_1}$ with the pullback functor $(\mathbf{F}_0\otimes\mathbf{F}_1)^*$, we can interpret the compatibility between pullback and dualization of Proposition \ref{PROP: Compatibility pullback and dualization} as a statement about the functors $\mathbf{G}^l_{\mathbf{F}_0,\mathbf{F}_1}$ and $\mathbf{G}^r_{\mathbf{F}_0,\mathbf{F}_1}$. This is the content of the following proposition which is a direct consequence of Propositions \ref{PROP: Compatibility dualization and Phi}, \ref{PROP: Compatibility pullback and dualization} and \ref{PROP:Compatibility G and pullback}.

\begin{prop}\label{PROP: Compatibility G and dualization} 
The following diagram commutes
$$\begindc{\commdiag}[25] 
\obj(0,20)[A]{$\mathit{fun}(\mathcal{A}',(\mathit{mod\mbox{--}}\mathcal{B}')^{opp})$}
\obj(55,20)[C]{$\mathit{fun}(\mathcal{A}',\mathcal{B}'\mathit{\mbox{--}mod})$} 
\obj(0,0)[A']{$\mathit{fun}(\mathcal{A},(\mathit{mod\mbox{--}}\mathcal{B})^{opp})$}
\obj(55,0)[C']{$\mathit{fun}(\mathcal{A},\mathcal{B}\mathit{\mbox{--}mod})$} 
\mor{A}{A'}{$\mathbf{G}^r_{\mathbf{F}_0,\mathbf{F}_1}$} 
\mor{A}{C}{$\mathbf{L}_{\mathbf{D}^{\mathit{opp}}}$}
\mor{A'}{C'}{$\mathbf{L}_{\mathbf{D}^{\mathit{opp}}}$} 
\mor{C}{C'}{$\mathbf{G}^l_{\mathbf{F}_1,\mathbf{F}_0}$}
\enddc$$
\end{prop}

\subsection{Yoneda functors and abstract Serre functors}\label{SUBSECTION: Yoneda functors}

Any $A_\infty$-category possesses canonical functors into left and right modules over the category, the left and right Yoneda functors.
\begin{defn}
The \textbf{left Yoneda functor} for the $A_\infty$-category $\mathcal{A}$ is the $A_\infty$-functor $\mathbf{Y}_{\mathcal{A}}^l:\mathcal{A}\to \mathcal{A}\mathit{\mbox{--}mod}$ which is defined as follows. On objects, the functor $\mathbf{Y}_{\mathcal{A}}^l$ is given by setting $\mathbf{Y}_{\mathcal{A}}^l(X)$ to be the left $\mathcal{A}$-module $\mathcal{M}_X^l$ defined by
\begin{align}
&\mathcal{M}_X^l(Y)=\mathcal{A}(Y,X),\nonumber\\
&\mu_{k|1}^{\mathcal{M}_X^l}:\mathcal{A}(Y_0,\ldots,Y_k)\otimes \mathcal{M}_X^l(Y_k)\to \mathcal{M}_X^l(Y_0),\\
&\mu_{k|1}^{\mathcal{M}_X^l}(y_1,\ldots,y_k,\mathbf{w})=\mu_{k+1}^\mathcal{A}(y_1,\ldots,y_k,\mathbf{w}).\nonumber
\end{align} 
The higher maps of the functor $\mathbf{Y}_{\mathcal{A}}^l$ are given by
\begin{align}
(\mathbf{Y}^l_{\mathcal{A}})_m:\mathcal{A}(X_m,\ldots,X_0)&\to \mathcal{A}\mathit{\mbox{--}mod}(\mathcal{M}^l_{X_m},\mathcal{M}^l_{X_0}),\nonumber\\
(x_m,\ldots,x_1)&\mapsto \nu_{(x_m,\ldots,x_1)}, 
\end{align}
where $\nu_{(x_m,\ldots,x_1)}$ is the module pre-morphism defined by
\begin{align}
(\nu_{(x_m,\ldots,x_1)})_{k|1}:\mathcal{A}(Y_0,\ldots,Y_k)\otimes\mathcal{M}^l_{X_m}(Y_k)&\to \mathcal{M}^l_{X_0}(Y_0),\nonumber\\
(y_1,\ldots,y_k,\mathbf{w})&\mapsto \mu^\mathcal{A}_{k+m+1}(y_1,\ldots,y_k,\mathbf{w},x_m,\ldots,x_1).
\end{align}
\end{defn}

\begin{defn}
The \textbf{right Yoneda functor} for the $A_\infty$-category $\mathcal{A}$ is the $A_\infty$-functor $\mathbf{Y}_{\mathcal{A}}^r:\mathcal{A}\to(\mathit{mod\mbox{--}}\mathcal{A})^{\mathit{opp}}$ which is defined as follows. On objects, the functor $\mathbf{Y}_{\mathcal{A}}^r$ is given by setting $\mathbf{Y}_{\mathcal{A}}^r(X)$ to be the right $\mathcal{A}$-module $\mathcal{M}_X^r$ defined by
\begin{align}
&\mathcal{M}_X^r(Y)=\mathcal{A}(X,Y),\nonumber\\
&\mu_{1|k}^{\mathcal{M}_X^r}:\mathcal{M}_X^r(Y_k)\otimes\mathcal{A}(Y_k,\ldots,Y_0)\to \mathcal{M}_X^r(Y_0),\\
&\mu_{1|k}^{\mathcal{M}_X^r}(\mathbf{z},y_k,\ldots,y_1)=\mu_{k+1}^\mathcal{A}(\mathbf{z},y_k,\ldots,y_1).\nonumber
\end{align} 
The higher maps of the functor $\mathbf{Y}_{\mathcal{A}}^r$ are given by
\begin{align}
(\mathbf{Y}^r_{\mathcal{A}})_m:\mathcal{A}(X_0,\ldots,X_m)&\to \mathit{mod\mbox{--}}\mathcal{A}(\mathcal{M}^r_{X_m},\mathcal{M}^r_{X_0}),\nonumber\\
(x_1,\ldots,x_m)&\mapsto \tau_{(x_1,\ldots,x_m)}, 
\end{align}
where $\tau_{(x_1,\ldots,x_m)}$ is the module pre-morphism defined by
\begin{align}
(\tau_{(x_1,\ldots,x_m)})_{1|k}:\mathcal{M}^r_{X_m}(Y_k)\otimes\mathcal{A}(Y_k,\ldots,Y_0)&\to \mathcal{M}^r_{X_0}(Y_0),\nonumber\\
(\mathbf{z},y_k,\ldots,y_1)&\mapsto \mu^\mathcal{A}_{m+k+1}(x_1,\ldots,x_m,\mathbf{z},y_k,\ldots,y_1).
\end{align}
\end{defn}

When the $A_\infty$-category $\mathcal{A}$ is homologically unital, the functors $\mathbf{Y}_{\mathcal{A}}^l$ and $\mathbf{Y}_{\mathcal{A}}^r$ are full and faithful on the level of homological categories \cite[Corollary 2.13]{Sei}. Hence these functors are often called the left and right Yoneda \emph{embeddings}.

There are also dual versions of the left and right Yoneda functors, as was mentioned in \cite[p.~189]{Sei}. 

\begin{defn}The \textbf{left abstract Serre functor} is the $A_\infty$ functor $(\mathbf{Y}_{\mathcal{A}}^\vee)^l:\mathcal{A}\to \mathcal{A}\mathit{\mbox{--}mod}$ which is defined as follows. On objects, the functor $(\mathbf{Y}_{\mathcal{A}}^\vee)^l$ is given by setting $(\mathbf{Y}_{\mathcal{A}}^\vee)^l(X)$ to be the left $\mathcal{A}$-module $\mathcal{N}_X^l$ defined by
\begin{align}
&(\mathcal{N}_X^l(Y))_q=(\mathcal{A}(X,Y))_q^\vee:=\mathit{hom}(\mathcal{A}(X,Y)_{-q},\Z_2),\nonumber\\
&\mu_{k|1}^{\mathcal{N}_X^l}:\mathcal{A}(Y_0,\ldots,Y_k)\otimes \mathcal{N}_X^l(Y_k)\to \mathcal{N}_X^l(Y_0),\\
&\langle\mu_{k|1}^{\mathcal{N}_X^l}(y_1,\ldots,y_k,\mathbf{f}),\mathbf{w}\rangle=\langle \mathbf{f},\mu_{k+1}^\mathcal{A}(\mathbf{w},y_1,\ldots,y_k)\rangle.\nonumber
\end{align} 
The higher maps of the functor $(\mathbf{Y}_{\mathcal{A}}^\vee)^l$ are given by
\begin{align}
(\mathbf{Y}_{\mathcal{A}}^\vee)^l_m:\mathcal{A}(X_m,\ldots,X_0)&\to \mathcal{A}\mathit{\mbox{--}mod}(\mathcal{N}^l_{X_m},\mathcal{N}^l_{X_0}),\nonumber\\
(x_m,\ldots,x_1)&\mapsto \rho_{(x_m,\ldots,x_1)}, 
\end{align}
where $\rho_{(x_m,\ldots,x_1)}$ is the module pre-morphism defined by
\begin{equation}
\begin{aligned}
&(\rho_{(x_m,\ldots,x_1)})_{k|1}:\mathcal{A}(Y_0,\ldots,Y_k)\otimes\mathcal{N}^l_{X_m}(Y_k)\to \mathcal{N}^l_{X_0}(Y_0),\\
&\langle(\rho_{(x_m,\ldots,x_1)})_{k|1}(y_1,\ldots,y_k,\mathbf{f}),\mathbf{w}\rangle\\
&\qquad\qquad=\langle\mathbf{f}, \mu^\mathcal{A}_{k+m+1}(x_m,\ldots,x_1, \mathbf{w}, y_1,\ldots,y_k)\rangle.
\end{aligned}
\end{equation}
\end{defn}

The left abstract Serre functor satisfies $(\mathbf{Y}_{\mathcal{A}}^\vee)^l=\mathbf{D}^{\mathit{opp}}\circ \mathbf{Y}_{\mathcal{A}}^r$. Although we will only make use of the left abstract Serre functor, one can similarly define the right abstract Serre functor $(\mathbf{Y}_{\mathcal{A}}^\vee)^r:\mathcal{A}\to (\mathit{mod\mbox{--}}\mathcal{A})^{\mathit{opp}}$. This functor satisfies $(\mathbf{Y}_{\mathcal{A}}^\vee)^r=\mathbf{D}\circ \mathbf{Y}_{\mathcal{A}}^l$.

Under the isomorphism $\Phi^l$ of \eqref{EQ: Category isos between functors into left modules and bimodules}, the left Yoneda functor corresponds to the diagonal bimodule $\mathcal{A}_\Delta$, and similarly the right Yoneda functor corresponds to $\mathcal{A}_\Delta$ under $\Phi^r$. The left and right abstract Serre functors correspond to the Serre bimodule $\mathcal{A}^\vee$ under $\Phi^l$ and $\Phi^r$ respectively.

\subsection{Hochschild homology and cohomology}\label{SUBSECTION: Hochschild homology and cohomology}

This section follows \cite{GanatraThesis} and \cite{SheridanFano}, although this material is presented there using cohomological conventions for $A_\infty$-categories as opposed to the homological conventions used here. 

\begin{defn}
For an $\mathcal{A}\mathit{\mbox{--}}\mathcal{A}$ bimodule $\mathcal{M}$, we define the \textbf{Hochschild cochain complex} of $\mathcal{A}$ with coefficients in $\mathcal{M}$ to be
\begin{equation}
CC^\bullet(\mathcal{A},\mathcal{M})=\prod_{X_0,\ldots,X_k}\mathrm{Hom}_{N_{\mathcal{A}}+k(1-N_{\mathcal{A}})+\bullet}(\mathcal{A}(X_0,\ldots,X_k),\mathcal{M}(X_0,X_k)).
\end{equation}
In other words, a Hochschild cochain $g\in CC^p(\mathcal{A},\mathcal{M})$ assigns to every $k\ge 0$ and every family of objects $X_0,\ldots, X_k$, a map
$$g_k:\mathcal{A}(X_0,\ldots,X_k)\to \mathcal{M}(X_0,X_k)$$
of degree $N_\mathcal{A}+k(1-N_{\mathcal{A}})+p$. The differential is defined by 
\begin{equation}\label{EQ: Hochschild complex}
\begin{aligned}
&(\partial g)_k(x_1,\ldots,x_k)=\sum_{j,j'}\mu_{j-1|1|k-j'}^{\mathcal{M}}(x_1,\ldots,g_{j'-j+1}(x_j,\ldots,x_{j'}),\ldots,x_k)\\
&\qquad\qquad\qquad\qquad\qquad+\sum_{j,j'}g_{k-j'+j}(x_1,\ldots,\mu_{j'-j+1}^{\mathcal{A}}(x_{j},\ldots,x_{j'}),\ldots,x_k).
\end{aligned}
\end{equation}
The homology of this complex is the \textbf{Hochschild cohomology} of $\mathcal{A}$ with coefficients in $\mathcal{M}$, denoted $HH^\bullet(\mathcal{A},\mathcal{M})$. The Hochschild cohomology of $\mathcal{A}$, $HH^\bullet(\mathcal{A})$, is defined to be the homology of $CC^\bullet(\mathcal{A}):=CC^\bullet(\mathcal{A},\mathcal{A}_\Delta)$.
\end{defn}

We emphasize that with our grading conventions, the differential on $CC^\bullet(\mathcal{A},\mathcal{M})$ \emph{lowers} degree.

There is an alternative description of Hochschild cohomology in terms of morphisms of bimodules.
\begin{defn} The \textbf{two-pointed Hochschild cochain complex} of $\mathcal{A}$ with coefficients in $\mathcal{M}$ is defined to be
\begin{equation}
_2CC^\bullet(\mathcal{A},\mathcal{M})=(\mathcal{\mathcal{A}\mathit{\mbox{--}mod\mbox{--}}\mathcal{A}}(\mathcal{A}_{\Delta},\mathcal{M}))_{\bullet}.
\end{equation}
We set $_2CC^\bullet(\mathcal{A})={}_2CC^\bullet(\mathcal{A},\mathcal{A}_\Delta)$.
\end{defn}

There is a chain map
\begin{equation}\label{EQ: Quasi-iso from CC_bullet to 2CC_bullet}
S:CC^\bullet(\mathcal{A},\mathcal{M})\to {}_2CC^\bullet(\mathcal{A},\mathcal{M}),
\end{equation}
which is a quasi-isomorphism when $\mathcal{M}$ is homologically unital \cite[Proposition 2.5]{GanatraThesis}. This is a version for homological categories of the map (2.200) in \cite{GanatraThesis}. It is defined by 
\begin{align}
&(S(g))_{k|1|m}(x_1,\ldots,x_k,\mathbf{w},x'_m,\ldots,x'_1)\nonumber\\
&\quad\quad =\sum_{}\mu^{\mathcal{M}}_{i-1|1|k-j+m+1}(x_1,\ldots,g_{j-i+1}(x_i,\ldots,x_j),\ldots,x_k,\mathbf{w},x'_m,\ldots,x'_1).
\end{align}

\begin{remk}
When $\mathcal{A}$ is homologically unital, the Hochschild cohomology $HH^\bullet(\mathcal{A})$ has the structure of a unital associative $\Z$-graded $\Z_2$-algebra, where the product is given on cochains in $_2CC^\bullet(\mathcal{A})$ by composition of module morphisms. Moreover, for a homologically unital $\mathcal{A}\mathit{\mbox{--}}\mathcal{A}$ bimodule $\mathcal{M}$, the Hochschild cohomology $HH^\bullet(\mathcal{A},\mathcal{M})$ has the structure of an $HH^\bullet(\mathcal{A})$-module. Again the $HH^\bullet(\mathcal{A})$ action is described on the level of two-pointed cochains by composition of module morphisms. We refer the reader to \cite{SheridanFano} and \cite{GanatraThesis} for more details, including a description of the product on ordinary Hochschild cochains. 
\end{remk}

\begin{defn} The \textbf{Hochschild chain complex} of $\mathcal{A}$ with coefficients in $\mathcal{M}$ is defined by
\begin{equation}
CC_\bullet(\mathcal{A},\mathcal{M})=\bigoplus_{X_0,\ldots,X_k}(\mathcal{A}(X_0,\ldots,X_k)\otimes\mathcal{M}(X_k,X_0))_{N_{\mathcal{A}}-k(1-N_{\mathcal{A}})+\bullet}
\end{equation}
with differential given by
\begin{equation}
\begin{aligned}
&\partial(x_1\otimes\cdots\otimes x_k\otimes\mathbf{w})\\
&\qquad=\sum x_1\otimes\cdots\otimes\mu^{\mathcal{A}}_{j-i+1}(x_i,\ldots,x_j)\otimes\cdots\otimes x_k\otimes \mathbf{w}\\
&\qquad\qquad +\sum x_{j+1}\otimes\cdots\otimes x_{i-1}\otimes\mu^{\mathcal{M}}_{k-i+1|1|j}(x_i,\ldots,x_k,\mathbf{w},x_1,\ldots,x_j).
\end{aligned}
\end{equation}
The homology of this complex is the \textbf{Hochschild homology} of $\mathcal{A}$ with coefficients in $\mathcal{M}$, denoted $HH_\bullet(\mathcal{A},\mathcal{M})$. The Hochschild homology of $\mathcal{A}$ is defined to be the homology of $CC_\bullet(\mathcal{A}):=CC_\bullet(\mathcal{A},\mathcal{A}_\Delta)$, and denoted $HH_\bullet(\mathcal{A})$.
\end{defn}

For any object $X$ of $\mathcal{A}$, there is an obvious inclusion of chain complexes
\begin{equation}\label{EQ: inclusion into Hochschild cochain complex}
\mathcal{M}(X,X)_{N_\mathcal{A}+\bullet}\hookrightarrow CC_\bullet(\mathcal{A},\mathcal{M}).
\end{equation}

A morphism of $\mathcal{A}\mathit{\mbox{--}}\mathcal{A}$ bimodules $\nu:\mathcal{M}\to \mathcal{M}'$ induces a map on Hochschild chain complexes
\begin{equation}\label{EQ: Hochschild cohomology functoriality 1}
\nu_*:CC_\bullet(\mathcal{A},\mathcal{M})\to CC_{\bullet+|\nu|}(\mathcal{A},\mathcal{M}').
\end{equation} 
This is defined by
\begin{equation}
\begin{aligned}
&\nu_*(x_1\otimes\cdots \otimes x_k\otimes\mathbf{w})\\
&\qquad=\sum_{i,j}x_{j+1}\otimes\cdots\otimes x_{k-i}\otimes \nu_{i|1|j}(x_{k-i+1},\ldots,x_k,\mathbf{w},x_1,\ldots,x_j).
\end{aligned}
\end{equation}
If $\nu$ is a quasi-isomorphism of bimodules, then $\nu_*$ is also a quasi-isomorphism.

The Hochschild chain complex is also functorial in the following sense. Fix a $\mathcal{B}\mathit{\mbox{--}}\mathcal{B}$ bimodule $\mathcal{N}$ and a functor $\mathbf{F}:\mathcal{A}\to\mathcal{B}$. There is an induced map of chain complexes
\begin{equation}\label{EQ: Hochschild cohomology functoriality 2}
\mathbf{F}_*:CC_\bullet(\mathcal{A},\mathbf{F}^*\mathcal{N})\to CC_{\bullet +N_{\mathcal{A}}-N_{\mathcal{B}}} (\mathcal{B},\mathcal{N}),
\end{equation}
which is given by
\begin{equation}
\begin{aligned}
&\mathbf{F}_*(x_1\otimes \cdots \otimes x_k\otimes\mathbf{z})=\\
&\qquad\sum_s\sum_{p_1,\ldots,p_s}\mathbf{F}_{p_1}(x_1,\ldots,x_{p_1})\otimes \cdots\otimes \mathbf{F}_{p_s}(x_{k-p_s+1},\ldots,x_k)\otimes \mathbf{z}.
\end{aligned}
\end{equation}
This is compatible with composition of functors, and if $\mathbf{F}$ is a quasi-isomorphism (i.e.\ the induced map on homological categories is an isomorphism), then so is $\mathbf{F}_*$.

As with Hochschild cohomology, there is an alternative description of Hochschild homology in terms of a `two-pointed complex'. For this, we require the following definition.
\begin{defn}
Given an $\mathcal{A}\mathit{\mbox{--}}\mathcal{B}$ bimodule $\mathcal{N}$ and a $\mathcal{B}\mathit{\mbox{--}}\mathcal{A}$ bimodule $\mathcal{N}'$, we can define their \textbf{bimodule tensor product}, $\mathcal{N}\otimes_{\mathcal{A}\mathit{\mbox{--}}\mathcal{B}}\mathcal{N}'$.
This is the chain complex given by
\begin{equation}
\mathcal{N}\otimes_{\mathcal{A}\mathit{\mbox{--}}\mathcal{B}}\mathcal{N}'=\bigoplus_{\substack{X_0,\ldots,X_k,\\Y_0,\ldots,Y_m}}\mathcal{N}(X_k,Y_m)\otimes\mathcal{B}(Y_m,\ldots,Y_0)\otimes\mathcal{N}'(Y_0,X_0)\otimes\mathcal{A}(X_0,\ldots,X_k),
\end{equation}
with differential
\begin{equation}
\begin{aligned}
&\partial(\mathbf{w}\otimes y_m\otimes \cdots\otimes y_1\otimes \mathbf{z}\otimes x_1\otimes \cdots\otimes x_k)\\
&\qquad=\sum_{i,j}\mu^{\mathcal{N}}_{k-i+1|1|m-j+1}(x_i,\ldots,x_k,\mathbf{w},y_m,\ldots,y_j)\otimes y_{j-1}\otimes\cdots\\
&\qquad\qquad\qquad\qquad\cdots\otimes y_1\otimes \mathbf{z}\otimes x_1\otimes\cdots\otimes x_{i-1}\\
&\qquad +\sum_{i,i'}\mathbf{w}\otimes y_m\otimes \cdots\otimes y_1\otimes \mathbf{z}\otimes x_1\otimes \cdots\otimes \mu^{\mathcal{A}}_{i'-i+1}(x_i,\ldots,x_{i'})\otimes\cdots\otimes x_k\\
&\qquad +\sum_{i,j}\mathbf{w}\otimes y_m\otimes \cdots\otimes\mu^{\mathcal{N'}}_{j|1|i}(y_j,\ldots,y_1,\mathbf{z},x_1,\ldots,x_i)\otimes \cdots\otimes x_k\\
&\qquad +\sum_{j,j'}\mathbf{w}\otimes y_m\otimes \cdots\otimes\mu^{\mathcal{B}}_{j'-j+1}(y_{j'},\ldots,y_j)\otimes\cdots\otimes y_1\otimes \mathbf{z}\otimes x_1\otimes \cdots\otimes x_k.
\end{aligned}
\end{equation}
The chains are graded as follows:
\begin{equation}
\begin{aligned}
&|\mathbf{w}\otimes y_m\otimes \cdots\otimes y_1\otimes \mathbf{z}\otimes x_1\otimes \cdots\otimes x_k|_{\mathcal{N}\otimes_{\mathcal{A}\mathit{\mbox{--}}\mathcal{B}}\mathcal{N}'}\\
&\quad\quad =\sum_{i=1}^k|x_i|+\sum_{j=1}^m |y_j| +|\mathbf{w}|+|\mathbf{z}|+k(1-N_{\mathcal{A}})+m(1-N_{\mathcal{B}})-N_{\mathcal{A}}-N_{\mathcal{B}}.
\end{aligned}
\end{equation}
\end{defn}

\begin{defn} The \textbf{two-pointed Hochschild chain complex} of $\mathcal{A}$ with coefficients in $\mathcal{M}$ is the bimodule tensor product of the diagonal bimodule $\mathcal{A}_\Delta$ with $\mathcal{M}$,
\begin{equation}
_2CC_\bullet(\mathcal{A},\mathcal{M})=\mathcal{A}_\Delta\otimes_{\mathcal{A}\mathit{\mbox{--}}\mathcal{A}}\mathcal{M}.
\end{equation}
We set $_2CC_\bullet(\mathcal{A})={}_2CC_\bullet(\mathcal{A},\mathcal{A}_\Delta)$.
\end{defn}
There is a chain map
\begin{equation}\label{EQ: iso from two-pointed cochain complex to one-pointed cochain complex}
T:{}_2CC_\bullet(\mathcal{A},\mathcal{M})\to CC_\bullet(\mathcal{A},\mathcal{M}),
\end{equation}
given by
\begin{equation}
\begin{aligned}
&T(\mathbf{w}\otimes y_m\otimes\cdots\otimes y_1\otimes \mathbf{z}\otimes x_1\otimes\cdots \otimes x_k)\\
&\qquad=\sum x_{j+1}\otimes\cdots\otimes x_{k-i}\otimes\mu^{\mathcal{M}}_{m+i+1|1|j}(x_{k-i+1},\ldots,x_k,\mathbf{w},y_m,\ldots\\
&\qquad\qquad\qquad\qquad\qquad\ldots,y_1,\mathbf{z},x_1,\ldots,x_j)
\end{aligned}
\end{equation}
This is a homological version of the chain map (2.196) in \cite{GanatraThesis}. When $\mathcal{M}$ is homologically unital, $T$ is a quasi-isomorphism \cite{GanatraThesis}.

\begin{remk}
The Hochschild homology $HH_\bullet(\mathcal{A},\mathcal{M})$ is a module over $HH^\bullet(\mathcal{A})$. Assuming $\mathcal{M}$ is homologically unital, the chain-level action can be described in terms of two-pointed complexes. It results from functoriality of $\mathcal{A}_\Delta\otimes_{\mathcal{A}\mathit{\mbox{--}}\mathcal{A}}\mathcal{M}$ with respect to bimodule endomorphisms of $\mathcal{A}_\Delta$. There is also a description of the action on the ordinary Hochschild complexes via the \emph{cap product} (see \cite{GanatraThesis,SheridanFano} for details).
\end{remk}

The following lemma appears as Lemma 6.2 in \cite{ganatra2015mirror} for $\mathcal{M}=\mathcal{A}_\Delta$. 

\begin{lem}\label{LEM: This iso 2CCn(A,Mvee) cong 2CC_n(A,M)vee}
There is a naturally defined isomorphism 
\begin{equation}
\Gamma:{}_2CC_\bullet(\mathcal{A},\mathcal{M})^\vee\xrightarrow{\cong} {}_2CC^{\bullet-2N_{\mathcal{A}}}(\mathcal{A},\mathcal{M}^\vee).
\end{equation}
\end{lem}

\begin{proof}
For $\alpha\in {}_2CC_\bullet(\mathcal{A},\mathcal{M})^\vee$, $\Gamma(\alpha)$ is the module pre-morphism from $\mathcal{A}_\Delta$ to $\mathcal{M}^\vee$ given by
\begin{equation}
\begin{aligned}
&\Gamma(\alpha)_{k|1|m}:\mathcal{A}(X_0,\ldots,X_k)\otimes \mathcal{A}_\Delta(X_k,Y_m)\otimes\mathcal{A}(Y_m,\ldots,Y_0)\to \mathcal{M}^\vee(X_0,Y_0),\\
&\langle\Gamma(\alpha)_{k|1|m}(x_1,\ldots,x_k,\mathbf{w},y_m,\ldots,y_1),\mathbf{z}\rangle\\
&\qquad\qquad\qquad\qquad = \langle \alpha, \mathbf{w}\otimes y_m\otimes\cdots\otimes y_1\otimes \mathbf{z}\otimes x_1\otimes\cdots\otimes x_k \rangle.
 \end{aligned}
 \end{equation}
One checks easily that this is an isomorphism of chain complexes.
\end{proof}

\begin{remk}\label{REMK: Ungraded Hochschild complexes} In subsequent sections, we will also consider Hochschild chain and cochain complexes for ungraded $A_\infty$-categories. Although these are ungraded chain complexes, we will continue to denote them $CC_\bullet(\mathcal{A},\mathcal{M})$ and $CC^\bullet(\mathcal{A},\mathcal{M})$ to distinguish the Hochschild chain complex from the Hochschild cochain complex (and similarly we continue to use $_2CC_\bullet(\mathcal{A},\mathcal{M})$ and $_2CC^\bullet(\mathcal{A},\mathcal{M})$ for the two-pointed complexes).
\end{remk}

\section{Geometric preliminaries}

In this section, we review the basic geometric ingredients that will be needed in subsequent sections. Specifically, we define Lagrangian cobordisms as well as the particular variants of Floer homology and the Fukaya category we use.

Throughout, $(M,\omega)$ will be a compact symplectic manifold of dimension $2n$. In general, Lagrangian submanifolds of $M$ will be taken to be closed.

\subsection{Lagrangian cobordisms}

We describe here Lagrangian cobordisms as defined by Biran and Cornea in \cite{BC13, BC14}. Let $\omega_0=dx \wedge dy$ be the standard symplectic form on $\R^2$. We set $\widetilde{M}=\R^2\times M$ and equip $\widetilde{M}$ with the product symplectic form $\widetilde{\omega}=\omega_0\oplus \omega$. Denote by $\pi:\R^2\times M\to\R^2$ the projection onto $\R^2$.

\begin{defn}
Let $(L_i)_{1\le i\le r}$ and $(L'_j)_{1\le j\le s}$ be two families of Lagrangian submanifolds in $M$. A Lagrangian cobordism $V$ from $(L_i)_{1\le i\le r}$ to $(L'_j)_{1\le j\le s}$, denoted $V:(L_1,\ldots,L_r)\to (L'_1,\ldots,L'_s)$, is a cobordism $V$ from $\sqcup_{i=1}^rL_i$ to $\sqcup_{j=1}^s L'_j$ together with a Lagrangian embedding $V\hookrightarrow ([0,1]\times\R)\times M$ satisfying the following condition. There exists $\epsilon>0$ such that
\begin{itemize}
 \item $V\cap \left(((1-\epsilon,1]\times \R)\times M\right)=\bigsqcup_{i=1}^r ((1-\epsilon,1]\times \{i\})\times L_i,$
 \item $V\cap \left(([0,\epsilon)\times \R)\times M\right)=\bigsqcup_{j=1}^s ([0,\epsilon)\times \{j\})\times L'_j.$
\end{itemize}
\end{defn}

A Lagrangian cobordism $V\subset [0,1]\times M$ can also be viewed as a submanifold of $\widetilde{M}$ by extending its ends trivially.
It is often useful to visualize Lagrangian cobordisms by projecting them onto $\R^2$, as in Figure \ref{fig:diagram5}. 

Some examples of Lagrangian cobordisms include the suspension of a Lagrangian in $M$ with respect to a Hamiltonian isotopy, and the trace of Lagrangian surgery performed on two transversely intersecting Lagrangians in $M$ \cite{BC13}.
\begin{figure}
\centering
\def\svgwidth{100mm}
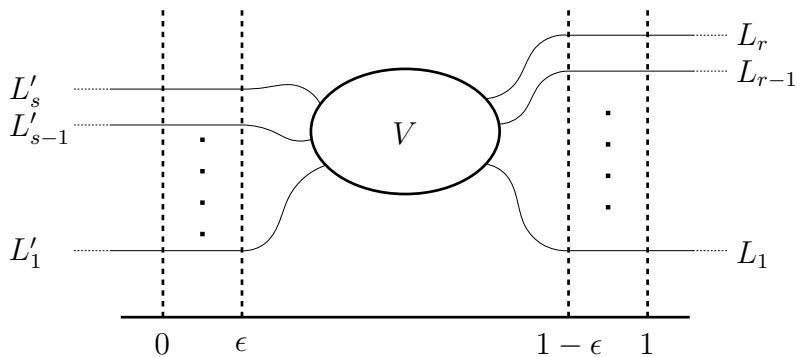
\caption{\label{fig:diagram5}A Lagrangian cobordism $V$ projected onto $\R^2$.}
\end{figure}

\subsection{Lagrangian Floer homology and Poincar\'{e} duality}\label{SUBSECTION: Lagrangian Floer homology and Poincare duality}

We review the construction of Lagrangian Floer homology, introduced by Floer \cite{Fl88} in his work on the Arnold conjecture and extended by Oh \cite{Oh2, Oh95Ad} to the monotone case. We present the variant used by Biran and Cornea in \cite{BC14}. 

Let $L\subset M$ be a Lagrangian submanifold. Recall that there are two homomorphisms
$$\mu_L:\pi_2(M,L)\to\Z,\quad \omega_L:\pi_2(M,L)\to \R.$$ 
The homomorphism $\mu_L$ is the Maslov index and the homomorphism $\omega_L$ is given on the class of a map $u:(D,\partial D)\to (M,L)$ by $\int_D u^*\omega$. Here $D$ is the unit disc in $\C$. The \textbf{minimal Maslov number} of $L$, denoted $N_L$, is the positive generator of $\mathrm{Image}(\mu_L)$. In the case where $\mathrm{Image}(\mu_L)=\{0\}$, we set $N_L=\infty$. The Lagrangian $L$ is said to be \textbf{monotone} if there is a constant $\tau\ge 0$ such that $\omega_L=\tau\mu_L$, and in addition, $N_L\ge 2$. The constant $\tau$, which is unique if $\mu_L\not\equiv 0$, is known as the \textbf{monotonicity constant} of $L$. 

To every monotone connected $L$, we can associate an invariant $d_L$ which is the number in $\Z_2$ of (unparametrized) $J$-holomorphic discs of Maslov index two passing through a generic point of $L$, for a generic compatible almost complex structure. This Gromov-Witten-type invariant was introduced in \cite{Oh2}, and appears in various constructions in the monotone setting, for example in \cite{Oh95Ad, Ch97, BC1}. By the same transversality and compactness arguments which are used to define Gromov-Witten invariants (see \cite{Gr85, McDS12}), the number $d_L$ is well-defined, and in particular independent of the choice of generic point in $L$ and generic compatible almost complex structure. The constant $d_L$ is also invariant under exact deformations of $L$.

Let $L$ and $N$ be two monotone connected Lagrangian submanifolds in $M$ which have the same monotonicity constant and satisfy $d_{L}=d_{N}$. Moreover, assume that the maps $\pi_1(L)\rightarrow \pi_1(M)$ and $\pi_1(N)\rightarrow \pi_1(M)$ induced by the inclusions of $L$ and $N$ in $M$ are trivial. 
Let $J=\{J_t\}_{t\in [0,1]}$ be a path of almost complex structures on $M$ compatible with $\omega$. Fix a time-dependent Hamiltonian function $H:M\times [0,1]\to\R$ whose associated Hamiltonian isotopy $\phi^H_t$ is such that $N$ and $\phi^H_1(L)$ intersect transversely.
The pair $(H,J)$ is called a \textbf{Floer datum} for the pair of Lagrangians $(L,N)$.
Define
$$\mathcal{O}(H)=\{\gamma \in C^\infty([0,1],M)|\;\gamma(0)\in L,\;\gamma(1)\in N,\;\gamma(t)=\phi_t^H(\gamma(0))\}.$$
The set $\mathcal{O}(H)$ consists of those paths from $L$ to $N$ which are time-1 orbits of $\phi^H_t$. 

Assume now that the Floer datum $(H,J)$ is chosen generically. The Floer complex associated to the pair $(L,N)$ and the Floer datum $(H,J)$ is defined by
$$CF(L,N;H,J)=\Z_2\langle \mathcal{O}(H)\rangle,$$ 
with differential given by counting elements of the moduli spaces of inhomogeneous pseudoholomorphic strips we define now. 

Let $\gamma_-$ and $\gamma_+$ be elements of $\mathcal{O}(H)$, and let $u\in C^\infty(\R\times [0,1], M)$ be a smooth map satisfying the conditions
\begin{equation}\label{EQ: Boundary conditions}
\begin{cases}
\lim\limits_{s\to -\infty}u(s,\cdot)=\gamma_-,\quad \lim\limits_{s\to +\infty}u(s,\cdot)=\gamma_+,\\
u(s,0)\in L,\quad u(s,1)\in N\quad \text{for all }s\in \R. 
\end{cases}
\end{equation}
We consider those maps $u$ which are solutions of Floer's equation,
\begin{equation}\label{EQ: Floer equation}
\partial_su + J_t(u)\partial_t u+\nabla H(t,u)=0,
\end{equation}
and which have finite energy,
\begin{equation}\label{EQ: Finite energy}
E(u):=\int_{\R\times [0,1]}u^*\omega<\infty.
\end{equation}
The moduli space of Floer orbits connecting $\gamma_-$ to $\gamma_+$ is defined to be
$$\widehat{\mathcal{M}}(\gamma_-,\gamma_+;H,J)=\{u\in C^\infty(\R\times [0,1],M)|\; u\text{ satisfies }\eqref{EQ: Boundary conditions}-\eqref{EQ: Finite energy}\}.$$ 
A generically chosen Floer datum is regular in the sense that these moduli spaces are regular for all $\gamma_-,\gamma_+\in \mathcal{O}(H)$. We let $\mathcal{M}(\gamma_-,\gamma_+;H,J)$ denote the quotient of $\widehat{\mathcal{M}}(\gamma_-,\gamma_+;H,J)$ by the $\R$-action by reparametrizations in the $s$ coordinate. The Floer differential counts the elements of the zero-dimensional component $\mathcal{M}(\gamma_-,\gamma_+;H,J)^0$ of $\mathcal{M}(\gamma_-,\gamma_+;H,J)$,
\begin{equation}
\begin{aligned}
&\partial:CF(L,N;H,J)\to CF(L,N;H,J), \\
&\partial(\gamma_-)=\sum_{\gamma_+\in\mathcal{O}(H)} \#_{\Z_2}\mathcal{M}(\gamma_-,\gamma_+;H,J)^0\gamma_+.
\end{aligned}
\end{equation}
Standard compactness and gluing arguments, together with the monotonicity condition, imply that this is a well-defined differential.

The Floer complex $CF(L,N;H,J)$ depends on the choice of regular Floer datum $(H,J)$, but different choices result in quasi-isomorphic chain complexes. Moreover, the isomorphism induced on Floer homology is natural. 

\begin{remk}
This construction, like all of the constructions in this thesis, also works when $M$ is not necessarily compact as long as the Gromov compactness theorem is still valid. This is the case for symplectic manifolds which are tame at infinity, for instance, and these provide a broad class of examples (see \cite[\S 4.2]{AudLaf94}).
\end{remk}

\subsubsection{Poincar\'{e} duality for Floer complexes}

Let $L$ and $N$ be as above and fix a regular Floer datum $(H_{L,N},J_{L,N})$ for the pair $(L,N)$. Define a time-dependent Hamiltonian $\bar{H}_{L,N}$ and time-dependent almost complex structure $\bar{J}_{L,N}$ on $M$ by
\begin{align}
&\bar{H}_{L,N}:[0,1]\times \R\to M, &&\bar{H}_{L,N}(t,p)=-H_{L,N}(1-t,p), \\
&\bar{J}_{L,N}=\{(\bar{J}_{L,N})_t\}_{t\in[0,1]}, &&(\bar{J}_{L,N})_t=(J_{L,N})_{1-t}. 
\end{align}
Then $(\bar{H}_{L,N},\bar{J}_{L,N})$ is a regular Floer datum for the pair $(N,L)$ and so the Floer complex $CF(N,L;\bar{H}_{L,N},\bar{J}_{L,N})$
is defined. For a path $\gamma$ in $M$, denote by $\bar{\gamma}$ the time-reversed path, $\bar{\gamma}(t)=\gamma(1-t)$. There is an isomorphism of chain complexes
\begin{equation}\label{EQ: Poincare duality quasi-iso for CF - trivial part}
\begin{aligned}
P:CF(L,N;H_{L,N},J_{L,N})&\xrightarrow{\cong} CF(N,L;\bar{H}_{L,N},\bar{J}_{L,N})^\vee,\\
\gamma &\mapsto \bar{\gamma}.
\end{aligned}
\end{equation}
Here $CF(N,L;\bar{H}_{L,N},\bar{J}_{L,N})^\vee$ denotes the dual of the complex $CF(N,L;\bar{H}_{L,N},\bar{J}_{L,N})$. In \eqref{EQ: Poincare duality quasi-iso for CF - trivial part}, we have implicitly used the vector space isomorphism between $CF(N,L;\bar{H}_{L,N},\bar{J}_{L,N})$ and its dual determined by the basis $\mathcal{O}(\bar{H}_{L,N})$.
The fact that the map \eqref{EQ: Poincare duality quasi-iso for CF - trivial part} is an isomorphism follows from the identification of moduli spaces
\begin{align}
\widehat{\mathcal{M}}(\gamma_-,\gamma_+;H_{L,N},J_{L,N})&\cong
\widehat{\mathcal{M}}(\bar{\gamma}_+,\bar{\gamma}_-;\bar{H}_{L,N},\bar{J}_{L,N}),\;
\nonumber\\
u&\mapsto \bar{u},
\end{align}
where $\bar{u}(s,t):=u(-s,1-t)$.

Given another regular Floer datum $(H_{N,L},J_{N,L})$ for the pair $(N,L)$, there is also a comparison map 
\begin{equation}\label{EQ: Floer homology comparison map}
\Phi_{\mathbf{H},\mathbf{J}}:CF(N,L;H_{N,L},J_{N,L})\to CF(N,L;\bar{H}_{L,N},\bar{J}_{L,N})
\end{equation}
which depends on a choice of regular homotopy $(\mathbf{H},\mathbf{J})$ from $(H_{N,L},J_{N,L})$ to $(\bar{H}_{L,N},\bar{J}_{L,N})$,
\begin{equation}\label{EQ: Homotopy from (H,J)_N,L to bar(H,J)_L,N}
\begin{aligned}
&\mathbf{H}:\R \to C^\infty([0,1],M),&\mathbf{H}|_{(-\infty,-\lambda]}&\equiv H_{N,L},&\mathbf{H}|_{[\lambda,\infty)}&\equiv\bar{H}_{L,N},\\
&\mathbf{J}:\R \to C^\infty([0,1],\mathcal{J}(M)),&\mathbf{J}|_{(-\infty,-\lambda]}&\equiv J_{N,L},&\mathbf{J}|_{[\lambda,\infty)}&\equiv\bar{J}_{L,N},
\end{aligned}
\end{equation}
where $\lambda>0$. Here $\mathcal{J}(M)$ is the space of compatible almost complex structures on $M$.
The map $\Phi_{\mathbf{H},\mathbf{J}}$ is defined by counting finite-energy solutions of
\begin{equation}
\partial_su + \mathbf{J}_t(u)\partial_t u+\nabla \mathbf{H}(t,u)=0
\end{equation}
which satisfy the conditions
\begin{equation}
\begin{cases}
\lim\limits_{s\to -\infty}u(s,\cdot)=\gamma_-,\quad \lim\limits_{s\to +\infty}u(s,\cdot)=\gamma_+,\\
u(s,0)\in N,\quad u(s,1)\in L\quad \text{for all }s\in \R. 
\end{cases}
\end{equation}
for $\gamma_-\in\mathcal{O}(H_{N,L})$ and $\gamma_+\in\mathcal{O}(\bar{H}_{L,N})$. The map induced on homology by $\Phi_{\mathbf{H},\mathbf{J}}$ is a canonically defined isomorphism.

Equivalently, by fixing a biholomorphic map between the unit disc in $\C$ with two boundary punctures and the strip $\R\times [0,1]$, $\Phi_{\mathbf{H},\mathbf{J}}$ can be defined by counting pseudoholomorphic discs. We can assume that one of these boundary punctures -- which we regard as the ``input'' puncture -- is at $-1$, and the other -- which we consider to be the ``output'' -- is at $+1$. In addition to these two boundary punctures, we fix an internal marked point at zero. No constraints are placed on the image of this marked point, which just serves to stabilize the domain.
 
We define the total Poincar\'{e} duality quasi-isomorphism $\phi_{\mathbf{H},\mathbf{J}}$ associated to the Floer data $(H_{L,N},J_{L,N})$ and $(H_{N,L},J_{N,L})$ to be the composition of the formal chain isomorphism $P$ with the dual of the comparison map quasi-isomorphism $\Phi_{\mathbf{H},\mathbf{J}}$:
\begin{align}\label{EQ: Poincare duality quasi-iso for CF total}
\phi_{\mathbf{H},\mathbf{J}}:CF(L,N;H_{L,N},J_{L,N})&\xrightarrow{P} CF(N,L;\bar{H}_{L,N},\bar{J}_{L,N})^\vee \nonumber\\
&\quad\quad\xrightarrow{\Phi_{\mathbf{H},\mathbf{J}}^\vee}CF(N,L;H_{N,L},J_{N,L})^\vee.
\end{align}
The map $\phi_{\mathbf{H},\mathbf{J}}$ is given geometrically by counting pseudoholomorphic discs with two ingoing boundary punctures and one internal marked point which satisfy boundary conditions along $L$ and $N$ -- see Figure \ref{FIG: Chain level duality map discs}.
\begin{figure}
\centering
\def\svgwidth{55mm}
\begingroup%
  \makeatletter%
  \providecommand\color[2][]{%
    \errmessage{(Inkscape) Color is used for the text in Inkscape, but the package 'color.sty' is not loaded}%
    \renewcommand\color[2][]{}%
  }%
  \providecommand\transparent[1]{%
    \errmessage{(Inkscape) Transparency is used (non-zero) for the text in Inkscape, but the package 'transparent.sty' is not loaded}%
    \renewcommand\transparent[1]{}%
  }%
  \providecommand\rotatebox[2]{#2}%
  \newcommand*\fsize{\dimexpr\f@size pt\relax}%
  \newcommand*\lineheight[1]{\fontsize{\fsize}{#1\fsize}\selectfont}%
  \ifx\svgwidth\undefined%
    \setlength{\unitlength}{254.19147401bp}%
    \ifx\svgscale\undefined%
      \relax%
    \else%
      \setlength{\unitlength}{\unitlength * \real{\svgscale}}%
    \fi%
  \else%
    \setlength{\unitlength}{\svgwidth}%
  \fi%
  \global\let\svgwidth\undefined%
  \global\let\svgscale\undefined%
  \makeatother%
  \begin{picture}(1,0.98223035)%
    \lineheight{1}%
    \setlength\tabcolsep{0pt}%
    \put(0,0){\includegraphics[width=\unitlength,page=1]{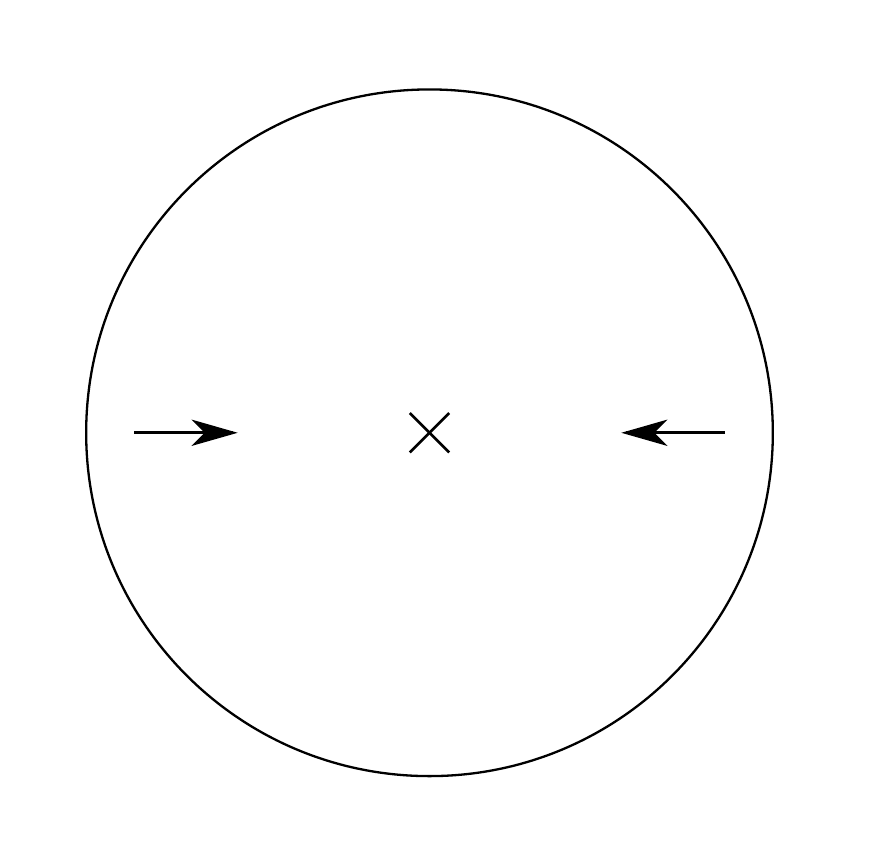}}%
    \put(0.45689773,0.91930114){\color[rgb]{0,0,0}\makebox(0,0)[lt]{\lineheight{2.13000011}\smash{\begin{tabular}[t]{l}$N$\end{tabular}}}}%
    \put(0.46561102,0.01770348){\color[rgb]{0,0,0}\makebox(0,0)[lt]{\lineheight{2.13000011}\smash{\begin{tabular}[t]{l}$L$\end{tabular}}}}%
    \put(0.9348839,0.47476863){\color[rgb]{0,0,0}\makebox(0,0)[lt]{\lineheight{2.13000011}\smash{\begin{tabular}[t]{l}$\xi$\end{tabular}}}}%
    \put(-0.00212069,0.47831851){\color[rgb]{0,0,0}\makebox(0,0)[lt]{\lineheight{2.13000011}\smash{\begin{tabular}[t]{l}$\gamma$\end{tabular}}}}%
    \put(0,0){\includegraphics[width=\unitlength,page=2]{circle_1.pdf}}%
  \end{picture}%
\endgroup%

\caption{\label{FIG: Chain level duality map discs}A disc contributing to the Poincar\'{e} duality quasi-isomorphism $\phi_{\mathbf{H},\mathbf{J}}$}
\end{figure}
The inhomogeneous Cauchy-Riemann equation we consider reduces to the Floer equation for $(H_{L,N},J_{L,N})$ near one puncture and for $(H_{N,L},J_{N,L})$ near the other puncture. More precisely, set $\mathcal{R}(\gamma,\xi)^0$ to be the zero-dimensional component of this space of discs which satisfy asymptotic conditions along $\gamma$ in $\mathcal{O}(H_{L,N})$ at one puncture and along $\xi\in \mathcal{O}(H_{N,L})$ at the other puncture. The map $\phi_{\mathbf{H},\mathbf{J}}$ is given by
\begin{equation}\label{EQ: Poincare duality quasi-iso for CF total formula}
\langle\phi_{\mathbf{H},\mathbf{J}}(\gamma),\xi\rangle = \#_{\Z_2} \mathcal{R}(\gamma,\xi)^0.
\end{equation}
This follows directly from the definitions of $P$ and $\Phi_{\mathbf{H},\mathbf{J}}$ as well as a trick for changing output punctures to input punctures which we explain in a more general context in Remark \ref{REMK: Changing inputs and outputs}. 

\begin{remk}
When $L$ is an exact Lagrangian, the isomorphism induced on homology by \eqref{EQ: Poincare duality quasi-iso for CF total} for the self-Floer complex of $L$, $HF(L,L)\cong HF(L,L)^\vee$, corresponds under the isomorphism $HF(L,L)\cong H(L)$ to the usual Poincar\'{e} duality isomorphism for the homology of the compact manifold $L$ with coefficients in $\Z_2$.
\end{remk}

\begin{remk}
In the case where $2c_1(M)=0$ and $L$ and $N$ are equipped with $\Z$-gradings (see \cite{Sei} for definitions), the complexes in \eqref{EQ: Poincare duality quasi-iso for CF total} admit induced $\Z$-gradings and the Poincar\'{e} duality quasi-isomorphism is a degree-zero map
\begin{equation}
CF_\bullet(L,N;H_{L,N},J_{L,N})\xrightarrow{\cong} CF_{\bullet-n}(N,L;H_{N,L},J_{N,L})^\vee
\end{equation}
Here we assume the dual complex $CF_{\bullet}(N,L;H_{N,L},J_{N,L})^\vee$ is graded homologically.
\end{remk}

\subsection{The monotone Fukaya category}\label{SUBSECTION: The monotone Fukaya category}
In this section, we briefly review the construction of the monotone Fukaya category associated to a compact symplectic manifold. This material is treated in \cite{BC14} using the same conventions and assumptions we use here. For an in-depth account of this construction in the exact context, we refer the reader to the foundational text \cite{Sei}. Here we suppress many details which we will however describe later in the context of Fukaya categories of Lagrangian cobordisms in Section \ref{SECTION: Fcob and its relative wCY pairing}.

Fix $d\in\Z_2$. We first define the class $\mathcal{L}^*_d(M)$ of Lagrangians in $M$ which forms the objects of the Fukaya category $\mathcal{F}\mathit{uk}(M)$. A family of monotone Lagrangians in $M$ is said to be \textbf{uniformly monotone} if all Lagrangians in the family have the same monotonicity constant. The class $\mathcal{L}^*_d(M)$ is defined to be the collection of uniformly monotone Lagrangians $L\subset M$, for a fixed monotonicity constant, satisfying $d_L=d$. Here $d_L$ is the invariant described in Section \ref{SUBSECTION: Lagrangian Floer homology and Poincare duality} which counts the number of discs of Maslov index two passing through a generic point of $L$. We also require that every Lagrangian $L$ in $\mathcal{L}^*_d(M)$ be non-narrow (i.e.\ $QH(L)\ne 0$) and that the inclusion of $\pi_1(L)$ in $\pi_1(M)$ be trivial. 

In order to define the morphisms in $\mathcal{F}\mathit{uk}(M)$, we fix for every pair $L,L'\in \mathcal{L}^*_d(M)$ a regular Floer datum $\mathscr{D}_{L,L'}=(H_{L,L'},J_{L,L'})$. The morphism space $\mathcal{F}\mathit{uk}(M)(L,L')$ is then taken to be the Floer chain complex $CF(L,L';\mathscr{D}_{L,L'})$. The maps $\mu_k^{\mathcal{F}\mathit{uk}(M)}$ count elements in zero-dimensional moduli spaces of inhomogeneous pseudoholomorphic polygons satisfying boundary conditions along Lagrangians in $\mathcal{L}^*_d(M)$. More precisely, in order to define the operation 
\begin{equation}
\mu^{\mathcal{F}\mathit{uk}(M)}_k:CF(L_0,L_1;\mathscr{D}_{L_0,L_1})\otimes\cdots\otimes CF(L_{k-1},L_k;\mathscr{D}_{L_{k-1},L_k})\to CF(L_0,L_k;\mathscr{D}_{L_0,L_k})
\end{equation}
 we consider maps $u:S\to M$, where $S$ is a disc with $k+1$ punctures on its boundary, which we label $z_1,\ldots,z_{k+1}$ in clockwise order. The connected component of the boundary of the disc between $z_i$ and $z_{i+1}$ must be mapped by $u$ into the Lagrangian $L_i$, for $i=1,\ldots,k-1$, and the connected component of the boundary between $z_{k+1}$ and $z_1$ must be mapped to $L_0$. At the puncture $z_i$, for $i=1,\ldots,k$, the map $u$ must tend asymptotically toward a Hamiltonian chord $\gamma_i\in\mathcal{O}(H_{L_{i-1}, L_i})$ and at the puncture $z_{k+1}$, the map $u$ must tend toward a Hamiltonian chord $\gamma_{k+1}\in \mathcal{O}(H_{L_0,L_k})$. In order to describe this asymptotic behaviour precisely, the punctured disc $S$ must be equipped with strip-like ends as we will describe in Section \ref{SUBSECTION: Pointed discs, strip-like ends, and sign data}. The maps $u$ in the moduli space we consider must satisfy an inhomogeneous nonlinear Cauchy-Riemann equation whose expression relies on a choice of perturbation data for the family $L_0,\ldots,L_k$ of Lagrangians in $\mathcal{L}^*_d(M)$. We will describe this data and the relevant equation for cobordism Fukaya categories in Section \ref{SECTION: Fcob and its relative wCY pairing}. We denote the moduli space of such maps $u$ by $\mathcal{R}_\mu^{k+1}(\gamma_1,\ldots,\gamma_{k+1})$. Assume now that the perturbation data is regular. Then if we allow for deformations of the conformal structure on $S$, the subset of the moduli space $\mathcal{R}_\mu^{k+1}(\gamma_1,\ldots,\gamma_{k+1})$ consisting of curves of a fixed Maslov index $N$ is a smooth manifold of dimension $N+k-2$. We denote by $\mathcal{R}_\mu^{k+1}(\gamma_1,\ldots,\gamma_{k+1})^0$ the zero-dimensional component of the moduli space. The map $\mu^{\mathcal{F}\mathit{uk}(M)}_k$ is defined by
\begin{equation}
\mu^{\mathcal{F}\mathit{uk}(M)}_k(\gamma_1,\ldots,\gamma_k)=\sum_{\gamma_{k+1}\in\mathcal{O}(H_{L_0,L_k})}\#_{\Z_2} \mathcal{R}_\mu^{k+1}(\gamma_1,\ldots,\gamma_{k+1})^0\gamma_{k+1}.
\end{equation} 
 
Standard compactness and gluing arguments show that these maps satisfy the $A_\infty$-relations, and hence $\mathcal{F}\mathit{uk}(M)$ is an ungraded $A_\infty$-category. Under additional assumptions, it is possible to introduce gradings (see \cite{Sei}), but we will work exclusively in the ungraded context. Moreover, the category $\mathcal{F}\mathit{uk}(M)$ is homologically unital, with the unit in $HF(L,L;\mathscr{D}_{L,L})$ for $L\in \mathcal{L}^*_d(M)$ being the fundamental class of $L$ under the isomorphism $QH(L)\cong HF(L,L;\mathscr{D}_{L,L})$.

The construction of the Fukaya category $\mathcal{F}\mathit{uk}(M)$ depends on all of the choices of data involved (strip-like ends, Floer and perturbation data), so we in fact obtain a family of $A_\infty$-categories. However for any two sets of data (for a fixed $d$ and monotonicity constant), the associated categories are quasi-isomorphic and the quasi-isomorphism is canonically defined on the underlying homological categories \cite{Sei}.

As $M$ will be fixed in what follows, we will abbreviate $\mathcal{F}\mathit{uk}(M)$ by $\mathcal{F}$.

\section{Weak Calabi-Yau structures and relative weak Calabi-Yau pairings}\label{SECTION: wCY structures}

There are two distinct but related approaches to describing duality for $A_\infty$-categories: The first is via a \emph{cyclic $A_\infty$ structure}, an approach studied by Costello \cite{Cost07,Cost09}; The second, which will be our focus, is via some manner of Calabi-Yau structure. Calabi-Yau structures were introduced by Kontsevich and Soibelman for $A_\infty$-algebras \cite{KontsevichSoibelman}. The concept also appears in \cite[\S I.12j]{Sei} and is further developed in \cite{GanatraOriginalThesis, SheridanFano, ganatra2015mirror}. Calabi-Yau structures for $A_\infty$-categories come in various flavours. The first distinction is between Calabi-Yau structures for \emph{smooth} $A_\infty$-categories and those for \emph{proper} $A_\infty$-categories. Each of these two types of Calabi-Yau structure in turn has a weak version, formulated in terms of Hochschild homology, and a strong version, formulated in terms of cyclic homology. We will consider exclusively \emph{weak proper Calabi-Yau structures}, also known as \emph{weak compact Calabi-Yau structures}. For $A_\infty$-categories over fields of characteristic zero, strong proper Calabi-Yau structures are closely related to Costello's cyclic $A_\infty$ structures (\cite[Theorem 10.2.2]{KontSoib06}), but away from characteristic zero the two concepts diverge (see Remark 6.5 in \cite{ganatra2015mirror}). We refer the reader to \cite[\S 6.1]{ganatra2015mirror} for a detailed explanation of the different notions of Calabi-Yau structures for $A_\infty$-categories. 

An $A_\infty$-category $\mathcal{A}$ is defined to be \textbf{proper} if the homology groups $H_p(\mathcal{A}(X,X'))$ are of finite rank for all $X,X'\in\mathrm{Ob}(\mathcal{A})$. We will not use the concept of smoothness, so we refer the reader to \cite[Definition 2.35]{GanatraThesis} for the definition. In \cite{SheridanFano}, Sheridan showed that the monotone Fukaya category of a compact symplectic manifold possesses a weak proper Calabi-Yau structure (a fact that was mentioned in \cite[\S I.12j]{Sei} for exact symplectic manifolds). Furthermore, the Fukaya category of a Calabi-Yau manifold satisfying some additional technical assumptions possesses a strong proper Calabi-Yau structure \cite{ganatra2015mirror}. In \cite{GanatraOriginalThesis}, Ganatra showed that the wrapped Fukaya category of a symplectic manifold satisfying a non-degeneracy condition possesses a weak smooth Calabi-Yau structure. When an $A_\infty$-category is both smooth and proper, the two notions of weak (and of strong) Calabi-Yau structures are equivalent \cite[Proposition 6.10]{ganatra2015mirror}. As we will only use weak proper Calabi-Yau structures, we suppress the `proper' in what follows.

In this section we will assume all $A_\infty$-categories and bimodules are homologically unital.

\subsection{Weak Calabi-Yau structures on $A_\infty$-categories} 

We first recall from \cite{Tra08} that an \textbf{$\infty$ inner product} on an $\mathcal{A}_\infty$-category $\mathcal{A}$ is an $A_\infty$-bimodule morphism from the diagonal bimodule associated to $\mathcal{A}$ to the Serre bimodule,
\begin{equation}
\phi:\mathcal{A}_\Delta \to \mathcal{A}^\vee.
\end{equation}
The $\infty$ inner product $\phi$ is said to be $n$-dimensional if $\phi$ is of degree $-n$. In other words, an $n$-dimensional $\infty$ inner product is a closed element of $_2 CC^{-n}(\mathcal{A},\mathcal{A}^\vee)$, and hence defines a class in $HH^{-n}(\mathcal{A},\mathcal{A}^\vee)$. Two $n$-dimensional $\infty$ inner products are said to be equivalent if they represent the same class in $HH^{-n}(\mathcal{A},\mathcal{A}^\vee)$.

\begin{defn}
A \textbf{weak Calabi-Yau structure} on $\mathcal{A}$ is the class of an $N_\mathcal{A}$-dimensional $\infty$ inner product which is a bimodule quasi-isomorphism. Here $N_\mathcal{A}$ is the degree of $\mathcal{A}$.
\end{defn}

Note that $\mathcal{A}$ must be proper in order to support a weak Calabi-Yau structure. This can be seen by considering the maps $\phi_{0|1|0}$ for an $N_\mathcal{A}$-dimensional $\infty$ inner product $\phi$. If $\phi$ is a quasi-isomorphism, these maps give isomorphisms 
\begin{equation}
\begin{aligned}
&H_p(\mathcal{A}(X,X'))\cong \mathit{hom}(H_{N_\mathcal{A}-p}(\mathcal{A}(X',X)),\Z_2),\\ 
&H_{N_\mathcal{A}-p}(\mathcal{A}(X',X))\cong \mathit{hom}(H_p(\mathcal{A}(X,X')),\Z_2),
\end{aligned}
\end{equation}
for any pair of objects $X,X'$ in $\mathcal{A}$. This in turn implies that $H_p(\mathcal{A}(X,X'))$ is isomorphic to its (algebraic) double dual and so must have finite dimension. 

\begin{remk} As a result of our choice to use homological conventions for $A_\infty$-categories, the dimension of an $\infty$ inner product representing a weak Calabi-Yau structure on $\mathcal{A}$ is intrinsically defined to be $N_\mathcal{A}$. For cohomological $A_\infty$-categories, which do not have an associated degree, there is no preferred choice of dimension for an $\infty$ inner product representing a weak Calabi-Yau structure. This dimension must be specified as an extra piece of data. 
\end{remk}

\begin{remk}
A weak Calabi-Yau structure $[\phi]$ on $\mathcal{A}$ induces an isomorphism 
\begin{equation}
HH^\bullet(\mathcal{A})\cong (HH_\bullet(\mathcal{A}))^\vee
\end{equation}
of degree $N_\mathcal{A}$. This is the homological map associated to the map of two-pointed complexes
\begin{equation}
_2CC^\bullet(\mathcal{A})\to  {}_2CC^{\bullet-N_{\mathcal{A}}}(\mathcal{A},\mathcal{A}^\vee),\; \alpha\mapsto \mu_2^{\mathcal{A}\mathit{\mbox{--}mod\mbox{--}}\mathcal{A}}(\alpha,\phi),
\end{equation}
where we implicitly use the isomorphism ${}_2CC_{\bullet+N_{\mathcal{A}}}(\mathcal{A})^\vee\cong {}_2CC^{\bullet-N_{\mathcal{A}}}(\mathcal{A},\mathcal{A}^\vee)$. See \cite[Lemma A.2]{SheridanFano}.
\end{remk}

Under the isomorphism of categories $ \mathcal{A}\mathit{\mbox{--}mod\mbox{--}}\mathcal{A}\cong \mathit{fun}(\mathcal{A},\mathcal{A}\mathit{\mbox{--}mod})$ in \eqref{EQ: Category isos between functors into left modules and bimodules}, an $\infty$ inner product $\phi$ corresponds to a natural transformation 
\begin{equation}
\delta_\phi:\mathbf{Y}_\mathcal{A}^l\to(\mathbf{Y}_\mathcal{A}^l)^\vee.
\end{equation} 
If $\phi$ is $n$-dimensional, then $\delta_\phi$ is of degree $-n$. We therefore have the following equivalent definition.

\begin{defn}\label{DEF: wCY structure as natural transformation}\cite[\S I.12j]{Sei}
A \textbf{weak Calabi-Yau structure} on $\mathcal{A}$ is the class of a natural quasi-isomorphism $\delta$ from $\mathbf{Y}_\mathcal{A}^l$ to $(\mathbf{Y}_\mathcal{A}^\vee)^l$ of degree $-N_{\mathcal{A}}$, i.e.\ $\delta$ represents an isomorphism in the category $H(fun(\mathcal{A},\mathcal{A}\mathit{\mbox{-}mod}))$. 
\end{defn} 

Via the isomorphism of chain complexes 
\begin{equation}
\Gamma:{}_2CC_\bullet(\mathcal{A},\mathcal{M})^\vee\xrightarrow{\cong} {}_2CC^{\bullet-2N_{\mathcal{A}}}(\mathcal{A},\mathcal{M}^\vee).
\end{equation}
of Lemma \ref{LEM: This iso 2CCn(A,Mvee) cong 2CC_n(A,M)vee} for $\mathcal{M}=\mathcal{A}_\Delta$, we can equivalently define an $N_\mathcal{A}$-dimensional $\infty$ inner product on $\mathcal{A}$ to be a closed element of ${}_2CC_{N_\mathcal{A}}(\mathcal{A})^\vee:=hom({}_2CC_{-N_\mathcal{A}}(\mathcal{A}),\Z_2)$. Then through the quasi-isomorphism
\begin{equation}
T^\vee:CC_\bullet(\mathcal{A},\mathcal{M})^\vee \to {}_2CC_\bullet(\mathcal{A},\mathcal{M})^\vee, 
\end{equation}
dual to the quasi-isomorphism \eqref{EQ: iso from two-pointed cochain complex to one-pointed cochain complex}, with $\mathcal{M}=\mathcal{A}_\Delta$, classes of $N_\mathcal{A}$-dimensional $\infty$ inner products correspond to elements of $HH_{N_{\mathcal{A}}}(\mathcal{A})^\vee$.

The following definition appears as Definition A.1 in \cite{SheridanFano} for $\mathcal{M}=\mathcal{A}_\Delta$.
\begin{defn}
A closed element $\sigma\in CC_{N_\mathcal{A}}(\mathcal{A},\mathcal{M})^\vee$ is defined to be \textbf{homologically non-degenerate} if the following composition is a perfect pairing
\begin{align}\label{EQ: homological non-degeneracy def pairing}
H_p(\mathcal{A}(X,X'))\otimes H_{N_\mathcal{A}-p}(\mathcal{M}(X',X))\xrightarrow{H(\mu^{\mathcal{M}}_{1|1|0})}H_0(\mathcal{M}(X,X))\nonumber\\
\overset{\iota}{\hookrightarrow} HH_{-N_{\mathcal{A}}}(\mathcal{A},\mathcal{M})\xrightarrow{\sigma}\Z_2
\end{align}
for any pair of objects $X,X'\in\mathrm{Ob}(\mathcal{A})$ and for all $p\in\Z$. The inclusion $\iota$ is induced by the chain-level map \eqref{EQ: inclusion into Hochschild cochain complex}. In other words, $\sigma$ is homologically non-degenerate if the map $H_p(\mathcal{A}(X,X'))\to hom(H_{N_\mathcal{A}-p}(\mathcal{M}(X',X)),\Z_2)$ induced by \eqref{EQ: homological non-degeneracy def pairing} is an isomorphism of $\Z_2$-vector spaces.
\end{defn}

The next lemma is a generalization of Lemma A.1 in \cite{SheridanFano} from $\mathcal{M}=\mathcal{A}_\Delta$ to arbitrary $\mathcal{M}$. The proof that appears there applies also for arbitrary $\mathcal{M}$.
\begin{lem}\label{LEM: Homological nondegeneracy condition}
A closed element $\sigma\in CC_{N_\mathcal{A}}(\mathcal{A},\mathcal{M})^\vee$ is homologically non-degenerate if and only if $\Gamma\circ T^\vee(\sigma)\in {}_2CC^{-N_{\mathcal{A}}}(\mathcal{A},\mathcal{M}^\vee)$ is a bimodule quasi-isomorphism.
\end{lem}

By Lemma \ref{LEM: Homological nondegeneracy condition}, the map induced on homology by $\Gamma\circ T^\vee$ identifies elements of $HH_{N_\mathcal{A}}(\mathcal{A})^\vee$ represented by homologically non-degenerate cycles with weak Calabi-Yau structures on $\mathcal{A}$. Hence we can equivalently define a weak Calabi-Yau structure on $\mathcal{A}$ as follows:

\begin{defn}\label{DEF: wCY structure HHdual version}
A \textbf{weak Calabi-Yau structure} on $\mathcal{A}$ is the class of a homologically non-degenerate element of $CC_{N_\mathcal{A}}(\mathcal{A})^\vee$.
\end{defn}

\subsection{Relative weak Calabi-Yau pairings}

Let $\mathcal{A}$ and $\mathcal{B}$ be $A_\infty$-categories with degrees satisfying $N_\mathcal{A}=N_\mathcal{B}+j$ and let 
\begin{equation}
\mathbf{I}:\mathcal{B}\to\mathcal{A}[j]
\end{equation}
be an $A_\infty$-functor. Consider the $j$-fold suspension $(\mathcal{A}_\Delta)[j]$ of $\mathcal{A}_\Delta$ as a module over $\mathcal{A}[j]$ (see Remark \ref{REMK: Module suspensions}). Note that there is a naturally defined $\mathcal{B}\mathit{\mbox{--}}\mathcal{B}$ bimodule morphism of degree zero
\begin{equation}
i:\mathcal{B}_\Delta\to \mathbf{I}^*(\mathcal{A}_\Delta)[j],
\end{equation}
given by
\begin{equation}
i_{k|1|m}(y_1,\ldots,y_k,\mathbf{z},y'_m,\ldots,y'_1)=\mathbf{I}_{k+m+1}(y_1,\ldots,y_k,\mathbf{z},y'_m,\ldots,y'_1).
\end{equation}

\begin{defn}\label{DEF: Rel wCY pairing bimodule quasi-iso version}
Let $\mathcal{A}^{rel}_\Delta$ be a homologically unital $\mathcal{A}\mathit{\mbox{--}}\mathcal{A}$ bimodule, and set $\mathcal{A}^\vee_{rel}=(\mathcal{A}^{rel}_\Delta)^\vee$. A \textbf{relative weak Calabi-Yau pairing} on $\mathcal{A}$ with coefficients in $\mathcal{A}^{rel}_\Delta$ consists of the following data: 
\begin{itemize}
\item the class of an $\mathcal{A}\mathit{\mbox{--}}\mathcal{A}$ bimodule quasi-isomorphism $
\phi^{\mathcal{A}}:\mathcal{A}_\Delta\to \mathcal{A}_{rel}^\vee$
of degree $-N_\mathcal{A}$, 
\item a $\mathcal{B}\mathit{\mbox{--}}\mathcal{B}$ bimodule morphism 
$i_{rel}:\mathcal{B}_\Delta\to \mathbf{I}^*\mathcal{A}^{rel}_\Delta[-j]$
of degree $j$.
\end{itemize}
The quasi-isomorphism $\phi^{\mathcal{A}}$ must satisfy an additional non-degeneracy condition which we describe now. First define a map of chain complexes of degree $j$
\begin{equation}
\Psi_{i,i_{rel}}:\mathcal{A}\mathit{\mbox{--}mod\mbox{--}}\mathcal{A}(\mathcal{A}_{\Delta},\mathcal{A}^\vee_{rel})\to \mathcal{B}\mathit{\mbox{--}mod\mbox{--}}\mathcal{B}(\mathcal{B}_\Delta,\mathcal{B}^\vee)
\end{equation}
by mapping $\alpha\in \mathcal{A}\mathit{\mbox{--}mod\mbox{--}}\mathcal{A}(\mathcal{A}_{\Delta},\mathcal{A}^\vee_{rel})$ to the composition
\begin{equation}
\Psi_{i,i_{rel}}(\alpha):\mathcal{B}_\Delta\xrightarrow{i} \mathbf{I}^*\mathcal{A}_\Delta[j]\xrightarrow{\mathbf{I}^*\alpha[j]}\mathbf{I}^*\mathcal{A}_{rel}^\vee[j]\xrightarrow{i_{rel}^\vee} \mathcal{B}^\vee.
\end{equation}
Here we use the identification $\mathbf{I}^*\mathcal{A}_{rel}^\vee[j]= (\mathbf{I}^*\mathcal{A}^{rel}_\Delta[-j])^\vee$. In order for the data described above to define a relative weak Calabi-Yau pairing, we require that the $\mathcal{B}\mathit{\mbox{--}}\mathcal{B}$ bimodule morphism
$\Psi_{i,i_{rel}}(\phi^{\mathcal{A}})$ be a quasi-isomorphism. In other words, since $\Psi_{i,i_{rel}}(\phi^{\mathcal{A}})$ is of degree $-N_\mathcal{B}$, $[\Psi_{i,i_{rel}}(\phi^{\mathcal{A}})]$ must define a weak Calabi-Yau structure on $\mathcal{B}$. 

We will refer to the bimodule $\mathcal{A}^{rel}_\Delta$ as the \textbf{relative diagonal bimodule} and the bimodule $\mathcal{A}^\vee_{rel}$ as the \textbf{relative Serre bimodule}. We will also refer to $[\phi^{\mathcal{A}}]$ as a relative weak Calabi-Yau pairing on $\mathcal{A}$ when it is clear what $\mathcal{A}^{rel}_\Delta$ and $i_{rel}$ are.
\end{defn}

Note that Definition \ref{DEF: Rel wCY pairing bimodule quasi-iso version} reduces to the absolute case when $\mathcal{B}=\mathcal{A}$, $\mathbf{I}$ is the identity functor, $\mathcal{A}^{rel}_\Delta=\mathcal{A}_\Delta$, and $i_\mathit{rel}=i$. 

As in the absolute case, this definition has an equivalent formulation in terms of Hochschild homology, but with coefficients in $\mathcal{A}_\Delta^{rel}$ in the present situation. This relies on the following lemma, which is a generalization of Proposition 4.1 in \cite{Brav-Dyck}. Set $\mathbf{I}^{\mathit{rel}}_*$ to be the map induced on Hochschild chain complexes by $i_{\mathit{rel}}$ and $\mathbf{I}$,
\begin{equation}\label{EQ: Def of Irel}
\mathbf{I}^{\mathit{rel}}_*:CC_\bullet (\mathcal{B})\xrightarrow{(i_{\mathit{rel}})_*}CC_{\bullet+j} (\mathcal{B},\mathbf{I}^*\mathcal{A}^{\mathit{rel}}_\Delta[-j])\xrightarrow{\mathbf{I}_*} CC_{\bullet+j} (\mathcal{A}[j],\mathcal{A}^{\mathit{rel}}_\Delta[-j])\cong CC_{\bullet-j} (\mathcal{A},\mathcal{A}^{\mathit{rel}}_\Delta).
\end{equation}

\begin{lem}\label{LEM: Correspondance between Psi and Idual}
The following diagram of chain complexes commutes up to homotopy.

$$\begindc{\commdiag}[25] 
\obj(0,20)[A]{$CC_\bullet(\mathcal{A},\mathcal{A}^{rel}_\Delta)^\vee$}
\obj(45,20)[B]{$_2CC_\bullet(\mathcal{A},\mathcal{A}^{rel}_\Delta)^\vee$} 
\obj(90,20)[C]{$_2CC^{\bullet-2N_\mathcal{A}}(\mathcal{A},\mathcal{A}_{rel}^\vee)$} 
\obj(0,0)[A']{$CC_{\bullet-j}(\mathcal{B})^\vee$}
\obj(45,0)[B']{$_2CC_{\bullet-j}(\mathcal{B})^\vee$} 
\obj(90,0)[C']{$_2CC^{\bullet-j-2N_\mathcal{B}}(\mathcal{B},\mathcal{B}^\vee)$} 
\mor{A}{B}{$T^\vee$} 
\mor{A}{B}{$\cong$}[\atright,\solidarrow]
\mor{B}{C}{$\Gamma$}
\mor{B}{C}{$\cong$}[\atright,\solidarrow]
\mor{A'}{B'}{$T^\vee$} 
\mor{A'}{B'}{$\cong$}[\atright,\solidarrow]
\mor{B'}{C'}{$\Gamma$}
\mor{B'}{C'}{$\cong$}[\atright,\solidarrow]
\mor{A}{A'}{$(\mathbf{I}_*^\mathit{rel})^\vee$}
\mor{C}{C'}{$\Psi_{i,i_{rel}}$} 
\enddc$$
\end{lem}

\begin{proof} By direct computation.
\end{proof}

\begin{prop}
A bimodule morphism $\phi$ of degree $-N_\mathcal{A}$ from $\mathcal{A}_{\Delta}$ to $\mathcal{A}^\vee_{rel}$ represents a relative weak Calabi-Yau pairing on $\mathcal{A}$ with coefficients in $\mathcal{A}^{rel}_\Delta$ if and only if any cycle $\sigma\in CC_{N_\mathcal{A}}(\mathcal{A},\mathcal{A}_\Delta^{rel})^\vee$ with $[\Gamma\circ T^\vee(\sigma)]=[\phi]$ is homologically non-degenerate and 
the cycle $(\mathbf{I}^{\mathit{rel}}_*)^\vee(\sigma)\in CC_{N_\mathcal{B}}(\mathcal{B})^\vee$ is homologically non-degenerate as well. In other words, $(\mathbf{I}^{\mathit{rel}}_*)^\vee(\sigma)$ represents a weak Calabi-Yau structure on $\mathcal{B}$ by Definition \ref{DEF: wCY structure HHdual version}.
\end{prop}

\begin{proof}
By Lemma \ref{LEM: Homological nondegeneracy condition}, $\phi$ is a quasi-isomorphism if and only if $\sigma$ is homologically non-degenerate. Moreover, by Lemma \ref{LEM: Correspondance between Psi and Idual}, $[(\mathbf{I}_*^{\mathit{rel}})^\vee(\sigma)]$ corresponds to $[\Psi_{i,i_{rel}}(\phi)]$ under the isomorphism 
\begin{equation}
H(\Gamma\circ T^\vee):HH_{N_\mathcal{B}}(\mathcal{B})^\vee\xrightarrow{\cong} H(\mathcal{B}\mathit{\mbox{--}mod\mbox{--}}\mathcal{B})(\mathcal{B}_\Delta,\mathcal{B}^\vee)_{-N_\mathcal{B}}.
\end{equation}
Again by Lemma \ref{LEM: Homological nondegeneracy condition}, The $\mathcal{B}\mbox{--}\mathcal{B}$ bimodule morphism $\Psi_{i,i_{rel}}(\phi)$ is a quasi-isomorphism precisely when $(\mathbf{I}^{\mathit{rel}}_*)^\vee(\sigma)$ is homologically non-degenerate.
\end{proof}

Therefore we have the following equivalent definition.
\begin{defn}\label{DEF: Relative weak CY pairing CC defn}
A \textbf{relative weak Calabi-Yau pairing} on $\mathcal{A}$ with coefficients in $\mathcal{A}^{rel}_\Delta$ consists of of the following data:
\begin{itemize}
\item the class of a homologically non-degenerate element $\sigma^\mathcal{A}\in CC_{N_\mathcal{A}}(\mathcal{A},\mathcal{A}^{\mathit{rel}}_\Delta)^\vee$,
\item a $\mathcal{B}\mathit{\mbox{--}}\mathcal{B}$ bimodule morphism 
$i_{rel}:\mathcal{B}_\Delta\to \mathbf{I}^*\mathcal{A}^{rel}_\Delta[-j]$
of degree $j$,
\end{itemize}
satisfying the condition that the cycle $(\mathbf{I}^{\mathit{rel}}_*)^\vee(\sigma)\in CC_{N_\mathcal{B}}(\mathcal{B})^\vee$ is also homologically non-degenerate. As with the previous definition, we will also refer to $[\sigma^\mathcal{A}]$ as a relative weak Calabi-Yau pairing on $\mathcal{A}$ when it is clear what $\mathcal{A}^{rel}_\Delta$ and $i_{rel}$ are. If $\mathcal{B}$ is equipped with a weak Calabi-Yau structure $[\sigma^\mathcal{B}]$, the relative weak Calabi-Yau pairing $[\sigma^\mathcal{A}]$ is said to be \textbf{compatible} with $[\sigma^\mathcal{B}]$ if $[(\mathbf{I}^{\mathit{rel}}_*)^\vee(\sigma^\mathcal{A})]=[\sigma^\mathcal{B}]$
\end{defn}

Finally, there is an interpretation of relative weak Calabi-Yau pairings in terms of Yoneda and abstract Serre functors. Consider the functors
\begin{align}
&\mathbf{G}^l_{\mathbf{I}}:\mathit{fun}(\mathcal{A}[j],\mathcal{A}[j]\mathit{\mbox{--}mod})\to \mathit{fun}(\mathcal{B},\mathcal{B}\mathit{\mbox{--}mod}),\\
&\mathbf{G}^r_{\mathbf{I}}:\mathit{fun}(\mathcal{A}[j],(\mathit{mod\mbox{--}}\mathcal{A}[j])^{\mathit{opp}}) \to \mathit{fun}(\mathcal{B},(mod\mathit{\mbox{--}}\mathcal{B})^{\mathit{opp}})
\end{align}
defined in \eqref{EQ: Def of G_F0,F1}.
The functor $\mathbf{I}:\mathcal{B}\to\mathcal{A}[j]$ induces a natural transformation
\begin{equation}
S_\mathbf{I}:\mathbf{Y}^l_{\mathcal{B}}\to \mathbf{G}^l_{\mathbf{I}}(\mathbf{Y}^l_{\mathcal{A}[j]})
\end{equation}
of degree zero. The components $(S_\mathbf{I})_k$ of $S_\mathbf{I}$ are given by
\begin{align}
&(S_\mathbf{I})_k:\mathcal{B}(Y_0,\ldots,Y_k)\to \mathcal{B}\mathit{\mbox{--}mod}(\mathbf{Y}_{\mathcal{B}}^l(Y_0),\mathbf{G}^l_{\mathbf{I}}(\mathbf{Y}^l_{\mathcal{A}[j]})(Y_k)),\\
&((S_\mathbf{I})_k(y_1,\ldots,y_k))_{m|1}:\mathcal{B}(Y'_m,\ldots,Y'_0)\otimes \mathbf{Y}_{\mathcal{B}}^l(Y_0)(Y'_0)\to \mathbf{G}^l_{\mathbf{I}}(\mathbf{Y}^l_{\mathcal{A}[j]})(Y_k)(Y'_m),\nonumber\\
&((S_\mathbf{I})_k(y_1,\ldots,y_k))_{m|1}(y'_m,\ldots,y'_1,\mathbf{w})=\mathbf{I}_{k+m+1}(y'_m,\ldots,y'_1,\mathbf{w},y_1,\ldots,y_k).\nonumber
\end{align}
Here we have used that $\mathbf{G}^l_{\mathbf{I}}(\mathbf{Y}^l_{\mathcal{A}[j]})(Y_k)(Y'_m)=\mathcal{A}[j](\mathbf{I}(Y'_m),\mathbf{I}(Y_k))$, which follows from the definition of $\mathbf{G}^l_{\mathbf{I}}$. 

In the following definition, we will also make use of the opposite functors associated to the dualization functors from right to left modules over $\mathcal{A}$ and over $\mathcal{B}$. Denote these by $\mathbf{D}_{\mathcal{A}}^{\mathit{opp}}$ and $\mathbf{D}_{\mathcal{B}}^{\mathit{opp}}$ respectively,
\begin{equation}
\mathbf{D}^{\mathit{opp}}_{\mathcal{A}}:(\mathit{mod\mbox{--}}\mathcal{A})^{\mathit{opp}}\to \mathcal{A}\mathit{\mbox{--}mod},\quad    \mathbf{D}^{\mathit{opp}}_{\mathcal{B}}:(\mathit{mod\mbox{--}}\mathcal{B})^{\mathit{opp}}\to \mathcal{B}\mathit{\mbox{--}mod}.
\end{equation}
Moreover we will make use of the suspension functors
\begin{equation}
\bm{\Sigma}^k:\mathcal{A}\mathit{\mbox{--}mod}\to\mathcal{A}\mathit{\mbox{--}mod},\quad \bm{\Sigma}^k:\mathit{mod\mbox{--}}\mathcal{A}\to\mathit{mod\mbox{--}}\mathcal{A}
\end{equation}
and their associated opposite functors
(see Remark \ref{REMK: Module suspensions}).
\begin{defn}\label{DEFN: Relative weak CY pairing Yoneda version}
Assume the category $\mathcal{A}$ is equipped with a functor $\mathbf{Y}^r_{rel}:\mathcal{A}\to (\mathit{mod\mbox{--}}\mathcal{A})^{\mathit{opp}}$ and set 
\begin{equation}
(\mathbf{Y}^\vee_{rel})^l=\mathbf{D}^{\mathit{opp}}_{\mathcal{A}}\circ \mathbf{Y}^r_{rel}:\mathcal{A}\to \mathcal{A}\mathit{\mbox{--}mod}.
\end{equation} 
By abuse of notation we also denote the induced functors $\mathcal{A}[j]\to (\mathit{mod\mbox{--}}\mathcal{A}[j])^{\mathit{opp}}$ and $\mathcal{A}[j]\to \mathcal{A}[j]\mathit{\mbox{--}mod}$ by $\mathbf{Y}^r_{rel}$ and $(\mathbf{Y}^\vee_{rel})^l$.
A \textbf{relative weak Calabi-Yau pairing} on $\mathcal{A}$ for the functor $\mathbf{Y}^r_{rel}$ consists of the following data:
\begin{itemize}
\item the class of a natural quasi-isomorphism $
\delta^\mathcal{A}:\mathbf{Y}_{\mathcal{A}}^l\to (\mathbf{Y}^\vee_{rel})^l$
of degree $-N_\mathcal{A}$,
\item a natural transformation $
S^{\mathit{rel}}:\mathbf{G}^r_\mathbf{I}((\bm{\Sigma}^{-j})^\mathit{opp}\circ\mathbf{Y}^r_{rel}) \to \mathbf{Y}^r_{\mathcal{B}}$ of degree $j$.
\end{itemize}
These data are subject to an additional non-degeneracy condition which we describe now. Using Proposition \ref{PROP: Compatibility G and dualization}, we have that $\mathbf{L}_{\mathbf{D}_{\mathcal{B}}^{\mathit{opp}}}(S^{\mathit{rel}})$ defines a natural transformation
\begin{equation}
\mathbf{L}_{\mathbf{D}^{\mathit{opp}}_{\mathcal{B}}}(S^{\mathit{rel}}):\mathbf{G}^l_{\mathbf{I}}(\bm{\Sigma}^j\circ(\mathbf{Y}^\vee_{rel})^l)\to (\mathbf{Y}^\vee_{\mathcal{B}})^l
\end{equation}
of degree $j$.
The above data define a relative weak Calabi-Yau pairing if the natural transformation $\mathcal{P}_\mathit{rel}(\delta^\mathcal{A}):\mathbf{Y}^l_{\mathcal{B}}\to (\mathbf{Y}^\vee_{\mathcal{B}})^l$ of degree $-N_\mathcal{B}$ defined as the composition
\begin{equation}\label{EQ: Def of mathcalP(delta)}
\mathcal{P}_\mathit{rel}(\delta^\mathcal{A}):\mathbf{Y}^l_{\mathcal{B}}\xrightarrow{S_{\mathbf{I}}} \mathbf{G}^l_{\mathbf{I}}(\mathbf{Y}^l_{\mathcal{A}[j]})\xrightarrow{\mathbf{G}^l_\mathbf{I}(\delta^{\mathcal{A}}[j])} \mathbf{G}^l_{\mathbf{I}}(\bm{\Sigma}^j\circ(\mathbf{Y}^\vee_{rel})^l)\xrightarrow{\mathbf{L}_{\mathbf{D}^{\mathit{opp}}_{\mathcal{B}}}(S^{\mathit{rel}})} (\mathbf{Y}^\vee_{\mathcal{B}})^l
\end{equation}
is a quasi-isomorphism. In other words, since $\mathcal{P}_\mathit{rel}(\delta^\mathcal{A})$ is of degree $-N_\mathcal{B}$, it represents a weak Calabi-Yau structure on $\mathcal{B}$. If $\mathcal{B}$ is equipped with a weak Calabi-Yau structure $[\delta^\mathcal{B}]$, the relative weak Calabi-Yau pairing $[\delta^\mathcal{A}]$ is said to be \textbf{compatible} with $[\delta^\mathcal{B}]$ if $[\mathcal{P}_\mathit{rel}(\delta^\mathcal{A})]=[\delta^{\mathcal{B}}]$. 

We will refer to the functor $\mathbf{Y}^r_{rel}$ as the \textbf{relative right Yoneda functor} and the functor $(\mathbf{Y}^\vee_{rel})^l$ as the \textbf{relative left abstract Serre functor}.
\end{defn}

\begin{remk} There is also a notion of a \emph{relative weak Calabi-Yau structure} which was introduced by Toen in \cite{Toe14} and whose strong version was further studied by Brav and Dyckerhoff in \cite{Brav-Dyck}. In \cite{Brav-Dyck} this is referred to as a `relative right Calabi-Yau structure' (`right' being their terminology for proper Calabi-Yau structures and `left' for smooth Calabi-Yau structures). This notion is distinct from our concept of a relative weak Calabi-Yau pairing. A relative weak Calabi-Yau structure for an $A_\infty$-functor $\mathbf{P}:\mathcal{C}\to\mathcal{D}$ between $A_\infty$-categories both of degree $N$ is defined to be an element of $HH_N(\mathbf{P})^\vee$ satisfying certain non-degeneracy conditions. Here
\begin{equation}
HH_\bullet(\mathbf{P}):=H^\bullet(\mathrm{cone}(\mathbf{P}_*:CC_\bullet(\mathcal{C})\to CC_\bullet(\mathcal{D}))).
\end{equation}
We set 
\begin{equation}
\mathcal{C}^{rel}_\Delta=\mathrm{cone}(p:\mathcal{C}_\Delta\to\mathbf{P}^*\mathcal{D}_
\Delta).
\end{equation}
As explained in \cite{Brav-Dyck}, a relative weak Calabi-Yau structure induces a module morphism $\mathcal{C}_\Delta \to (\mathcal{C}^{rel}_\Delta)^\vee$. This can be viewed as a duality morphism for the category $\mathcal{C}$ relative to $\mathcal{D}$. We see then that a relative weak Calabi-Yau structure for $\mathbf{P}:\mathcal{C}\to\mathcal{D}$ describes duality for the category $\mathcal{C}$ relative to $\mathcal{D}$, whereas a relative weak Calabi-Yau pairing for $\mathbf{I}:\mathcal{B}\to \mathcal{A}$ describes a type of duality for the category $\mathcal{A}$ which induces duality on $\mathcal{B}$ in the form of an absolute weak Calabi-Yau structure. Moreover, for a relative weak Calabi-Yau structure, the module $\mathcal{C}^{rel}_\Delta$ is defined intrinsically in terms of the functor $\mathbf{P}$, but to define a relative weak Calabi-Yau pairing, the module $\mathcal{A}^{rel}_\Delta$ is an extra structure that must be fixed.  

In the special case of $A_\infty$-algebras over fields of characteristic zero, Toen and Brav-Dyckerhoff's notion of a relative (strong) right Calabi-Yau structure is related to the earlier concept of Seidel of an \emph{$A_\infty$-algebra with boundary} \cite[\S 3.3]{SeiLef12}.  
\end{remk}

\subsection{The weak Calabi-Yau structure on $\mathcal{F}$}

We now describe the weak Calabi-Yau structure on the monotone Fukaya category. This construction was introduced by Seidel in the exact context \cite[\S I.12j]{Sei}. It was proved by Sheridan that the construction defines a weak Calabi-Yau structure in the monotone case  \cite{SheridanFano}. The technical assumptions used by Sheridan are slightly different from those given in Section \ref{SUBSECTION: The monotone Fukaya category} because he works in the graded context (although not necessarily $\Z$-graded) and with coefficients in $\C$. However the same arguments carry over to the ungraded context and $\Z_2$-coefficients. 

The construction of the weak Calabi-Yau structure on $\mathcal{F}$ relies on counting elements of moduli spaces of inhomogeneous pseudoholomorphic curves in $M$ which generalize the moduli spaces involved in the definition of the Poincar\'{e} duality quasi-isomorphism for Floer complexes. We present two equivalent versions of the construction which correspond to Definitions \ref{DEF: wCY structure as natural transformation} and \ref{DEF: wCY structure HHdual version} of weak Calabi-Yau structures on $A_\infty$-categories. As we work in the ungraded context, Hochschild homology and cohomology spaces are ungraded vector spaces, as are spaces of $\mathcal{F}\mathit{\mbox{--}}\mathcal{F}$ bimodule pre-morphisms, and there is no requirement on the degree of a weak Calabi-Yau structure (see also Remark \ref{REMK: Ungraded Hochschild complexes}). 

\subsubsection{As an element of $HH_\bullet(\mathcal{F})^\vee$}\label{SUBSUBSECTION: wCY structure on F CC version}

To define the weak Calabi-Yau structure on $\mathcal{F}$ as an element of $HH_\bullet(\mathcal{F})^\vee$, we consider moduli spaces of inhomogeneous pseudoholomorphic curves with domain the unit disc in $\C$ with at least one entry point along the boundary, no exit points, and a single internal marked point fixed at zero. We assume one of the entries along the boundary is fixed at $+1$. As in the usual construction of Fukaya categories \cite{Sei}, we fix strip-like ends and a generic choice of perturbation data for these curves so that, when considered together with the data for defining $\mathcal{F}$, all data is consistent with respect to gluing. We discuss this type of construction in detail for cobordism categories in Section \ref{SECTION: Fcob and its relative wCY pairing}. 

For a Hochschild chain of the form
\begin{equation}
\gamma_1\otimes \cdots\otimes\gamma_m\otimes \bm{\gamma} \in \mathcal{F}(L_0,\ldots, L_m)\otimes \mathcal{F}(L_m,L_0)\subset CC_\bullet(\mathcal{F}),
\end{equation}
we let $\mathcal{R}^{m+1;1}(\gamma_1,\ldots,\gamma_m,\bm{\gamma})$ denote the moduli space of such pseudoholomorphic discs satisfying boundary conditions along $L_0,\ldots,L_m,L_0$ (labelled in clockwise order) and asymptotic conditions along $\gamma_1,\ldots, \gamma_m,\bm{\gamma}$, as shown in Figure \ref{FIG: Hochschild CY curve configurations in M}. No constraints are placed on the image of the internal marked point, which is present for dimensional reasons. The subset of $\mathcal{R}^{m+1;1}(\gamma_1,\ldots,\gamma_m,\bm{\gamma})$ consisting of discs of a fixed Maslov index is a smooth manifold whose dimension is given by the Fredholm index of the linearized operator associated to the relevant inhomogeneous Cauchy-Riemann equation. We denote the $s$-dimensional component of $\mathcal{R}^{m+1;1}(\gamma_1,\ldots,\gamma_m,\bm{\gamma})$ by $\mathcal{R}^{m+1;1}(\gamma_1,\ldots,\gamma_m,\bm{\gamma})^s$. 

\begin{figure}
\centering
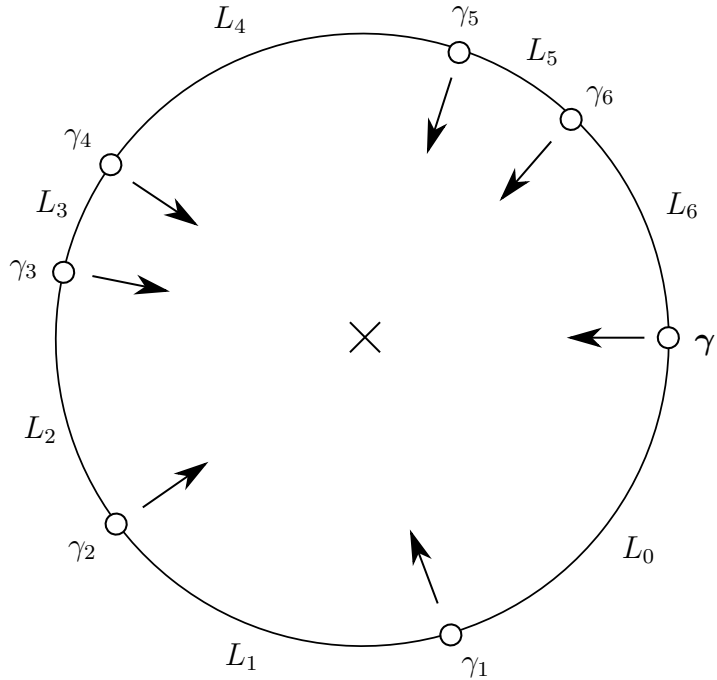
\caption[An inhomogeneous pseudoholomorphic polygon in \hfill \break$\mathcal{R}^{7;1}(\gamma_1,\ldots,\gamma_6,\bm{\gamma})$]{An inhomogeneous pseudoholomorphic polygon in $M$ in the moduli space $\mathcal{R}^{7;1}(\gamma_1,\ldots,\gamma_6,\bm{\gamma})$ used to define the representative $\sigma^\mathcal{F}\in CC_\bullet(\mathcal{F})^\vee$ of the weak Calabi-Yau structure on $\mathcal{F}$. The inward pointing arrows indicate that the punctures along the boundary are inputs. The cross at zero indicates the marked point. }
\label{FIG: Hochschild CY curve configurations in M}
\end{figure}

The zero-dimensional component $\mathcal{R}^{m+1;1}(\gamma_1,\ldots,\gamma_m,\bm{\gamma})^0$ is compact and the one-dimensional component $\mathcal{R}^{m+1;1}(\gamma_1,\ldots,\gamma_m,\bm{\gamma})^1$ can be compactified by once-broken configurations. These broken configurations are pairs of discs in $M$ joined at a boundary point, each with $m\ge 1$ entry chords along the boundary, and with one disc containing the image of the internal marked point.

Define an element $\sigma^\mathcal{F}\in CC_\bullet(\mathcal{F})^\vee$ by
\begin{equation}
\langle\sigma^\mathcal{F},\gamma_1\otimes\cdots\otimes\gamma_m\otimes\bm{\gamma}\rangle=\#_{\Z_2} \mathcal{R}^{m+1;1}(\gamma_1,\ldots,\gamma_m,\bm{\gamma})^0,
\end{equation}
for $\gamma_i\in \mathcal{O}(H_{L_{i-1},L_i})$, $i=1,\ldots,m$, and $\bm{\gamma}\in \mathcal{O}(H_{L_m,L_0})$.

Consider now $\partial^\vee\sigma^\mathcal{F}$ applied to $\gamma_1\otimes \cdots\otimes\gamma_m\otimes \bm{\gamma}$, where $\partial^\vee$ is the dual of the differential on the Hochschild chain complex $CC_{\bullet}(\mathcal{F})$. This is a count of the broken configurations appearing as elements of the boundary of $\mathcal{R}^{m+1;1}(\gamma_1,\ldots,\gamma_m,\bm{\gamma})^1$. It follows that $\partial^\vee\sigma^\mathcal{F}=0$, and so $\sigma^\mathcal{F}$ represents a class in $HH_\bullet(\mathcal{F})^\vee$. Moreover, as proved in \cite[Lemma 2.4]{SheridanFano}, $\sigma^\mathcal{F}$ is homologically non-degenerate, and therefore represents a weak Calabi-Yau structure on $\mathcal{F}$. This is a reflection of the fact that the chain-level comparison map \eqref{EQ: Floer homology comparison map} is a quasi-isomorphism.

\subsubsection{As the class of a natural quasi-isomorphism $\mathbf{Y}_\mathcal{F}^l\to (\mathbf{Y}_\mathcal{F}^\vee)^l$}\label{SUBSUBSECTION: wCY structure on F Yoneda version}

To define the weak Calabi-Yau structure on $\mathcal{F}$ as the class of a natural transformation from the left Yoneda functor for $\mathcal{F}$ to the left abstract Serre functor, we consider moduli spaces of \emph{two-pointed open-closed $(m,p)$-discs} in $M$. These have as domain the unit disc in $\C$ with $m+p+2$ entry punctures around the boundary: one puncture fixed at $-1$, one puncture at $+1$, $m$ punctures on the lower half of the boundary of the disc and $p$ punctures on the upper half of the boundary of the disc. There is also a marked point fixed at zero. We fix strip-like ends and a generic choice of perturbation data for these curves so that, when considered together with the data for defining $\mathcal{F}$, all data is consistent with respect to gluing. 

Fix $m,p\ge 0$ and 
Hamiltonian orbits
\begin{equation}
\begin{aligned}
&\gamma_i\in \mathcal{O}(H_{L_{i-1},L_i}),\; i=1,\ldots,m,\\
&\eta_j\in \mathcal{O}(H_{N_j,N_{j-1}}),\; j=1,\ldots,p,\\
&\bm{\xi}\in \mathcal{O}(H_{L_m,N_p}), \;\bm{\xi}'\in \mathcal{O}(H_{N_0,L_0}).
\end{aligned}
\end{equation}
We denote by $\mathcal{R}^{m,p;1}(\gamma_1,\ldots,\gamma_m,\bm{\xi};\eta_p,\ldots,\eta_1,\bm{\xi}')$ the moduli space of two-pointed open-closed $(m,p)$-discs in $M$ which satisfy boundary conditions along $L_0,\ldots,L_m,N_p,\ldots,N_0$ and asymptotic conditions along $\gamma_1,\ldots,\gamma_m,\bm{\xi},\eta_p,\ldots,\eta_1,\bm{\xi}'$ (see Figure \ref{FIG: Natural transformation CY curve configurations in M}). We denote its $s$-dimensional component by $\mathcal{R}^{m,p;1}(\gamma_1,\ldots,\gamma_m,\bm{\xi};\eta_p,\ldots,\eta_1,\bm{\xi}')^s$. 

\begin{figure}
  \centering
  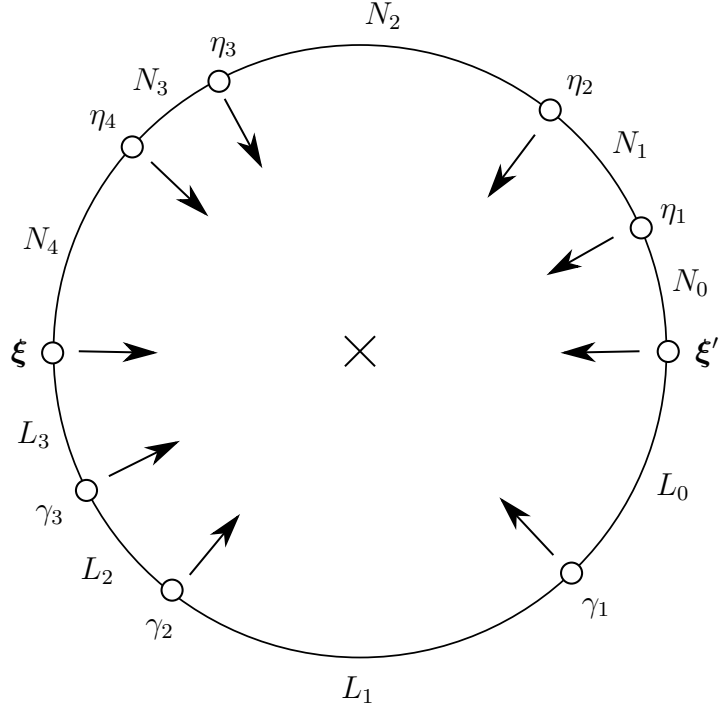
  \caption[An inhomogeneous pseudoholomorphic polygon in \hfill \break $\mathcal{R}^{3,4;1}(\gamma_1,\gamma_2,\gamma_3,\bm{\xi};\eta_4,\ldots,\eta_1,\bm{\xi}')$]{An inhomogeneous pseudoholomorphic polygon in $M$ in the moduli space $\mathcal{R}^{3,4;1}(\gamma_1,\gamma_2,\gamma_3,\bm{\xi};\eta_4,\ldots,\eta_1,\bm{\xi}')$ used to define the representative $\delta^\mathcal{F}$ in $\mathit{fun}(\mathcal{F},\mathcal{F}\mathit{\mbox{--}mod})(\mathbf{Y}^l_\mathcal{F},(\mathbf{Y}^\vee_\mathcal{F})^l)$ of the weak Calabi-Yau structure on $\mathcal{F}$}
  \label{FIG: Natural transformation CY curve configurations in M}
  \end{figure}

The weak Calabi-Yau structure on $\mathcal{F}$ is represented by the natural transformation
\begin{equation}
\delta^\mathcal{F}=(\delta^\mathcal{F}_0,\delta^\mathcal{F}_1,\ldots): \mathbf{Y}_\mathcal{F}^l\to (\mathbf{Y}_\mathcal{F}^\vee)^l,
\end{equation}
defined as follows. The map
\begin{equation}
\delta^\mathcal{F}_m:\mathcal{F}(N_p,\ldots,N_0)\to \mathcal{F}\mathit{\mbox{--}mod}(\mathbf{Y}_\mathcal{F}^l(N_p),(\mathbf{Y}_\mathcal{F}^\vee)^l(N_0))
\end{equation} 
applied to the Hamiltonian chords $\eta_j\in \mathcal{O}(H_{N_j,N_{j-1}})$, $j=p,\ldots,1$, gives a pre-morphism of left $\mathcal{F}$-modules
\begin{equation}
\delta^\mathcal{F}_p(\eta_p,\ldots, \eta_1):\mathbf{Y}_\mathcal{F}^l(N_p)\to (\mathbf{Y}_\mathcal{F}^\vee)^l(N_0).
\end{equation}
This in turn is specified by maps
\begin{equation}
\delta^\mathcal{F}_p(\eta_p,\ldots, \eta_1)_{m|1}:\mathcal{F}(L_0,\ldots, L_m)\otimes \mathcal{F}(L_m,N_p)\to (\mathcal{F}(N_0,L_0))^\vee. 
\end{equation}
On the Hamiltonian chords $\gamma_i\in \mathcal{O}(H_{L_{i-1},L_i})$, $i=1,\ldots,m$, $\bm{\xi}\in \mathcal{O}(H_{L_m,N_p})$, and $\bm{\xi}'\in \mathcal{O}(H_{N_0,L_0})$, the map $\delta^\mathcal{F}_p(\eta_p,\ldots, \eta_1)_{m|1}$ is given by
\begin{equation}
\begin{aligned}
&\langle\delta^\mathcal{F}_p(\eta_p,\ldots, \eta_1)_{m|1}(\gamma_1,\ldots,\gamma_m,\bm{\xi}),\bm{\xi}'\rangle\\
&\qquad\qquad=\#_{\Z_2} \mathcal{R}^{m,p;1}(\gamma_1,\ldots,\gamma_m,\bm{\xi};\eta_p,\ldots,\eta_1,\bm{\xi}')^0. 
\end{aligned}
\end{equation}
This description defines a pre-natural transformation which can be seen to be a natural transformation, i.e.\ $\mu_1^{\mathit{fun}(\mathcal{F},\mathcal{F}\mathit{\mbox{--}mod})}(\delta^\mathcal{F})=0$, by examining the configurations which appear as the boundary of the compactification of $\mathcal{R}^{m,p;1}(\gamma_1,\ldots,\gamma_m,\bm{\xi}';\eta_p,\ldots,\eta_1,\bm{\xi})^1$. Moreover the fact that $\delta^\mathcal{F}$ represents an isomorphism in the category $H(\mathit{fun}(\mathcal{F},\mathcal{F}\mathit{\mbox{--}mod}))$ again follows from the fact that the chain-level comparison map \eqref{EQ: Floer homology comparison map} is a quasi-isomorphism. 

\begin{remk}\label{REMK: The wCY structure as a bimodule quasi-iso}
The bimodule quasi-isomorphism $\phi^\mathcal{F}:\mathcal{F}_\Delta\to\mathcal{F}^\vee$ corresponding to $\delta^\mathcal{F}$ under the isomorphism $\Phi^l$ of $\eqref{EQ: Category isos between functors into left modules and bimodules}$ is given by
\begin{equation}
\begin{aligned}
&\phi^\mathcal{F}_{m|1|p}:\mathcal{F}(L_0,\ldots,L_m)\otimes\mathcal{F}(L_m,N_p)\otimes \mathcal{F}(N_p,\ldots,N_0)\to (\mathcal{F}(N_0,L_0))^\vee,\\
&\langle\phi^\mathcal{F}_{m|1|p}(\gamma_1,\ldots,\gamma_m,\bm{\xi},\eta_p,\ldots,\eta_1),\bm{\xi}'\rangle\\
&\qquad\qquad = \#_{\Z_2} \mathcal{R}^{m,p;1}(\gamma_1,\ldots,\gamma_m,\bm{\xi};\eta_p,\ldots,\eta_1,\bm{\xi}')^0.
\end{aligned}
\end{equation}
Moreover, the chain-level map $\phi^\mathcal{F}_{0|1|0}$ can be identified with the Poincar\'{e} duality quasi-isomorphism \eqref{EQ: Poincare duality quasi-iso for CF total formula}. To see this, first note that a homotopy $(\mathbf{H},\mathbf{J})$ interpolating between Floer data $(H_{N,L},J_{N,L})$ and $(\bar{H}_{L,N},\bar{J}_{L,N})$ as in \eqref{EQ: Homotopy from (H,J)_N,L to bar(H,J)_L,N} determines a perturbation datum for defining $\mathcal{R}^{0,0;1}(\gamma;\xi)$ (we define perturbation data and the associated inhomogeneous Cauchy-Riemann equation in the more general setting of Fukaya categories of Lagrangian cobordisms in Section \ref{SUBSECTION: Floer data and perturbation data}). With this perturbation datum, there is an identification between $\mathcal{R}^{0,0;1}(\gamma;\xi)$ and the moduli space $\mathcal{R}(\gamma,\xi)$ in \eqref{EQ: Poincare duality quasi-iso for CF total formula}.
\end{remk}

\subsubsection{Relating the two descriptions}

As remarked in \cite[Remark 2.7]{SheridanFano}, the descriptions of the weak Calabi-Yau structure on $\mathcal{F}$ in the previous two sections are related via the quasi-isomorphism from the two-pointed Hochschild chain complex to the ordinary Hochschild chain complex. We spell this out in the following lemma which is a simple consequence of a version of Proposition 5.6 in \cite{GanatraThesis}. The proof, which we omit, uses the same idea as the proof of the corresponding result for the relative weak Calabi-Yau pairing on the Lagrangian cobordism Fukaya category (see part \eqref{THM: Main theorem, equivalence part} of Theorem \ref{THM: Main theorem}).

\begin{lem}\label{LEM: Relating geometric wCY structures}
The isomorphism induced on homology by the quasi-isomorphism of chain complexes
\begin{equation}
CC_\bullet(\mathcal{F})^\vee \xrightarrow{T^\vee} {}_2CC_\bullet(\mathcal{F})^\vee\xrightarrow{\Gamma} {}_2CC^\bullet(\mathcal{F},\mathcal{F}^\vee), 
\end{equation}
takes $[\sigma^\mathcal{F}]$ to $[\phi^\mathcal{F}]$. Here $T^\vee$ is the quasi-isomorphism dual to \eqref{EQ: iso from two-pointed cochain complex to one-pointed cochain complex}, and $\Gamma$ is the isomorphism of Lemma \eqref{LEM: This iso 2CCn(A,Mvee) cong 2CC_n(A,M)vee}.
\end{lem}

\begin{remk}
If we work instead with the graded Fukaya category, which is an $A_\infty$-category of degree $n=\tfrac{1}{2}\mathrm{dim}(M)$, the weak Calabi-Yau structure has the appropriate degree, i.e.\ $|\phi^\mathcal{F}|=|\delta^\mathcal{F}|=-n$ and $\sigma^\mathcal{F}\in HH_n(\mathcal{F})^\vee$. Moreover, this construction of the weak Calabi-Yau structure is valid for coefficients in an arbitrary field.  
\end{remk}

\begin{remk}
The usual description of the weak Calabi-Yau structure on $\mathcal{F}$ is in terms of open-closed maps (see \cite{SheridanFano}). The two incarnations $\sigma^\mathcal{F}$ and $\phi^\mathcal{F}$ of the structure on $\mathcal{F}$ can be described using the ordinary open-closed map and the two-pointed open-closed map respectively: 
\begin{equation}
\mathcal{OC}: CC_\bullet(\mathcal{F})\to (\mathcal{C}_\bullet(f),*),\quad {}_2\mathcal{OC}: {}_2CC_\bullet(\mathcal{F})\to (\mathcal{C}_\bullet(f),*).
\end{equation}
Here $(\mathcal{C}_\bullet(f),*)$ is the Morse complex associated to a Morse function $f$ on $M$ with coefficients in $\Z_2$, equipped with the quantum product, i.e.\ $(\mathcal{C}_\bullet(f),*)$ computes the quantum homology of $M$ with $\Z_2$ coefficients. The Hochschild homology $HH_\bullet(\mathcal{F})$ has a natural $QH_\bullet(M)$-module structure, where the chain level description of the action can be described on either the ordinary or two-pointed Hochschild chain complexes. Both $\mathcal{OC}$ and $_2\mathcal{OC}$ induce maps of $QH_\bullet(M)$-modules on homology. The maps ${}_2\mathcal{OC}$ and $\mathcal{OC}$ are defined by counting the same kind of configurations as in Figures \ref{FIG: Hochschild CY curve configurations in M} and \ref{FIG: Natural transformation CY curve configurations in M}, but additionally geometric constraints are placed on the image of the marked point at zero. The representatives $\sigma^\mathcal{F}$ and $\phi^\mathcal{F}$ of the weak Calabi-Yau structure on $\mathcal{F}$ satisfy
\begin{equation}
\sigma^\mathcal{F}(a)=\langle \mathcal{OC}(a),e\rangle,\quad (\Gamma^{-1}(\phi^\mathcal{F}))(b)=\langle {}_2\mathcal{OC}(b),e\rangle,
\end{equation}
for all $a\in CC_\bullet(\mathcal{F})$ and $b\in {}_2CC_\bullet(\mathcal{F})$, where $e\in \mathcal{C}_\bullet(f)$ is a representative of the fundamental class in $QH_\bullet(M)$. Lemma \ref{LEM: Relating geometric wCY structures} is a reflection of the fact that up to homotopy, $\mathcal{OC}\circ T={}_2\mathcal{OC}$, which is a version of Proposition 5.6 in \cite{GanatraThesis}.  
\end{remk}

\section{The Fukaya category of cobordisms and its relative weak Calabi-Yau pairing}\label{SECTION: Fcob and its relative wCY pairing}
In this section, we will recall the construction of the Fukaya category of cobordisms from \cite{BC14} while introducing the new ingredients needed to define the relative weak Calabi-Yau pairing. We will view cobordisms as Lagrangian submanifolds of $\widetilde{M}=\C\times M$ by identifying cobordisms in $([0,1]\times\R)\times M$ with their trivial $\R$-extensions. 

The objects of the Fukaya category of cobordisms $\mathcal{F}uk_{cob}(\widetilde{M})$ are connected cobordisms $W\subset \widetilde{M}$ with $d_W=d$, for a fixed $d\in\Z_2$, which are uniformly monotone and satisfy the condition
\begin{equation}
\pi_1(W)\xrightarrow{\iota_*}\pi_1(\widetilde{M})\text{ is trivial. }
\end{equation}
Moreover, we assume that the Lagrangians in $M$ forming the ends of the cobordism are in $\mathcal{L}_d^*(M)$.
We denote this class of cobordisms by $\mathcal{CL}_d(\widetilde{M})$. The construction of $\mathcal{F}uk_{cob}(\widetilde{M})$ proceeds along the same lines as the construction of $\mathcal{F}uk(M)$, but with changes to account for the compactness issues that arise from the non-compactness of the cobordisms themselves. 

In what follows, $M$ will be fixed and we will denote $\mathcal{F}uk_{cob}(\widetilde{M})$ by $\mathcal{F}_{c}$.

\subsection{Pointed discs, strip-like ends, and sign data}\label{SUBSECTION: Pointed discs, strip-like ends, and sign data}

Fix $k\ge 2$. We first describe the domains of the inhomogeneous pseudoholomorphic curves we will be interested in. This description appears in \cite{Sei} and \cite{BC14} for those curves contributing to the operations in the Fukaya category. Let $D$ denote the unit disc in $\C$ and define $\mathrm{Conf}_{k+1}(\partial D)\subset (\partial D)^{k+1}$ to be the space of configurations of $k+1$ distinct points $(z_1,\ldots,z_{k+1})$ in $\partial D$ ordered clockwise. The group $\mathrm{Aut}(D)$ of biholomorphisms of the disc acts freely and properly on $\mathrm{Conf}_{k+1}(\partial D)$ and we define 
\begin{align*}
\mathcal{R}^{k+1}&= \mathrm{Conf}_{k+1}(\partial D)/\mathrm{Aut}(D),\\
\widehat{\mathcal{S}}^{k+1}&=\mathrm{Conf}_{k+1}(\partial D)\times_{\mathrm{Aut}(D)} D.
\end{align*}
The projection $\widehat{\mathcal{S}}^{k+1}\to \mathcal{R}^{k+1}$ has sections $\zeta_i([z_1,\ldots,z_{k+1}])=[z_1,\ldots,z_{k+1},z_i]$. We set $\mathcal{S}^{k+1} = \widehat{\mathcal{S}}^{k+1}\setminus \cup_{i}\zeta_i(\mathcal{R}^{k+1})$. The fibre bundle $\mathcal{S}^{k+1}\to \mathcal{R}^{k+1}$ is called a \textbf{universal family of $(k+1)$-pointed discs}, and its fibres $\mathcal{S}^{k+1}_r$, $r\in \mathcal{R}^{k+1}$, are called \textbf{$(k+1)$-pointed discs}. We extend this definition to the case $k=1$ by setting $\mathcal{R}^2=\{0\}$ and $\mathcal{S}^2=D\setminus\{-1,1\}$. 

\sloppy We also consider discs with an interior marked point. Fix $m\ge 0$  and define $\mathrm{Conf}_{m+1;1}(\partial D;\mathrm{Int}(D))\subset (\partial D)^{m+1}\times \mathrm{Int}(D)$ to be the space of configurations of $m+2$ distinct points $(z_1,\ldots, z_{m+1},y)$ in $D$, where $z_1,\ldots, z_{m+1}\in \partial D$ are ordered clockwise and $y\in \mathrm{Int}(D)$. Again, the group $\mathrm{Aut}(D)$ acts freely and properly on $\mathrm{Conf}_{m+1;1}(\partial D;\mathrm{Int}(D))$ and we define 
\begin{align*}
\mathcal{R}^{m+1;1}&= \mathrm{Conf}_{m+1;1}(\partial D;\mathrm{Int}(D))/\mathrm{Aut}(D),\\
\widehat{\mathcal{S}}^{m+1;1}&=\mathrm{Conf}_{m+1;1}(\partial D;\mathrm{Int}(D))\times_{\mathrm{Aut}(D)} D.
\end{align*}
The projection $\widehat{\mathcal{S}}^{m+1;1}\to \mathcal{R}^{m+1;1}$ has sections $\xi_i([z_1,\ldots, z_{m+1},y])=[z_1,\ldots, z_{m+1},y,z_i]$.
We set 
$\mathcal{S}^{m+1;1} = \widehat{\mathcal{S}}^{m+1;1}\setminus \cup_{i}\xi_i(\mathcal{R}^{m+1;1})$.
The fibre bundle $\mathcal{S}^{m+1;1}\to \mathcal{R}^{m+1;1}$ is called a \textbf{universal family of $(m+1)$-pointed discs with one interior marked point}, and its fibres are called \textbf{$(m+1)$-pointed discs with one interior marked point}. For $k\ge 2$, there is a natural projection $p_{k+1}:\mathcal{R}^{k+1;1}\to \mathcal{R}^{k+1}$ and we have $\mathcal{S}^{k+1;1}=p_{k+1}^* \mathcal{S}^{k+1}$. 

Now fix $m,p\ge 0$. We recall from \cite{GanatraThesis} the definition of two-pointed open-closed discs. Consider the space $\mathcal{R}^{m+p+2;1}$ and assume the elements of $\mathrm{Conf}_{m+p+2;1}(\partial D;\mathrm{Int}(D))$ are labelled $(z_1,\ldots, z_m,z',w_p,\ldots,w_1,w',y)$, where $z_1,\ldots, z_m,z',w_p,\ldots,w_1,w'\in \partial D$ are ordered clockwise and $y\in \mathrm{Int}(D)$. We define $\mathcal{R}^{m,p;1}$ to be the codimension-one submanifold of $\mathcal{R}^{m+p+2;1}$ consisting of those $[z_1,\ldots, z_m,z',w_p,\ldots,w_1,w',y]\in \mathcal{R}^{m+p+2;1}$ where $z'=-1, w'=1,$ and $y=0$. We set  
$\mathcal{S}^{m,p;1}=\mathcal{S}^{m+p+2;1}|_{\mathcal{R}^{m,p;1}}$. This bundle is the universal family of \textbf{two-pointed open-closed discs with $(m,p)$ punctures}. The points $z'$ and $w'$ are referred to as \textbf{special points}. 
 
We now recall from \cite{Sei} the notion of a universal choice of strip-like ends. Fix a subset $\mathcal{I}^{k+1}\subset\{1,\ldots, k+1\}$, $k\ge 1$. Set $Z^+= [0,\infty)\times [0,1]$, $Z^-=(-\infty,0]\times[0,1]$. Let $S$ be a $(k+1)$-pointed disc with boundary punctures $(z_1,\ldots, z_{k+1})$. A choice of \textbf{strip-like ends} for $S$ with inputs $\mathcal{I}^{k+1}$ is a collection of proper, holomorphic embeddings
\begin{align*}
&\epsilon^S_i:Z^-\to S,\;i\in \mathcal{I}^{k+1},\\
&\epsilon^S_i:Z^+\to S,\;i\in \{1,\ldots, k+1\}\setminus\mathcal{I}^{k+1},
\end{align*}
which satisfy the conditions
\begin{enumerate}
\item For $i\in\mathcal{I}^{k+1}$, $(\epsilon^S_i)^{-1}(\partial S)=(-\infty,0]\times\{0,1\}$ and $\lim_{s\to-\infty}\epsilon_i^S(s,t)=z_i$.
\item For $i\in\{1,\ldots, k+1\}\setminus\mathcal{I}^{k+1}$, $(\epsilon^S_i)^{-1}(\partial S)=[0,\infty)\times\{0,1\}$ and $\lim_{s\to\infty}\epsilon_i^S(s,t)=z_i$.
\end{enumerate}
We also require that the $\epsilon_i^S$ have pairwise disjoint images. Those punctures $z_i$ for $i\in \mathcal{I}^{k+1}$ are referred to as \textbf{inputs} or \textbf{entries}, and $z_{i}$ for $i\in \{1,\ldots, k+1\}\setminus \mathcal{I}^{k+1}$ are referred to as \textbf{outputs} or \textbf{exits}. 

A \textbf{universal choice of strip-like ends} with inputs $\mathcal{I}^{k+1}$ for the bundle $\mathcal{S}^{k+1}\to \mathcal{R}^{k+1}$ is a choice of embeddings
\begin{align*}
&\epsilon_i:\mathcal{R}^{k+1}\times Z^-\to \mathcal{S}^{k+1},\;i\in \mathcal{I}^{k+1},\\
&\epsilon_i:\mathcal{R}^{k+1}\times Z^+\to \mathcal{S}^{k+1},\;i\in \{1,\ldots, k+1\}\setminus\mathcal{I}^{k+1},
\end{align*}
whose restrictions $\epsilon_i |_{r\times Z^{\pm}}$ for any $r\in \mathcal{R}^{k+1}$ are a choice of strip-like ends for $\mathcal{S}^{k+1}_r$ with inputs $\mathcal{I}^{k+1}$.

Similarly one defines strip-like ends for $(m+1)$-pointed discs with one interior marked point ($m\ge 0$) and for two-pointed open-closed discs with $(m,p)$ punctures ($m,p\ge 0$), but there we include the additional requirement that the interior marked point not be contained in the union of the images of the $\epsilon^S_i$. These definitions can likewise be extended to universal choices of strip-like ends for the bundles $\mathcal{S}^{m+1;1}$ and $\mathcal{S}^{m,p;1}$. A universal choice of strip-like ends for $\mathcal{S}^{m+p+2;1}$ determines one for $\mathcal{S}^{m,p;1}$ by restriction. 

\begin{remk}
We note that a universal choice of strip-like ends $\{\epsilon_i\}$ for some fixed input set $\mathcal{I}^{k+1}$ determines a universal choice of strip-like ends for any input set $(\mathcal{I}')^{k+1}$. Indeed, a negative strip-like end $\epsilon_i$, i.e.\ a strip-like end $\epsilon_i$ for $i\in \mathcal{I}^{k+1}$, can be changed to a positive strip-like end $\epsilon_i'$ (and vice versa) by defining $\epsilon_i'(s,t)=\epsilon(-s,1-t)$. 
\end{remk}

A pair of pointed discs equipped with strip-like ends can be glued along an input for one disc and an output for the other disc to produce another pointed disc. This operation is defined as follows. Fix $\rho>0$. Let $S$ be a $(k+1)$-pointed disc with strip-like ends $\{\epsilon^S_i\}$ and $S'$ be a $(k'+1)$-pointed disc with strip-like ends $\{\epsilon_j^{S'}\}$. Assume that $z_{i_0}$ is an input puncture for $S$ and $z'_{j_0}$ is an output puncture for $S'$. The glued disc $S\#_\rho S'$ with gluing parameter $\rho$ is defined by first taking the disjoint union obtained by removing a piece of each disc around the punctures $z_{i_0}$ and $z'_{j_0}$ 
$$(S\setminus \epsilon^S_{i_0}((-\infty,-\rho)))\coprod(S'\setminus\epsilon^{S'}_{j_0}((\rho,\infty))),$$ 
and then making the identifications
$$\epsilon_{i_0}^S(s-\rho,t)\sim \epsilon_{j_0}^{S'}(s,t),\; (s,t)\in [0,\rho]\times [0,1].$$
The surface $S\#_\rho S'$ inherits a complex structure from $S$ and $S'$ and has $k+k'$ boundary punctures. Therefore it is biholomorphic to a unique fibre of the bundle $\mathcal{S}^{k+k'}\to \mathcal{R}^{k+k'}$. Moreover, the strip-like ends $\{\epsilon^S_i\}$ and $\{\epsilon_j^{S'}\}$ induce strip-like ends on the glued disc $S\#_\rho S'$. The glued disc $S\#_\rho S'$ comes equipped with a thick-thin decomposition, given by
\begin{align*}
(S\#_\rho S')^{thin}&=\coprod_{i\ne i_0}\epsilon^S_i(Z^{\pm})\sqcup\coprod_{j\ne j_0}\epsilon^{S'}_j(Z^\pm)\coprod\epsilon_{i_0}([-\rho,0]\times[0,1]),\\
(S\#_\rho S')^{thick}&=(S\#_\rho S')\setminus(S\#_\rho S')^{thin}.
\end{align*} 

Similarly a pointed disc $S''$ with one interior marked point can be glued to the pointed disc $S$ along an input for $S''$ and an output for $S$ (or vice versa). The glued disc $S\#_\rho S''$ is a pointed disc with one interior marked point. Again, strip-like ends on $S$ and $S''$ induce strip-like ends on $S\#_\rho S'$, and $S\#_\rho S'$ has a thick-thin decomposition. We note however that the gluing of a two-pointed open-closed disc to a pointed disc does not necessarily produce a two-pointed open-closed disc, but rather only a pointed disc with one interior marked point. This is because, after gluing, the resulting disc may not be biholomorphic to one where the special punctures and interior marked point lie at $-1$, $1$ and $0$ respectively. However, as the gluing parameter $\rho$ approaches zero, this configuration still degenerates to a two-pointed open-closed disc joined to a pointed disc at a boundary point.

The spaces $\mathcal{R}^{k+2}$ ($k\ge 2$) and $\mathcal{R}^{m+1;1}$ ($m\ge 0$) have natural compactifications (see \cite{Sei}). The strata of the compactification $\overline{\mathcal{R}}^{k+1}$ of $\mathcal{R}^{k+1}$ are indexed by stable trees with $k+1$ exterior edges, and those of the compactification $\overline{\mathcal{R}}^{m+1;1}$ of $\mathcal{R}^{m+1;1}$ are indexed by stable trees with one distinguished vertex and $m+1$ exterior edges. We note that the stability condition for the distinguished vertex only requires it to have one adjacent edge, as opposed to three for the non-distinguished vertices. As a submanifold of $\mathcal{R}^{m+p+2;1}$, the space $\mathcal{R}^{m,p;1}$ can be compactified by adding strata which are indexed by a subset of the set of stable trees with $m+p+2$ exterior edges and one distinguished vertex. The bundles $\mathcal{S}^{k+1}$ and $\mathcal{S}^{m+1;1}$ have partial compactifications $\overline{\mathcal{S}}^{k+1}$ and $\overline{\mathcal{S}}^{m+1;1}$. These are no longer fibre bundles, but only smooth manifolds with corners. They are equipped with smooth submersions $\overline{\mathcal{S}}^{k+1}\to \overline{\mathcal{R}}^{k+1}$ and $\overline{\mathcal{S}}^{m+1;1}\to \overline{\mathcal{R}}^{m+1;1}$ that extend the respective projections. It follows that the subbundle $\mathcal{S}^{m,p;1}\subset \mathcal{S}^{m+p+2;1}$ also has a partial compactification $\overline{\mathcal{S}}^{m,p;1}\subset \overline{\mathcal{S}}^{m+p+2;1}$ together with a smooth submersion $\overline{\mathcal{S}}^{m,p;1}\to\overline{\mathcal{R}}^{m,p;1}$ extending the projection. 

For the structures we will consider, we will limit ourselves to the following fixed input sets: 
\begin{enumerate}
\item For $\mathcal{S}^{k+1}$, $k\ge 1$, $\mathcal{I}_{\mu}^{k+1}:=\{1,\ldots,k\}\subset \{1,\ldots,k+1\}$. 
\item For $\mathcal{S}^{m+1;1}$, $m\ge 0$, $\mathcal{I}_{CY}^{m+1}=\{1,\ldots,m+1\}$.
\item For $\mathcal{S}^{m,p;1}$, $m,p\ge 0$, $\mathcal{I}^{m+p+2}_{CY}$.
\end{enumerate}

We will assume that the strip-like ends on $\mathcal{S}^2$ for the input set $\mathcal{I}_\mu^2$ are those coming from a fixed choice of biholomorphic map between the strip $\R\times [0,1]$ and $\mathcal{S}^2$.

Given universal choices of strip-like ends $\{\epsilon_i^{k+1}\}_{i=1,\ldots,k+1}$ for $\mathcal{S}^{k+1}$ with inputs $\mathcal{I}^{k+1}_\mu$ for all $k\ge 2$ and universal choices of strip-like ends $\{\epsilon_j^{m+1;1}\}_{j=1,\ldots,m+1}$ for $\mathcal{S}^{m+1;1}$ with inputs $\mathcal{I}_{CY}^{m+1}$ for all $m\ge 0$, there is a notion of consistency of these choices with respect to gluing. Here consistency refers to consistency among all of the $\{\epsilon_i^{k+1}\}$ and $\{\epsilon_j^{m+1;1}\}$ for different choices of $k$ and $m$. This is a minor extension of the notion introduced in \cite{Sei}. In the present context, the consistency condition requires the choices of strip-like ends near the boundary of $\overline{\mathcal{R}}^{k+1}$ and of $\overline{\mathcal{R}}^{m+1;1}$ to be equal to those obtained by gluing configurations of discs associated to stable trees with $k+1$ exterior edges, or respectively to stable trees with one distinguished vertex and $m+1$ exterior edges. By arguments similar to those in \cite{Sei}, consistent universal choices of strip-like ends for $\mathcal{S}^{k+1}$ with inputs $\mathcal{I}^{k+1}_\mu$ and for $\mathcal{S}^{m+1;1}$ with inputs $\mathcal{I}_{CY}^{m+1}$ exist. Similarly there is a notion of consistency between universal choices of strip-like ends for $\mathcal{S}^{k+1}$ with inputs $\mathcal{I}^{k+1}_{\mu}$ and for $\mathcal{S}^{m,p;1}$ with inputs $\mathcal{I}_{CY}^{m+p+2}$. 
The consistency condition is satisfied by taking the universal choice of strip-like ends on $\mathcal{S}^{m,p;1}$ to be that given by restriction from $\mathcal{S}^{m+p+2;1}$. 

We also require a choice of metrics $\rho^{k+1}$ on $\overline{\mathcal{S}}^{k+1}$ and $\rho^{m+1;1}$ on $\overline{\mathcal{S}}^{m+1;1}$. Like the universal choices of strip-like ends, these metrics must be consistent with respect to gluing. Furthermore, we assume there exist constants $A_{k+1}$, $k\ge 2$, and $B_{m+1}$, $m\ge 0$, such that  
\begin{align}
&\mathrm{length}_{\rho^{k+1}}(\partial \mathcal{S}^{k+1}_r)\le A_{k+1}\text{ for all }r\in\mathcal{R}^{k+1}\text{ and all } k\ge 2,\\
&\mathrm{length}_{\rho^{m+1;1}}(\partial \mathcal{S}^{m+1;1}_{r'})\le B_{m+1}\text{ for all }r'\in\mathcal{R}^{m+1;1}\text{ and all } m\ge 0.
\end{align}
Note that the metric $\rho^{m+p+2;1}$ induces a metric $\rho^{m,p;1}$ on $\overline{\mathcal{S}}^{m,p;1}$ for $m,p\ge 0$ which is consistent with the metrics $\rho^{k+1}$ and satisfies
\begin{equation}
\mathrm{length}_{\rho^{m,p;1}}(\partial \mathcal{S}^{m,p;1}_{r''})\le B_{m+p+2}\text{ for all }r''\in\mathcal{R}^{m,p;1}\text{ and all } m,p\ge 0.
\end{equation}
These metrics play a role in proving the energy estimates in Lemma \ref{LEM: Energy bounds}.

We will also associate to a universal family of discs $\mathcal{S}^{k+1}$ or $\mathcal{S}^{k+1;1}$, with fixed choice of inputs a \textbf{$(k+1)$ sign datum} $\bar{\alpha}=(\alpha_{1},\ldots, \alpha_{k+1})$ consisting of an ordered family of $k+1$ signs $\alpha_i\in\{-,+\}$. The $\alpha_i$ indicate which of two types of Floer data will be used for the pairs of Lagrangians associated to the $i$th puncture, as we describe in Section \ref{SUBSECTION: Floer data and perturbation data}. Similarly, we associate to a universal family of two-pointed open-closed discs with $(k,m)$ punctures and fixed choice of inputs a $(k+m+2)$ sign datum. Note that sign data do not appear in \cite{BC14}, but are needed in the present context to describe the duality structure on $\mathcal{F}_c$.

\subsection{Transition functions}\label{SUBSECTION: Transition functions}

Following \cite{BC14}, we require an ingredient that does not appear in \cite{Sei} -- so-called transition functions. These are used in defining the naturality transformation that arises in the proof of compactness for the moduli spaces of curves in $\widetilde{M}$ that we will consider. The transition functions we define here are a slight generalization of those appearing in \cite{BC14}. 

Let $S$ be a $(k+1)$-pointed disc with or without interior marked point. We assume $k\ge 1$
in the case without interior marked point and $k\ge 0$ in the case with interior marked point. Fix inputs $\mathcal{I}^{k+1}\subset\{1,\ldots,k+1\}$ and equip $S$ with a choice of strip-like ends $\{\epsilon_i^S\}_{i=1,\ldots,k+1}$ for the input set $\mathcal{I}^{k+1}$ and a $(k+1)$ sign datum $\bar{\alpha}$. A \textbf{transition function} for $S$ is a smooth function $a:S\to [0,1]$ satisfying the following conditions (see Figure \ref{FIG: Transition functions}): 
\begin{enumerate}
\item\label{DEF Trans fun 1} For $i\in\mathcal{I}^{k+1}$ with $\alpha_i=+$, we have:
\begin{enumerate}
\item\label{DEF Trans fun 1a} $a\circ \epsilon_i^S(s,t)=t,\;\text{for all }(s,t)\in (-\infty,-1]\times [0,1].$
\item $\partial_s(a\circ \epsilon_i^S)(s,1)\le 0\text{ for }s\in [-1,0].$
\item $a\circ \epsilon_i^S(s,t)= 0\text{ for }(s,t)\in ((-\infty,0]\times\{0\})\cup(\{0\}\times[0,1]).$
\end{enumerate}
\item\label{DEF Trans fun 2} For $i\in\{1,\ldots,k+1\}\setminus \mathcal{I}^{k+1}$ with $\alpha_i=+$, we have:
\begin{enumerate}
\item $a\circ \epsilon_i^S(s,t)=t,\;\text{for all }(s,t)\in [1,\infty)\times [0,1].$
\item $\partial_s(a\circ \epsilon_i^S)(s,1)\geq 0\text{ for }s\in [0,1].$
\item $a\circ \epsilon_i^S(s,t)= 0\text{ for }(s,t)\in ([0,+\infty)\times\{0\})\cup(\{0\}\times[0,1]).$
\end{enumerate}
\item\label{DEF Trans fun 3} For $i\in\mathcal{I}^{k+1}$ with $\alpha_i=-$, we have:
\begin{enumerate}
\item $a\circ \epsilon_i^S(s,t)=1-t,\;\text{for all }(s,t)\in (-\infty,-1]\times [0,1].$
\item $\partial_s(a\circ \epsilon_i^S)(s,0)\le 0\text{ for }s\in [-1,0].$
\item $a\circ \epsilon_i^S(s,t)= 0\text{ for }(s,t)\in ((-\infty,0]\times\{1\})\cup(\{0\}\times[0,1]).$
\end{enumerate}
\item\label{DEF Trans fun 4} For $i\in\{1,\ldots,k+1\}\setminus \mathcal{I}^{k+1}$ with $\alpha_i=-$, we have:
\begin{enumerate}
\item $a\circ \epsilon_i^S(s,t)=1-t,\;\text{for all }(s,t)\in [1,\infty)\times [0,1].$
\item $\partial_s(a\circ \epsilon_i^S)(s,0)\geq 0\text{ for }s\in [0,1].$
\item $a\circ \epsilon_i^S(s,t)= 0\text{ for }(s,t)\in ([0,+\infty)\times\{1\})\cup(\{0\}\times[0,1]).$
\end{enumerate}
\end{enumerate}

\begin{figure}
\def\svgwidth{10.7cm}
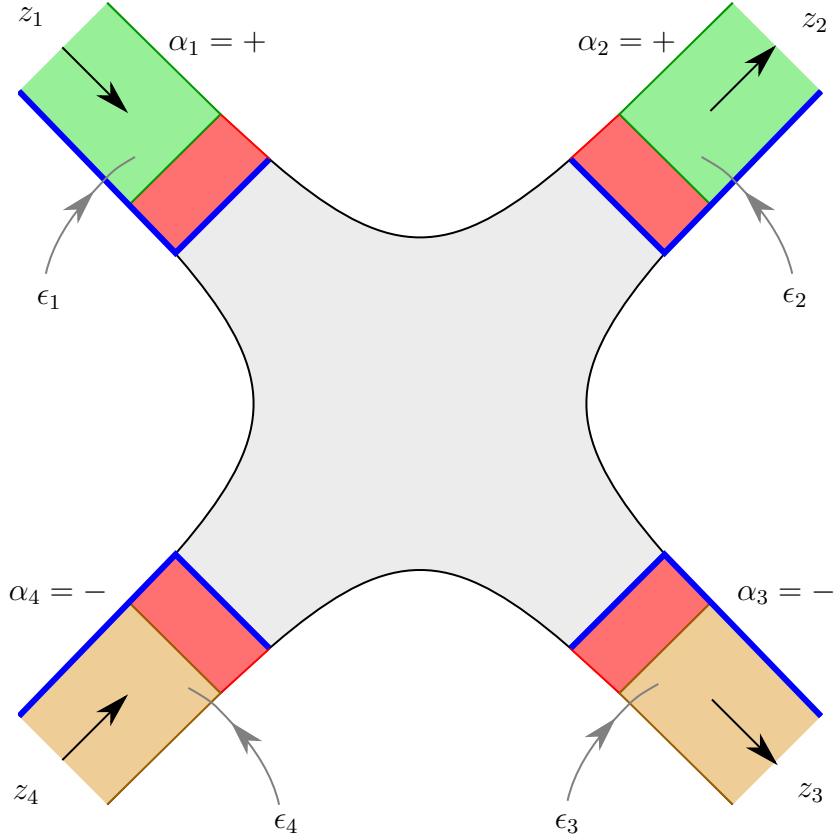
\caption[The conditions on transition functions]{A 4-pointed disc with strip-like ends $\epsilon_1,\ldots,\epsilon_4$ illustrating the conditions imposed on transition functions.
 The punctures are of the following types: $z_1$ is an input with $\alpha_1=+$, $z_2$ is an output with $\alpha_2=+$, $z_3$ is an output with $\alpha_3=-$, and $z_4$ is an input with $\alpha_4=-$. On the green regions the transition function is equal to $(s,t)\mapsto t$; on the brown regions, it is equal to $(s,t)\mapsto 1-t$; along the blue lines, the function vanishes; the red regions are transition zones. There are no conditions imposed on the function in the grey region.}
\label{FIG: Transition functions}
\end{figure}

We consider now the case without interior marked point. Assume we have a universal choice of strip-like ends $\epsilon_i$ for $\mathcal{S}^{k+1}$ with input set $\mathcal{I}^{k+1}$. A \textbf{global transition function} for $\mathcal{S}^{k+1}$ with sign datum $\bar{\alpha}$ is a smooth function $\bm{a}:\mathcal{S}^{k+1}\to [0,1]$ such that $a_r:=\bm{a}|_{\mathcal{S}_r}$ is a transition function on $\mathcal{S}_r$ for each $r\in \mathcal{R}^{k+1}$. Global transition functions are defined similarly for $\mathcal{S}^{k+1;1}$ and $\mathcal{S}^{k,m;1}$.

Let $S$ be a $(k+1)$-pointed disc with $(k+1)$ sign datum $\bar{\alpha}$ and transition function $a:S\to [0,1]$, and let $S'$ be a $(k'+1)$-pointed disc with $(k'+1)$ sign datum $\bar{\beta}$ and transition function $b:S\to [0,1]$. We allow for the possibility that one of $S$ or $S'$ has an interior marked point. Fix $\rho\in (1,\infty)$ and consider the disc $S\#_\rho S'$ obtained by gluing along the $i$th puncture for $S$, which we assume to be an input, and the $j$th puncture for $S'$, which we assume to be an output. We further assume that $\alpha_i=\beta_j$. By conditions \ref{DEF Trans fun 1} -- \ref{DEF Trans fun 4} (a), the transition functions $a$ and $b$ determine a transition function on $S\#_\rho S'$.

Suppose we have fixed consistent universal choices of strip-like ends $\{\epsilon^{k+1}_i\}_{i=1,\ldots, k+1}$ on $\mathcal{S}^{k+1}$ for the input set $\mathcal{I}_\mu^{k+1}$, $k\ge 1$, and $\{\epsilon^{m+1;1}_i\}_{i=1,\ldots,m+1}$ on $\mathcal{S}^{m+1;1}$ for the input set $\mathcal{I}_{CY}^{m+1}$, $m\ge 0$. We take as strip-like ends $\{\epsilon_i^{m,p;1}\}_{i=1,\ldots,m+p+2}$ on $\mathcal{S}^{m,p;1}$ for the input set $\mathcal{I}^{m+p+2}_{CY}$, $m,p\ge 0$, those given by restriction of the $\epsilon^{m+p+2;1}_i$. We will make use of the following types of global transition functions in what follows:

\begin{enumerate}
\item For defining the maps $\mu^{\mathcal{F}_{c}}_k$ in the category $\mathcal{F}_{c}$, we will need global transition functions $\bm{a}_\mu:\mathcal{S}^{k+1}\to [0,1]$ for the input set $\mathcal{I}^{k+1}_\mu$ and the $(k+1)$ sign datum $\bar{\alpha}_\mu$ given by $(\alpha_{\mu})_i=+$ for $i\in\{1,\ldots,k+1\}$.

\item For defining the relative right Yoneda functor $\mathbf{Y}^r_{\mathit{rel}}$ for $\mathcal{F}_{c}$ (or equivalently the relative diagonal bimodule $(\mathcal{F}_c)_\Delta^{\mathit{rel}}$), we will need for every pair $(m,p)$ with $m,p\ge 0$ global transition functions $\bm{a}^{m,p}_{\mathbf{Y}}:\mathcal{S}^{m+p+2}\to [0,1]$ for the input set $\mathcal{I}^{m+p+2}_\mu$ and the $(m+p+2)$ sign datum $\bar{\alpha}_{\mathbf{Y}}(m,p)$ given by
\begin{align*}
(\alpha_{\mathbf{Y}}(m,p))_i=\begin{cases} -, & i\in \{m+1,m+p+2\}\\
+, & i\in \{1,\ldots,m+p+2\}\setminus \{m+1,m+p+2\}.
\end{cases}
\end{align*}
\item For defining the relative weak Calabi-Yau pairing on $\mathcal{F}_{c}$ as the class of a natural quasi-isomorphism from $\mathbf{Y}_{\mathcal{F}_{c}}^l$ to $(\mathbf{Y}^\vee_{\mathit{rel}})^l$, we will need for every $m,p\ge 0$ a global transition function $\bm{a}_{\delta}:\mathcal{S}^{m,p;1}\to [0,1]$ for the input set $\mathcal{I}^{m+p+2}_{CY}$ and the $m+p+2$ sign datum $\bar{\alpha}_{CY}$ given by
\begin{align}
(\alpha_{CY})_i=\begin{cases} -, & i=m+p+2,\\
+, & i\in \{1,\ldots,m+p+1\}.
\end{cases}
\end{align}
\item For defining the relative weak Calabi-Yau pairing on $\mathcal{F}_c$ in terms of Hochschild homology, we will need for every $m\ge 0$ a global transition function $\bm{a}_{\sigma}:\mathcal{S}^{m+1;1}\to [0,1]$ for the input set $\mathcal{I}^{m+1}_{CY}$ and the $(m+1)$ sign datum $\bar{\alpha}_{CY}$.
\end{enumerate}

By the same type of arguments used to show the existence of consistent universal choices of strip-like ends in \cite{Sei}, there exist consistent choices of global transition functions $\bm{a}_{\mu}$, $\bm{a}^{m,p}_{\mathbf{Y}}$, and $\bm{a}_\sigma$. We take $\bm{a}_{\delta}$ to be given by restriction of $\bm{a}_\sigma$ to $\mathcal{S}^{m,p;1}$.
Moreover, we can assume that the transition function $\bm{a}_{\mu}$ on $\mathcal{S}^2$ is given by $\bm{a}_{\mu}(s,t)=t$ and the transition function $\bm{a}^{0,0}_{\mathbf{Y}}$ on $\mathcal{S}^2$ is given by $\bm{a}_{\mu}(s,t)=1-t$, where we make use of the same fixed biholomorphic map between the strip $\R\times [0,1]$ and $\mathcal{S}^2=D\setminus\{-1,1\}$ that was used to define strip-like ends on $\mathcal{S}^2$.

\subsection{Floer data, perturbation data, and the curve configurations}\label{SUBSECTION: Floer data and perturbation data}

We describe here the two types of Floer data we will use to define Floer complexes for pairs of cobordisms. We then define a general notion of compatible perturbation data that covers both the data used to define the pseudoholomorphic polygons appearing in the construction of the Fukaya category of cobordisms as well as the data used to describe the curves which define the relative right Yoneda functor $\mathbf{Y}^r_{\mathit{rel}}$ on $\mathcal{F}_c$ and the relative weak Calabi-Yau pairing.

As in \cite{BC14}, we fix a smooth function $h:\R^2\to\R$ satisfying the following conditions (see Figure \ref{FIG: profile function}):

\begin{enumerate}
\item\label{DEF: Profile fn COND 1} There exists $\epsilon\in (0,\tfrac{1}{4})$ such that the support of $h$ is contained in the union of the sets
$$W_i^-=(-\infty,-1]\times [i-\epsilon,i+\epsilon],\;W_i^+=[2,\infty)\times [i-\epsilon,i+\epsilon],\; i\in\Z.$$
\item\label{DEF: Profile fn COND 2} Set $T_i^-=(-\infty,-1]\times [i-\tfrac{\epsilon}{2},i+\tfrac{\epsilon}{2}]$ and $T_i^+=[2,\infty)\times [i-\tfrac{\epsilon}{2},i+\tfrac{\epsilon}{2}]$. The restrictions $h|_{T^-_i}$ and $h|_{T^+_i}$, $i\in\Z$, are of the form $h|_{T^-_i}(x,y)=h_-(x)$ and $h|_{T^+_i}(x,y)=h_+(x)$ where the smooth functions $h_-:(-\infty,-1]\to\R$ and $h_+:[2,\infty)\to\R$ satisfy:
	\begin{enumerate}
	\item $h_-|_{(-\infty,-2)}=\lambda_0^-x+\lambda_1^-$ for some constants $\lambda_0^-$, $\lambda_1^-\in\R$ with $\lambda_0^->0$. Moreover, on the interval $(-\infty,-1)$, $h_-$ has a single critical point at $-\tfrac{3}{2}$ and this point is a non-degenerate maximum.
	\item $h_+|_{(3,\infty)}=\lambda_0^+x+\lambda_1^+$ for some constants $\lambda_0^+$, $\lambda_1^+\in\R$ with $\lambda_0^+<0$. Moreover, on the interval $(2,\infty)$, $h_+$ has a single critical point at $\tfrac{5}{2}$ and this point is also a non-degenerate maximum.
	\end{enumerate}
\item\label{DEF: Profile fn COND 3} The Hamiltonian isotopy $\phi_t^h:\R^2\to\R^2$ exists for all $t\in\R$ and satisfies:
\begin{enumerate}
	\item For all $t\in [-1,1]$, $$\phi_t^h((\infty,-1]\times\{i\})\subset T_i^-\text{ and }\phi_t^h([2,\infty)\times\{i\})\subset T_i^+.$$ In other words, we require the derivatives of the functions $h_{\pm}$ to be sufficiently small.	
	\item\label{DEF: Profile fn COND 3b} For all $t\in\R$, $\phi_t^h([-\tfrac{3}{2},\tfrac{5}{2}]\times\R)=[-\tfrac{3}{2},\tfrac{5}{2}]\times\R$.
\end{enumerate}
\end{enumerate}
The function $h$ is called a \textbf{profile function}. 

\begin{figure}
\centering
\def\svgwidth{\linewidth}
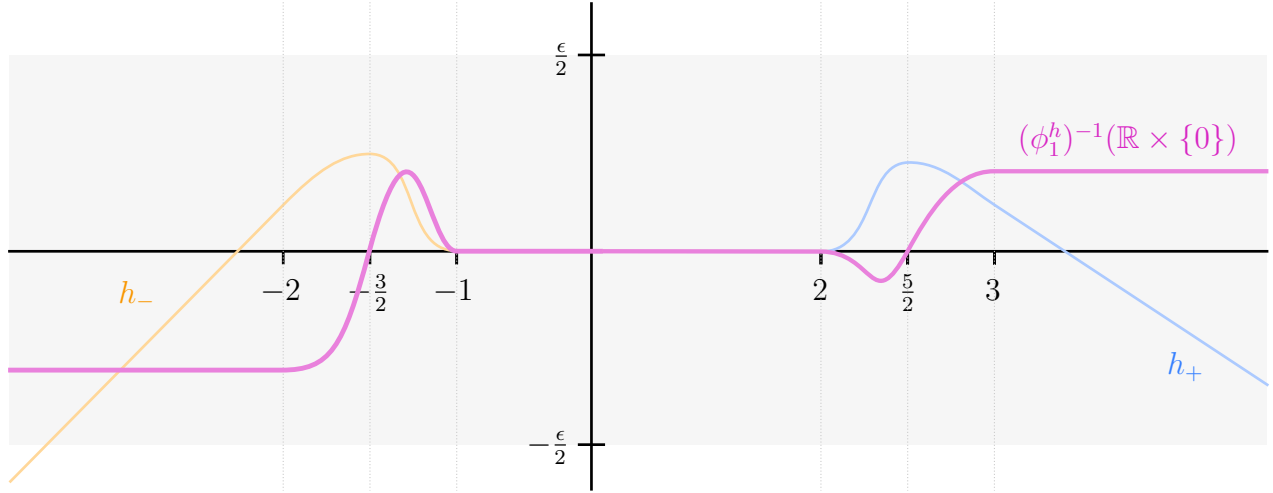
\caption[Graphs of $h_-$ and $h_+$ with $(\phi_1^h)^{-1}(\R\times\{0\})$]{The functions $h_-:(-\infty,-1]\to\R$ and $h_+:[2,\infty)\to\R$, together with the image of the real line under the Hamiltonian diffeomorphism $(\phi_1^h)^{-1}$. The points $(-\tfrac{3}{2},0)$ and $(\tfrac{5}{2},0)$ are ``bottlenecks''.}
\label{FIG: profile function}
\end{figure}

\begin{remk}
The critical points of $h$ at $x=(-\tfrac{3}{2},i)$ and $x=(\tfrac{5}{2},i)$ for $i\in\Z$ are referred to as \textbf{bottlenecks}. The descriptive term arises from the role these critical points play in the compactness theorem for moduli spaces of  inhomogeneous pseudoholomorphic curves satisfying boundary conditions along Lagrangian cobordisms. This proof relies on performing a naturality transformation on such curves, and the resulting transformed curves satisfy moving boundary conditions. The techniques developed in \cite{BC13} for proving Gromov compactness in the cobordism setting need to be modified in the presence of moving boundary conditions. The bottlenecks introduced in \cite{BC14} create a common point that a transformed curve must ``pass through'' in order to escape a compact region, hence the name.  
\end{remk}
 
Fix a pair of cobordisms $W,W'\in \mathcal{CL}_d(\widetilde{M})$ and a sign $\alpha\in\{-,+\}$. We associate to the pair of cobordisms $W,W'$ a \textbf{Floer datum} $\mathscr{D}^\alpha_{W,W'}=(H^{\alpha}_{W,W'}, J^{\alpha}_{W,W'})$ consisting of a time-dependent Hamiltonian and almost complex structure on $\widetilde{M}$, satisfying the following conditions:
\begin{enumerate}
\item\label{DEF: Floer datum cond 1} $\phi_1^{H^{\alpha}_{W,W'}}(W)$ is transverse to $W'$.
\item\label{DEF: Floer datum cond 2} Write points of $\widetilde{M}$ as $(x,y,p)$ with $x+iy\in \C$ and $p\in M$. We require that there exist a compact set $K_{W,W'}\subset (-\tfrac{5}{4},\tfrac{9}{4})\times\R\subset \C$ such that outside of $K_{W,W'}\times M$, we have 
$$H^\alpha_{W,W'}(t,x,y,p)=\alpha h(x,y)+G^\alpha_{W,W'}(t,p),$$
for some $G^\alpha_{W,W'}:[0,1]\times M\to \R$.
\item\label{DEF: Floer datum cond 3} Outside of $K_{W,W'}\times M$, the projection $\pi:\widetilde{M}\to \C$ satisfies:
\begin{itemize}
\item For $\alpha=+$, $\pi$ is $(J^\alpha_{W,W'}(t),(\phi_t^h)_*i)$-holomorphic for all $t\in [0,1]$,
\item For $\alpha=-$, $\pi$ is $(J^\alpha_{W,W'}(t),(\phi_{1-t}^h)_*i)$-holomorphic for all $t\in [0,1]$.
\end{itemize}
\end{enumerate}
When $\alpha=+$, these are the conditions used in \cite{BC14}. We will refer to Floer data of the type $\mathscr{D}^+_{W,W'}$ as \textbf{positive profile Floer data} and to data of the type $\mathscr{D}^-_{W,W'}$ as \textbf{negative profile Floer data}.

\begin{remk}\label{REMK: Chords are contained and finite}
It follows from the form of $H^\alpha_{W,W'}$ and $h$ that the set $\mathcal{O}(H^\alpha_{W,W'})$ of Hamiltonian chords connecting $W$ to $W'$ is finite and that every $\gamma\in \mathcal{O}(H^\alpha_{W,W'})$ satisfies $\mathrm{Image}(\pi\circ \gamma)\subset [-\tfrac{3}{2},\tfrac{5}{2}]\times \R$. To see this, consider a chord $\gamma(t)=\phi_t^{H^\alpha_{W,W'}}(x_0,y_0,p_0)\in \mathcal{O}(H^\alpha_{W,W'})$. If $x_0\in \R\setminus(-\tfrac{5}{4},\tfrac{9}{4})$, it follows from conditions \eqref{DEF: Profile fn COND 2} and \eqref{DEF: Profile fn COND 3} on $h$ that $x_0\in\{-\tfrac{3}{2},\tfrac{5}{2}\}$ and $\pi\circ \gamma$ is a constant curve at $(-\tfrac{3}{2},i)$ or $(\tfrac{5}{2},i)$ for some $i\in\Z$. We conclude that any $\gamma$ in $\mathcal{O}(H^\alpha_{W,W'})$ is of the form $\gamma(t)=\phi_t^{H^\alpha_{W,W'}}(x_0,y_0,p_0)$ with $x_0\in [-\tfrac{3}{2},\tfrac{5}{2}]$. Since $W$ and $(\phi_1^{H^\alpha_{W,W'}})^{-1}(W')$ intersect transversely, there are only finitely many such $\gamma$. Finally, condition \eqref{DEF: Profile fn COND 3b} on $h$ implies that $\mathrm{Image}(\pi\circ \gamma)\subset [-\tfrac{3}{2},\tfrac{5}{2}]\times \R$. 
\end{remk}

We now define the higher level perturbation data. Let $\mathcal{S}\to \mathcal{R}$ denote one of the bundles $\mathcal{S}^{k+1}\to \mathcal{R}^{k+1}$ ($k\ge 2$),  $\mathcal{S}^{k+1;1}\to \mathcal{R}^{k+1;1}$ ($k\ge 0$), or $\mathcal{S}^{m,p;1}\to \mathcal{R}^{m,p;1}$ ($k:=m+p+1\ge 1$). Fix cobordisms $W_0,\ldots,W_k\in \mathcal{CL}_d(\widetilde{M})$, a set of inputs $\mathcal{I}^{k+1}$, a universal choice of strip-like ends $\{\epsilon_i\}_{i=1,\ldots,k+1}$, a $(k+1)$ sign datum $\bar{\alpha}=(\alpha_1,\ldots,\alpha_{k+1})$, and a global transition function $\bm{a}:\mathcal{S}\to [0,1]$. Fix also Floer data $(H_{W_{i-1},W_i}^{\alpha_i},J_{W_{i-1},W_i}^{\alpha_i})$ for $i\in\mathcal{I}^{k+1}$ and  $(H_{W_{i},W_{i-1}}^{\alpha_i},J_{W_{i},W_{i-1}}^{\alpha_i})$ for $i\in\{1,\ldots,k+1\}\setminus \mathcal{I}^{k+1}$. We first define some notation. Let $s_{W_0,\ldots, W_k}\in\N$ be the smallest natural number satisfying $\pi(W_0\cup\cdots\cup W_k)\subset \R\times (-s_{W_0,\ldots, W_k},s_{W_0,\ldots, W_k})$. Let $\tilde{h}=h\circ \pi:\widetilde{M}\to\R$. For $r\in\mathcal{R}$, define
$$U_i^r=\begin{cases} \epsilon_i(r,(-\infty,-1]\times[0,1])\subset \mathcal{S}_r,&i\in \mathcal{I}^{k+1},\\
\epsilon_i(r,[1,\infty)\times[0,1])\subset \mathcal{S}_r,&i\in \{1,\ldots,k+1\}\setminus\mathcal{I}^{k+1}.
\end{cases}$$
We denote the connected components of $\partial \mathcal{S}_r$ by $C_0,\ldots, C_k$, where $C_0$ is the component connecting $z_{k+1}$ to $z_1$ and the indices increase in the clockwise direction.

To the collection $W_0,\ldots, W_k$, the set of inputs $\mathcal{I}^{k+1}$, the universal choice of strip-like ends $\{\epsilon_i\}_{i=1,\ldots,k+1}$, the $(k+1)$ sign datum $\bar{\alpha}$, and the transition function $\bm{a}$, we can associate a \textbf{perturbation datum} $\mathscr{D}_{W_0,\ldots, W_k}=(\mathbf{\Theta},\mathbf{J})$ on the bundle $\mathcal{S}\to\mathcal{R}$ where:
\begin{enumerate}
\item $\mathbf{\Theta}=\{\Theta_r\}_{r\in\mathcal{R}}$ is a family of one-forms indexed by $r\in\mathcal{R}$, where for a fixed $r\in\mathcal{R}$, $\Theta_r\in \Omega^1(\mathcal{S}_r,C^\infty(\widetilde{M}))$, i.e.\ $\Theta_r$ is a one-form on $\mathcal{S}_r$ with values in the space of smooth functions on $\widetilde{M}$.
\item $\mathbf{J}=\{J_{r,z}\}_{(r,z)\in\mathcal{S}}$ is a family of $\widetilde{\omega}$-compatible almost complex structures on $\widetilde{M}$ parametrized by $r\in \mathcal{R}$ and $z\in \mathcal{S}_r$.
\end{enumerate}
For each $r\in\mathcal{R}$, the one-form $\Theta_r$ induces a one-form $Y_r$ on $\mathcal{S}_r$ with values in the space of vector fields on $\widetilde{M}$, $Y_r\in \Omega^1(\mathcal{S}_r,C^\infty(T\widetilde{M}))$. This is defined by $Y_r(v)=X^{\Theta_r(v)}$ for $v\in T\widetilde{M}$, i.e.\ $Y_r(v)$ is the vector field on $\widetilde{M}$ associated to the autonomous Hamiltonian $\Theta_r(v)$. 

The perturbation datum $\mathscr{D}_{W_0,\ldots, W_k}$ is required to satisfy the following conditions:

\begin{enumerate}
\item\label{DEF: Pert datum 1} (\emph{Asymptotic conditions}) For all $r\in\mathcal{R}$, 
\begin{align*}
\mathbf{\Theta}|_{U_i^r}&=\begin{cases}
H^{\alpha_i}_{W_{i-1},W_i}dt,&i\in\mathcal{I}^{k+1},\\
H^{\alpha_i}_{W_{i},W_{i-1}}dt,&i\in \{1,\ldots,k+1\}\setminus\mathcal{I}^{k+1},
\end{cases}\\\\
\mathbf{J}|_{U_i^{r}}&=\begin{cases}
J^{\alpha_i}_{W_{i-1},W_i},&i\in\mathcal{I}^{k+1},\\
J^{\alpha_i}_{W_{i},W_{i-1}},&i\in \{1,\ldots,k+1\}\setminus\mathcal{I}^{k+1},
\end{cases}
\end{align*}
We use the convention $W_{k+1}=W_0$ throughout.

\item\label{DEF: Pert datum 2} For the family of forms $(\Theta_0)_r\in \Omega^1(\mathcal{S}_r,C^\infty(\widetilde{M}))$
 defined by
\begin{equation}\label{EQ: Def of Theta_0}
(\Theta_0)_r=\Theta_r-da_r\otimes \tilde{h},
\end{equation}
we have:
	\begin{enumerate}
	\item \begin{equation} (\Theta_0)_r(\xi)=0\text{ for all }\xi\in TC_i\subset T\partial \mathcal{S}_r
	\end{equation}
	\item There is a compact set $K_{W_0,\ldots, W_{k}}\subset \left(-\tfrac{3}{2},\tfrac{5}{2}\right)\times\R$ which does not depend on $r\in \mathcal{R}$ such that:
		\begin{enumerate}
		\item For all $i=0,\ldots, k$,
		$$K_{W_0,\ldots, W_k}\supset \begin{cases}
		K_{W_{i-1},W_i},&i\in \mathcal{I}^{k+1},\\
		K_{W_i,W_{i-1}},& i\in \{1,\ldots,k+1\}\setminus \mathcal{I}^{k+1}.
		\end{cases}$$
		\item $K_{W_0,\ldots, W_{k}}\supset \left(\left[-\frac{5}{4},\frac{9}{4}\right]\times[-s_{W_0,\ldots, W_k},+s_{W_0,\ldots, W_k}]\right),$
		\item For every $r\in\mathcal{R}$, outside of $K_{W_0,\ldots, W_{k}}\times M$ we have $D\pi(Y_0)=0$, where $Y_0=X^{(\Theta_0)_r}.$
		\end{enumerate}
	\end{enumerate}
\item\label{Perturbation data conditions, PART projection holomorphic} Outside of $K_{W_0,\ldots, W_{k}}\times M$, the projection $\pi$ is $(J_{r,z},(\phi^h_{a_r(z)})_*i)$-holomorphic for every $r\in \mathcal{R}$ and $z\in \mathcal{S}_r$.
\end{enumerate}

These conditions are a slight generalization of the conditions appearing in \cite{BC14} which in turn are based on the conditions in \cite{Sei}.

The inhomogeneous Cauchy-Riemann equation associated to the data above is
\begin{equation}\label{EQ: Perturbed CR equation}
u:\mathcal{S}_r\to \widetilde{M},\; Du +J_{r,z}(u)\circ Du\circ j=Y_r+J_{r,z}(u)\circ Y_r\circ j,\; u(C_i)\subset W_i.
\end{equation}
Here $j$ is the complex structure on $\mathcal{S}_r$. On the entry strip-like ends, $u\circ\epsilon_i$ must approach a time-1 Hamiltonian chord $\gamma_i\in \mathcal{O}(H^{\alpha_i}_{W_{i-1},W_i})$ as $s\to -\infty$.  On the exit strip-like ends, $u\circ\epsilon_i$ must approach a time-1 Hamiltonian chord $\gamma_i\in \mathcal{O}(H^{\alpha_i}_{W_i,W_{i-1}})$ as $s\to +\infty$. 

\begin{remk}\label{REMK: Changing inputs and outputs}
For a solution $u:\mathcal{S}_r\to \widetilde{M}$ of \eqref{EQ: Perturbed CR equation}, it is possible to change any number of inputs to outputs, and vice versa, if we also make a change of Floer data for the corresponding puncture. For instance, consider a puncture $z_{i_0}$ which is an entry, i.e.~$i_0\in \mathcal{I}^{k+1}$. Set $(\mathcal{I}')^{k+1}=\mathcal{I}^{k+1}\setminus\{i_0\}$ and define a new $(k+1)$ sign datum $(\bar{\alpha}')$ by
\begin{equation}
(\alpha')_i=\begin{cases} \alpha_i, & i\ne i_0,\\
-\alpha_i, & i=i_0.
\end{cases}
\end{equation}
Define a positive strip-like end $\epsilon'_{i_0}:[0,\infty)\times [0,1]\to \mathcal{S}_r$ for the puncture $z_{i_0}$ by $\epsilon_{i_0}'(s,t)=\epsilon_{i_0}(-s,1-t)$. Note that the global transition function $\bm{a}:\mathcal{S}\to [0,1]$ satisfies conditions (\ref{DEF Trans fun 1})--(\ref{DEF Trans fun 4}) in Section \ref{SUBSECTION: Transition functions} for the input set $(\mathcal{I}')^{k+1}$, for the choice of strip-like ends where $\epsilon_{i_0}$ is replaced by $\epsilon'_{i_0}$, and for the sign datum $(\bar{\alpha}')$. Moreover, the perturbation datum $\mathscr{D}_{W_0,\ldots,W_k}=(\mathbf{\Theta},\mathbf{J})$ satisfies the asymptotic condition near the puncture $z_{i_0}$ viewed as an exit for the Floer datum $\mathscr{D}^{-\alpha_{i_0}}_{W_{i_0},W_{i_0-1}}=(\bar{H}^{\alpha_{i_0}}_{W_{i_0-1},W_{i_0}},\bar{J}^{\alpha_{i_0}}_{W_{i_0-1},W_{i_0}})$, where $\bar{H}^{\alpha_{i_0}}_{W_{i_0-1},W_{i_0}}(t)= -H^{\alpha_{i_0}}_{W_{i_0-1}, W_{i_0}}(1-t)$ and $\bar{J}^{\alpha_{i_0}}_{W_{i_0-1},W_{i_0}}(t)=J^{\alpha_{i_0}}_{W_{i_0-1},W_{i_0}}(1-t)$. The datum $\mathscr{D}^{-\alpha_{i_0}}_{W_{i_0},W_{i_0-1}}$ satisfies the conditions (\ref{DEF: Floer datum cond 1}) -- (\ref{DEF: Floer datum cond 3}) on the Floer data for pairs of cobordisms. The perturbation datum $\mathscr{D}_{W_0,\ldots,W_k}$ also satisfies conditions (\ref{DEF: Pert datum 2}) and (\ref{Perturbation data conditions, PART projection holomorphic}) for the input set $(\mathcal{I}')^{k+1}$ and the global transition function $\bm{a}$. The map $u$ is a solution of \eqref{EQ: Perturbed CR equation} with boundary conditions along $W_0,\ldots, W_k$, sign datum $(\bar{\alpha}')$, and perturbation datum $\mathscr{D}_{W_0,\ldots,W_k}$. The Hamiltonian chord in $\mathcal{O}(\bar{H}^{\alpha_{i_0}}_{W_{i_0-1},W_{i_0}})$ giving the asymptotic condition for $u$ at the puncture $z_{i_0}$ is $\bar{\gamma}_{i_0}$ where $\bar{\gamma}_{i_0}(t)=\gamma_{i_0}(1-t)$.

\end{remk}

Assume that we have fixed consistent universal choices of strip-like ends $\{\epsilon^{k+1}_i\}_{1=1,\ldots,k+1}$ on the bundles $\mathcal{S}^{k+1}$ for the input sets $\mathcal{I}^{k+1}_\mu$ and consistent choices of global transition functions $\bm{a}_{\mu}:\mathcal{S}^{k+1}\to [0,1]$ as described in Section \ref{SUBSECTION: Transition functions}.
In order to define the Fukaya category of cobordisms without any additional structure (as in \cite{BC14}), we require the following Floer and perturbation data:
\begin{enumerate}
\item\label{Original Floer data} For every pair of cobordisms $W,W'\in \mathcal{CL}_d(\widetilde{M})$, a Floer datum $\mathscr{D}^+_{W,W'}=(H^+_{W,W'}, J^+_{W,W'})$.
\item\label{Original Perturbation data} For every $k\ge 2$ and every collection of cobordisms $W_0,\ldots, W_k$ in $\mathcal{CL}_d(\widetilde{M})$, a perturbation datum $\mathscr{D}^\mu_{W_0,\ldots, W_k}=(\mathbf{\Theta}^\mu,\mathbf{J}^\mu)$ on the bundle $\mathcal{S}^{k+1}\to\mathcal{R}^{k+1}$ for the set of inputs $\mathcal{I}_{\mu}^{k+1}$, the universal choice of strip-like ends $\{\epsilon^{k+1}_i\}_{i=1,\ldots,k+1}$, the $(k+1)$ sign datum $\bar{\alpha}_{\mu}$, and the global transition function $\bm{a}_{\mu}$.
\end{enumerate}
Moreover, this data is required to be consistent with respect to gluing.

Now we describe the additional data needed to define the weak Calabi-Yau pairing on $\mathcal{F}_c$. First assume we have fixed universal choices of strip-like ends $\{\epsilon_i^{m+1;1}\}_{i=1,\ldots,m+1}$ and global transition functions $\bm{a}^{m,p}_{\mathbf{Y}}$, $\bm{a}_{\delta}$, and $\bm{a}_\sigma$ as described in Section \ref{SUBSECTION: Transition functions} so that all strip-like ends and global transition functions are consistent.
Then we must fix the following additional Floer and perturbation data:
\begin{enumerate}
\item[($\mathit{1^\prime}$)]\label{Extra Data 1} For every pair of cobordisms $W,W'\in \mathcal{CL}_d(\widetilde{M})$, a Floer datum $\mathscr{D}^-_{W,W'}=(H^-_{W,W'}, J^-_{W,W'})$.
\item[($\mathit{2^\prime}$)] For every pair $(m,p)$ with $m+p\ge 1$ and every collection of cobordisms $W_0,\ldots, W_m,X_p,\ldots,X_0$ in $\mathcal{CL}_d(\widetilde{M})$, a perturbation datum $\mathscr{D}^\mathbf{Y}_{W_0,\ldots,W_m;X_p,\ldots, X_0}=(\mathbf{\Theta}^\mathbf{Y},\mathbf{J}^{\mathbf{Y}})$ on the bundle $\mathcal{S}^{m+p+2}\to\mathcal{R}^{m+p+2}$ for the set of inputs $\mathcal{I}_{\mu}^{m+p+2}$, the universal choice of strip-like ends $\{\epsilon^{m+p+2}_i\}_{i=1,\ldots,m+p+2}$, the $(m+p+2)$ sign datum $\bar{\alpha}_{\mathbf{Y}}(m,p)$, and the global transition function $\bm{a}^{m,p}_{\mathbf{Y}}$. 
\item[($\mathit{3^\prime}$)] For every pair $(m,p)$ with $m,p\ge 0$ and every collection of cobordisms $W_0,\ldots, W_m,X_p,\ldots,X_0$ in $\mathcal{CL}_d(\widetilde{M})$, a perturbation datum $\mathscr{D}^\delta_{W_0,\ldots,W_m;X_p,\ldots, X_0}=(\mathbf{\Theta}^\delta,\mathbf{J}^\delta)$ on the bundle $\mathcal{S}^{m,p;1}\to\mathcal{R}^{m,p;1}$ for the set of inputs $\mathcal{I}_{CY}^{m+p+2}$, the universal choice of strip-like ends $\{\epsilon^{m,p;1}_i\}_{i=1,\ldots,m+p+2}$, the $(m+p+2)$ sign datum $\bar{\alpha}_{CY}$, and the global transition function $\bm{a}_{\delta}$.
\item[($\mathit{4^\prime}$)]\label{Extra Data 4} For every pair $m\ge 0$ and every collection of cobordisms $W_0,\ldots, W_m$ in $\mathcal{CL}_d(\widetilde{M})$, a perturbation datum $\mathscr{D}^\sigma_{W_0,\ldots,W_m}=(\mathbf{\Theta}^\sigma,\mathbf{J}^\sigma)$ on the bundle $\mathcal{S}^{m+1;1}\to\mathcal{R}^{m+1;1}$ for the set of inputs $\mathcal{I}_{CY}^{m+1}$, the universal choice of strip-like ends $\{\epsilon^{m+1;1}_i\}_{i=1,\ldots,m+1}$, the $(m+1)$ sign datum $\bar{\alpha}_{CY}$, and the global transition function $\bm{a}_\sigma$.
\end{enumerate}

The additional data ($\mathit{1^\prime}$)--($\mathit{4^\prime}$) together with the data (\ref{Original Floer data})--(\ref{Original Perturbation data}) must all be consistent with respect to gluing. The consistency condition is satisfied if the data $\mathscr{D}^\delta_{W_0,\ldots,W_m;X_p,\ldots, X_0}$ in ($\mathit{4^\prime}$) are given by restriction of $\mathscr{D}^\sigma_{W_0,\ldots,W_m,X_p,\ldots, X_0}$ to the subbundle $\mathcal{S}^{m,p;1}$ of $\mathcal{S}^{m+p+2;1}$. However, care must be taken because regular data on $\mathcal{S}^{m+p+2;1}$ will not in general restrict to regular data on $\mathcal{S}^{m,p;1}$.

\subsubsection{The curve configurations}\label{SUBSUBSECTION: The curve configurations}

We introduce the following notation for the moduli spaces of solutions of the inhomogeneous Cauchy-Riemann equation \eqref{EQ: Perturbed CR equation} associated to the Floer and perturbation data introduced above (see Figure \ref{FIG: Curve configurations}).

\begin{enumerate}
\item\label{ENUM: Curve configuration 1} For the Floer datum $\mathscr{D}^\pm_{W,W'}=(H^\pm_{W,W'},J^\pm_{W,W'})$, define a perturbation datum by $\Theta=dt\otimes H_{W,W'}^\pm$ and $J(t)=J^\pm_{W,W'}(t)$ (where we have identified $\mathcal{S}^2$ with the strip $\R\times [0,1]$). The moduli space of solutions of \eqref{EQ: Perturbed CR equation} for this perturbation datum satisfying boundary conditions along $W,W'$ and asymptotic conditions along $\gamma,\gamma'\in \mathcal{O}(H^\pm)$, modulo $\R$-action, is denoted $\mathcal{R}_\pm^2(\gamma,\gamma')$.
\item The moduli space of pairs $(r,u)$ where $u:\mathcal{S}_r^{k+1}\to \widetilde{M}$ is a solution of \eqref{EQ: Perturbed CR equation} for the perturbation datum $\mathscr{D}^\mu_{W_0,\ldots, W_k}$ satisfying boundary conditions along $W_0,\ldots, W_k$ and asymptotic conditions along $\gamma_i\in \mathcal{O}(H^+_{W_{i-1},W_i})$, $i=1,\ldots,k$, and $\gamma_{k+1}\in \mathcal{O}(H^+_{W_0,W_k})$ is denoted $\mathcal{R}^{k+1}_\mu(\gamma_1,\ldots,\gamma_{k+1})$.
\item The moduli space of pairs $(r,u)$ where $u:\mathcal{S}_r^{m+p+2}\to \widetilde{M}$ is a solution of \eqref{EQ: Perturbed CR equation} for the perturbation datum $\mathscr{D}^\mathbf{Y}_{W_0,\ldots,W_m;X_p,\ldots, X_0}$ satisfying boundary conditions along $W_0,\ldots,W_m,X_p,\ldots, X_0$ and the following asymptotic conditions 
\begin{equation}
\begin{aligned}
&\gamma_i\in \mathcal{O}(H^+_{W_{i-1},W_i}),\; i=1,\ldots,m,\\
&\bm{\zeta}\in \mathcal{O}(H^-_{W_m,X_p}),\\ 
&\eta_j\in \mathcal{O}(H^+_{X_j,X_{j-1}}),\;j=p,\ldots,1,\\
& \bm{\zeta}'\in \mathcal{O}(H^-_{W_0,X_0}),
\end{aligned}
\end{equation}
is denoted $\mathcal{R}_{\mathbf{Y}}^{m+p+2}(\gamma_1,\ldots,\gamma_m,\bm{\zeta};\eta_p,\ldots,\eta_1,\bm{\zeta}')$.
\item The moduli space of pairs $(r,u)$ where $u:\mathcal{S}_r^{m,p;1}\to \widetilde{M}$ is a solution of \eqref{EQ: Perturbed CR equation} for the perturbation datum $\mathscr{D}^\delta_{W_0,\ldots,W_m;X_p,\ldots,X_0}$ satisfying boundary conditions along $W_0,\ldots,W_m$,$X_p,\ldots, X_0$ and the following asymptotic conditions 
\begin{equation}
\begin{aligned}
&\gamma_i\in \mathcal{O}(H^+_{W_{i-1},W_i}),\;i=1,\ldots,m,\\ 
&\bm{\xi}\in \mathcal{O}(H^+_{W_m,X_p}),\\ 
&\eta_j\in \mathcal{O}(H^+_{X_j,X_{j-1}}),\;j=p,\ldots,1,\\
& \bm{\xi}'\in \mathcal{O}(H^-_{X_0,W_0}),
\end{aligned}
\end{equation}
is denoted $\mathcal{R}_{\delta}^{m,p;1}(\gamma_1,\ldots,\gamma_m,\bm{\xi};\eta_p,\ldots,\eta_1,\bm{\xi}')$.
\item\label{ENUM: Curve configuration 5} The moduli space of pairs $(r,u)$ where $u:\mathcal{S}_r^{m+1;1}\to \widetilde{M}$ is a solution of \eqref{EQ: Perturbed CR equation} for the perturbation datum $\mathscr{D}^\sigma_{W_0,\ldots,W_m}$ satisfying boundary conditions along $W_0,\ldots,W_m$ and asymptotic conditions along $\gamma_i\in \mathcal{O}(H^+_{W_{i-1},W_i})$, $i=1,\ldots,m$, and $\bm{\gamma}\in \mathcal{O}(H^-_{W_m,W_0})$
is denoted $\mathcal{R}_{\sigma}^{m+1;1}(\gamma_1,\ldots,\gamma_m,\bm{\gamma})$.
\end{enumerate}
At certain points, we will also make use of the conventions $\mathcal{R}^2_\mu(\gamma_1,\gamma_2)=\mathcal{R}^2_+(\gamma_1,\gamma_2)$ and $\mathcal{R}_{\mathbf{Y}}^{2}(\bm{\zeta};\bm{\zeta}')=\mathcal{R}^2_-(\bm{\zeta},\bm{\zeta}')$.

\begin{figure}
\begin{subfigure}{.43\textwidth}
  \centering
  \def\svgwidth{.6\linewidth}
\begingroup%
  \makeatletter%
  \providecommand\color[2][]{%
    \errmessage{(Inkscape) Color is used for the text in Inkscape, but the package 'color.sty' is not loaded}%
    \renewcommand\color[2][]{}%
  }%
  \providecommand\transparent[1]{%
    \errmessage{(Inkscape) Transparency is used (non-zero) for the text in Inkscape, but the package 'transparent.sty' is not loaded}%
    \renewcommand\transparent[1]{}%
  }%
  \providecommand\rotatebox[2]{#2}%
  \newcommand*\fsize{\dimexpr\f@size pt\relax}%
  \newcommand*\lineheight[1]{\fontsize{\fsize}{#1\fsize}\selectfont}%
  \ifx\svgwidth\undefined%
    \setlength{\unitlength}{254.19147401bp}%
    \ifx\svgscale\undefined%
      \relax%
    \else%
      \setlength{\unitlength}{\unitlength * \real{\svgscale}}%
    \fi%
  \else%
    \setlength{\unitlength}{\svgwidth}%
  \fi%
  \global\let\svgwidth\undefined%
  \global\let\svgscale\undefined%
  \makeatother%
  \begin{picture}(1,0.98223035)%
    \lineheight{1}%
    \setlength\tabcolsep{0pt}%
    \put(0,0){\includegraphics[width=\unitlength,page=1]{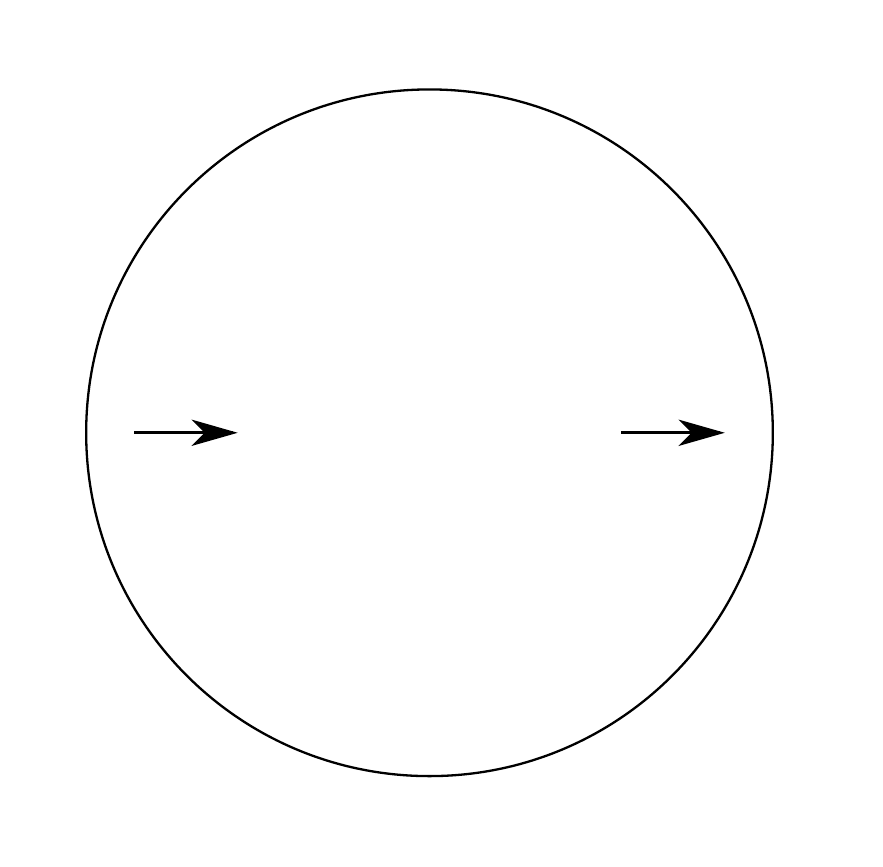}}%
    \put(0.43919454,0.92520229){\color[rgb]{0,0,0}\makebox(0,0)[lt]{\lineheight{2.13000011}\smash{\begin{tabular}[t]{l}$W'$\end{tabular}}}}%
    \put(0.44790783,0.01180233){\color[rgb]{0,0,0}\makebox(0,0)[lt]{\lineheight{2.13000011}\smash{\begin{tabular}[t]{l}$W$\end{tabular}}}}%
    \put(0.92898283,0.48066978){\color[rgb]{0,0,0}\makebox(0,0)[lt]{\lineheight{2.13000011}\smash{\begin{tabular}[t]{l}$\gamma'$\end{tabular}}}}%
    \put(0,0){\includegraphics[width=\unitlength,page=2]{circle_2.pdf}}%
    \put(0.00378038,0.48421966){\color[rgb]{0,0,0}\makebox(0,0)[lt]{\lineheight{2.13000011}\smash{\begin{tabular}[t]{l}$\gamma$\end{tabular}}}}%
    \put(0,0){\includegraphics[width=\unitlength,page=3]{circle_2.pdf}}%
  \end{picture}%
\endgroup%
\medskip
  \caption{$\mathcal{R}_+^2(\gamma,\gamma')$}
  \label{a}
\end{subfigure}
\begin{subfigure}{.43\textwidth}
  \centering
\def\svgwidth{.6\linewidth}
\begingroup%
  \makeatletter%
  \providecommand\color[2][]{%
    \errmessage{(Inkscape) Color is used for the text in Inkscape, but the package 'color.sty' is not loaded}%
    \renewcommand\color[2][]{}%
  }%
  \providecommand\transparent[1]{%
    \errmessage{(Inkscape) Transparency is used (non-zero) for the text in Inkscape, but the package 'transparent.sty' is not loaded}%
    \renewcommand\transparent[1]{}%
  }%
  \providecommand\rotatebox[2]{#2}%
  \newcommand*\fsize{\dimexpr\f@size pt\relax}%
  \newcommand*\lineheight[1]{\fontsize{\fsize}{#1\fsize}\selectfont}%
  \ifx\svgwidth\undefined%
    \setlength{\unitlength}{254.19147401bp}%
    \ifx\svgscale\undefined%
      \relax%
    \else%
      \setlength{\unitlength}{\unitlength * \real{\svgscale}}%
    \fi%
  \else%
    \setlength{\unitlength}{\svgwidth}%
  \fi%
  \global\let\svgwidth\undefined%
  \global\let\svgscale\undefined%
  \makeatother%
  \begin{picture}(1,0.98223035)%
    \lineheight{1}%
    \setlength\tabcolsep{0pt}%
    \put(0,0){\includegraphics[width=\unitlength,page=1]{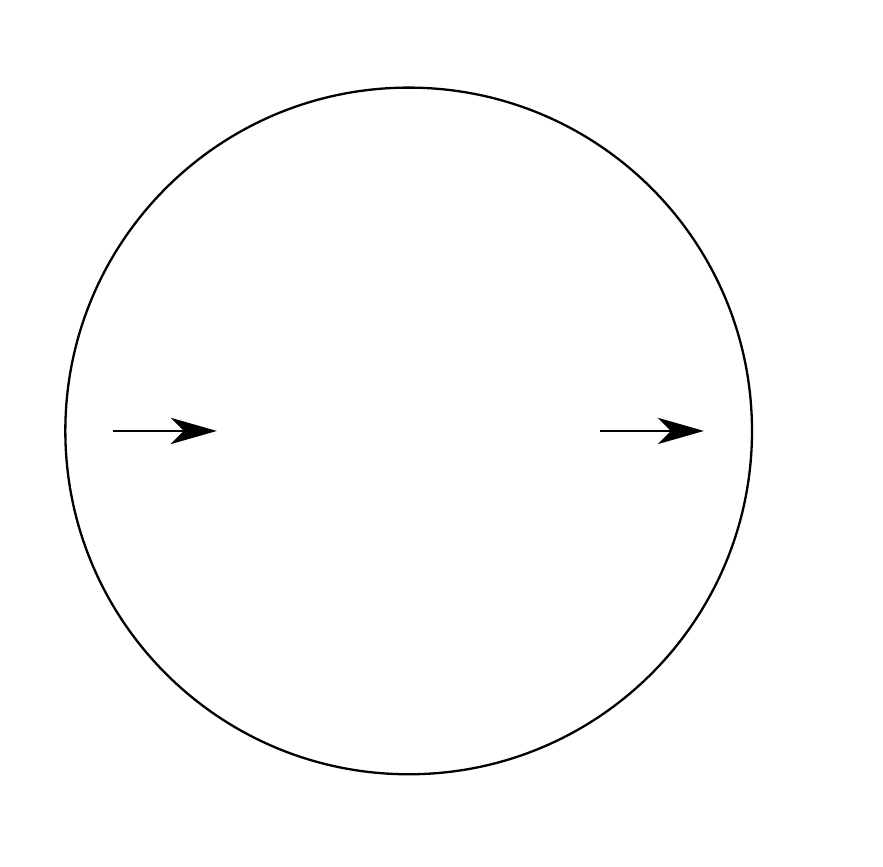}}%
    \put(0.42149134,0.92727654){\color[rgb]{0,0,0}\makebox(0,0)[lt]{\lineheight{2.13000011}\smash{\begin{tabular}[t]{l}$W'$\end{tabular}}}}%
    \put(0.43020456,0.01387658){\color[rgb]{0,0,0}\makebox(0,0)[lt]{\lineheight{2.13000011}\smash{\begin{tabular}[t]{l}$W$\end{tabular}}}}%
    \put(0.89947748,0.48274403){\color[rgb]{0,0,0}\makebox(0,0)[lt]{\lineheight{2.13000011}\smash{\begin{tabular}[t]{l}$\gamma'$\end{tabular}}}}%
    \put(-0.02572496,0.48629391){\color[rgb]{0,0,0}\makebox(0,0)[lt]{\lineheight{2.13000011}\smash{\begin{tabular}[t]{l}$\gamma$\end{tabular}}}}%
    \put(0,0){\includegraphics[width=\unitlength,page=2]{circle_5.pdf}}%
  \end{picture}%
\endgroup%
\medskip
  \caption{$\mathcal{R}_-^2(\gamma,\gamma')$}
  \label{b}
\end{subfigure}\vspace{1.5\baselineskip}
\begin{subfigure}{.43\textwidth}
  \centering
  \def\svgwidth{.8\linewidth}
  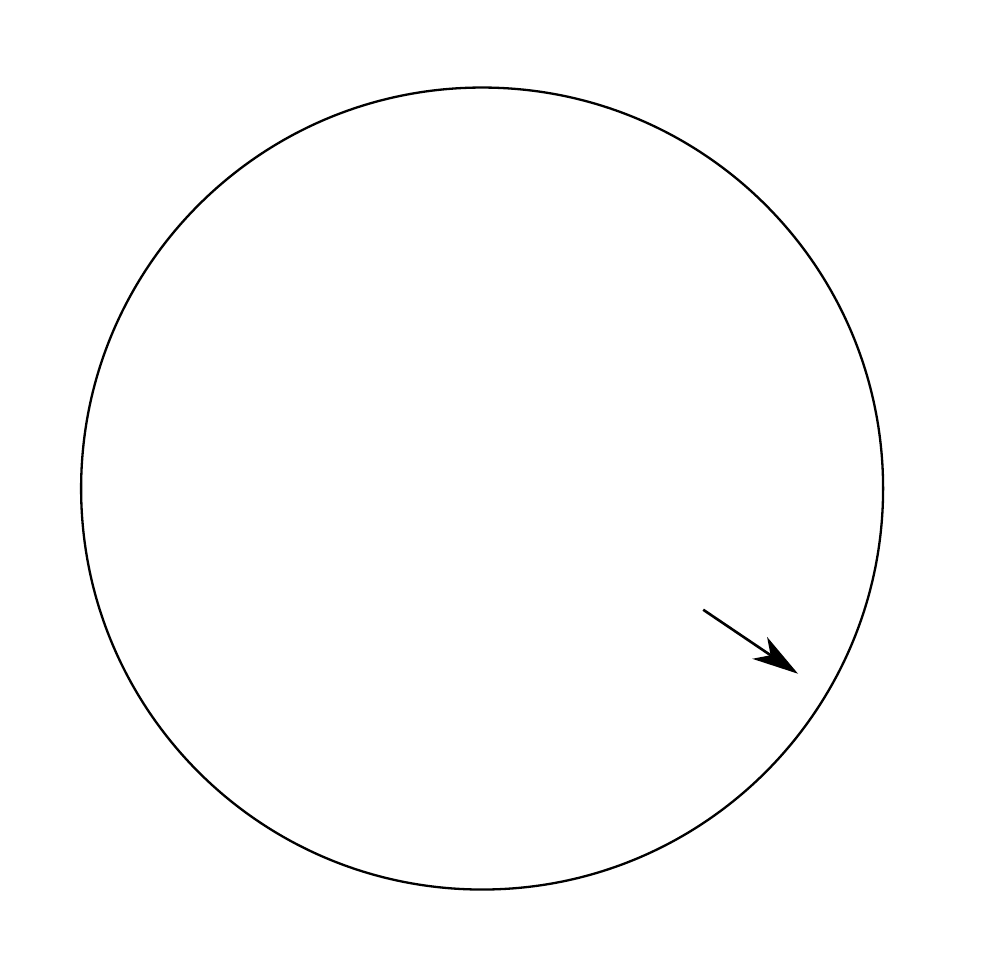\medskip
  \caption{$\mathcal{R}_{\mu}^{4}(\gamma_1,\ldots,\gamma_4)$}
  \label{c}
\end{subfigure}%
\begin{subfigure}{.43\textwidth}
  \centering
  \def\svgwidth{.8\linewidth}
  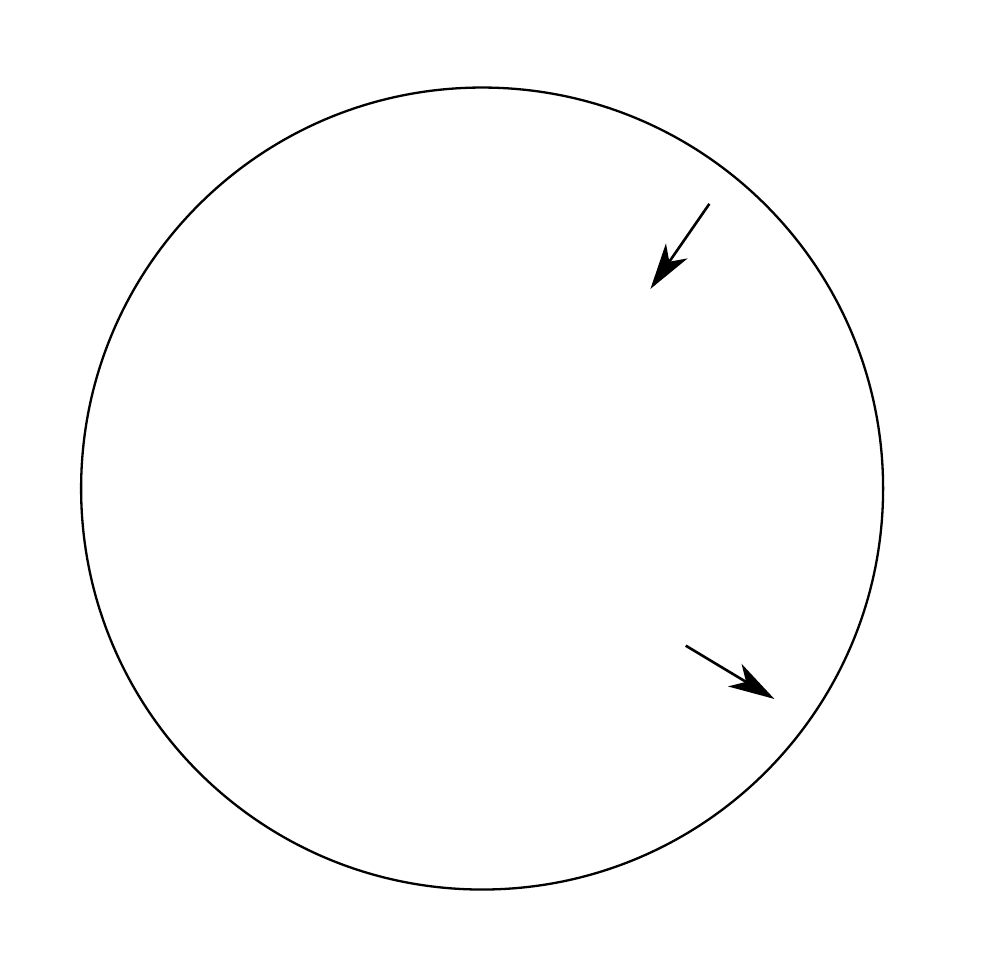\medskip
  \caption{$\mathcal{R}_{\mathbf{Y}}^{5}(\gamma_1,\gamma_2,\bm{\zeta};\eta_1,\bm{\zeta}')$}
  \label{d}
\end{subfigure}\vspace{1.5\baselineskip}
\begin{subfigure}{.43\textwidth}
  \centering
  \def\svgwidth{.8\linewidth}
  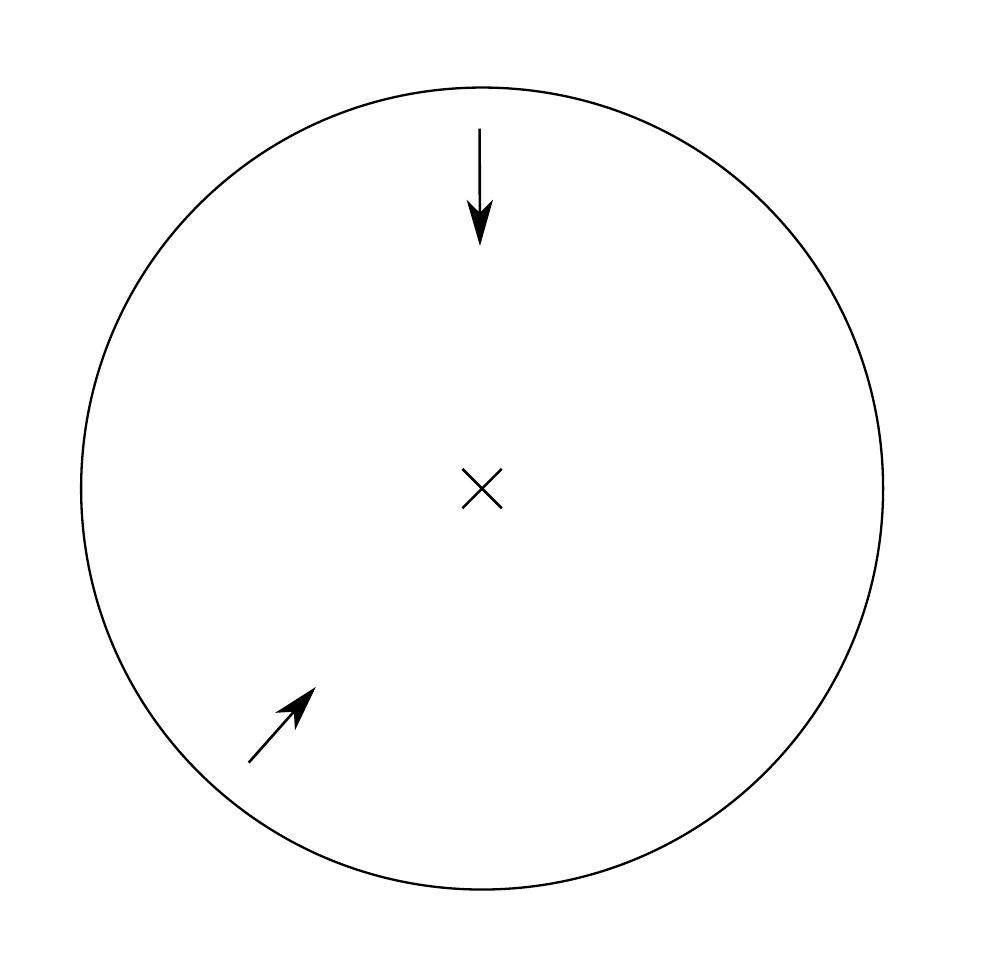\medskip
  \caption{$\mathcal{R}_{\delta}^{2,3;1}(\gamma_1,\gamma_2,\bm{\xi};\eta_3,\eta_2,\eta_1,\bm{\xi}')$}
  \label{e}
  \end{subfigure}
  \begin{subfigure}{.43\textwidth}
  \centering
  \def\svgwidth{.8\linewidth}
  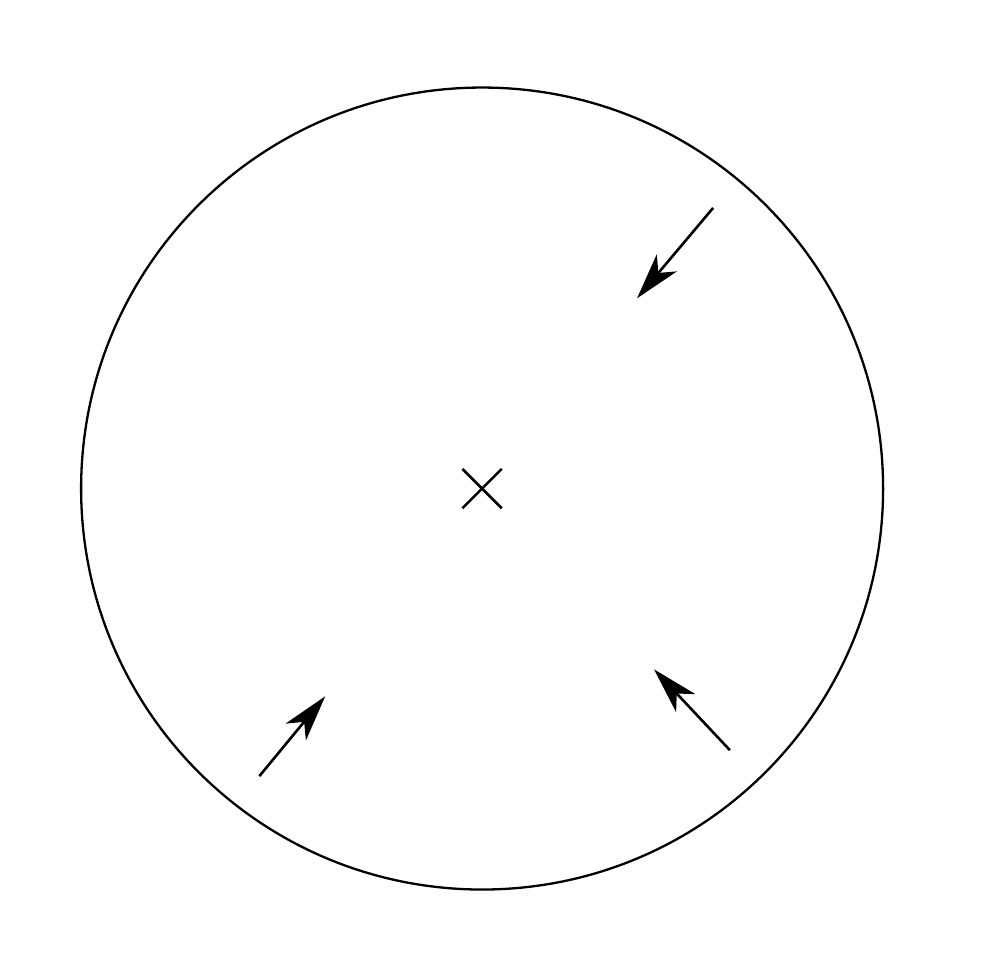\medskip
  \caption{$\mathcal{R}_{\sigma}^{4;1}(\gamma_1,\ldots,\gamma_3,\bm{\gamma})$}
  \label{f}
\end{subfigure}%
\caption[The six different types of moduli spaces]{The six different types of moduli spaces used to define the $A_\infty$-structure and the relative weak Calabi-Yau pairing on $\mathcal{F}_c$. The signs near the boundary punctures indicate which type of Floer data -- positive or negative profile -- appears in the asymptotic conditions for the perturbation data near the puncture.}\label{FIG: Curve configurations}
\end{figure}

\subsection{The Fukaya category of cobordisms}

Biran and Cornea proved in \cite{BC14} that for a generic consistent choice of Floer data and perturbation data as described in (\ref{Original Floer data})--(\ref{Original Perturbation data}) in Section \ref{SUBSECTION: Floer data and perturbation data}, the moduli spaces of the type $\mathcal{R}_+^2(\gamma,\gamma')$ and $\mathcal{R}^{k+1}_\mu(\gamma_1,\ldots,\gamma_{k+1})$ are regular and satisfy Gromov compactness. As a result, they were able to adapt the methods in \cite{Sei} for constructing ordinary Fukaya categories to construct Fukaya categories of cobordisms. The Fukaya category of cobordisms $\mathcal{F}_c$ associated to a fixed choice of Floer and perturbation data has as objects cobordisms in $\mathcal{CL}_d(\widetilde{M})$, as morphism spaces between $W,W'\in\mathcal{CL}_d(\widetilde{M})$ the Floer complex for the Floer datum $\mathscr{D}^+_{W,W'}$, 
\begin{equation}
\mathcal{F}_c(W,W')=CF(W,W';\mathscr{D}^+_{W,W'}):=(\Z_2\langle\mathcal{O}(H_{W,W'}^+)\rangle,\mu^{\mathcal{F}_c}_1).
\end{equation}
Here the differential $\mu^{\mathcal{F}_c}_1$ counts elements in the zero-dimensional component of the moduli spaces $\mathcal{R}^2_+(\gamma,\gamma')$. Since the moduli spaces  $\mathcal{R}^2_+(\gamma,\gamma')$ are well-behaved, this complex is well-defined and satisfies the expected properties of the Floer complex associated to a pair of Lagrangians.
The maps $\mu_k^{\mathcal{F}_c}$ are defined by counting elements in the zero-dimensional components of the moduli spaces $\mathcal{R}^{k+1}_\mu(\gamma_1,\ldots,\gamma_{k+1})$. By the usual gluing and compactness arguments, the $\mu_k^{\mathcal{F}_c}$ satisfy the $A_\infty$-relations. 

It had been shown previously in \cite{BC13} that Floer complexes associated to pairs of cobordisms are well-defined for both types of data $\mathscr{D}^+_{W,W'}$ and $\mathscr{D}^-_{W,W'}$, as well as for more general types of Floer data. It was also established there that for the type of profile function $h$ we consider here, which in particular has bottlenecks corresponding to local maxima, there are $PSS$-type isomorphisms 
\begin{align}
&QH(W,\partial W):=H(\mathcal{C}(f,J^+_{W,W}))\cong HF(W,W;\mathscr{D}^+_{W,W}),\label{EQ: Rel PSS iso}\\
&QH(W):=H(\mathcal{C}(g,J^-_{W,W}))\cong HF(W,W;\mathscr{D}^-_{W,W}).\label{EQ: Abs PSS iso} 
\end{align}  
Here $QH(W,\partial W)$ and $QH(W)$ are the relative and absolute quantum homology for cobordisms introduced in \cite{BC13}. They are taken to be ungraded and with coefficients in $\Z_2$. The relative quantum homology $QH(W,\partial W)$ is the homology of the pearl complex $\mathcal{C}(f,J_{W,W})$ associated to a Morse function $f$ on $W$ whose negative gradient points outward along $\partial W$, and the absolute quantum homology $QH(W)$ is the homology of the pearl complex $\mathcal{C}(g,J_{W,W})$ associated to a Morse function $g$ on $W$ whose negative gradient points inward along $\partial W$. The images of the fundamental classes in $QH(W,\partial W)$ for $W\in \mathcal{CL}_d(\widetilde{M})$ under the isomorphism \eqref{EQ: Rel PSS iso} serve as identity isomorphisms in the category $H(\mathcal{F}_c)$, and thus $\mathcal{F}_c$ is homologically unital (see also Remark 3.5.1 in \cite{BC14}).
  
The category $\mathcal{F}_c$ depends on the choices of data used in its construction: universal strip-like ends, a Floer datum $\mathscr{D}^+_{W,W'}$ for every pair of objects $W,W'$, a perturbation datum $\mathscr{D}^\mu_{W_0,\ldots,W_k}$ for every family of objects $W_0,\ldots,W_k$ ($k\ge 2$), as well as a profile function $h$. However, any two choices of data result in quasi-isomorphic categories. Moreover, choosing a profile function with positions other than $-3/2$ and $5/2$ for the bottlenecks likewise results in a quasi-isomorphic category.

\subsection{The relative weak Calabi-Yau pairing on $\mathcal{F}_c$}

The definitions of the structures in this section rely on regularity and compactness results for the moduli spaces defined in Section \ref{SUBSUBSECTION: The curve configurations}. We prove these results in Section \ref{SECTION: Proofs of the main results}. 

\subsubsection{Poincar\'{e} duality for Lagrangian cobordism Floer complexes}

The Poincar\'{e} duality quasi-isomorphism for Floer complexes associated to pairs of cobordisms differs from the ordinary Poincar\'{e} duality quasi-isomorphism for Floer complexes \eqref{EQ: Poincare duality quasi-iso for CF total} due to the fact that cobordisms generally have non-empty boundary (when viewed as existing in $\R\times [0,1]\times M$). It is instructive to first consider the description of Poincar\'{e} duality for Morse complexes on a smooth manifold $X$ with boundary. Recall that for a Morse function $f$ on $X$ whose negative gradient is transverse to the boundary of $X$ and points outward, the Morse complex $\mathcal{C}(f)$ is well-defined and computes the homology of $X$ relative to $\partial X$. We use the convention here that the Morse differential counts flow lines of $-\nabla f$. For a Morse function $g$ on $X$ whose negative gradient is transverse to the boundary of $X$ and points inward, the Morse complex $\mathcal{C}(g)$ is similarly well-defined, but in this case the complex computes the absolute homology of $X$. We define the Poincar\'{e} duality quasi-isomorphism for the choice of Morse functions $f$ and $g$ to be the composition
\begin{equation}\label{EQ: Morse relative poincare duality}
\mathcal{C}(f)\to \mathcal{C}(-f)^\vee\to \mathcal{C}(g)^\vee.
\end{equation}
The first map is the isomorphism which acts as the identity on critical points of $f$. This map corresponds to the ``formal'' part \eqref{EQ: Poincare duality quasi-iso for CF - trivial part} of the duality map on Floer complexes associated to pairs of closed Lagrangians. The second map in \eqref{EQ: Morse relative poincare duality} is the dual of a choice of comparison quasi-isomorphism $\mathcal{C}(g)\to \mathcal{C}(-f)$ interpolating between the Morse functions $g$ and $-f$ whose negative gradients both point inward along $\partial X$. When $\partial X=\emptyset$ it is possible to take $f=g$. However, in general there is no comparison map interpolating between a Morse function with negative gradient pointing outward along $\partial X$ and a Morse function with negative gradient pointing inward along $\partial X$.

We now consider duality for Floer complexes associated to a pair of cobordisms $W,W'\in \mathcal{CL}_d(\widetilde{M})$. Fix a regular positive profile Floer datum $\mathscr{D}^+_{W,W'}$ for the pair $W, W'$ and a regular negative profile Floer datum $\mathscr{D}^-_{W',W}$ for the pair $W',W$. As in the Morse case, the Poincar\'{e} duality quasi-isomorphism for the complex $CF(W,W';\mathscr{D}^+_{W,W'})$ is a composition of a formal map and a continuation map. The formal map is defined identically to the formal part of the duality map for Floer complexes for closed Lagrangians. In the present situation, the map is
\begin{equation}\label{EQ: Cobordism duality morphism formal part}
\begin{aligned}
CF(W,W';H^+_{W,W'},J^+_{W,W'})&\to CF(W',W;\bar{H}^+_{W,W'},\bar{J}^+_{W,W'})^\vee,\\
\gamma&\mapsto \bar{\gamma}.
\end{aligned}
\end{equation}

In general there is only a well-defined continuation map interpolating between Floer complexes for a pair of cobordisms if both complexes are computed with respect to the same type of Floer data: either positive profile or negative profile. This mirrors the situation in Morse theory where one can only interpolate between Morse functions whose negative gradients point the same way along the boundary of the manifold. Note that for the positive profile Floer datum $\mathscr{D}^+_{W,W'}=(H^+_{W,W'},J^+_{W,W'})$ associated to the pair $W,W'$, the datum $(\bar{H}^+_{W,W'}, \bar{J}^+_{W,W'})$ is a negative profile Floer datum for the pair $W',W$, i.e.\ it satisfies conditions (\ref{DEF: Floer datum cond 1})--(\ref{DEF: Floer datum cond 3}) on Floer data in Section \ref{SUBSECTION: Floer data and perturbation data} for $\alpha=-$. Therefore there is a continuation map 
\begin{equation}\label{EQ: Cobordism continuation map}
CF(W',W;\mathscr{D}^-_{W',W})\to CF(W',W;\bar{H}^+_{W,W'}, \bar{J}^+_{W,W'}). 
\end{equation}
As usual, this continuation map depends on a choice of regular homotopy between the Floer data $\mathscr{D}^-_{W',W}$ and $(\bar{H}^+_{W,W'}, \bar{J}^+_{W,W'})$.
We define the total Poincar\'{e} duality quasi-isomorphism for Floer complexes associated to pairs of cobordisms to be the composition of the map \eqref{EQ: Cobordism duality morphism formal part} with the dual of the continuation map \eqref{EQ: Cobordism continuation map}: 
\begin{equation}\label{EQ: Total duality map for cobordisms}
CF(W,W';\mathscr{D}^+_{W,W'})\to CF(W',W;\bar{H}^+_{W,W'}, \bar{J}^+_{W,W'})^\vee\to CF(W',W;\mathscr{D}^-_{W',W})^\vee.
\end{equation}
For an appropriate choice of homotopy determining the map \eqref{EQ: Cobordism continuation map}, the total Poincar\'{e} duality quasi-isomorphism is defined by counting elements in the zero-dimensional component of $\mathcal{R}^{1,1;1}(\bm{\xi},\bm{\xi}')$ for $\bm{\xi}\in \mathcal{O}(H^+_{W,W'})$ and $\bm{\xi}'\in \mathcal{O}(H^-_{W',W})$. We note that in the case where $W=W'$, under the PSS isomorphisms \eqref{EQ: Rel PSS iso} and \eqref{EQ: Abs PSS iso}, the map \eqref{EQ: Total duality map for cobordisms} induces an isomorphism $QH(W,\partial W)\cong QH(W)^\vee$.

\subsubsection{The relative right Yoneda functor for $\mathcal{F}_c$}

The relative right Yoneda functor for the category $\mathcal{F}_c$ is the functor
$$\mathbf{Y}^r_\mathit{rel}:\mathcal{F}_c\to (\mathit{mod\mbox{--}}\mathcal{F}_c)^{\mathit{opp}}$$
which is described as follows. On objects $W$ in $\mathcal{F}_c$, $\mathbf{Y}^r_{\mathit{rel}}(W)=\mathcal{M}_W^{r,-}$ where $\mathcal{M}_W^{r,-}$ is the right $\mathcal{F}_c$-module defined by
\begin{align*}
&\mathcal{M}_W^{r,-}(X) = CF(W,X;\mathscr{D}_{W,X}^-),\\
&\mu_{1|p}^{\mathcal{M}_W^{r,-}}:CF(W,X_p;\mathscr{D}_{W,X_p}^-)\otimes \mathcal{F}_c(X_p,\ldots, X_0)\to CF(W,X_0;\mathscr{D}_{W,X_0}^-),\\
&\langle \mu_{1|p}^{\mathcal{M}_W^{r,-}}(\bm{\zeta},\eta_p,\ldots, \eta_1) ,\bm{\zeta}'\rangle=\#_{\Z_2}\mathcal{R}^{p+2}_{\mathbf{Y}}(\bm{\zeta};\eta_p,\ldots,\eta_1,\bm{\zeta}')^0,
\end{align*}
for Hamiltonian chords $\bm{\zeta}\in \mathcal{O}(H^-_{W,X_p})$, $\eta_j\in \mathcal{O}(H^+_{X_j,X_{j-1}})$, $j=p,\ldots,1$, $\bm{\zeta}'\in \mathcal{O}(H^-_{W,X_0})$.
In particular, $\mu_{1|0}^{\mathcal{M}_W^{r,-}}$ is the Floer differential on $CF(W,X;\mathscr{D}_{W,X}^-)$. The higher maps of $\mathbf{Y}^r_\mathit{rel}$ are defined by
\begin{equation}
\begin{aligned}
&(\mathbf{Y}^r_\mathit{rel})_m:\mathcal{F}_c(W_0,\ldots, W_m)\to \mathit{mod\mbox{--}}\mathcal{F}_c(\mathcal{M}_{W_m}^{r,-},\mathcal{M}_{W_0}^{r,-}),\\ 
&(\gamma_1,\ldots,\gamma_m)\mapsto \tau_{(\gamma_1,\ldots,\gamma_m)},\text{ for }\gamma_i\in \mathcal{O}(H^+_{W_{i-1},W_i}),\;i=1,\ldots,m,
\end{aligned}
\end{equation}
where $\tau_{(\gamma_1,\ldots,\gamma_m)}$ is the module pre-morphism specified by
\begin{align*}
&(\tau_{(\gamma_1,\ldots,\gamma_m)})_{1|p}:CF(W_m,X_p;\mathscr{D}_{W_m,X_p}^-)\otimes \mathcal{F}_c(X_p,\ldots, X_0)\to CF(W_0,X_0;\mathscr{D}_{W_0,X_0}^-),\\
&\langle(\tau_{(\gamma_1,\ldots,\gamma_m)})_{1|p}(\bm{\zeta},\eta_p,\ldots, \eta_1),\bm{\zeta}'\rangle=\#_{\Z_2}\mathcal{R}^{m+p+2}_{\mathbf{Y}}(\gamma_1,\ldots,\gamma_m,\bm{\zeta};\eta_p,\ldots,\eta_1,\bm{\zeta}')^0,
\end{align*}
for $\bm{\zeta}\in \mathcal{O}(H^-_{W_m,X_p})$, $\eta_j\in \mathcal{O}(H^+_{X_j,X_{j-1}})$, $j=p,\ldots,1$, $\bm{\zeta}'\in \mathcal{O}(H^-_{W_0,X_0})$.
Here we use regularity of $\mathcal{R}^{m+p+2}_{\mathbf{Y}}(\gamma_1,\ldots,\gamma_m,\bm{\zeta};\eta_p,\ldots,\eta_1,\bm{\zeta}')$ and compactness of its zero-dimensional component. The fact that this definition of $\mathbf{Y}^r_\mathit{rel}$ yields a functor results from analysing once-broken configurations in the compactification of the one-dimensional component of $\mathcal{R}^{m+p+2}_{\mathbf{Y}}(\gamma_1,\ldots,\gamma_m,\bm{\zeta};\eta_p,\ldots,\eta_1,\bm{\zeta}')$.

\subsubsection{The natural transformation $\delta^{\mathcal{F}_{\mathit{c}}}$}
The relative weak Calabi-Yau pairing on $\mathcal{F}_c$ is represented by the natural transformation
\begin{equation}
\delta^{\mathcal{F}_{\mathit{c}}}=(\delta^{\mathcal{F}_{\mathit{c}}}_0,\delta^{\mathcal{F}_{\mathit{c}}}_1,\ldots): \mathbf{Y}_{\mathcal{F}_{\mathit{c}}}^l\to (\mathbf{Y}_{\mathit{rel}}^\vee)^l,
\end{equation}
defined as follows. The map
\begin{equation}
\delta^{\mathcal{F}_{\mathit{c}}}_p:\mathcal{F}_c(X_p,\ldots,X_0)\to \mathcal{F}_c\mathit{\mbox{--}mod}(\mathbf{Y}_{\mathcal{F}_{\mathit{c}}}^l(X_p),(\mathbf{Y}_{\mathit{rel}}^\vee)^l(X_0))
\end{equation} 
applied to the Hamiltonian chords $\eta_j\in \mathcal{O}(H^+_{X_j,X_{j-1}})$, $j=p,\ldots,1$, gives a module pre-morphism 
\begin{equation}
\delta^{\mathcal{F}_{\mathit{c}}}_p(\eta_p,\ldots, \eta_1):\mathbf{Y}_{\mathcal{F}_{\mathit{c}}}^l(X_p)\to (\mathbf{Y}_{\mathit{rel}}^\vee)^l(X_0),
\end{equation}
which in turn is specified by maps
\begin{equation}
\begin{aligned}
\delta^{\mathcal{F}_{\mathit{c}}}_p(\eta_p,\ldots, \eta_1)_{m|1}:&\mathcal{F}_c(W_0,\ldots, W_m)\otimes \mathcal{F}_c(W_m,X_p)\\
&\qquad\qquad\qquad\qquad\qquad\to CF(X_0,W_0;\mathscr{D}^-_{X_0,W_0})^\vee. 
\end{aligned}
\end{equation}
On the Hamiltonian chords $\gamma_i\in \mathcal{O}(H^+_{W_{i-1},W_i})$, $i=1,\ldots,m$, $\bm{\xi}\in \mathcal{O}(H^+_{W_m,X_p})$, and $\bm{\xi}'\in \mathcal{O}(H^-_{X_0,W_0})$, this is given by
\begin{equation}
\begin{aligned}
&\langle(\delta^{\mathcal{F}_{\mathit{c}}}_p(\eta_p,\ldots, \eta_1)_{m|1}(\gamma_1,\ldots,\gamma_m,\bm{\xi}),\bm{\xi}'\rangle\\
&\qquad\qquad=\#_{\Z_2} \mathcal{R}^{m,p;1}_\delta(\gamma_1,\ldots,\gamma_m,\bm{\xi};\eta_p,\ldots,\eta_1,\bm{\xi}')^0. 
\end{aligned}
\end{equation}
Here again we are relying on regularity and compactness results. In particular, the relation $\mu_1^{\mathit{fun}(\mathcal{F}_c,\mathcal{F}_c\mathit{\mbox{--}mod})} (\delta^{\mathcal{F}_{\mathit{c}}})=0$ arises from counting configurations in the boundary of the compactification of $\mathcal{R}^{m,p;1}_\delta(\gamma_1,\ldots,\gamma_m,\bm{\xi};\eta_p,\ldots,\eta_1,\bm{\xi}')^1$.

\subsubsection{The dual Hochschild cycle $\sigma^{\mathcal{F}_{\mathit{c}}}$}

We denote by $(\mathcal{F}_c)_\Delta^\mathit{rel}$ the $\mathcal{F}_c\mathit{\mbox{--}}\mathcal{F}_c$ bimodule corresponding to $\mathbf{Y}^r_{\mathit{rel}}$. Explicitly, $(\mathcal{F}_c)_\Delta^\mathit{rel}$ is defined as follows:
\begin{equation}
\begin{aligned}
&(\mathcal{F}_c)_\Delta^\mathit{rel} (W,X)=CF(W,X;\mathscr{D}^-_{W,X}),\\
&\mu^{(\mathcal{F}_c)_\Delta^\mathit{rel}}_{m|1|p}:\mathcal{F}_c(W_0,\ldots,W_m)\otimes CF(W_m,X_p;\mathscr{D}^-_{W_m,X_p})\otimes \mathcal{F}_c(X_p,\ldots,X_0)\\
&\qquad\qquad\qquad\qquad\qquad\qquad\qquad\qquad\qquad\rightarrow CF(W_0,X_0;\mathscr{D}^-_{W_0,X_0}),\\
&\langle\mu^{(\mathcal{F}_c)_\Delta^\mathit{rel}}_{m|1|p}(\gamma_1,\ldots,\gamma_m,\bm{\zeta},\eta_p,\ldots,\eta_1),\bm{\zeta}'\rangle=\#_{\Z_2}\mathcal{R}_{\mathbf{Y}}^{m+p+2}(\gamma_1,\ldots,\gamma_m,\bm{\zeta};\eta_p,\ldots,\eta_1,\bm{\zeta}')^0,
\end{aligned}
\end{equation}
for $\gamma_i\in \mathcal{O}(H^+_{W_{i-1},W_i})$, $i=1,\ldots,m$,
$\bm{\zeta}\in \mathcal{O}(H^-_{W_m,X_p})$,
$\eta_j\in \mathcal{O}(H^+_{X_j,X_{j-1}})$, $j=p,\ldots,1$,
$\bm{\zeta}'\in \mathcal{O}(H^-_{W_0,X_0})$.

We define a dual Hochschild cycle $\sigma^{\mathcal{F}_{\mathit{c}}}\in CC_\bullet(\mathcal{F}_c,(\mathcal{F}_c)_\Delta^\mathit{rel})^\vee$ by
\begin{equation}
\sigma^{\mathcal{F}_{\mathit{c}}}(\gamma_1\otimes\cdots\otimes\gamma_m\otimes\bm{\gamma})=\#_{\Z_2}\mathcal{R}_\sigma^{m+1;1}(\gamma_1,\ldots,\gamma_m,\bm{\gamma})^0,
\end{equation}
for $\gamma_i\in \mathcal{O}(H^+_{W_{i-1},W_i})$, $i=1,\ldots,m$, and $\bm{\gamma}\in \mathcal{O}(H^-_{W_m,W_0})$. It follows from the properties of the moduli spaces $\mathcal{R}_\sigma^{m+1;1}(\gamma_1,\ldots,\gamma_m,\bm{\gamma})^0$ and $\mathcal{R}_\sigma^{m+1;1}(\gamma_1,\ldots,\gamma_m,\bm{\gamma})^1$ that this is a well-defined closed element of $CC_\bullet(\mathcal{F}_c,(\mathcal{F}_c)_\Delta^\mathit{rel})^\vee$.

\subsubsection{The main theorem}
\begin{thm}\label{THM: Main theorem}
For a generic consistent choice of Floer and perturbation data as in Section \ref{SUBSECTION: Floer data and perturbation data}, the following statements hold:
\begin{enumerate}
\item\label{THM: Main theorem Hochschild part} 
\begin{enumerate}
\item\label{THM: Main theorem, Hochschild part, part I} The $\mathcal{F}_c\mathit{\mbox{--}}\mathcal{F}_c$ bimodule $(\mathcal{F}_c)_\Delta^\mathit{rel}$ is well-defined and $\sigma^{\mathcal{F}_{\mathit{c}}}$ is a closed homologically non-degenerate element of $CC_\bullet(\mathcal{F}_c,(\mathcal{F}_c)_\Delta^\mathit{rel})^\vee$. 
\item\label{THM: Main theorem, Hochschild part, part II} There is a functor $\mathbf{I}_\gamma:\mathcal{F}\to\mathcal{F}_{\mathit{c}}$, which depends on a choice of curve $\gamma$ in $\C$, and there is an $\mathcal{F}\mbox{--}\mathcal{F}$ bimodule morphism $i_{\mathit{rel}}:\mathcal{F}_\Delta\to (\mathbf{I}_\gamma)^*(\mathcal{F}_c)_\Delta^\mathit{rel}$ such that $[(\mathbf{I}_{\gamma}^{\mathit{rel}})_*^\vee \sigma^{\mathcal{F}_{\mathit{c}}}]=[\sigma^\mathcal{F}]$. Here $(\mathbf{I}_\gamma^{\mathit{rel}})_*$ is the map on Hochschild homology defined in \eqref{EQ: Def of Irel} for the functor $\mathbf{I}_\gamma$. 
\end{enumerate}
In other words, $\sigma^{\mathcal{F}_{\mathit{c}}}$ represents a relative weak Calabi-Yau pairing on $\mathcal{F}_c$ by Definition \ref{DEF: Relative weak CY pairing CC defn} with coefficients in $(\mathcal{F}_c)_\Delta^\mathit{rel}$, and this pairing is compatible with the usual weak Calabi-Yau structure $\sigma^\mathcal{F}$ on $\mathcal{F}$ as defined in Section \ref{SUBSUBSECTION: wCY structure on F CC version}.
\item\label{THM: Main theorem nat transf part}
\begin{enumerate}
\item\label{THM: Main theorem, Nat transf part, part I} $\mathbf{Y}^r_{\mathit{rel}}$ is a well-defined functor from $\mathcal{F}_c$ to $(\mathit{mod\mbox{--}}\mathcal{F}_c)^{\mathit{opp}}$ and $\delta^{\mathcal{F}_{\mathit{c}}}:\mathbf{Y}^l_{\mathcal{F}_{\mathit{c}}}\to (\mathbf{Y}^\vee_{\mathit{rel}})^l$ is a natural transformation which is a quasi-isomorphism. 
\item\label{THM: Main theorem, Nat transf part, part II} There is a natural transformation $S^{\mathit{rel}}:\mathbf{G}^r_{\mathbf{I}_\gamma}(\mathbf{Y}^r_\mathit{rel})\to \mathbf{Y}^r_\mathcal{F}$ such that the natural transformation $\mathcal{P}_\mathit{rel}(\delta^{\mathcal{F}_c}):\mathbf{Y}^l_\mathcal{F}\to (\mathbf{Y}^\vee_\mathcal{F})^l$ induced by $\delta^{\mathcal{F}_{\mathit{c}}}$ as in \eqref{EQ: Def of mathcalP(delta)} satisfies $[\mathcal{P}_\mathit{rel}(\delta^{\mathcal{F}_c})]=[\delta^\mathcal{F}]$. 
\end{enumerate}
In other words, $\delta^{\mathcal{F}_{\mathit{c}}}$ represents a relative weak Calabi-Yau pairing on $\mathcal{F}_c$ by Definition \ref{DEFN: Relative weak CY pairing Yoneda version} for the functor $\mathbf{I}_\gamma$ and the relative right Yoneda functor $\mathbf{Y}^r_{\mathit{rel}}$, and this pairing is compatible with the usual weak Calabi-Yau structure $\delta^\mathcal{F}$ on $\mathcal{F}$ as defined in Section \ref{SUBSUBSECTION: wCY structure on F Yoneda version}.
\item\label{THM: Main theorem, equivalence part} Denote by $\phi^{\mathcal{F}_c}:(\mathcal{F}_c)_\Delta\to (\mathcal{F}_c)_\mathit{rel}^\vee$ the bimodule quasi-isomorphism corresponding to the weak Calabi-Yau pairing in (\ref{THM: Main theorem nat transf part}). The isomorphism induced on homology by the quasi-isomorphism of chain complexes
\begin{equation}
CC_\bullet(\mathcal{F}_c,(\mathcal{F}_c)_\Delta^\mathit{rel})^\vee\xrightarrow{T^\vee} {}_2CC_\bullet (\mathcal{F}_c,(\mathcal{F}_c)_\Delta^\mathit{rel})^\vee\xrightarrow{\Gamma} {}_2CC^\bullet (\mathcal{F}_c,(\mathcal{F}_c)^\vee_\mathit{rel})
\end{equation}
takes $[\sigma^{\mathcal{F}_{\mathit{c}}}]$ to $[\phi^{\mathcal{F}_c}]$. Here $T^\vee$ is the quasi-isomorphism dual to \eqref{EQ: iso from two-pointed cochain complex to one-pointed cochain complex}, and $\Gamma$ is the isomorphism of Lemma \eqref{LEM: This iso 2CCn(A,Mvee) cong 2CC_n(A,M)vee}.
\end{enumerate}

\end{thm}

\section{Proof of the main theorem}\label{SECTION: Proofs of the main results}

\subsection{Compactness}

The main goal of this section is to prove a general compactness result that extends Lemma 3.3.2 in \cite{BC14}. This compactness result, Lemma \ref{LEM: Images of projections of curves are contained}, implies that for all of the moduli spaces of inhomogeneous pseudoholomorphic polygons in $\widetilde{M}$ that we consider, there is a compact region of $\widetilde{M}$ containing the images of all of these curves. This, together with uniform energy bounds on the curves, allows us to apply the Gromov compactness theorem even in this non-compact setting. The latter part of this section deals with these energy bounds.

We will make use of the following general proposition which appears as Proposition 3.3.1 in \cite{BC14}:

\begin{prop}\label{PROP: Open mapping theorem compactness prop} 
Let $\Sigma$ and $\Gamma$ be Riemann surfaces, not necessarily compact, $\Sigma$ possibly with boundary and $\Gamma$ without boundary. Let $w:\Sigma\to\Gamma$ be a continuous map and $U\subset \Gamma$ be an open connected subset. Suppose the following conditions are satisfied:
\begin{enumerate}
\item $\mathrm{Image}(w)\cap U\ne \emptyset$.
\item $w$ is holomorphic over $U$, i.e.\ $w|_{w^{-1}(U)}:w^{-1}(U)\to U$ is holomorphic.
\item $w(\partial \Sigma)\cap U=\emptyset$.
\item $(\overline{\mathrm{Image}(w)}\setminus \mathrm{Image}(w))\cap U=\emptyset$.
\end{enumerate}
Then $\mathrm{Image}(w)\supset U$. In particular, if $\overline{\mathrm{Image}(w)}\subset\Gamma$ is compact, then so is $\overline{U}$.
\end{prop}

Let $\mathcal{S}\to \mathcal{R}$ denote any of the bundles $\mathcal{S}^{k+1}\to \mathcal{R}^{k+1}$ ($k\ge 1$),  $\mathcal{S}^{k+1;1}\to \mathcal{R}^{k+1;1}$ ($k\ge 0$), or $\mathcal{S}^{m,p;1}\to \mathcal{R}^{m,p;1}$ ($k:=m+p+1\ge 1$). For convenience we associate to an input set $\mathcal{I}^{k+1}$ a $(k+1)$-tuple of signs $\bar{\kappa}=(\kappa_1,\ldots,\kappa_{k+1})$ defined by $\kappa_i=-$ for $i\in \mathcal{I}^{k+1}$ and $\kappa_i=+$ for $i\not\in \mathcal{I}^{k+1}$. The input set $\mathcal{I}^{k+1}$ is defined to be \textbf{compatible} with 
a $(k+1)$ sign datum $\bar{\alpha}$ if there exist $i,i'\in\{1,\ldots, k+1\}$ such that 
$\alpha_i=\kappa_i$ and $\alpha_{i'}\not=\kappa_{i'}$.

\begin{lem}\label{LEM: Images of projections of curves are contained} 
Let $W_0,\ldots, W_k$ be $k+1$ cobordisms in the class $\mathcal{CL}_d(\widetilde{M})$. Fix an input set $\mathcal{I}^{k+1}$ and a $(k+1)$ sign datum $\bar{\alpha}$ which are compatible, as well as a universal choice of strip-like ends $\{\epsilon_i\}$ on $\mathcal{S}$ for the input set $\mathcal{I}^{k+1}$, and a global transition function $\bm{a}:\mathcal{S}\to [0,1]$ for the input set $\mathcal{I}^{k+1}$ and the $(k+1)$ sign datum $\bar{\alpha}$. Also fix a Floer datum $\mathscr{D}^{\alpha_i}_{W_{i-1},W_i}$ for all $i\in \mathcal{I}^{k+1}$ and a Floer datum $\mathscr{D}^{\alpha_i}_{W_{i},W_{i-1}}$ for all $i\in \{1,\ldots,k+1\}\setminus\mathcal{I}^{k+1}$ (where as usual we set $W_{k+1}=W_0$), and a perturbation datum  $\mathscr{D}_{W_0,\ldots,W_k}$ satisfying the conditions in Section \ref{SUBSECTION: Floer data and perturbation data} for this choice of data. Then there exists a constant $C=C_{W_0,\ldots, W_k}$ which depends only on $W_0,\ldots, W_k$ and the Floer and perturbation data (but not on the strip-like ends or global transition function) such that for all $r\in \mathcal{R}$ and every solution $u:\mathcal{S}_r\to \widetilde{M}$ of \eqref{EQ: Perturbed CR equation}, we have $u(\mathcal{S}_r)\subset B_{W_0,\ldots, W_k}\times M$, where $B_{W_0,\ldots, W_k}=[-\tfrac{3}{2},\tfrac{5}{2}]\times [-C,C]$. 
\end{lem}

\begin{proof} 

This proof follows the proof of Lemma 3.3.2 in \cite{BC14} closely. We indicate the main adjustments that need to be made.

We will arrive at the constant $C$ via an auxiliary constant $C'$ which we define now. Take $C'>0$ to be large enough such that the set $B'=\left[-\tfrac{3}{2},\tfrac{5}{2}\right]\times[-C',C']$ satisfies the following conditions for every $t\in[0,1]$:
\begin{enumerate}
\item  $B'\supset (\phi_t^h)^{-1}(\pi (W_i)\cap \left[-\tfrac{3}{2},\tfrac{5}{2}\right]\times\R ),\;i=0,\ldots,k,$
\item $B'\supset (\phi_t^h)^{-1}(K_{W_0,\ldots,W_k}),$
\item 
\begin{enumerate}
\item For every $i\in \{1,\ldots, k+1\}$ such that $(\kappa_i,\alpha_i)=(-,+)$, $B'\supset(\phi_t^h)^{-1}(\pi(\gamma(t)))$ for every chord $\gamma\in\mathcal{O}(H^{\alpha_i}_{W_{i-1},W_i})$. 
\item For every $i\in \{1,\ldots, k+1\}$ such that $(\kappa_i,\alpha_i)=(+,+)$, $B'\supset (\phi_t^h)^{-1}(\pi(\gamma(t)))$ for every chord $\gamma\in\mathcal{O}(H^{\alpha_i}_{W_i,W_{i-1}})$.
\item For every $i\in \{1,\ldots, k+1\}$ such that $(\kappa_i,\alpha_i)=(-,-)$, $B'\supset (\phi_{1-t}^h)^{-1}(\pi(\gamma(t)))$ for every chord $\gamma\in\mathcal{O}(H^{\alpha_i}_{W_{i-1},W_i})$.
\item For every $i\in \{1,\ldots, k+1\}$ such that $(\kappa_i,\alpha_i)=(+,-)$, $B'\supset(\phi_{1-t}^h)^{-1}(\pi(\gamma(t)))$ for every chord $\gamma\in\mathcal{O}(H^{\alpha_i}_{W_{i},W_{i-1}})$.
\end{enumerate}
\end{enumerate}

The existence of such a constant $C'$ relies on condition \eqref{DEF: Profile fn COND 3b} on the profile function $h$, which says that $\phi_t^h$ preserves the strip $\left[-\tfrac{3}{2},\tfrac{5}{2}\right]\times \R$ for all $t$. It also makes use of the fact mentioned in Remark \ref{REMK: Chords are contained and finite} that the projections of the time-1 Hamiltonian chords involved are contained in the strip $[-\tfrac{3}{2},\tfrac{5}{2}]\times\R$.

The proof is based on the following Auxiliary Lemma which also appears in \cite{BC14}, but here there is a dependence of the form of the transition functions $a_r$ on the input set $\mathcal{I}^{k+1}$ and on the $(k+1)$ sign datum $\bar{\alpha}$.

\vspace*{\baselineskip}

\noindent\textbf{Auxiliary Lemma}. Let $r\in \mathcal{R}$ and let $u:\mathcal{S}_r\to \widetilde{M}$ be a solution of \eqref{EQ: Perturbed CR equation}. Take $v:\mathcal{S}_r\to\widetilde{M}$ to be defined by the naturality transformation formula:
\begin{equation}\label{EQ: naturality transformation}
v(z)=\left(\phi_{a_r(z)}^{\tilde{h}}\right)^{-1}(u(z)),
\end{equation}
where $\tilde{h}=h\circ \pi$. Then $\mathrm{Image}(v)\subset B'\times M$.

\vspace*{\baselineskip}

\noindent\textbf{Proof of Auxiliary Lemma}. A computation shows that the inhomogeneous Cauchy-Riemann equation \eqref{EQ: Perturbed CR equation} for $u$ transforms to the following equation for $v$:
\begin{equation}\label{EQ: Transformed CR equation}
Dv+J'_{r,z}(v)\circ Dv\circ j = Y'_r + J'_{r,z}(v)\circ Y'_r\circ j.
\end{equation}
Here $Y'_r\in \Omega^1(\mathcal{S}_r,C^\infty(T\widetilde{M}))$ and the almost complex structure $J'_{r,z}$ on $\widetilde{M}$ are defined by
\begin{equation}\label{EQ: Def of Y' and J'}
Y_r=D\phi^{\tilde{h}}_{a_r(z)}(Y'_r)+da_r\otimes X^{\tilde{h}},\quad J_{r,z}=(\phi^{\tilde{h}}_{a_r(z)})_*J'_{r,z}.
\end{equation}
The map $v$ satisfies the moving boundary conditions:
\begin{equation}
\forall z\in C_i,\; v(z)\in \left(\phi^{\tilde{h}}_{a_r(z)}\right)^{-1}(W_i).
\end{equation}
In contrast to the scenario in \cite{BC14}, the asymptotic conditions on $v$ depend on the input set $\mathcal{I}^{k+1}$ and the $(k+1)$ sign datum $\bar{\alpha}$. They are as follows:
\begin{itemize}
\item For $i\in\{1,\ldots, k+1\}$ such that $(\kappa_i,\alpha_i)=(-,+)$, $v(\epsilon_i(s,t))$ tends as $s\to -\infty$ to a time-$1$ chord of $(\phi_t^{\tilde{h}})^{-1}\circ(\phi_t^{H^+_{W_{i-1},W_i}})$ starting on $W_{i-1}$ and ending on $(\phi_1^{\tilde{h}})^{-1}(W_i)$.
\item For $i\in\{1,\ldots, k+1\}$ such that $(\kappa_i,\alpha_i)=(+,+)$, $v(\epsilon_i(s,t))$ tends as $s\to +\infty$ to a time-$1$ chord of $(\phi_t^{\tilde{h}})^{-1}\circ(\phi_t^{H^+_{W_i,W_{i-1}}})$ starting on $W_i$ and ending on $(\phi_1^{\tilde{h}})^{-1}(W_{i-1})$.
\item For $i\in\{1,\ldots, k+1\}$ such that $(\kappa_i,\alpha_i)=(-,-)$, $v(\epsilon_i(s,t))$ tends as $s\to -\infty$ to a time-$1$ chord of $(\phi_{1-t}^{\tilde{h}})^{-1}\circ(\phi_t^{H^-_{W_{i-1},W_i}})$ starting on $(\phi_1^{\tilde{h}})^{-1}(W_{i-1})$ and ending on $W_i$.
\item For $i\in\{1,\ldots, k+1\}$ such that $(\kappa_i,\alpha_i)=(+,-)$, $v(\epsilon_i(s,t))$ tends as $s\to +\infty$ to a time-$1$ chord of $(\phi_{1-t}^{\tilde{h}})^{-1}\circ(\phi_t^{H^-_{W_i,W_{i-1}}})$ starting on $(\phi_1^{\tilde{h}})^{-1}(W_i)$ and ending on $W_{i-1}$.
\end{itemize}

As in \cite{BC14}, we prove the Auxiliary Lemma using three claims concerning the map $v':=\pi\circ v:\mathcal{S}_r\to\C$. Although in our case the form of $v'$ depends on the more general transition function $a_r$, no significant alteration is required for the proofs of these claims. Indeed, the proofs of Claims 2 and 3 are the same as those that appear in \cite{BC14}. The proof of Claim 1 makes use of the more general perturbation data and transition function in our case; however no other adjustment is required. 

\vspace*{\baselineskip}

\noindent \textbf{CLAIM 1.} There exists $\delta>0$ small enough such that $v'$ is $(j,i)$-holomorphic over $\C\setminus([-\tfrac{3}{2}+\delta,\tfrac{5}{2}-\delta]\times [-C',C'])$. 

\vspace*{\baselineskip}

Using the definition of $(\mathbf{\Theta}_0)_r$ (Equation \eqref{EQ: Def of Theta_0}), the one-form $Y_r\in\Omega^1(\mathcal{S}_r,C^\infty(T\widetilde{M}))$ is given by 
\begin{equation}\label{EQ: The form Y_r}
Y_r=da_r\otimes X^{\tilde{h}}+Y_0,
\end{equation} 
where $(D\pi)_{(w,p)}(Y_0)=0$ for all $(w,p)\in (\C\setminus K_{W_0,\ldots, W_k})\times M$. Together with the definition of $Y'$ in Equation \eqref{EQ: Def of Y' and J'}, we obtain $Y_0=D\phi^{\tilde{h}}_{a(z)}(Y')$. Since the vector field $X^{\tilde{h}}$ is horizontal with respect to the projection $\pi:\widetilde{M}\to \C$, its flow $\phi_t^{\tilde{h}}$ carries vertical vector fields  to vertical vector fields. Therefore
$Y'$ satisfies 
\begin{equation}\label{EQ: Y' is vertical}
(D\pi)_{(w,p)}(Y')=0\text{ for }(w,p)\in (\phi^{h}_{a(z)})^{-1}(\C\setminus K_{W_0,\ldots,W_k})\times M. 
\end{equation}
Choose $\delta>0$ such that $(\phi^h_t)^{-1}(K_{W_0,\ldots,W_k})\subset [-\tfrac{3}{2}+\delta,\tfrac{5}{2}-\delta]\times [-C',C']$ for all $t\in[0,1]$. Then for all $z\in \mathcal{S}_r$ such that $v'(z)$ is contained in the set $\C\setminus ([-\tfrac{3}{2}+\delta,\tfrac{5}{2}-\delta]\times [-C',C'])$, we have $(D\pi)_{v(z)}(Y')=0$.  

By condition \eqref{Perturbation data conditions, PART projection holomorphic}  on the perturbation datum $\mathscr{D}_{W_0,\ldots, W_k}$, the projection $\pi$ is $\left(J_z,\left(\phi^h_{a_r(z)}\right)_*i\right)$-holomorphic over $(\C\setminus K_{W_0,\ldots, W_k})$. Equivalently, $\pi$
is $((\phi^{\tilde{h}}_{a_r(z)})^{-1}_*J_z,i)$-holomorphic over $(\C\setminus (\phi^h_{a_r(z)})^{-1}(K_{W_0,\ldots, W_k}))$. This together with the inhomogeneous Cauchy-Riemann equation \eqref{EQ: Transformed CR equation} for $v$ and \eqref{EQ: Y' is vertical} gives 
\begin{equation*}
Dv'+iDv'\circ j=0
\end{equation*}
for all $z\in \mathcal{S}_r$ such that $v'(z)\in  (\phi^{h}_{a(z)})^{-1}(\C\setminus K_{W_0,\ldots,W_k})$. Therefore the claim holds for the chosen $\delta$. 

\vspace*{\baselineskip}

Set
\begin{align*}
R &= \left(\bigcup_{i=0}^k\bigcup_{t\in[0,1]}(\phi_t^h)^{-1}(\pi(W_i))\right)\cap (\C\setminus B'), &Q=\C\setminus (R\cup B'),\\
A&=\left(\left\{(-\tfrac{3}{2},q)|\;q\in\Z\right\}\cup\left\{(\tfrac{5}{2},q)|\;q\in\Z\right\}\right)\cap B'.
\end{align*}

\vspace*{\baselineskip}

\noindent\textbf{CLAIM 2.} $v'(\mathcal{S}_r)\cap Q=\emptyset$.

\vspace*{\baselineskip}

Assume to the contrary that $v'(\mathcal{S}_r)\cap Q\ne \emptyset$. Let $Q_0$ be one of the connected components of $Q$ that has non-empty intersection with $v'(\mathcal{S}_r)$. Note that $v'$ maps all of the boundary arcs $C_i$ connecting the punctures of $\mathcal{S}_r$ into $R\cup B'$. Moreover all of the Hamiltonian chords forming the asymptotic conditions of $v$ are contained in $B'\times M$. Therefore $v'$ satisfies the conditions $v'(\partial \mathcal{S}_r)\cap Q_0=\emptyset$ and $(\overline{v'(\mathcal{S}_r)}\setminus v'(\mathcal{S}_r))\cap Q_0=\emptyset$. Since $Q_0$ is open, we can apply Proposition \ref{PROP: Open mapping theorem compactness prop} with $w=v'$, $\Sigma=\mathcal{S}_r$, $\Gamma=\C$ and $U=Q_0$ to conclude that $Q_0\subset v'(\mathcal{S}_r)$. This is a contradiction as all of the connected components of $Q$ are unbounded, but $\overline{v'(\mathcal{S}_r)}$ is compact. This completes the proof of Claim 2.

\vspace*{\baselineskip}

\noindent\textbf{CLAIM 3.} $v'(\mathrm{Int}(\mathcal{S}_r))\cap A=\emptyset$.

\vspace*{\baselineskip}

This is a direct consequence of the open mapping theorem together with Claim 1: By Claim 1, $v'$ is holomorphic over a neighbourhood of $A$. If there were a point $z\in \mathrm{Int}(\mathcal{S}_r)$ and a point $P\in A$ such that $v'(z)=P$, then a neighbourhood of $P$ would have to be contained in $v'(\mathcal{S}_r)$. But every neighbourhood of $P$ intersects $Q$, and therefore this would contradict Claim 2. 

\vspace*{\baselineskip}

We now use the three claims to prove the Auxiliary Lemma. Since $\C=Q\sqcup R\sqcup B'$, by Claim 2 it suffices to show that $\mathrm{Image}(v')\cap R=\emptyset$. Suppose by way of contradiction that $\mathrm{Image}(v')\cap R\not =\emptyset$. 

There are two possibilities:
\begin{enumerate}
\item\label{Enum: Possibilites for v' 1} $\mathrm{Image}(v')\subset R\cup A$.
\item\label{Enum: Possibilites for v' 2} $\mathrm{Image}(v')\cap (B'\setminus A)\ne\emptyset$. 
\end{enumerate}

In the second case, since we are assuming $\mathrm{Image}(v')\cap R\not =\emptyset$, there must be two points $z_0,z_1\in \mathcal{S}_r$ such that $v'(z_0)\in R$ and $v'(z_1)\in B'\setminus A$. Let $\{z_t\}_{t\in[0,1]}$ be a path in $\mathcal{S}_r$ from $z_0$ to $z_1$ such that $z_t\in \mathrm{Int}(\mathcal{S}_r)$ for all $t\in (0,1)$. Since $\partial B'\cap \partial R\subset A$, there must exist $t_0\in (0,1)$ such that $v'(z_{t_0})\in A$. This contradicts Claim 3, and hence rules out the second possibility.

Suppose now that the first possibility occurs. Then for one of the connected components $R_0$ of $R$, we have $\mathrm{Image}(v')\subset \bar{R}_0=R_0\cup \{P\}$. Here $P$ is a point of $A$ and so is of the form $(-\tfrac{3}{2},q)$ or $(\tfrac{5}{2},q)$ for some $q\in\Z$. Consider the case where $P=(-\tfrac{3}{2},q)$. Since $\mathrm{Image}(v')\subset \bar{R}_0$, the Hamiltonian chords giving the asymptotic conditions for $v$ are contained in $\bar{R}_0\times M$ and hence project to constant chords at $P$ (see Remark \ref{REMK: Chords are contained and finite}). By assumption, the input set $\mathcal{I}^{k+1}$ and the sign datum $\bar{\alpha}$ are compatible, and so there is a puncture $z_i$ with $(\kappa_i,\alpha_i)=(\pm,\pm)$. We first treat the case where $(\kappa_i,\alpha_i)=(+,+)$, which corresponds to the case considered in \cite{BC14}, and we follow their proof. In this case, $v'$ satisfies the asymptotic condition
\begin{equation}
\lim_{s\to \infty}(v'(\epsilon_i(s,t)))=P.
\end{equation}
Set $v''=v'\circ \epsilon_i:Z^+\to\C$. Then from the boundary conditions on $u$, $v''(s,0)\in \pi((\phi^{\tilde{h}}_{a_r(\epsilon_i(s,0))})^{-1}(W_i))\cap \bar{R}_0=\pi (W_i)\cap \bar{R}_0=(-\infty, -\tfrac{3}{2}]\times \{q\}$ for all $s\in [0,\infty)$. Therefore $\partial_sv''(s,0)$ is real-valued for all $s\in [0,\infty)$. Moreover, for every $s\ge 0$, $\partial_s v''(s,0)\le 0$. Indeed, if we have $s_0\ge 0$ such that $\partial_s v''(s_0,0)>0$, then since $v''$ is holomorphic, we obtain $\partial_t v''(s_0,0)=i\partial_s v''(s_0,0)\in i \R_{>0}$. This implies that for $t\in [0,1]$ close enough to $0$, $\mathrm{Im}(v''(s_0,t))>q$ and hence $\mathrm{Image}(v'')\cap Q\ne \emptyset$, contradicting Claim 2. From $\partial_s v''(s,0)\le 0$ for all $s\ge 0$ and $\lim_{s\to \infty}v''(s,0)=P$, we conclude that $v''(s,0)=P$ for all $s\ge 0$. This makes it possible to extend $v''$ by Schwarz reflection to a holomorphic map $v'':[0,\infty)\times [-1,1]\to\C$. Since the extended function satisfies $v''([0,\infty)\times\{0\})=P$, by the open mapping theorem, $v''$ must be constant at the point $P$. Then $v'$ itself must be constant at $P$, contradicting the assumption that $\mathrm{Image}(v')\cap R\not=\emptyset$.

In contrast to \cite{BC14}, we must also consider the case where $(\kappa_i,\alpha_i)=(-,-)$. In this case, define a positive universal strip-like end $\epsilon_i':\mathcal{R}\times Z^+\to \mathcal{S}$ by $\epsilon'_i(r,s,t)=\epsilon_i(r,-s,1-t)$. Then set $v''':=v\circ \epsilon'_i :Z^+\to\C$. The map $v'''$ is holomorphic and satisfies $v'''([0,\infty)\times \{0\})\subset \pi(W_i)$, $v'''([1,\infty)\times\{1\})\subset (\phi_1^h)^{-1}(\pi(W_{i-1}))$, and $v'''$ converges as $s\to \infty$ to a time-1 chord of $(\phi_t^h)^{-1}\circ (\phi_{1-t}^{H^-_{W_{i-1},W_i}})$. It again follows from Remark \ref{REMK: Chords are contained and finite} that this chord is constant at $P$. Therefore the same argument as above applies to show that $v'''$ is constant at $P$. 

In the case where $P=(\tfrac{5}{2},q)$, similar arguments applied to the puncture with $(\kappa_i,\alpha_i)=(\pm,\mp)$ show that $v'$ must be constant at $P$. This completes the proof of the Auxiliary Lemma.

\vspace*{\baselineskip}

We now use the Auxiliary Lemma to complete the proof of Lemma \ref{LEM: Images of projections of curves are contained}. This part of the proof proceeds as in \cite{BC14}. Let $C=C_{W_0,\ldots,W_k}>0$ be a constant large enough such that the set $B=[-\tfrac{3}{2},\tfrac{5}{2}]\times [-C,C]$ satisfies
$$\bigcup_{\tau\in [0,1]} \phi_\tau^{\tilde{h}}(B')\subset B.$$
Let $r\in \mathcal{R}$ and consider a solution $u:\mathcal{S}_r\to \widetilde{M}$ of \eqref{EQ: Perturbed CR equation}. Let $v:\mathcal{S}_r\to\C$ be defined as in \eqref{EQ: naturality transformation}. By the Auxiliary Lemma we have
$$v(z)\in B'\times M,\;\forall z\in \mathcal{S}_r.$$
Then $u$ satisfies
$$u(z)=\phi_{a_r(z)}^{\tilde{h}}(v(z))\in \phi^{\tilde{h}}_{a_r(z)}(B')\times M\subset B\times M,\;\forall z\in \mathcal{S}_r.$$
Therefore the statement of Lemma \ref{LEM: Images of projections of curves are contained} holds for this choice of constant $C$. 
\end{proof}

\subsubsection{Energy bounds}The next step in showing that Gromov compactness continues to hold in this setting is to establish uniform bounds on the energy of the curves in the moduli spaces we consider. Recall that the energy of a solution $u:\mathcal{S}_r\to\widetilde{M}$ of \eqref{EQ: Perturbed CR equation} is defined by
\begin{equation}\label{EQ: energy definition}
E(u):=\int_{\mathcal{S}_r} \tfrac{1}{2}|Du-Y|_\mathbf{J}^2 \sigma,
\end{equation}
where $\sigma$ is an arbitrary choice of volume form on $\mathcal{S}_r$. The energy is independent of this choice, although the energy density $|Du-Y|_\mathbf{J}^2$ does depend on $\sigma$.

Let $\mathcal{R}(k+1;\Gamma_{k+1};\mathscr{D})$ denote any one of the moduli spaces of inhomogeneous pseudoholomorphic curves of type \eqref{ENUM: Curve configuration 1}--\eqref{ENUM: Curve configuration 5} in Section \ref{SUBSUBSECTION: The curve configurations}. Here $k+1$ is the total number of boundary punctures on the domain, $\Gamma_{k+1}$ is the $(k+1)$-tuple of Hamiltonian chords giving the asymptotic conditions, and $\mathscr{D}=(\mathbf{\Theta},\mathbf{J})$ is the perturbation datum. We use $\mathcal{R}(k+1;\Gamma_{k+1};\mathscr{D})^s$ to denote the component of $\mathcal{R}(k+1;\Gamma_{k+1};\mathscr{D})$ of virtual dimension $s$.

\begin{lem}\label{LEM: Energy bounds} For each of the moduli spaces $\mathcal{R}(k+1;\Gamma_{k+1};\mathscr{D})^s$ there exists a constant $A$ depending on $s\in\N$, on the consistent choice of global transition functions $\bm{a}$ (i.e.\ the functions $\bm{a}_\mu$, $\bm{a}^{m,p}_{\mathbf{Y}}$, $\bm{a}_{\delta}$, and $\bm{a}_{\sigma}$), on the Hamiltonians $H^\pm_{W_{i-1},W_i}$ and $H^\pm_{W_{i},W_{i-1}}$ giving the asymptotic conditions on $\mathbf{\Theta}$, as well as on the profile function $h$, the perturbation datum $\mathscr{D}$, and the $(k+1)$-tuple of Hamiltonian chords $\Gamma_{k+1}$, such that for any $(r,u)\in \mathcal{R}(k+1;\Gamma_{k+1};\mathscr{D})^s$,
\begin{equation}
E(u)\le A.
\end{equation}
\end{lem}

\begin{proof}
The proof proceeds precisely as in the case considered in \cite{BC14} and we only provide an outline. There are three steps involved. 
The first step is to establish the following property of the global transition functions $\bm{a}$ which is a generalization to our context of Lemma 3.1.1 in \cite{BC14}: There exists a constant $C^k_{\bm{a}}$, which depends on the metric $\rho$ on $\overline{\mathcal{S}}$, so that for any $r\in\overline{\mathcal{R}}$ and any $\xi\in T(\partial \overline{\mathcal{S}}_r)$, we have
\begin{equation}\label{EQ: transition function inequality}
|da_r(\xi)|\le C^k_{\bm{a}} |\xi|_\rho
\end{equation}
The proof of this property relies on the compatibility of the metrics $\rho$ as well as of the global transition functions. Compatibility of the global transition functions implies in particular that $\bm{a}$ extends to the partial compactification $\overline{\mathcal{S}}$ of $\mathcal{S}$, and so the above inequality makes sense. The steps in the proof of \eqref{EQ: transition function inequality} are outlined in \cite{BC14}. 

The second step in the proof of Lemma \ref{LEM: Energy bounds} is to bound the energy of curves $u$ with $(r,u)\in \mathcal{R}(k+1;\Gamma_{k+1};\mathscr{D})$ in terms of their symplectic area. More precisely, we can show that there exists a constant $C$ depending on the global transition functions $\bm{a}$, the Hamiltonians $H^\pm_{W_{i-1},W_i}$ and $H^\pm_{W_{i},W_{i-1}}$ giving the asymptotic conditions on $\mathbf{\Theta}$, the profile function $h$, and the perturbation datum $\mathscr{D}$, such that for $(r,u)\in \mathcal{R}(k+1;\Gamma_{k+1};\mathscr{D})$ with $u$ of any index,
\begin{equation}\label{EQ: Energy bound in terms of symplectic area}
E(u)\le \int_{\mathcal{S}_r}u^*\widetilde{\omega}+C.
\end{equation}
This is a generalization of Lemma 3.3.3 in \cite{BC14}. The proof involves re-expressing the integrand in \eqref{EQ: energy definition} as
\begin{equation}\label{EQ: Rewriting energy integrand}
\tfrac{1}{2}|Du-Y|_\mathbf{J}^2 \sigma=u^*\widetilde{\omega}-d(u^*\Theta_r)+R^{\Theta_r}(u).
\end{equation} 
Here the term $R^{\Theta_r} \in\Omega^2(\mathcal{S}_r,C^\infty(\widetilde{M}))$ is the curvature form associated to $\Theta_r$. The integral of $R^{\Theta_r}(u)$ satisfies a uniform bound (independent of $r$ and $u$) since for every pair $(r,u)\in \mathcal{R}(k+1;\Gamma_{k+1};\mathscr{D})$, the curve $u$ has image contained in the compact set $B\times M$ by Lemma \ref{LEM: Images of projections of curves are contained}. The integral of the term $-d(u^*\Theta_r)$ in \eqref{EQ: Rewriting energy integrand} can be bounded by another application of Lemma \ref{LEM: Images of projections of curves are contained} and using the existence of the constant $C^k_{\bm{a}}$ above. 

The final step in the proof of Lemma \ref{LEM: Energy bounds} is to establish a uniform bound on the symplectic area of solutions of \eqref{EQ: Perturbed CR equation} of a fixed index. This is done by showing that for $(r,u), (r',u')\in \mathcal{R}(k+1;\Gamma_{k+1};\mathscr{D})$,
\begin{equation}\label{EQ: symplectic area Maslov index relation}
\int_{\mathcal{S}_r} u^*\widetilde{\omega}-\int_{\mathcal{S}_{r'}} u'^*\widetilde{\omega}=\tau(\mu(u)-\mu(u')),
\end{equation}
where $\mu$ is the Maslov index and $\tau$ is the monotonicity constant. The identity \eqref{EQ: symplectic area Maslov index relation} is proved using a variation of an argument due to Oh for Floer strips \cite{Oh2}, and uses the condition on $W\in \mathcal{CL}_d(\widetilde{M})$ that the inclusion of $\pi_1(W)$ in $\pi_1(\widetilde{M})$ is trivial. From \eqref{EQ: symplectic area Maslov index relation} it follows that all $(r,u)\in \mathcal{R}(k+1;\Gamma_{k+1};\mathscr{D})^s$ have the same symplectic area, and we then obtain the uniform energy bound from \eqref{EQ: Energy bound in terms of symplectic area}.
\end{proof}

\subsection{Transversality}

In order to define the $A_\infty$-category $\mathcal{F}_c$ together with its relative weak Calabi-Yau pairing, we need to show that all of the Floer and perturbation data involved can be chosen in a generic fashion to be both consistent and regular. The proof of the next lemma is based on the arguments in \cite[\S 3.4]{BC14}, which in turn adapt the methods in \cite{Sei} to the cobordism context.
\begin{lem}\label{LEM: Transversality}
We consider collections of data consisting of the following:
\begin{enumerate}
\item Floer data $\mathscr{D}^+_{W,W'}$ and $\mathscr{D}^-_{W,W'}$ for all pairs of cobordisms $W,W'\in \mathcal{CL}_d(\widetilde{M})$.
\item Perturbation data $\mathscr{D}^\mu_{W_0,\ldots,W_k}$ for all families $W_0,\ldots, W_k\in \mathcal{CL}_d(\widetilde{M})$, $k\ge 2$. 
\item Perturbation data $\mathscr{D}^\mathbf{Y}_{W_0,\ldots,W_m;X_p,\ldots, X_0}$ for $m,p$ with $m+p\ge 1$ and perturbation data $\mathscr{D}^\delta_{W_0,\ldots,W_m;X_p,\ldots, X_0}$ for all $m,p\ge 0$ and all families $W_0,\ldots,W_m;X_p,\ldots, X_0\in \mathcal{CL}_d(\widetilde{M})$.
\item Perturbation data $\mathscr{D}^\sigma_{W_0,\ldots,W_m}$ for every $m\ge 0$ and every family $W_0,\ldots,W_m\in \mathcal{CL}_d(\widetilde{M})$. 
\end{enumerate}
Such a collection can be chosen generically to be both consistent and regular.
\end{lem}

\begin{proof}

Assume we have fixed the following data: universal choices of strip-like ends $\{\epsilon_i^{k+1}\}$ on the bundles $\mathcal{S}^{k+1}\to \mathcal{R}^{k+1}$, $k\ge 1$, for the input sets $\mathcal{I}_{\mu}^{k+1}$ and universal choices of strip-like ends $\{\epsilon^{m+1;1}_j\}$ on $\mathcal{S}^{m+1;1}\to\mathcal{R}^{m+1;1}$, $m\ge 0$, for the input sets $\mathcal{I}_{CY}^{m+1}$; global transition functions $\bm{a}_{\mu}:\mathcal{S}^{k+1}\to [0,1]$, $\bm{a}^{m,p}_{\mathbf{Y}}:\mathcal{S}^{m+p+2}\to [0,1]$, and $\bm{a}_{\sigma}:\mathcal{S}^{m+1;1}\to [0,1]$ as described in Section \ref{SUBSECTION: Transition functions}; and a profile function $h:\C\to\R$. We take the universal choices of strip-like ends and global transition functions on the bundles $\mathcal{S}^{m,p;1}\to \mathcal{R}^{m,p;1}$, $m,p\ge 0$, to be given by restriction from $\mathcal{S}^{m+p+2;1}$.

As in the argument in \cite{Sei}, we will choose the perturbation data by induction on the number of punctures. We emphasize that in order to achieve the necessary consistency among all of the data, i.e.\ the data used to define the maps $\mu^{\mathcal{F}_c}_k$ as well as those used to define the $\mathcal{F}_c\mbox{--}\mathcal{F}_c$ bimodule $(\mathcal{F}_c)_\Delta^\mathit{rel}$, the dual Hochschild cycle $\sigma^{\mathcal{F}_{\mathit{c}}}\in CC_\bullet(\mathcal{F}_c,(\mathcal{F}_c)_\Delta^\mathit{rel})^\vee$, and the natural transformation $\delta^{\mathcal{F}_c}:\mathbf{Y}_{\mathcal{F}_{\mathit{c}}}^l\to (\mathbf{Y}_{\mathit{rel}}^\vee)^l$, it is necessary to perform the induction for the different types of data in parallel. More precisely, the induction step is the following:

\begin{em}
Assume we have chosen the following consistent regular data: 
\begin{enumerate}
\item\label{Induction hypothesis: Data 1} $\mathscr{D}^+_{W,W'}$ and $\mathscr{D}^-_{W,W'}$ for all pairs of cobordisms $W,W'\in\mathcal{CL}_d(\widetilde{M})$.
\item $\mathscr{D}^\mu_{W_0,\ldots, W_k}$ for all $k\in \{2,\ldots,N\}$ and all families $W_0,\ldots, W_k\in \mathcal{CL}_d(\widetilde{M})$.
\item \sloppy$\mathscr{D}^\mathbf{Y}_{W_0,\ldots,W_m;X_p,\ldots, X_0}$ for all $m,p$ with $0\le m+p\le N-1$ and all families $W_0,\ldots,W_m,X_p,\ldots, X_0\in \mathcal{CL}_d(\widetilde{M})$,
\item\label{Induction hypothesis: Data 3} $\mathscr{D}^\delta_{W_0,\ldots,W_m;X_p,\ldots, X_0}$ for all $m,p\ge 0$ with $m+p\le N-3$ and all families $W_0,\ldots,W_m,X_p,\ldots, X_0\in \mathcal{CL}_d(\widetilde{M})$,
\item\label{Induction hypothesis: Data 4} $\mathscr{D}^\sigma_{W_0,\ldots, W_m}$ for all $m\in\{0,\ldots,N-2\}$ and all families $W_0,\ldots, W_m\in \mathcal{CL}_d(\widetilde{M})$.
\end{enumerate}
Then it is possible to extend this choice to consistent regular data: 
\begin{enumerate}
\item[($\mathit{1^\prime}$)]\label{Induction step: Data 1} $\mathscr{D}^\mu_{W_0,\ldots, W_{N+1}}$, 
\item[($\mathit{2^\prime}$)] $\mathscr{D}^\mathbf{Y}_{W_0,\ldots,W_m;X_p,\ldots, X_0}$ for all $m,p$ with $m+p=N$,
\item[($\mathit{3^\prime}$)]\label{Induction step: Data 3}$\mathscr{D}^\delta_{W_0,\ldots,W_m;X_p,\ldots, X_0}$ for all $m,p$ with $m+p= N-2$,
\item[($\mathit{4^\prime}$)]\label{Induction step: Data 4}$\mathscr{D}^\sigma_{W_0,\ldots, W_{N-1}}$.
\end{enumerate}
\end{em}

We begin by making regular choices of Floer data $\mathscr{D}^+_{W,W'}$ and $\mathscr{D}^-_{W,W'}$ for all pairs of cobordisms $W,W'\in\mathcal{CL}_d(\widetilde{M})$ in a generic way. This is done in two steps: We first choose generic data such that solutions $u$ of the Floer equation, i.e.\ solutions of \eqref{EQ: Perturbed CR equation} for the perturbation datum $(dt\otimes H_{W,W'}^\pm,J^\pm_{W,W'}(t))$, are regular if they have image contained in the fibre over a bottleneck; then we introduce a generic perturbation in this data so that all solutions of the Floer equation are regular. To accomplish the first step, consider a solution $u$ of the Floer equation for a Floer datum $(H^{\pm}_{W,W'},J^{\pm}_{W,W'})$ satisfying conditions \eqref{DEF: Floer datum cond 1} -- \eqref{DEF: Floer datum cond 3} on Floer data in Section \ref{SUBSECTION: Floer data and perturbation data} and assume $\pi\circ u$ is constant at a bottleneck point $q\in\C$. The linearized operator associated to \eqref{EQ: Perturbed CR equation} at $u$ splits into a horizontal and a vertical part. The vertical part is the linearized operator for the Floer equation for $u$ in $\pi^{-1}(q)$ for the data $(G^\pm_{W,W'},J^\pm|_{\pi^{-1}(q)})$ (where $G^\pm_{W,W'}$ is the Hamiltonian on $M$ satisfying condition \eqref{DEF: Floer datum cond 2} in the definition of Floer data for cobordisms). Hence if we choose the data $\mathscr{D}^+_{W,W'}$ and $\mathscr{D}^-_{W,W'}$ such that the ``vertical'' data $(G^\pm_{W,W'},J^\pm|_{\pi^{-1}(q)})$ are regular data on $\pi^{-1}(q)$ for all bottlenecks $q\in\C$, the vertical part of the linearized operator will be surjective. To demonstrate surjectivity for the horizontal part of the linearized operator, we make use of the specific form of the profile functions around a bottleneck. Part \ref{COR: Cor of the index formula, I_F, 2 punctures} of Corollary \ref{COR: Cor of the index formula} applied with $f_1=f_2=h_\pm$ in the case of data $\mathscr{D}^+_{W,W'}$, and with $f_1=f_2=-h_\pm$ in the case of data $\mathscr{D}^-_{W,W'}$, implies that the horizontal part of the linearized operator is surjective. This proves regularity for solutions with constant projection when the data $\mathscr{D}^+_{W,W'}$ and $\mathscr{D}^-_{W,W'}$ are chosen so that the corresponding vertical data $(G^\pm_{W,W'},J^\pm|_{\pi^{-1}(q)})$ are generic for all bottlenecks $q\in\C$.

We now consider solutions $u$ of the Floer equation with non-constant projection, and we simplify the argument by using the naturality transformation \eqref{EQ: naturality transformation} to transform these to curves $v$. Note that $u$ projects to a constant at a bottleneck if and only if $v$ does. Since by assumption $u$ does not project to a constant at a bottleneck, it is not possible for the image of $v':=\pi\circ v$ to remain in a small neighbourhood of a bottleneck. This would mean that the bottleneck serves as both entry and exit for $v'$. However, by orientation considerations and an application of the open mapping theorem, a bottleneck cannot be an exit in the case of data $\mathscr{D}^+_{W,W'}$, and cannot be an entry in the case of data $\mathscr{D}^-_{W,W'}$, unless $v'$ is constant. It then follows that for small enough $\delta$ (which does not depend on the original curve $u$) the image of $u$ must enter the region $(-\tfrac{5}{4}+\delta,\tfrac{9}{4}-\delta)\times \R\times M$. Enlarging $K_{W,W'}$ within $(-\tfrac{5}{4},\tfrac{9}{4})\times \R$ if necessary, we can assume $u$ enters $\mathrm{Int}(K_{W,W'})\times M$. Standard regularity arguments (as in \cite[Lemma 3.4.3]{McDS12} and \cite{Sei}) then imply that transversality for the pair $(H^{\pm}_{W,W'},J^{\pm}_{W,W'})$ can be attained by a generic perturbation supported on $K_{W,W'}\times M$.

To prove the induction step, we first fix choices of perturbation data ($1^\prime$)--($4^\prime$) which are consistent with the choices (\ref{Induction hypothesis: Data 1})--(\ref{Induction hypothesis: Data 4}). This is done by a simple adaptation of the procedure explained in \cite{Sei} for choosing perturbation data consistently, and relies on the fact that the perturbation data induced on a glued configuration of discs will satisfy conditions \eqref{DEF: Pert datum 1} -- \eqref{Perturbation data conditions, PART projection holomorphic} on perturbation data in Section \ref{SUBSECTION: Floer data and perturbation data} if the data on the individual discs satisfy these conditions. We then modify this data in two steps using two different types of perturbation. Both types of perturbation are subclasses of the class of perturbations introduced in \cite[\S II.9k]{Sei}, and so in particular are supported on subsets $\Omega\subset\mathcal{S}^{N+2}$, $\Omega'\subset\mathcal{S}^{N;1}$ and $\Omega''\subset \mathcal{S}^{m,p;1}$ which are disjoint from the strip-like ends and, for small enough gluing parameters, are contained in the thick part of the thick-thin decomposition associated to discs obtained by gluing. It will also be convenient to arrange for the data on $\mathcal{S}^{m,p;1}$ to be given by restriction of the data on $\mathcal{S}^{N;1}$. This can be achieved by first choosing regular data for $\mathcal{S}^{m,p;1}$, then extending this to consistent data on $\mathcal{S}^{N;1}$, and finally perturbing the extended data away from the subbundle $\mathcal{S}^{m,p;1}\subset \mathcal{S}^{N;1}$. 

We now describe the two types of perturbation we use. We first perturb the initial data by vertical perturbations supported in a small neighbourhood of the union of the fibres over the bottlenecks. If a generic perturbation of this sort is applied to a datum $(\mathbf{\Theta},\mathbf{J})$, the perturbed datum $(\widetilde{\mathbf{\Theta}},\widetilde{\mathbf{J}})$ will satisfy the condition that each solution $(r,u)$ of \eqref{EQ: Perturbed CR equation} which has constant projection at a bottleneck $q$ will be regular as a curve in $\pi^{-1}(q)$ for the perturbed vertical datum $(\widetilde{\mathbf{\Theta}}_0|_{\pi^{-1}(q)},\widetilde{\mathbf{J}}|_{\pi^{-1}(q)})$.   
Essentially the same argument that was applied to Floer strips shows that these curves are also regular as curves in $\widetilde{M}$, although in this case we must consider the extended linearized operator $\mathcal{D}_{\mathcal{S},r,u}$ associated to \eqref{EQ: Perturbed CR equation} (i.e.\ the linearized operator which takes into account variations of $r$ within $\mathcal{R}$). Recall that this operator is a map
\begin{equation}
\mathcal{D}_{\mathcal{S},r,u}:(T\mathcal{R})_r\times (T\mathcal{B}_{\mathcal{S}_r})_u\to (\mathcal{E}_{\mathcal{S}_r})_u,
\end{equation}
where $\mathcal{B}_{\mathcal{S}_r}$ is the Banach manifold of locally $W^{1,p}$-maps $\mathcal{S}_r\to \widetilde{M}$ converging on the strip-like ends to the chords giving the asymptotic conditions for \eqref{EQ: Perturbed CR equation}, and $\mathcal{E}_{\mathcal{S}_r}\to \mathcal{B}_{\mathcal{S}_r}$ is the Banach vector bundle with fibre over $u\in \mathcal{B}_{\mathcal{S}_r}$ given by $L^p(\mathcal{S}_r,\Omega^{0,1}(\mathcal{S}_r)\otimes u^*T\widetilde{M})$ (see \cite[\S II.9h]{Sei}). The second component of the operator $\mathcal{D}_{\mathcal{S},r,u}$ is simply the linearized operator associated to \eqref{EQ: Perturbed CR equation} for the fixed domain $\mathcal{S}_r$. When $u$ projects to a constant at a bottleneck, the operator $\mathcal{D}_{\mathcal{S},r,u}$ respects the splitting of $(T\mathcal{B}_{\mathcal{S}_r})_u$ and $(\mathcal{E}_{\mathcal{S}_r})_u$ coming from the product structure on $\C\times M$. By this we mean that $\mathcal{D}_{\mathcal{S},r,u}$ is of the form
\begin{equation}
\mathcal{D}_{\mathcal{S},r,u}(\nu,(\xi_1,\xi_2))=(\mathcal{D}_{\mathcal{S},r,w}(\nu,\xi_1),\mathcal{D}_{\mathcal{S},r,w'}(\nu,\xi_2)).
\end{equation}
Here $\mathcal{D}_{\mathcal{S},r,w}$ is the extended linearized operator for the inhomogeneous pseudoholomorphic curve equation in $\C$ at $(r,w)$ where $w=\pi\circ u$, and $\mathcal{D}_{\mathcal{S},r,w'}$ is the extended linearized operator for the inhomogeneous pseudoholomorphic curve equation in $M$ at $(r,w')$ where $w'=\pi_M\circ u$ and $\pi_M:\C\times M\to M$ is the projection onto the second factor. The vertical operator $\mathcal{D}_{\mathcal{S},r,w'}$ is surjective by regularity of the vertical data. Surjectivity of the total operator follows from the fact that the operator $\mathcal{D}_{\mathcal{S},r,w}(\nu,\cdot)$ is surjective for every $\nu\in (T\mathcal{R})_r$. To see this, we apply \eqref{COR: Cor of the index formula, I_F, more punctures} of Corollary \ref{COR: Cor of the index formula} for solutions $u:\mathcal{S}_r^{N+2}\to\widetilde{M}$ of \eqref{EQ: Perturbed CR equation} with data $\mathscr{D}^\mu_{W_0,\ldots,W_{N+1}}$ or $\mathscr{D}^\mathbf{Y}_{W_0,\ldots,W_m;X_p,\ldots, X_0}$ and part \eqref{COR: Cor of the index formula, I_rho} of Corollary \ref{COR: Cor of the index formula} for solutions $u:\mathcal{S}_r^{m,p;1}\to \widetilde{M}$ of \eqref{EQ: Perturbed CR equation} with data $\mathscr{D}^\delta_{W_0,\ldots, W_m;X_p,\ldots,X_0}$ and for solutions $u:\mathcal{S}_r^{N;1}\to \widetilde{M}$ with data $\mathscr{D}^\sigma_{W_0,\ldots, W_{N-1}}$. From this we obtain that the horizontal linearized operator for every fixed domain is surjective, and therefore $\mathcal{D}_{\mathcal{S},r,w}(\nu,\cdot)$ is surjective for all $\nu\in (T\mathcal{R})_r$.

For solutions $u$ which do not satisfy the condition that $\pi\circ u$ is constant at a bottleneck, we again apply the naturality transformation \eqref{EQ: naturality transformation}. The types of solutions under consideration all have either an exit $z_i$ with $\alpha_i=+$ or an entry $z_j$ with $\alpha_j=-$ and so by the open mapping theorem together with orientation considerations, the transformed curves $v$ cannot have projection that remains in a small neighbourhood of a bottleneck. Let $K$ be the compact subset appearing in conditions \eqref{DEF: Pert datum 2} and \eqref{Perturbation data conditions, PART projection holomorphic} on perturbation data in Section \ref{SUBSECTION: Floer data and perturbation data} for the datum under consideration (i.e.\ $\mathscr{D}^\mu_{W_0,\ldots,W_{N+1}}$, $\mathscr{D}^\mathbf{Y}_{W_0,\ldots,W_m;X_p,\ldots, X_0}$, $\mathscr{D}^\delta_{W_0,\ldots,W_m;X_p,\ldots, X_0}$, or $\mathscr{D}^\sigma_{W_0,\ldots, W_{N-1}}$). It follows from the form of the profile function and the conditions on the set $K$ that the curve $u$ must enter the region $\mathrm{Int}(K)\times M$. Hence we can achieve transversality for these curves by introducing a generic perturbation of the perturbation datum supported in $K\times M$. 
\end{proof}

We are now in a position to prove parts \eqref{THM: Main theorem, Hochschild part, part I}, \eqref{THM: Main theorem, Nat transf part, part I}, and \eqref{THM: Main theorem, equivalence part} of Theorem \ref{THM: Main theorem}.

\begin{proof}[Proof of Theorem \ref{THM: Main theorem}, \eqref{THM: Main theorem, Hochschild part, part I}, \eqref{THM: Main theorem, Nat transf part, part I}, and \eqref{THM: Main theorem, equivalence part}]
By Lemma \ref{LEM: Transversality}, for a generic consistent choice of Floer and perturbation data as in Section \ref{SUBSECTION: Floer data and perturbation data}, the moduli spaces defined in Section \ref{SUBSUBSECTION: The curve configurations} are smooth manifolds of dimension equal to the Fredholm index of the extended linearized operator associated to \eqref{EQ: Perturbed CR equation}. By Lemmas \ref{LEM: Images of projections of curves are contained} and \ref{LEM: Energy bounds}, these moduli spaces also satisfy Gromov compactness, i.e.\ they can be compactified by nodal curves. The monotonicity assumption on the cobordisms in $\mathcal{CL}_d(\widetilde{M})$ precludes disc and sphere bubbling for the moduli spaces of dimension zero and one. Index considerations then imply that the zero-dimensional moduli spaces are compact, and that the one-dimensional moduli spaces can be compactified by once-broken configurations of polygons, and the usual gluing argument shows that all once-broken polygon configurations occur in the boundary of the compactified one-dimensional moduli spaces. By compactness of the index-zero moduli spaces, the curve counts used in defining $\mathbf{Y}^r_{\mathit{rel}}$, $\delta^{\mathcal{F}_{\mathit{c}}}$, and $\sigma^{\mathcal{F}_{\mathit{c}}}$ make sense. The terms appearing in the $A_\infty$-functor relations $\eqref{A_infty functor relation}$ for $\mathbf{Y}^r_{\mathit{rel}}$ correspond to the configurations in the boundary of the compactifications of the moduli spaces $\mathcal{R}_{\mathbf{Y}}^{m+p+2}(\gamma_1,\ldots,\gamma_m,\bm{\zeta};\eta_p,\ldots,\eta_1,\bm{\zeta}')^1$, and hence $\mathbf{Y}^r_{\mathit{rel}}$ is a functor. The corresponding $\mathcal{F}_c\mathit{\mbox{--}}\mathcal{F}_c$ bimodule $(\mathcal{F}_c)_\Delta^\mathit{rel}$ is therefore also well-defined. Similarly, the terms appearing in the relation $\mu^{\mathit{fun}(\mathcal{F}_c,\mathcal{F}_c\mathit{\mbox{--}mod})}_1(\delta^{\mathcal{F}_{\mathit{c}}})=0$ correspond to the configurations in the boundary of the compactifications of the moduli spaces $\mathcal{R}_{\delta}^{m,p;1}(\gamma_1,\ldots,\gamma_m,\bm{\xi};\eta_p,\ldots,\eta_1,\bm{\xi}')^1$, and so $\delta^{\mathcal{F}_{\mathit{c}}}:\mathbf{Y}^l_{\mathcal{F}_{\mathit{c}}}\to \mathbf{Y}^\vee_{\mathit{rel}}$ is a natural transformation. By examining the configurations in the boundary of the compactifications of the spaces $\mathcal{R}_{\sigma}^{m+1;1}(\gamma_1,\ldots,\gamma_m,\bm{\gamma})^1$, we also see that $\sigma^{\mathcal{F}_{\mathit{c}}}\in CC_\bullet(\mathcal{F}_c,(\mathcal{F}_c)_\Delta^\mathit{rel})^\vee$ is a cycle.

To see that the relationship $[\phi^{\mathcal{F}_c}]=\Gamma\circ T^\vee([\sigma^{\mathcal{F}_{\mathit{c}}}])$ holds, we apply an argument due to Ganatra \cite[Proposition 5.6]{GanatraThesis}. This involves considering curves defined on a subbundle $\mathcal{T}^{m,p;1}$ of $\mathcal{S}^{m+p+2;1}$ which is given by restricting $\mathcal{S}^{m+p+2;1}$ to a particular subset $\widetilde{\mathcal{R}}^{m,p;1}$ of $\mathcal{R}^{m+p+2;1}$. The subset $\widetilde{\mathcal{R}}^{m,p;1}$ is defined to be the set of elements $[z_1,\ldots, z_m,z',w_p,\ldots,w_1,w',y]$ of $\mathcal{R}^{m+p+2;1}$ where $z'=e^{i\pi t}$ for some $t\in(0,1)$, $w'=+1$, and $y=0$. We assign to this bundle the input set $\mathcal{I}^{m+p+2}_{CY}$ and the $(m+p+2)$ sign datum $\bar{\alpha}_{CY}$. We take the universal strip-like ends, global transition function, and perturbation data on $\mathcal{T}^{m,p;1}$ to be given by restriction of the corresponding data on $\mathcal{S}^{m+p+2;1}$. Denote by $\widetilde{\mathcal{R}}^{m,p;1}(\gamma_1,\ldots,\gamma_m,\bm{\xi};\eta_p,\ldots,\eta_1,\bm{\xi}')$ the moduli space of pairs $(r,u)$ where $r\in \widetilde{\mathcal{R}}^{m,p;1}$ and $u:\mathcal{T}_r^{m,p;1}\to \widetilde{M}$ is a solution of \eqref{EQ: Perturbed CR equation} satisfying boundary conditions along $W_0,\ldots,W_m$,$X_p,\ldots, X_0$ and asymptotic conditions along the following Hamiltonian chords
\begin{equation}
\begin{aligned}
&\gamma_i\in \mathcal{O}(H^+_{W_{i-1},W_i}),\;i=1,\ldots,m,\\ 
&\bm{\xi}\in \mathcal{O}(H^+_{W_m,X_p}),\\ 
&\eta_j\in \mathcal{O}(H^+_{X_j,X_{j-1}}),\;j=p,\ldots,1,\\
& \bm{\xi}'\in \mathcal{O}(H^-_{X_0,W_0}).
\end{aligned}
\end{equation}
The moduli space $\widetilde{\mathcal{R}}^{m,p;1}(\gamma_1,\ldots,\gamma_m,\bm{\xi};\eta_p,\ldots,\eta_1,\bm{\xi}')$ is regular and satisfies Gromov compactness. Define an element $\beta\in CC_\bullet(\mathcal{F}_c,(\mathcal{F}_c)_\Delta^\mathit{rel})^\vee$ by
\begin{equation}
\begin{aligned}
&\langle \beta, \gamma_1\otimes\cdots\otimes\gamma_m\otimes\bm{\xi}\otimes\eta_p\otimes\cdots\otimes\eta_1\otimes\bm{\xi}' \rangle \\
&\qquad\qquad= \#_{\Z_2}\widetilde{\mathcal{R}}^{m,p;1}(\gamma_1,\ldots,\gamma_m,\bm{\xi};\eta_p,\ldots,\eta_1,\bm{\xi}')^0.
\end{aligned}
\end{equation} 
Set $\bar{\delta}^{\mathcal{F}_c}=\Gamma^{-1}(\phi^{\mathcal{F}_c})$. We claim that the cycle $\bar{\delta}^{\mathcal{F}_c}-T^\vee (\sigma^{\mathcal{F}_c})$ in ${}_2CC_\bullet (\mathcal{F}_c,(\mathcal{F}_c)_\Delta^\mathit{rel})^\vee$ is equal to the boundary of $\beta$. Indeed, the terms in the relation 
\begin{equation}
\langle\partial(\beta)-\bar{\delta}^{\mathcal{F}_c}+T^\vee (\sigma^{\mathcal{F}_c}),\gamma_1\otimes\cdots\otimes\gamma_m\otimes\bm{\xi}\otimes\eta_p\otimes\cdots\otimes\eta_1\otimes\bm{\xi}'\rangle=0
\end{equation}
correspond to configurations in the boundary of the compactification of $\widetilde{\mathcal{R}}^{m,p;1}(\gamma_1,\ldots,\bm{\xi};\allowbreak\ldots,\eta_1,\bm{\xi}')^1$ (including those configurations obtained as the limit when the parameter $t$ approaches $0$ and $1$ -- see Figure \ref{FIG: circle_10}). This proves that $[\phi^{\mathcal{F}_c}]=\Gamma\circ T^\vee([\sigma^{\mathcal{F}_{\mathit{c}}}])$.

\begin{figure}
\centering
\def\svgwidth{\linewidth}
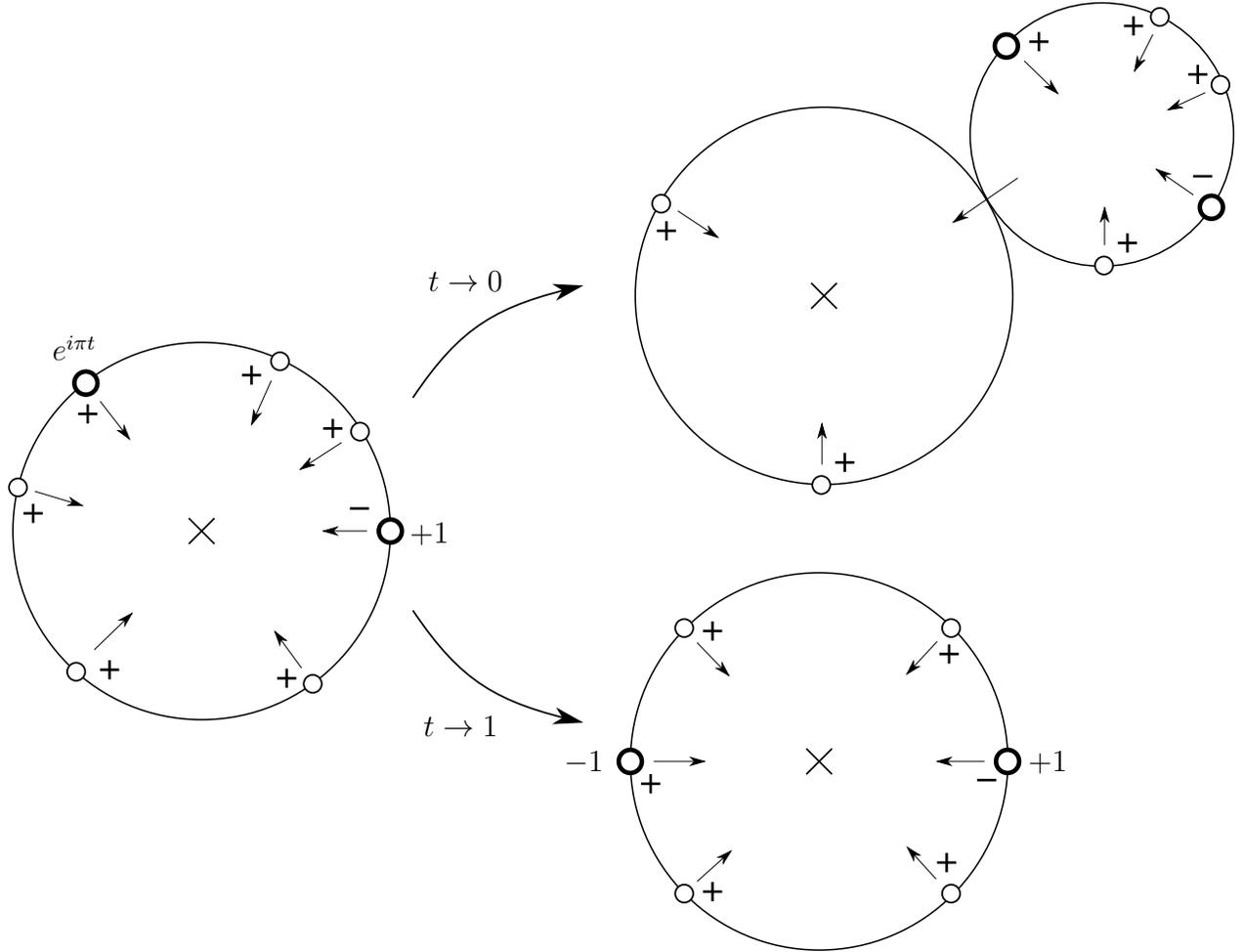
\caption[A fibre of $\mathcal{T}^{3,2;1}$ with limit configurations]{\label{FIG: circle_10} A punctured disc which is a fibre of $\mathcal{T}^{3,2;1}$ with $t\to 0$ and $t\to 1$ limit configurations.}
\end{figure}

It remains to prove that $\delta^{\mathcal{F}_{\mathit{c}}}$ is a quasi-isomorphism and that $\sigma^{\mathcal{F}_{\mathit{c}}}$ is homologically non-degenerate. Note that by the relation $[\phi^{\mathcal{F}_c}]=\Gamma\circ T^\vee([\sigma^{\mathcal{F}_{\mathit{c}}}])$ and Lemma \ref{LEM: Homological nondegeneracy condition} these statements are equivalent. We will show that $\delta^{\mathcal{F}_{\mathit{c}}}$ is a quasi-isomorphism. By the definition of a natural transformation between $A_\infty$-functors, the component $\delta^{\mathcal{F}_{\mathit{c}}}_0$ of $\delta^{\mathcal{F}_{\mathit{c}}}=(\delta^{\mathcal{F}_{\mathit{c}}}_0,\delta^{\mathcal{F}_{\mathit{c}}}_1,\ldots)$ consists of a family of module morphisms
\begin{equation}
(\delta^{\mathcal{F}_{\mathit{c}}}_0)_W:\mathbf{Y}^l_{\mathcal{F}_{\mathit{c}}}(W)\to \mathbf{Y}^\vee_{\mathit{rel}}(W),\; W\in \mathcal{CL}_d(\widetilde{M}).
\end{equation}
To show that $\delta^{\mathcal{F}_{\mathit{c}}}$ is a quasi-isomorphism, we claim that it is sufficient to show that the chain maps
\begin{equation}\label{EQ: delta chain maps}
((\delta^{\mathcal{F}_{\mathit{c}}}_0)_W)_{0|1}:\mathbf{Y}^l_{\mathcal{F}_{\mathit{c}}}(W)(X)\to \mathbf{Y}^\vee_{\mathit{rel}}(W)(X)
\end{equation}
are quasi-isomorphisms for all $X,W\in \mathcal{CL}_d(\widetilde{M})$. This is a result of a general algebraic fact that for two $A_\infty$-functors $\mathbf{G}_0,\mathbf{G}_1:\mathcal{A}\to \mathcal{B}$, where $\mathcal{B}$ is homologically unital, a natural transformation $T$ from $\mathbf{G}_0$ to $\mathbf{G}_1$ is a quasi-isomorphism if $(T_0)_X\in \mathcal{B}(\mathbf{G}_0(X),\mathbf{G}_1(X))$ is a quasi-isomorphism for all $X\in \mathrm{Ob}(\mathcal{A})$. Applying this algebraic result first to the natural transformation $\delta^{\mathcal{F}_{\mathit{c}}}$ and then to the module morphism $\delta^{\mathcal{F}_{\mathit{c}}}_0$, viewed as a natural transformation itself, shows that it is sufficient to check that the chain maps \eqref{EQ: delta chain maps} are quasi-isomorphisms. The algebraic result in turn is a consequence of the existence of an obvious functor
\begin{equation}
H(\mathit{fun}(\mathcal{A},\mathcal{B}))\to \mathit{fun}(H(\mathcal{A}),H(\mathcal{B})),
\end{equation}
which is unital if $\mathcal{B}$ is homologically unital. Finally, the fact that the map 
\begin{equation}
((\delta^{\mathcal{F}_{\mathit{c}}}_0)_W)_{0|1}:CF(X,W;\mathscr{D}^+_{X,W})\to CF(W,X;\mathscr{D}^-_{W,X})^\vee 
\end{equation}
is a quasi-isomorphism is clear from its explicit description, since it is the total Poincar\'{e} duality quasi-isomorphism for Floer complexes of Lagrangian cobordisms.
\end{proof}

\subsection{Index and regularity results for curves in $\R^2$}

In this section we state some results about the index and regularity of polygons in $\R^2$ satisfying an inhomogeneous nonlinear Cauchy-Riemann equation and having boundary along an embedded path with horizontal ends. We consider two versions of this equation, the second more general than the first. For the first version, we will concern ourselves only with constant solutions, and for the second more general version of the equation, we treat only non-constant solutions. Corollary \ref{COR: Cor of the index formula} concerning constant solutions has already been applied in the previous section in the proof of regularity of the moduli spaces of curves involved in the definition of $\mathcal{F}_c$ and its weak Calabi-Yau pairing. Both this corollary and Lemma \ref{LEM: Index formula 2 fixed} concerning non-constant curves will be used in the following sections to define the inclusion functor $\mathcal{I}_\gamma$, as well as the bimodule morphism $i_{\mathit{rel}}$ and the natural transformation $S^{\mathit{rel}}$ appearing in Theorem \ref{THM: Main theorem}, and to prove their properties. The results of this section are generalizations of those appearing in Section 4.3 of \cite{BC14}. As the proofs involve only minor alterations to the proofs there, we do not include them here.

Let $\gamma\subset\R^2$ be a properly embedded curve which is diffeomorphic to $\R$ and has horizontal ends. Fix a punctured disc $S$ which is a fibre of one of the bundles $\mathcal{S}^{k+1}$ ($k\ge 2$), $\mathcal{S}^{k+1;1}$ ($k\ge 0$), or $\mathcal{S}^{m,p;1}$ ($m,p\ge 0$). In the last case, set $k=m+p+1$. Denote by $j$ the complex structure on $S$. Assume we have a fixed input set $\mathcal{I}^{k+1}$, strip-like ends $\{\epsilon^S_i\}_{i=1,\ldots,k+1}$, and a $(k+1)$ sign datum $\bar{\alpha}$, as well as a corresponding transition function $a:S\to \R$. Let $f_1,\ldots,f_{k+1}:\gamma\to\R$ be Morse functions with a common critical point $x\in\gamma$ and let $\{\hat{f}_z:\gamma\to\R\}_{z\in S}$ be a family of functions parametrized by $S$ satisfying the following properties:
\begin{enumerate}
\item For all $z\in S$, $d\hat{f}_z(x)=0$.
\item For all $i\in \mathcal{I}^{k+1}$, there exists a compact set $K^-\subset Z^-$ and for all $i\in \{1,\ldots,k+1\}\setminus\mathcal{I}^{k+1}$, there exists a compact set $K^+\subset Z^+$ such that $\hat{f}_z=f_i$ for all $z\in \epsilon^S_i(Z^\pm\setminus K^{\pm})$. 
\end{enumerate}
We extend the functions $\hat{f}_z$ and $f_i$ to $\R^2$ using the identification $\R^2\cong T^*\gamma$ and defining the extension to be constant along the fibres. Let $X^{\hat{f}}=\{X^{\hat{f}_z}\}_{z\in S}$ be the associated family of Hamiltonian vector fields on $\R^2$ parametrized by $S$, and define a 1-form on $S$ with values in the space of Hamiltonian vector fields on $\R^2$ by $Z=da\otimes X^{\hat{f}}$. Lastly define a family of complex structures $\hat{i}=\{\hat{i}_z\}_{z\in S}$ on $\R^2$ parametrized by $S$ by $\hat{i}_z=(\phi^{\hat{f}_z}_{a(z)})_*i$. 

For the case without interior marked point, i.e.\ where $S$ is a fibre of $\mathcal{S}^{k+1}$, we also extend these choices to the case $k=1$ for the input set $\mathcal{I}^2=\mathcal{I}_\mu^2$ and sign datum $\bar{\alpha}=(\alpha_1,\alpha_2)$ satisfying $\alpha_1=\alpha_2$. Here we require that $f_0=f_1$ and that $\hat{f}_z$ be constant with respect to $z$. For the case where $\bar{\alpha}=(+,+)$, we take $a(s,t)=\bm{a}_{\mu}(s,t):=t$, and for $\bar{\alpha}=(-,-)$, we take $a(s,t)=\bm{a}^{0,0}_{\mathbf{Y}}(s,t):=1-t$.

We consider the following equation:
\begin{equation}\label{EQ: Perturbed CR eq plane}
w:(S,\partial S)\to (\R^2,\gamma),\quad Dw + \hat{i}_z(w)\circ Dw\circ j = Z+\hat{i}_z(w)\circ Z\circ j.
\end{equation}
From our definitions, we have that $Z_z(x)=0$ for all $z\in S$, and so the constant map $w_0(z)\equiv x$ is a solution of \eqref{EQ: Perturbed CR eq plane}. Let $\mathcal{D}_{w_0}$ denote the linearized operator associated to \eqref{EQ: Perturbed CR eq plane} for this solution. 

The following lemma is a generalization of Lemma 4.3.1 in \cite{BC14}. The proof, which we omit, is an application of the theory developed in \cite{SeiLef12}.

\begin{lem}\label{LEM: Index formula}
 The Fredholm index of $\mathcal{D}_{w_0}$ is given by
\begin{equation}\label{EQ: Index formula}
\mathrm{ind}(\mathcal{D}_{w_0})=\sum_{i\in \mathcal{I}^{k+1}}(|x|_{f_i}-1) - \sum_{i\not\in\mathcal{I}^{k+1}}|x|_{f_i}+1,
\end{equation}
where $|x|_{f_i}$ is the Morse index of $x$ as a critical point of $f_i$. Moreover, when the index \eqref{EQ: Index formula} is zero, the operator $\mathcal{D}_{w_0}$ is surjective and therefore the constant solution $w_0$ is regular.
\end{lem}

The following corollary is simply a statement of the previous lemma for the specific cases we will be interested in.
 
\begin{cor}\label{COR: Cor of the index formula}
 In the following cases, $\mathrm{ind}(\mathcal{D}_{w_0})=0$.
\begin{enumerate} 
 \item\label{COR: Cor of the index formula, I_F} The input set is $\mathcal{I}^{k+1}=\mathcal{I}^{k+1}_{\mu}$ and one of the following applies:
 	\begin{enumerate}
 	\item\label{COR: Cor of the index formula, I_F, 2 punctures} $k=1$ and $|x|_{f_1}=|x|_{f_2}$ 
 	\item\label{COR: Cor of the index formula, I_F, more punctures} $k\ge 2$, and there exists $i_0\in\{1,\ldots,k\}$ such that $|x|_{f_i}=1$ for $i\in \{1,\ldots, k\}\setminus \{i_0\}$ and $|x|_{f_{i_0}}=|x|_{f_{k+1}}$. 
 	\end{enumerate}
\item\label{COR: Cor of the index formula, I_rho} The input set is $\mathcal{I}^{k+1}=\mathcal{I}^{k+1}_{CY}$, and we have $|x|_{f_{k+1}}=0$, and $|x|_{f_i}=1$ for $i\in\{1,\ldots,k\}$. 
\end{enumerate}
Moreover in all of these case $\mathcal{D}_{w_0}$ is surjective and therefore the constant solution $w_0$ is regular.
\end{cor}

We now describe a more general version of Equation \ref{EQ: Perturbed CR eq plane}. For this we consider Morse functions $f_i:\gamma\to\R$, now possibly with different critical points. Assume $x_i\in\gamma$ is a critical point of $f_i$. We still require a family of functions $\{\hat{f}_z:\gamma\to\R\}_{z\in S}$ identically equal to $f_i$ near the $i$th puncture, but there is no longer a requirement on the critical points of the functions $\hat{f}_z$. We also allow the one-form $Z$ which appears in \eqref{EQ: Perturbed CR eq plane} to have the more general form 
\begin{equation}\label{EQ: Second definition of Z}
Z=da\otimes X^{\hat{f}}+X^q.
\end{equation}
Here $q$ is a form in $\Omega^1(S,C^\infty(\R^2))$ which is required to have compact support contained in the interior of $S$.  

\begin{remk}\label{REMK: The mistake 2}
In the course of completing this work, we discovered an error in Section 4.3.2 of  the published version of \cite{BC14}. This error, which had no impact on the rest of \cite{BC14}, has been corrected in the 
most recent arXiv version of the paper. Our references below to this section and, in particular, to 
Lemma 4.3.3 and Corollary 4.3.4 refer to the forms that appear in this revision.
\end{remk}

The next lemma is a generalization of Lemma 4.3.3 in \cite{BC14}. Again we omit the proof which requires only minor modifications to the proof which appears there. The transversality result in particular follows from the methods in \cite[\S II.13a]{Sei}. 

\begin{lem}\label{LEM: Index formula 2 fixed} Let $w:(\mathcal{S}_r,\partial \mathcal{S}_r)\to (\R^2,\gamma)$ be a non-constant solution of \eqref{EQ: Perturbed CR eq plane}. Then
the linearized operator $\mathcal{D}_{w}$ associated to \eqref{EQ: Perturbed CR eq plane} for the fixed surface $\mathcal{S}_r$ has index given by
\begin{equation}
\mathrm{ind}(\mathcal{D}_{w})=\sum_{i\in \mathcal{I}^{k+1}}(|x_i|_{f_i}-1) - \sum_{i\not\in\mathcal{I}^{k+1}}|x_i|_{f_i}+1.
\end{equation}
Therefore the index of the extended linearized operator $\mathcal{D}_{\mathcal{S},r,w}$ is given by
\begin{equation}
\mathrm{ind}(\mathcal{D}_{\mathcal{S},r,w})=\begin{cases}\sum_{i\in \mathcal{I}^{k+1}}(|x_i|_{f_i}-1) - \sum_{i\not\in\mathcal{I}^{k+1}}|x_i|_{f_i}+k-1,&\mathcal{S}=\mathcal{S}^{k+1},\;(k\ge 2),\\
\sum_{i\in \mathcal{I}^{k+1}}(|x_i|_{f_i}-1) - \sum_{i\not\in\mathcal{I}^{k+1}}|x_i|_{f_i}+k+1,&\mathcal{S}=\mathcal{S}^{k+1;1},\\
\sum_{i\in \mathcal{I}^{k+1}}(|x_i|_{f_i}-1) - \sum_{i\not\in\mathcal{I}^{k+1}}|x_i|_{f_i}+m+p+1, &\mathcal{S}=\mathcal{S}^{m,p;1}\; (k=m+p+1).
\end{cases}
\end{equation}
Moreover, in all of these cases $\mathcal{D}_{\mathcal{S},r,w}$ is surjective.
\end{lem}
The difference in the indices of the ordinary and extended linearized operators accounts for repositioning of the marked points, hence is given by the dimension of $\mathcal{R}$.

\subsection{The inclusion functor $\mathbf{I}_\gamma$}\label{SUBSECTION: The inclusion functor I_gamma}
The purpose of this section is to define the functor $\mathbf{I}_\gamma:\mathcal{F}\to\mathcal{F}_{\mathit{c}}$ appearing in the statement of Theorem \ref{THM: Main theorem}. This definition is based on the one in \cite[\S 4.2]{BC14}, although some modifications are required in this setting. Here $\gamma\subset\R^2$ is an embedded curve which is horizontal outside of $[0,1]\times \R$ and whose ends have $y$-coordinate in $\Z$. In other words, $\gamma$ is a Lagrangian cobordism in $\R^2\times\{\mathrm{pt}\}$. 

The definition of $\mathbf{I}_\gamma$ relies on restricting the class of data used to define the category $\mathcal{F}_{\mathit{c}}$. More precisely, we place restrictions on the Floer and perturbation data associated to families of cobordisms in $\mathcal{CL}_d(\widetilde{M})$ which are of the form $\gamma\times L$ for $L\in \mathcal{L}^*_d(M)$. For notational convenience we set $\widetilde{L}=\gamma\times L$. We first fix a profile function $h$ satisfying the conditions in Section \ref{SUBSECTION: Floer data and perturbation data}. We also choose a function $h':\R^2\to\R$ which is equal to $h$ outside of a small neighbourhood of $[-1,2]\times[-\lambda,\lambda]$ for some $\lambda>0$ and such that $(\phi_1^{h'})^{-1}(\gamma)$ is of the form shown in Figure \ref{FIG: diagram2}. In particular, we require that $\gamma$ and $(\phi^{h'}_1)^{-1}(\gamma)$ intersect transversely and that the number of intersection points is equal to one modulo four. Note that the second condition is different from the one imposed in \cite{BC14} where the number of intersection points is only required to be odd. This added assumption plays a role in proving Propositions \ref{PROP: sigmaF_c pulls back to sigmaF} and \ref{PROP: deltaFc pulls back to deltaF} relating $\sigma^{\mathcal{F}_c}$ to $\sigma^\mathcal{F}$ and $\delta^{\mathcal{F}_c}$ to $\delta^\mathcal{F}$. We label the intersection points $o_1,\ldots,o_l$ from right to left along $\gamma$ (see Figure \ref{FIG: diagram2}). In particular, $o_1$ and $o_l$ are bottlenecks. We also fix regular Floer data $\mathscr{D}_{L,L'}=(H_{L,L'},J_{L,L'})$ on $M$ for all pairs of Lagrangians $L,L'\in \mathcal{L}^*_d(M)$. The positive and negative profile Floer data $\mathscr{D}^{\pm}_{\widetilde{L},\widetilde{L}'}=(H^{\pm}_{\widetilde{L},\widetilde{L}'},J^{\pm}_{\widetilde{L},\widetilde{L}'})$ for the pair $\widetilde{L},\widetilde{L}'\in\mathcal{CL}_d(\widetilde{M})$ are taken to be
\begin{align}
&H^+_{\widetilde{L},\widetilde{L}'}:= h'\oplus H_{L,L'}\text{ and }J^+_{\widetilde{L},\widetilde{L}'}(t)=((\phi_t^{h'})_*i)\oplus J_{L,L'}(t),\label{EQ: Split positive profile Floer data}\\
&H^-_{\widetilde{L},\widetilde{L}'}:= -h'\oplus H_{L,L'}\text{ and }J^-_{\widetilde{L},\widetilde{L}'}(t)=((\phi_{1-t}^{h'})_*i)\oplus J_{L,L'}(t),\; t\in[0,1].\label{EQ: Split negative profile Floer data}
\end{align}
The data $\mathscr{D}^{\pm}_{\widetilde{L},\widetilde{L}'}$ are regular owing to the splitting of the linearized operator associated to \eqref{EQ: Perturbed CR equation} into horizontal and vertical parts, each of which can be seen to be surjective. Indeed, at a solution $u=(u_1,u_2)$ of \eqref{EQ: Perturbed CR equation}, where $u_1:\R\times [0,1]\to \C$ and $u_2:\R\times [0,1]\to M$, the vertical part of the linearized operator is the linearized operator at $u_2$ for the Floer equation with data $\mathscr{D}_{L,L'}$, and thus surjective by regularity of $\mathscr{D}_{L,L'}$. The horizontal part of the linearized operator is the linearized operator at $u_1$ for the Floer equation with data $(h',(\phi_t^{h'})_*i)$ in the positive profile case and with data $(-h',(\phi_{1-t}^{h'})_*i)$ in the negative profile case. To show surjectivity of this component of the linearized operator, we consider two cases. When $u_1$ is constant, part \eqref{COR: Cor of the index formula, I_F, 2 punctures} of Corollary \ref{COR: Cor of the index formula} implies surjectivity. For non-constant $u_1$, surjectivity follows from automatic regularity for holomorphic curves in $\C$ (see \cite[\S II.13a]{Sei} and \cite{dSRS}).

\begin{figure}
\centering
\def\svgwidth{130mm}
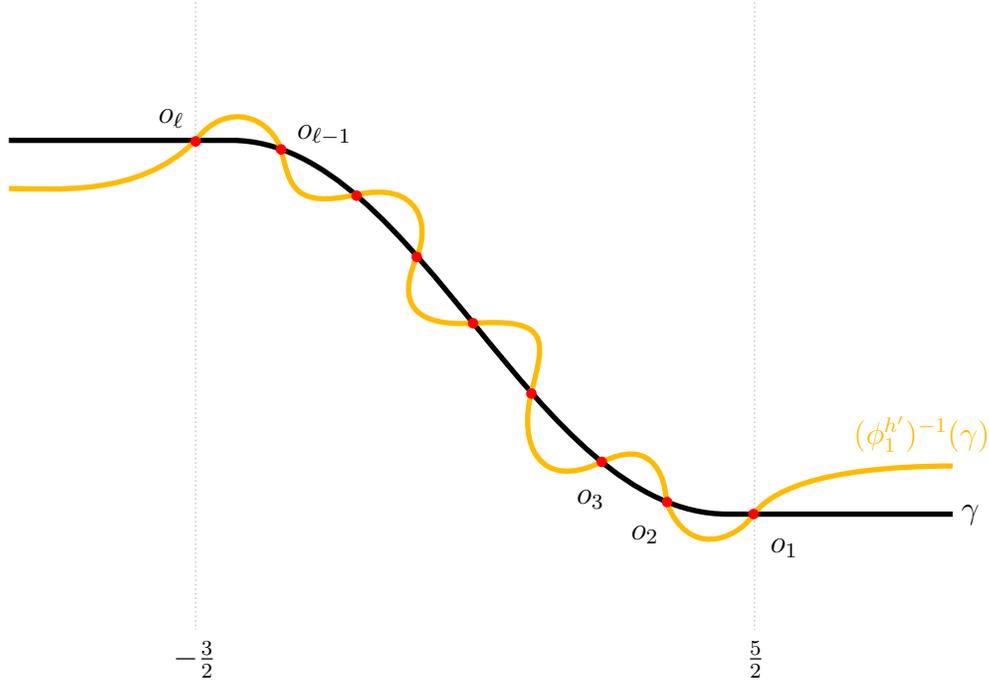
\caption{\label{FIG: diagram2}The path $\gamma$ and its image under the Hamiltonian diffeomorphism $(\phi_1^{h'})^{-1}$.}
\end{figure}

We extend the above choice of Floer data for pairs of cobordisms of the form $\gamma\times L$ to a choice of regular positive and negative profile Floer data for all pairs of cobordisms in $\mathcal{CL}_d(\widetilde{M})$ (for the profile function $h$). 

Next we make a choice of perturbation data for both $\mathcal{F}$ and $\mathcal{F}_c$ which includes data for defining the weak Calabi-Yau structure on $\mathcal{F}$ and the relative weak Calabi-Yau pairing on $\mathcal{F}_c$. Denote the data on $M$ by $\mathscr{D}_M^\mu$, $\mathscr{D}_M^\mathbf{Y}$, $\mathscr{D}_M^\delta$, and $\mathscr{D}_M^\sigma$. The data $\mathscr{D}^\mu$, $\mathscr{D}^\mathbf{Y}$, $\mathscr{D}^\delta$, and $\mathscr{D}^\sigma$ on $\widetilde{M}$ associated to families of cobordisms of the form $\gamma\times L$ for $L\in \mathcal{L}^*_d(M)$ must be of the form
\begin{align}
&\Theta_r^\mu=da^\mu_r\otimes h'+(\Theta^\mu_M)_r+Q_r^\mu,&&J_r^\mu=(\widetilde{i}^\mu_{h'})_r\oplus(J^\mu_M)_r,\;r\in\mathcal{R}^{k+1},\label{EQ: Split perturbation data 1}\\
&\Theta_r^\mathbf{Y}=da^\mathbf{Y}_r\otimes h'+ (\Theta^\mathbf{Y}_M)_r+Q_r^\mathbf{Y},&& J_r^\mathbf{Y}=(\widetilde{i}^\mathbf{Y}_{h'})_r\oplus(J^\mathbf{Y}_M)_r,\; r\in\mathcal{R}^{m+p+2},\label{EQ: Y perturbation data for B}\\
&\Theta_r^\sigma=da^\sigma_r\otimes h'+(\Theta^\sigma_M)_r+Q_r^\sigma,&& J_r^\sigma=(\widetilde{i}^\sigma_{h'})_r\oplus(J^\sigma_M)_r,\;r\in\mathcal{R}^{m+1;1},\\
&\Theta_r^\delta=da^\delta_r\otimes h'+(\Theta^\delta_M)_r+Q_r^\delta,&& J_r^\delta=(\widetilde{i}^\delta_{h'})_r\oplus(J^\delta_M)_r,\;r\in\mathcal{R}^{m,p;1}.\label{EQ: Split perturbation data 4}
\end{align}
Here $(\widetilde{i}^\mu_{h'})_r=\{i^\mu_{h'}(z)\}_{z\in \mathcal{S}^{k+1}_r}$ is the family of almost complex structures on $\R^2$ defined by $(i^\mu_{h'})_r(z)=(\phi^{h'}_{a^\mu_r(z)})_*i$. Similarly $(\widetilde{i}^\mathbf{Y}_{h'})_r$, $(\widetilde{i}^\delta_{h'})_r$, and $(\widetilde{i}^\sigma_{h'})_r$ are families of almost complex structures on $\R^2$ parametrized by $\mathcal{S}^{m+p+2}$, $\mathcal{S}^{m,p;1}$, and $\mathcal{S}^{m+1;1}$ respectively. They are defined analogously to $(\widetilde{i}^\mu_{h'})$. The term $Q^\mu_r$ is a one-form in $\Omega^1(\mathcal{S}^{k+1}_r,C^\infty(\widetilde{M}))$ which is required to be compactly supported on the interior of $\mathcal{S}^{k+1}_r$ and away from the strip-like ends. The form $Q^\mu_r$ is subject to the following additional constraints:
\begin{enumerate}
\item There is a one-form $q^\mu_r\in \Omega^1(\mathcal{S}^{k+1}_r,C^\infty(\C))$ such that for all $z\in \mathcal{S}^{k+1}_r$ and $v\in T_z\mathcal{S}^{k+1}_r$, the function $(Q^\mu_r)_z(v):\widetilde{M}\to\R$ satisfies $(Q^\mu_r)_z(v)=(q^\mu_{r})_z(v)\circ\pi$. 
\item For all $z\in \mathcal{S}^{k+1}_r$ and $v\in T_z\mathcal{S}^{k+1}_r$, the function $(q^\mu_{r})_z(v)\in C^\infty(\C)$ is supported on a union $U_h$ of small neighbourhoods of the points $o_i\in\C$ with $i$ even.
\item\label{DEF: conditions on Q part 3} For every even $i$, there exist $z$ and $v$ such that the function $(q^\mu_{r})_z(v)$ does not have a critical point at $o_i$.
\end{enumerate}
The terms $Q_r^\mathbf{Y}$, $Q_r^\delta$, and $Q_r^\sigma$ are likewise one-forms with values in $C^\infty(\widetilde{M})$ which are subject to constraints analogous to the above.
As usual, the data are chosen inductively to achieve consistency. In this situation, the induction for the data for $\mathcal{F}$ and $\mathcal{F}_c$ must take place in parallel.

Note that the conditions on the Floer data $\mathscr{D}^+_{\widetilde{L},\widetilde{L}'}$ and the perturbation data $\mathscr{D}^\mu$ are those which appeared in \cite{BC14}. The additional conditions involving the data $\mathscr{D}^-_{\widetilde{L},\widetilde{L}'}$, $\mathscr{D}^\mathbf{Y}$, $\mathscr{D}^\sigma$ and $\mathscr{D}^\delta$ are needed to define the bimodule morphism $i_{\mathit{rel}}$ and the natural transformation $S^{\mathit{rel}}$ in the statement of Theorem \ref{THM: Main theorem} and to prove their properties.

\begin{remk}\label{REMK: Curves can't project to even points}
Suppose that $u$ is a solution of \eqref{EQ: Perturbed CR equation} for a perturbation datum $(\mathbf{\Theta},\mathbf{J})$ as in \eqref{EQ: Split perturbation data 1}--\eqref{EQ: Split perturbation data 4}. In particular $u$ satisfies boundary conditions along cobordisms of the form $\widetilde{L}$ for $L\in\mathcal{L}^*_d(M)$. Then condition \eqref{DEF: conditions on Q part 3} above on $(q_{r})_z(v)$ implies that $u$ cannot have constant projection at an intersection point $o_i\in\gamma\cap (\phi^{h'}_1)^{-1}(\gamma)$ with $i$ even. This is in contrast to the case of $o_i$ with $i$ odd, where if $u$ is an inhomogeneous pseudoholomorphic curve in $M$ for the vertical datum $((\mathbf{\Theta}_M)_{L_0,\ldots,L_k},(\mathbf{J}_M)_{L_0,\ldots,L_k})$, then the curve $(o_i,u)$ in $\widetilde{M}$ is a solution of \eqref{EQ: Perturbed CR equation} for the datum $(\mathbf{\Theta}_{\widetilde{L}_0,\ldots,\widetilde{L}_k},\mathbf{J}_{\widetilde{L}_0,\ldots,\widetilde{L}_k})$.
\end{remk}

By the following lemma, this class of perturbation data is sufficient to achieve regularity for all of the moduli spaces in Section \ref{SUBSUBSECTION: The curve configurations}. The corresponding fact concerning only the moduli spaces associated to the data of type $(\mathbf{\Theta}^\mu,\mathbf{J}^\mu)$ appears in Corollary 4.3.4 of \cite{BC14}.

\begin{lem}\label{LEM: regularity of data for the category B}
For a generic choice of perturbation data on $M$, all of the moduli spaces in Section \ref{SUBSUBSECTION: The curve configurations} of curves $(r,u)$ in $\widetilde{M}$ associated to the perturbation data \eqref{EQ: Split perturbation data 1}--\eqref{EQ: Split perturbation data 4} are regular.
\end{lem}

\begin{proof}
We will prove this for a moduli space of type $\mathcal{R}_{\sigma}^{m+1;1}(\xi_1,\ldots,\xi_m,\bm{\xi})$ of curves in $\widetilde{M}$ satisfying boundary conditions along $\widetilde{N_0},\ldots,\widetilde{N}_m$. The other cases are similar. First notice that, due to the split form of the Floer data on $\widetilde{M}$, the orbits $\xi_1,\ldots,\xi_m$ are of the form $\xi_i=(o_{j_i},\xi'_i)$ for some $j_i\in\{1,\ldots,l\}$ and $\xi'_i\in \mathcal{O}(H_{N_{i-1},N_i})$, and the orbit $\bm{\xi}$ is of the form $(o_s,\bm{\xi}')$, for some $s\in \{1,\ldots,l\}$ and $\bm{\xi}'\in \mathcal{O}(H_{N_m,N_0})$. Moreover, for every $(r,u)\in \mathcal{R}_{\sigma}^{m+1;1}(\xi_1,\ldots,\xi_m,\bm{\xi})$ the curve $w:=\pi\circ u$ in $\C$ satisfies the inhomogeneous pseudoholomorphic curve equation in $\C$ for the horizontal perturbation datum $(\mathbf{\Theta}^\sigma_\C,\mathbf{J}^\sigma_\C)=(d\bm{a}^\sigma\otimes h'+q^\sigma,\widetilde{i}^\sigma_{h'})$, and the curve $w':=\pi_M\circ u$ satisfies the inhomogeneous pseudoholomorphic curve equation in $M$ for the vertical perturbation datum $(\mathbf{\Theta}^\sigma_M,\mathbf{J}^\sigma_M)$. Denote by $\mathcal{R}^\sigma_\C$ the moduli space of curves in $\C$ associated to the perturbation datum $(\mathbf{\Theta}^\sigma_\C,\mathbf{J}^\sigma_\C)$ satisfying asymptotic conditions along the constant trajectories at $o_{j_1},\ldots,o_{j_m},o_s$ of the flow of $h'$. These curves have boundary conditions along $\gamma$. Similarly denote by $\mathcal{R}^\sigma_M$ the moduli space of curves in $M$ associated to the vertical perturbation datum $(\mathbf{\Theta}^\sigma_M,\mathbf{J}^\sigma_M)$ satisfying asymptotic conditions along the Hamiltonian trajectories $\xi'_1,\ldots,\xi'_m,\bm{\xi}'$. The boundary conditions for these curves are along the Lagrangians $N_0,\ldots,N_m\subset M$. 

The moduli spaces $\mathcal{R}^\sigma_\C$ and $\mathcal{R}^\sigma_M$ are equipped with smooth projections,
\begin{equation}
p_1:\mathcal{R}^\sigma_\C\to \mathcal{R}^{m+1;1},\quad p_2:\mathcal{R}^\sigma_M\to \mathcal{R}^{m+1;1}.
\end{equation}
The moduli space $\mathcal{R}_{\sigma}^{m+1;1}(\xi_1,\ldots,\xi_m,\bm{\xi})$ of curves in $\widetilde{M}$ can be described as the fibre product of these projections. By Lemma \ref{LEM: Index formula 2 fixed}, the moduli spaces $\mathcal{R}^\sigma_\C$ (of all dimensions) are regular. The moduli spaces $\mathcal{R}^\sigma_M$ (of all dimensions) are regular by genericity of the vertical data. By standard arguments, the projection $p_2$ is transverse to $p_1$ for a generic choice of data $(\mathbf{\Theta}^\sigma_M,\mathbf{J}^\sigma_M)$. Therefore the moduli space $\mathcal{R}_{\sigma}^{m+1;1}(\xi_1,\ldots,\xi_m,\bm{\xi})$ is also regular.
\end{proof}

\begin{remk}
In general, for defining Fukaya categories and associated structures, it is only necessary to have regularity of moduli spaces of dimension zero and one. The preceding lemma however requires that the moduli spaces of curves in $\C$ and in $M$ of higher dimensions also be regular. This can be achieved easily by the usual methods.
\end{remk}

We assume from this point on that the perturbation data for $M$ in \eqref{EQ: Split perturbation data 1}--\eqref{EQ: Split perturbation data 4} are chosen generically as in Lemma \ref{LEM: regularity of data for the category B} so that the corresponding data for $\widetilde{M}$ are regular. We also assume that the perturbation data for $\widetilde{M}$ associated to families of cobordisms not necessarily of the form $\gamma\times L$ are regular.

\subsubsection{Definition of $\mathbf{I}_\gamma$}

For $L,L'\in \mathcal{L}^*_d(M)$, both types of Floer complexes $CF(\widetilde{L},\widetilde{L}';\mathscr{D}^+_{\widetilde{L},\widetilde{L}'})$ and $CF(\widetilde{L},\widetilde{L}';\mathscr{D}^-_{\widetilde{L},\widetilde{L}'})$ split as vector spaces as
\begin{equation}
CF(\widetilde{L},\widetilde{L}';\mathscr{D}^{\pm}_{\widetilde{L},\widetilde{L}'})=\bigoplus_{i=1}^l CF(L,L';\mathscr{D}_{L,L'}).
\end{equation}
The $i$th summand corresponds to the point $o_i$, and it will be convenient to identify it as $CF(L,L';\mathscr{D}_{L,L'})^{o_i,-}$ for the negative profile Floer complex and as $CF(L,L';\mathscr{D}_{L,L'})^{o_i,+}$ for the positive profile Floer complex. For $x\in CF(L,L';\mathscr{D}_{L,L'})$, we write $x^{(i,\pm)}$ for the element $(0,\ldots,0,x,0,\ldots,0)$ of $CF(\widetilde{L},\widetilde{L}';\mathscr{D}^{\pm}_{\widetilde{L},\widetilde{L}'})$, where $x$ is placed at the $i$th position.

\begin{remk}\label{REMK: The differential on CF(tildeL,tildeL')}
As remarked in \cite[\S 4.2]{BC14}, the differential $\mu^{\mathcal{F}_c}_1$ on $CF(\widetilde{L},\widetilde{L}';\mathscr{D}^+_{\widetilde{L},\widetilde{L}'})$ has a simple description in terms of the differential $\mu^\mathcal{F}_1$ on $CF(L,L';\mathscr{D}_{L,L'})$. It is given by 
\begin{equation}\label{EQ: Differential on CF(tildeL,tildeL') pos profile}
\mu^{\mathcal{F}_c}_1(x^{(j,+)})=\begin{cases}(\mu^\mathcal{F}_1(x))^{(j,+)}+x^{(j-1,+)}-x^{(j+1,+)}, &j \text{ odd}, \\
(\mu^\mathcal{F}_1(x))^{(j,+)},& j\text{ even}.
\end{cases}
\end{equation}
Here we set $x^{(j-1,+)}=0$ if $j=1$ and $x^{(j+1,+)}=0$ if $j=l$. Formula \eqref{EQ: Differential on CF(tildeL,tildeL') pos profile} is a simple consequence of the split form of the Floer data \ref{EQ: Split positive profile Floer data} together with index considerations and an open mapping theorem argument. The same arguments show that the differential $\partial^-$ on $CF(\widetilde{L},\widetilde{L}';\mathscr{D}^-_{\widetilde{L},\widetilde{L}'})$ is given by
\begin{equation}\label{EQ: Differential on CF(tildeL,tildeL') neg profile}
\partial^-(x^{(j,-)})=\begin{cases}(\mu^\mathcal{F}_1(x))^{(j,-)}+x^{(j-1,-)}-x^{(j+1,-)}, &j \text{ even}, \\
(\mu^\mathcal{F}_1(x))^{(j,-)},& j\text{ odd}.
\end{cases}
\end{equation}
As a result, we see that the complexes $CF(L,L';\mathscr{D}_{L,L'})$, $CF(\widetilde{L},\widetilde{L}';\mathscr{D}^+_{\widetilde{L},\widetilde{L}'})$, and $CF(\widetilde{L},\widetilde{L}';\mathscr{D}^-_{\widetilde{L},\widetilde{L}'})$ are all quasi-isomorphic.
\end{remk}

The functor $\mathbf{I}_\gamma:\mathcal{F}\to\mathcal{F}_c$ is defined on objects by $\mathbf{I}_\gamma(L)=\widetilde{L}$. The map $(\mathbf{I}_\gamma)_1$ is given by
\begin{equation}
\begin{aligned}
&(\mathbf{I}_\gamma)_1:CF(L,L';\mathscr{D}_{L,L'})\to CF(\widetilde{L},\widetilde{L}';\mathscr{D}^+_{\widetilde{L},\widetilde{L}'}),\\
&x\mapsto x^{(1,+)}+x^{(3,+)}+\cdots+x^{(l-2,+)}+x^{(l,+)}.
\end{aligned}
\end{equation}
The maps $(\mathbf{I}_\gamma)_k$ for $k\ge 2$ are all zero. It is an easy consequence of the corollary of the next lemma that $\mathbf{I}_\gamma$ is a functor. We omit the proofs of this lemma (which appears as part of Corollary 4.3.4 in \cite{BC14}) and its corollary as they are very similar to the proofs of Lemma \ref{LEM: Y curves in the category B} and Corollary \ref{COR: operations in B} in the next section. 

\begin{lem}\label{LEM: mu curves in the category B}
Fix $k\ge 2$ and consider the moduli space
\begin{equation}
\mathcal{R}_{\mu}^{k+1}(a_1^{(j_1,+)},\ldots,a_k^{(j_k,+)},\bm{a}^{(s,+)}),
\end{equation} 
of curves in $\widetilde{M}$ satisfying boundary conditions along $\widetilde{L}_0,\ldots,\widetilde{L}_k$. Here $a_i\in \mathcal{O}(H_{L_{i-1},L_i})$ for $i=1,\ldots,k$, and $\bm{a}\in \mathcal{O}(H_{L_0,L_k})$. Suppose that the zero-dimensional component of this moduli space is non-empty.
Then one of the following possibilities occurs:
\begin{enumerate}
\item The index $s$ is odd and we have $j_1=\cdots=j_k=s$. In this case, $\pi\circ u$ is constant at $o_s$. 
\item The index $s$ is even, and among the indices $j_1,\ldots,j_k$ at least one is also even.
\end{enumerate}
\end{lem}

\begin{cor}\label{COR: mu in the category B}
Let $j_1,\ldots,j_k\in\{1,\ldots,l\}$ be odd. Then the following relation is satisfied:
\begin{equation}
\mu_k^{\mathcal{F}_c}(x_1^{(j_1,+)},\ldots,x_k^{(j_k,+)})
=\begin{cases}\mu^{\mathcal{F}}_{k}(x_1,\ldots,x_k)^{(j_1,+)},& j_1=\cdots=j_k,\\
0,&\text{otherwise},
\end{cases}
\end{equation}
for all $x_i\in CF(L_{i-1},L_i;\mathscr{D}_{L_{i-1},L_i})$, $i=1,\ldots,k$.
\end{cor}

The functor $\mathbf{I}_\gamma$ is called the \textbf{inclusion functor}. If we allow $\gamma$ and the function $h'$, as well as the profile function $h$, to vary, we in fact obtain a whole family of inclusion functors,
\begin{equation}
\mathbf{I}_{\gamma,h,h'}:\mathcal{F}\mathit{uk}(M)\to \mathcal{F}\mathit{uk}_{\mathit{cob}}(\widetilde{M};\mathscr{D}(\gamma,h,h')). 
\end{equation}
Here we use $\mathscr{D}(\gamma,h,h')$ to denote the collection of Floer and perturbation data on $\widetilde{M}$. This data is assumed to satisfy the conditions in Section \ref{SUBSECTION: Floer data and perturbation data} as well as \eqref{EQ: Split positive profile Floer data}--\eqref{EQ: Split perturbation data 4} for the curve $\gamma$ and the functions $h$ and $h'$. 
For any two choices of data $\mathscr{D}(\gamma_1,h_1,h_1')$ and $\mathscr{D}(\gamma_2,h_2,h_2')$, there is a comparison quasi-isomorphism
\begin{equation}
\mathbf{K}^{\mathscr{D}(\gamma_1,h_1,h_1')}_{\mathscr{D}(\gamma_2,h_2,h_2')}:\mathcal{F}\mathit{uk}_{\mathit{cob}}(\widetilde{M};\mathscr{D}(\gamma_1,h_1,h_1'))\to \mathcal{F}\mathit{uk}_{\mathit{cob}}(\widetilde{M};\mathscr{D}(\gamma_2,h_2,h_2')).
\end{equation}
Suppose $\gamma_1$ and $\gamma_2$ are horizontally isotopic paths, i.e.\ they are related by a Hamiltonian isotopy which is only permitted to slide the ends of the path in the horizontal direction and cannot change their $y$-coordinate. Then the functors $\mathbf{K}^{\mathscr{D}(\gamma_1,h_1,h_1')}_{\mathscr{D}(\gamma_2,h_2,h_2')}\circ \mathbf{I}_{\gamma_1,h_1,h_1'}$ and $\mathbf{I}_{\gamma_2,h_2,h_2'}$ are related by a natural quasi-isomorphism which is canonically defined on homology. See \cite[Proposition 4.2.5]{BC14}.

\subsection{The relative weak Calabi-Yau pairing and compatibility}

In order to complete the proof of Theorem \ref{THM: Main theorem}, it remains to define the $\mathcal{F}\mbox{--}\mathcal{F}$ bimodule morphism $i_{\mathit{rel}}:\mathcal{F}_\Delta\to (\mathbf{I}_\gamma)^*(\mathcal{F}_c)_\Delta^\mathit{rel}$ in part \eqref{THM: Main theorem Hochschild part} of the theorem and the natural transformation $S^{\mathit{rel}}:\mathbf{G}^r_{\mathbf{I}_\gamma}(\mathbf{Y}^r_\mathit{rel})\to \mathbf{Y}^r_\mathcal{F}$ in part \eqref{THM: Main theorem nat transf part}, and to verify that they satisfy the stated properties.

\subsubsection{The bimodule morphism $i_{\mathit{rel}}$}

As a module pre-morphism, $i_{\mathit{rel}}:\mathcal{F}_\Delta\to (\mathbf{I}_\gamma)^*(\mathcal{F}_c)_\Delta^\mathit{rel}$ is defined by 
\begin{equation}
\begin{aligned}
&(i_{\mathit{rel}})_{0|1|0}:CF(L,L';\mathscr{D}_{L,L'})\to CF(\widetilde{L},\widetilde{L}';\mathscr{D}^-_{\widetilde{L},\widetilde{L}'}),\\
&(i_{\mathit{rel}})_{0|1|0}(x)=x^{(1,-)}+x^{(3,-)}+\cdots+x^{(l,-)},\\ 
&(i_{\mathit{rel}})_{m|1|p}=0,\;  \text{for }(m,p)\ne (0,0).
\end{aligned}
\end{equation}

To verify that $i_{\mathit{rel}}$ is in fact a module morphism and to show that $[(\mathbf{I}_{\gamma}^{\mathit{rel}})_*^\vee \sigma^{\mathcal{F}_{\mathit{c}}}]=[\sigma^\mathcal{F}]$, we will use the following results about solutions $(r,u)$ of \eqref{EQ: Perturbed CR equation} where $u$ satisfies boundary conditions along cobordisms of the form $\gamma\times L$.

\begin{lem}\label{LEM: Y curves in the category B}
Assume $m+p\ge 1$ and consider the moduli space
\begin{equation}\label{EQ: Y moduli space in the category B}
\mathcal{R}_{\mathbf{Y}}^{m+p+2}(a_1^{(j_1,+)},\ldots,a_m^{(j_m,+)},\bm{c}^{(s,-)};b_p^{(j'_p,+)},\ldots,b_1^{(j'_1,+)},\bm{c}'^{(s',-)}),
\end{equation} 
of curves in $\widetilde{M}$ satisfying boundary conditions along $\widetilde{L}_0,\ldots,\widetilde{L}_m,\widetilde{N}_p,\ldots,\widetilde{N}_0$. Here $a_i\in \mathcal{O}(H_{L_{i-1},L_i})$ for $i=1,\ldots,m$, $\bm{c}\in \mathcal{O}(H_{L_m,N_p})$, $b_i\in \mathcal{O}(H_{N_i,N_{i-1}})$ for $i=p,\ldots,1$, and $\bm{c}'\in \mathcal{O}(H_{L_0,N_0})$. Suppose that the zero-dimensional component of this moduli space is non-empty.
Then one of the following possibilities occurs:
\begin{enumerate}
\item The index $s$ is odd and we have $j_1=\cdots=j_m=s=j'_p=\cdots=j'_1=s'$. In this case, $\pi\circ u$ is constant at $o_s$. 
\item The index $s$ is even, and among the indices $j_1,\ldots,j_m,j'_p,\ldots,j'_1,s'$ at least one is also even.
\end{enumerate}
\end{lem}

\begin{proof}Set $\mathcal{S}=\mathcal{S}^{m+p+2}$ and let $(r,u)$ belong to the zero-dimensional component of the moduli space \eqref{EQ: Y moduli space in the category B}.
Note that with the choice of perturbation data \eqref{EQ: Y perturbation data for B}, when $u$ is transformed by the naturality transformation \eqref{EQ: naturality transformation}, it becomes holomorphic over a neighbourhood of the $o_i$ with $i$ odd. Therefore the odd intersection points in $\gamma\cap (\phi^{h'}_1)^{-1}(\gamma)$ behave as bottlenecks and in particular are entry points for positive profile Floer data and exit points for negative profile Floer data. It follows that every curve $u:\mathcal{S}_r\to\widetilde{M}$ that has an entry over a bottleneck with negative profile Floer data must project to a constant at the bottleneck. This proves that $j_1=\cdots=j_m=s=j'_p=\cdots=j'_1=s'$ and that $\pi\circ u$ is constant if $s$ is odd. 

The other possibility is that $s$ is even. In this case, set $w=\pi\circ u$ and $w'=\pi_M\circ u$. It follows from the split form of the perturbation data that the index of the extended linearized operator $\mathcal{D}_{\mathcal{S},r,u}$ associated to \eqref{EQ: Perturbed CR equation} at $(r,u)$ can be expressed in terms of the indices of the operators $\mathcal{D}_{\mathcal{S},r,w}$ and $\mathcal{D}_{\mathcal{S},r,w'}$ by
\begin{equation}\label{EQ: Extended linearized operator sum formula}
\mathrm{ind}(\mathcal{D}_{\mathcal{S},r,u})=\mathrm{ind}(\mathcal{D}_{\mathcal{S},r,w})+\mathrm{ind}(\mathcal{D}_{\mathcal{S},r,w'})-(m+p-1).
\end{equation}
Suppose toward a contradiction that $j_1,\ldots,j_m,j'_p,\ldots,j'_1,s'$ are all odd. Then by Lemma \ref{LEM: Index formula 2 fixed} we obtain $\mathrm{ind}(\mathcal{D}_{\mathcal{S},r,w})=m+p$. By assumption $\mathrm{ind}(\mathcal{D}_{\mathcal{S},r,u})=0$, and so from \eqref{EQ: Extended linearized operator sum formula} we have $\mathrm{ind}(\mathcal{D}_{\mathcal{S},r,w'})=-1$.
This contradicts the fact that the vertical perturbation datum $(\mathbf{\Theta}_M^\mathbf{Y},\mathbf{J}_M^\mathbf{Y})$ is regular. 
Therefore one of the indices $j_1,\ldots,j_m,j'_p,\ldots,j'_1,s'$ must be even.
\end{proof}

The following two lemmas are proved using the same argument as in the proof of Lemma \ref{LEM: Y curves in the category B}.
\begin{lem}\label{LEM: sigma curves in the category B}
Consider the moduli space 
$\mathcal{R}_{\sigma}^{m+1;1}(a_1^{(j_1,+)},\ldots,a_m^{(j_m,+)},\bm{a}^{(s,-)})$
of curves in $\widetilde{M}$ satisfying boundary conditions along $\widetilde{L}_0,\ldots,\widetilde{L}_m$. Here $a_i\in \mathcal{O}(H_{L_{i-1},L_i})$ for $i=1,\ldots,m$, and $\bm{a}\in \mathcal{O}(H_{L_m,L_0})$. Suppose that the zero-dimensional component of this moduli space is non-empty.
Then one of the following possibilities occurs:
\begin{enumerate}
\item The index $s$ is odd and we have $j_1=\cdots=j_m=s$. In this case, $\pi\circ u$ is constant at $o_s$. 
\item The index $s$ is even, and among the indices $j_1,\ldots,j_m$ at least one is also even.
\end{enumerate}
\end{lem}

\begin{lem}\label{LEM: delta curves in the category B}
Consider the moduli space 
$$\mathcal{R}_{\delta}^{m,p;1}(a_1^{(j_1,+)},\ldots,a_m^{(j_m,+)},\bm{d}^{(s,+)};b_p^{(j'_p,+)},\ldots,b_1^{(j'_1,+)},\bm{d}'^{(s',-)})$$
of curves in $\widetilde{M}$ satisfying boundary conditions along $\widetilde{L}_0,\ldots,\widetilde{L}_m, \widetilde{N}_p,\ldots,\widetilde{N}_0$. Here $a_i\in \mathcal{O}(H_{L_{i-1},L_i})$ for $i=1,\ldots,m$, $\bm{d}\in \mathcal{O}(H_{L_m,N_p})$, $b_i\in \mathcal{O}(H_{N_i,N_{i-1}})$ for $i=p,\ldots,1$, and $\bm{d}'\in \mathcal{O}(H_{N_0,L_0})$. Suppose that the zero-dimensional component of this moduli space is non-empty.
Then one of the following possibilities occurs:
\begin{enumerate}
\item The index $s$ is odd and we have $j_1=\cdots=j_m=s=j'_p=\cdots=j'_1=s'$. In this case, $\pi\circ u$ is constant at $o_s$. 
\item The index $s$ is even, and among the indices $j_1,\ldots,j_m,j'_p,\ldots,j'_1,s'$ at least one is also even.
\end{enumerate}
\end{lem}

We now use the preceding three lemmas to prove a corollary which is the analogue of Corollary \ref{COR: mu in the category B} for the module structure maps $\mu_{m|1|p}^{(\mathcal{F}_c)^{\mathit{rel}}_\Delta}$, the dual Hochschild cycle $\sigma^{\mathcal{F}_c}$, and the natural transformation $\delta^{\mathcal{F}_c}$.

\begin{cor}\label{COR: operations in B}
Assume $j_1,\ldots,j_m,s,j'_p,\ldots,j'_1,s'$ are all odd. Then the following relations are satisfied:
\begin{enumerate}
\item\label{COR: operations in B, muFDeltarel}
$\begin{aligned}[t]
&\mu_{m|1|p}^{(\mathcal{F}_c)^{\mathit{rel}}_\Delta}(a_1^{(j_1,+)},\ldots,a_m^{(j_m,+)},\bm{c}^{(s,-)},b_p^{(j'_p,+)},\ldots,b_1^{(j'_1,+)})\\
&\quad =\begin{cases}\mu^{\mathcal{F}}_{m+p+1}(a_1,\ldots,a_m,\bm{c},b_p,\ldots,b_1)^{(s,-)},&j_1=\cdots=j_m=s\\
&\qquad=j'_p=\cdots=j'_1,\\
0,&\text{otherwise}.
\end{cases}
\end{aligned}$
\item\label{COR: operations in B, sigma}
$
\begin{aligned}[t]
&\langle \sigma^{\mathcal{F}_c},a_1^{(j_1,+)}\otimes \cdots\otimes a_m^{(j_m,+)}\otimes\bm{a}^{(s,-)}\rangle\\
&\quad=\begin{cases} \langle\sigma^{\mathcal{F}},a_1\otimes\ldots\otimes a_m\otimes\bm{a}\rangle,& j_1=\cdots=j_m=s,\\
0,&\text{otherwise}.
\end{cases}
\end{aligned}$
\item\label{COR: operations in B, delta}
$
\begin{aligned}[t]
&\langle(\delta^{\mathcal{F}_c}_p(b_p^{(j'_p,+)},\ldots,b_1^{(j'_1,+)}))_{m|1}(a_1^{(j_1,+)},\ldots,a_m^{(j_m,+)},\bm{d}^{(s,+)}),\bm{d}'^{(s',-)}\rangle\\
&\quad =\begin{cases}\langle((\delta^\mathcal{F})_p(b_p,\ldots,b_1))_{m|1}(a_1,\ldots,a_m,\bm{d}),\bm{d}'\rangle,& j_1=\cdots=j_m=s\\
&\qquad=j'_p=\cdots=j'_1=s',\\
0,&\text{otherwise}.
\end{cases}
\end{aligned}
$
\end{enumerate}
\end{cor}

\begin{proof} In order to prove \eqref{COR: operations in B, muFDeltarel}, first notice that the identity holds for $m=p=0$ as a result of the form \eqref{EQ: Differential on CF(tildeL,tildeL') neg profile} of the differential on $CF(\widetilde{L},\widetilde{L}';\mathscr{D}^-_{\widetilde{L},\widetilde{L}'})$. Assume therefore that $(m,p)\ne (0,0)$. Then Lemma \ref{LEM: Y curves in the category B} implies that the left-hand side of the equation in part \eqref{COR: operations in B, muFDeltarel} vanishes unless $j_1=\cdots=j_m=s=j'_p=\cdots=j'_1$, since otherwise the moduli spaces $\mathcal{R}_{\mathbf{Y}}^{m+p+2}(a_1^{(j_1,+)},\ldots,a_m^{(j_m,+)},\bm{c}^{(s,-)};b_p^{(j'_p,+)},\ldots,b_1^{(j'_1,+)},\bm{c}'^{(s',-)})^0$ are empty. 
For the case where $j_1=\cdots=j_m=s=j'_p=\cdots=j'_1$, we apply part \eqref{COR: Cor of the index formula, I_F, more punctures} of Corollary \ref{COR: Cor of the index formula} with $k=m+p+1$ and $i_0=m+1$. From this, we obtain that the index of the horizontal part of the linearized operator (for a fixed domain) associated to \eqref{EQ: Perturbed CR equation} at a curve $u$ in $\widetilde{M}$ that projects to an $o_i$ with $i$ odd vanishes. It follows that the index of $u$ as a curve in $\widetilde{M}$ is equal to its index as a curve in $M$, and we obtain the identity in part \eqref{COR: operations in B, muFDeltarel}.

The second and third parts of the corollary are proved by the same argument, using the results of Lemmas \ref{LEM: sigma curves in the category B} and \ref{LEM: delta curves in the category B}. 
\end{proof}

\begin{prop}\label{PROP: irel is a module morphism} The module pre-morphism $i_{\mathit{rel}}:\mathcal{F}_\Delta\to (\mathbf{I}_\gamma)^*(\mathcal{F}_c)_\Delta^\mathit{rel}$ is a module morphism, i.e.\ $\mu_1^{\mathcal{F}\mathit{\mbox{--}mod\mbox{--}}\mathcal{F}}(i_{\mathit{rel}})=0$.
\end{prop}

\begin{proof} We must verify the identity
\begin{align}\label{EQ: Module morphism identity for i_rel}
&\sum\mu^{(\mathbf{I}_\gamma)^*(\mathcal{F}_c)_\Delta^\mathit{rel}}_{i-1|1|i'-1}(x_1,\ldots,(i_{\mathit{rel}})_{m-i+1|1|p-i'+1}(x_i,\ldots,x_m,\bm{z},y_p,\ldots,y_{i'}), \ldots, y_1) \nonumber\\
&\quad+ \sum(i_{\mathit{rel}})_{i-1|1|i'-1}(x_1,\ldots,\mu^{\mathcal{F}_\Delta}_{m-i+1|1|p-i'+1}(x_i,\ldots,x_m,\bm{z},y_p,\ldots,y_{i'}),\ldots,y_1) \nonumber\\
& \quad + \sum (i_{\mathit{rel}})_{m-j'+j|1|p}(x_1,\ldots,\mu^{\mathcal{F}}_{j'-j+1}(x_{j},\ldots,x_{j'}),\ldots,x_m,\bm{z},y_p,\ldots,y_1)\nonumber\\
& \quad + \sum (i_{\mathit{rel}})_{m|1|p-j'+j}(x_1,\ldots,x_m,\bm{z},y_p,\ldots,\mu^{\mathcal{F}}_{j'-j+1}(y_{j'},\ldots,y_{j}),\ldots,y_1)\nonumber\\
&\quad\quad=0
\end{align}
Since $(i_{\mathit{rel}})_{m|1|p}=0$ for $(m,p)\ne (0,0)$, the last two sums on the left-hand side of \eqref{EQ: Module morphism identity for i_rel} are empty. Simply applying the definitions of $i_{\mathit{rel}}$, of the functor $\mathbf{I}_\gamma$, and of the pullback of a bimodule, the first sum is equal to
\begin{equation}
\begin{aligned}\label{EQ: i_rel is a bimodule morphism sum 1}
&\mu_{m|1|p}^{(\mathcal{F}^{\mathit{rel}}_c)_\Delta}((x_1)^{(1,+)}+(x_1)^{(3,+)}+\cdots+(x_1)^{(l,+)},\ldots,\\
&\qquad(x_m)^{(1,+)}+(x_m)^{(3,+)}+\cdots+(x_m)^{(l,+)},
\bm{z}^{(1,-)}+\bm{z}^{(3,-)}+\cdots+\bm{z}^{(l,-)},\\
&\qquad(y_p)^{(1,+)}+(y_p)^{(3,+)}+\cdots+(y_p)^{(l,+)},
\ldots,(y_1)^{(1,+)}+(y_1)^{(3,+)}+\cdots+(y_1)^{(l,+)}).
\end{aligned}
\end{equation}
One sees directly that the second sum is equal to
\begin{equation}\label{EQ: i_rel is a bimodule morphism sum 2}
\mu_{m+p+1}^\mathcal{F}(x_1,\ldots,\bm{z},\ldots,y_1)^{(1,-)}+\cdots+\mu_{m+p+1}^\mathcal{F}(x_1,\ldots,\bm{z},\ldots,y_1)^{(l,-)}.
\end{equation}
By part \eqref{COR: operations in B, muFDeltarel} of Corollary \ref{COR: operations in B}, we have equality of \eqref{EQ: i_rel is a bimodule morphism sum 1} and \eqref{EQ: i_rel is a bimodule morphism sum 2}.
\end{proof}

\begin{prop}\label{PROP: sigmaF_c pulls back to sigmaF}
The dual Hochschild cycles $\sigma^{\mathcal{F}_c}\in CC_\bullet (\mathcal{F}_c,(\mathcal{F}_c)^{\mathit{rel}}_\Delta)^\vee$ and $\sigma^\mathcal{F}\in CC_\bullet(\mathcal{F})^\vee$ are related by
\begin{equation}
(\mathbf{I}^{\mathit{rel}}_\gamma)^\vee_*\sigma^{\mathcal{F}_c}=\sigma^\mathcal{F}. 
\end{equation}
\end{prop}

\begin{proof}
Recall that $(\mathbf{I}_\gamma^{\mathit{rel}})_*$ is defined to be the composition
\begin{equation}
(\mathbf{I}_\gamma^{\mathit{rel}})_*:CC_\bullet (\mathcal{F})\xrightarrow{(i_{\mathit{rel}})_*}CC_\bullet (\mathcal{F},\mathbf{I}_\gamma^*(\mathcal{F}_c)^{\mathit{rel}}_\Delta)\xrightarrow{(\mathbf{I}_\gamma)_*} CC_\bullet (\mathcal{F}_c,(\mathcal{F}_c)^{\mathit{rel}}_\Delta).
\end{equation}
Using the definitions of $\mathbf{I}_\gamma$ and $i_{\mathit{rel}}$ as well as of the pushforward maps \eqref{EQ: Hochschild cohomology functoriality 1} and \eqref{EQ: Hochschild cohomology functoriality 2} on Hochschild homology, we obtain
\begin{align}
&\langle (\mathbf{I}_\gamma^{\mathit{rel}})_*^\vee \sigma^{\mathcal{F}_c},\xi_1\otimes \cdots\otimes \xi_m\otimes\bm{\xi}\rangle\\
&\qquad = \langle \sigma^{\mathcal{F}_c},(\mathbf{I}_\gamma)_1(\xi_1)\otimes \cdots\otimes (\mathbf{I}_\gamma)_1(\xi_m)\otimes (i_{\mathit{rel}})_{0|1|0}(\bm{\xi})\rangle\\
&\qquad = \langle \sigma^{\mathcal{F}_c},(\xi_1^{(1,+)}+\xi_1^{(3,+)}+\cdots +\xi_1^{(l,+)})\otimes \cdots \\
&\qquad\qquad\otimes(\xi_m^{(1,+)}+\xi_m^{(3,+)}+\cdots +\xi_m^{(l,+)})\otimes(\bm{\xi}^{(1,-)}+\bm{\xi}^{(3,-)}+\cdots +\bm{\xi}^{(l,-)}) \rangle\nonumber.
\end{align}
Then from part \eqref{COR: operations in B, sigma} of Corollary \ref{COR: operations in B}, we have
\begin{equation}
\langle (\mathbf{I}_\gamma^{\mathit{rel}})_*^\vee \sigma^{\mathcal{F}_c},\xi_1\otimes \cdots\otimes \xi_m\otimes\bm{\xi}\rangle= \#_{\Z_2}\{o_i|\;i\text{ is odd}\}\langle \sigma^\mathcal{F},\xi_1\otimes\cdots\otimes \xi_m\otimes\bm{\xi} \rangle.
\end{equation}
Finally, from the fact that the number $l$ of intersection points between $\gamma$ and $(\phi_1^{h'})^{-1}(\gamma)$ is equal to one in $\Z_4$, we obtain the result. 
\end{proof}

\subsubsection{The natural transformation $S^{\mathit{rel}}$}

As a pre-natural transformation, $S^{\mathit{rel}}=(S^{\mathit{rel}}_0,S^{\mathit{rel}}_1,\ldots):\mathbf{G}^r_{\mathbf{I}_\gamma}(\mathbf{Y}^r_\mathit{rel})\to \mathbf{Y}^r_\mathcal{F}$ is defined as follows. First, in order to define the component $S^{\mathit{rel}}_0$, we must specify a morphism $(S^{\mathit{rel}}_0)_L$ in the category $(\mathit{mod\mbox{--}}\mathcal{F})^{\mathit{opp}}$ from $\mathbf{G}^r_{\mathbf{I}_\gamma}(\mathbf{Y}^r_{\mathit{rel}})(L)$ to $\mathbf{Y}^r_\mathcal{F}(L)$ for every object $L$ in $\mathrm{Ob}(\mathcal{F})$. In other words, $(S^{\mathit{rel}}_0)_L$ must be a pre-morphism of right $\mathcal{F}$-modules from $\mathbf{Y}^r_\mathcal{F}(L)$ to $\mathbf{G}^r_{\mathbf{I}_\gamma}(\mathbf{Y}^r_{\mathit{rel}})(L)$. Note that for $N\in \mathrm{Ob}(\mathcal{F})$, 
\begin{equation}
\mathbf{G}^r_{\mathbf{I}_\gamma}(\mathbf{Y}^r_{\mathit{rel}})(L)(N)=CF(\widetilde{L},\widetilde{N};\mathscr{D}^-_{\widetilde{L},\widetilde{N}}).
\end{equation}
We define $(S^{\mathit{rel}}_0)_L$ by
\begin{equation}
\begin{aligned}
((S^{\mathit{rel}}_0)_L)_{1|0}:&CF(L,N;\mathscr{D}_{L,N})\to CF(\widetilde{L},\widetilde{N};\mathscr{D}^-_{\widetilde{L},\widetilde{N}}),\\
&x\mapsto x^{(1,-)}+x^{(3,-)}+\cdots+x^{(l,-)},\\
\end{aligned}
\end{equation}
and $((S^{\mathit{rel}}_0)_L)_{1|p}=0$ for $p\ge 1$. Finally, all of the components $S^{\mathit{rel}}_m$ for $m\ge 1$ are set to zero.

\begin{prop}\label{PROP: Trel is a natural transformation}
The pre-natural transformation, $S^{\mathit{rel}}$ is a natural transformation, i.e.\ $\mu^{\mathit{fun}(\mathcal{F},(\mathit{mod\mbox{--}}\mathcal{F})^{\mathit{opp}})}_1(S^{\mathit{rel}})=0$.
\end{prop}

\begin{proof} We omit the details of this proof since it is very similar to the proof of Proposition \ref{PROP: irel is a module morphism}. It simply involves using the definitions of the differential $\mu^{\mathit{fun}(\mathcal{F},(\mathit{mod\mbox{--}}\mathcal{F})^{\mathit{opp}})}_1$, of the functors $\mathbf{G}^r_{\mathbf{I}_\gamma}$ and $\mathbf{Y}^r_\mathit{rel}$, and of the pre-natural transformation $S^{\mathit{rel}}$, together with the identity in part \eqref{COR: operations in B, muFDeltarel} of Corollary \ref{COR: operations in B}.
\end{proof}

\begin{remk}
The identities $\mu^{\mathit{fun}(\mathcal{F},(\mathit{mod\mbox{--}}\mathcal{F})^{\mathit{opp}})}_1(S^{\mathit{rel}})=0$ and $\mu_1^{\mathcal{F}\mathit{\mbox{--}mod\mbox{--}}\mathcal{F}}(i_{\mathit{rel}})=0$ are in fact equivalent. This is a result of $S^{\mathit{rel}}$ and $i_{\mathit{rel}}$ being related by the isomorphism of dg-categories 
\begin{equation}
\Phi^r:\mathit{fun}(\mathcal{F},(\mathit{mod}\mbox{--}\mathcal{F})^{\mathit{opp}})\xrightarrow{\cong} (\mathcal{F}\mathit{\mbox{--}mod\mbox{--}}\mathcal{F})^{\mathit{opp}}
\end{equation}
from \eqref{EQ: Category isos between functors into left modules and bimodules}. Indeed, the object map of $\Phi^r$ satisfies $\Phi^r(\mathbf{Y}^r_\mathcal{F})=\mathcal{F}_\Delta$ and $\Phi^r(\mathbf{G}^r_{\mathbf{I}_\gamma}(\mathbf{Y}^r_\mathit{rel}))=(\mathbf{I}_\gamma)^*(\mathcal{F}_c)_\Delta^\mathit{rel}$. The latter equality makes use of Proposition \ref{PROP:Compatibility G and pullback}. One checks easily that the map on morphisms for $\Phi^r$ satisfies $\Phi^r(S^{\mathit{rel}})=i_{\mathit{rel}}$. 
\end{remk}

\begin{prop}\label{PROP: deltaFc pulls back to deltaF}
The natural transformation $\mathcal{P}_\mathit{rel}(\delta^{\mathcal{F}_c}):\mathbf{Y}^l_\mathcal{F}\to (\mathbf{Y}^\vee_\mathcal{F})^l$ induced by $\delta^{\mathcal{F}_{\mathit{c}}}$ as in \eqref{EQ: Def of mathcalP(delta)} satisfies $\mathcal{P}_\mathit{rel}(\delta^{\mathcal{F}_c})=\delta^\mathcal{F}$.
\end{prop}

\begin{proof} By computation and using part \eqref{COR: operations in B, delta} of Corollary \ref{COR: operations in B}. As in the proof of Proposition \ref{PROP: sigmaF_c pulls back to sigmaF}, this makes use of the fact that the number $l$ of intersection points between $\gamma$ and $(\phi_1^{h'})^{-1}(\gamma)$ is equal to one in $\Z_4$.
\end{proof}

Together Propositions \ref{PROP: irel is a module morphism}, \ref{PROP: sigmaF_c pulls back to sigmaF}, \ref{PROP: Trel is a natural transformation}, and \ref{PROP: deltaFc pulls back to deltaF} complete the proof of Theorem \ref{THM: Main theorem}.

\section{Further directions: Cone decompositions and duality}

Our aim in this section is to explore the implications of the existence of the weak Calabi-Yau pairing on $\mathcal{F}_c$ for the cone decomposition in the derived Fukaya category of $M$ associated to a cobordism. In Section \ref{SECTION: General cone decompositions}, we will recall in full generality Theorem A of \cite{BC14} establishing the existence of this cone decomposition. We begin however by considering the special case of the cobordism associated to the Lagrangian surgery of two Lagrangians intersecting transversely in a single point. We present a speculative result in Section \ref{SECTION: General cone decompositions} which generalizes this example involving Lagrangian surgery to arbitrary cone decompositions associated to cobordisms. The proof of this generalization requires the full machinery used by Biran and Cornea to construct a cone decomposition from a cobordism, and is beyond the scope of this thesis. However, we will provide a very broad idea of what is involved.

\subsection{An example: Duality and the exact triangle associated to Lagrangian surgery}

Let $L_1,L_2\in \mathcal{L}^*_d(M)$ be two Lagrangians which intersect in a single point, which we assume to be a transverse intersection. We first recall how to perform Lagrangian surgery at this point, a construction due to Lalonde-Sikarov and Polterovich \cite{LalSik, Pol}. The construction begins with the local picture, i.e. we assume $L_1=\R^n\subset\C^n$ and $L_2=i\R^n\subset \C^n$. We define a smooth curve $c\subset\C$ of the form $c(t)=a(t)+ib(t)$, $t\in\R$, where the functions $a$ and $b$ satisfy:
\begin{enumerate}
\item $(a(t),b(t))=(t,0)$ for $t\in (-\infty,-1]$,
\item $(a(t),b(t))=(0,t)$ for $t\in [1,\infty)$,
\item $a'(t),b'(t)>0$ for $t\in (-1,1)$.
\end{enumerate}
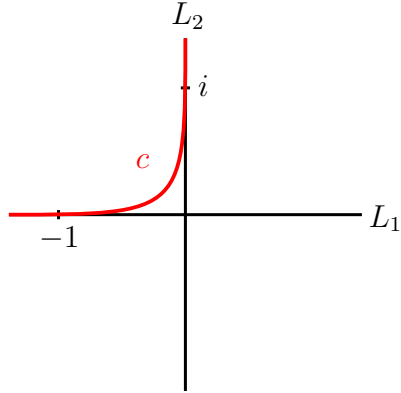
\begin{figure}
\centering
\def\svgwidth{50mm}
\begingroup%
  \makeatletter%
  \providecommand\color[2][]{%
    \errmessage{(Inkscape) Color is used for the text in Inkscape, but the package 'color.sty' is not loaded}%
    \renewcommand\color[2][]{}%
  }%
  \providecommand\transparent[1]{%
    \errmessage{(Inkscape) Transparency is used (non-zero) for the text in Inkscape, but the package 'transparent.sty' is not loaded}%
    \renewcommand\transparent[1]{}%
  }%
  \providecommand\rotatebox[2]{#2}%
  \newcommand*\fsize{\dimexpr\f@size pt\relax}%
  \newcommand*\lineheight[1]{\fontsize{\fsize}{#1\fsize}\selectfont}%
  \ifx\svgwidth\undefined%
    \setlength{\unitlength}{141.73228346bp}%
    \ifx\svgscale\undefined%
      \relax%
    \else%
      \setlength{\unitlength}{\unitlength * \real{\svgscale}}%
    \fi%
  \else%
    \setlength{\unitlength}{\svgwidth}%
  \fi%
  \global\let\svgwidth\undefined%
  \global\let\svgscale\undefined%
  \makeatother%
  \begin{picture}(1,1)%
    \lineheight{1}%
    \setlength\tabcolsep{0pt}%
    \put(0.96054858,0.43738052){\color[rgb]{0,0,0}\makebox(0,0)[lt]{\lineheight{2.13000011}\smash{\begin{tabular}[t]{l}$L_1$\end{tabular}}}}%
    \put(0.43455928,0.97593506){\color[rgb]{0,0,0}\makebox(0,0)[lt]{\lineheight{2.13000011}\smash{\begin{tabular}[t]{l}$L_2$\end{tabular}}}}%
    \put(0,0){\includegraphics[width=\unitlength,page=1]{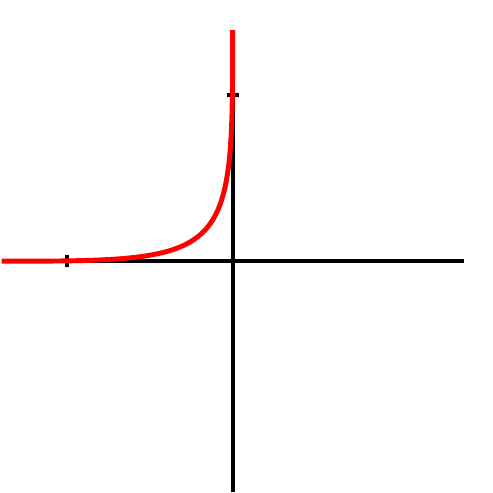}}%
    \put(0.34085879,0.59337738){\color[rgb]{1,0,0}\makebox(0,0)[lt]{\lineheight{2.13000011}\smash{\begin{tabular}[t]{l}$c$\end{tabular}}}}%
    \put(0.08425921,0.38272186){\color[rgb]{0,0,0}\makebox(0,0)[lt]{\lineheight{2.13000011}\smash{\begin{tabular}[t]{l}$-1$\end{tabular}}}}%
    \put(0.50719401,0.7853952){\color[rgb]{0,0,0}\makebox(0,0)[lt]{\lineheight{2.13000011}\smash{\begin{tabular}[t]{l}$i$\end{tabular}}}}%
  \end{picture}%
\endgroup%

\caption{\label{fig:diagram3}The curve $c$}
\end{figure}
See Figure \ref{fig:diagram3}. 

The submanifold $L_1\# L_2\subset \C^n$ obtained by performing Lagrangian surgery of $L_1$ and $L_2$ at $0\in\C^n$ is by definition $c\cdot S^{n-1}$. Explicitly,
\begin{equation} 
L_1\# L_2=\{((a(t)+ib(t))x_1,\ldots,(a(t)+ib(t))x_n)|\;t\in\R,\;(x_1,\ldots,x_n)\in \R^n,\;\sum x_i^2=1\}.
\end{equation}
A computation shows that the submanifold $L_1\# L_2\subset \C^n$ is Lagrangian (see \cite{Pol}). 

To describe the global picture for Lagrangian surgery of $L_1$ and $L_2$, we choose symplectic coordinates around the intersection point in such a way that $L_1$ maps to $\R^n\subset \C^n$ and $L_2$ maps to $i\R^n\subset \C^n$. We then apply the construction above in this coordinate system. 

In \cite{BC13}, Biran and Cornea showed the existence of a cobordism $V:L_1\# L_2\to (L_1,L_2)$. The cobordism $V$ which they constructed is monotone with the same monotonicity constant as $L_1$ and $L_2$. It will be convenient for us to assume all ends of the cobordism are in fact negative. This arrangement can be accomplished by bending the $L_1\# L_2$ end of the cobordism around to the left as shown in Figure \ref{fig:diagram6}. We call the resulting cobordism $W$, and we assume that $W$ belongs to the class $\mathcal{CL}_d(\widetilde{M})$.

\begin{figure}
\centering
\def\svgwidth{135mm}
\begingroup%
  \makeatletter%
  \providecommand\color[2][]{%
    \errmessage{(Inkscape) Color is used for the text in Inkscape, but the package 'color.sty' is not loaded}%
    \renewcommand\color[2][]{}%
  }%
  \providecommand\transparent[1]{%
    \errmessage{(Inkscape) Transparency is used (non-zero) for the text in Inkscape, but the package 'transparent.sty' is not loaded}%
    \renewcommand\transparent[1]{}%
  }%
  \providecommand\rotatebox[2]{#2}%
  \newcommand*\fsize{\dimexpr\f@size pt\relax}%
  \newcommand*\lineheight[1]{\fontsize{\fsize}{#1\fsize}\selectfont}%
  \ifx\svgwidth\undefined%
    \setlength{\unitlength}{382.67716535bp}%
    \ifx\svgscale\undefined%
      \relax%
    \else%
      \setlength{\unitlength}{\unitlength * \real{\svgscale}}%
    \fi%
  \else%
    \setlength{\unitlength}{\svgwidth}%
  \fi%
  \global\let\svgwidth\undefined%
  \global\let\svgscale\undefined%
  \makeatother%
  \begin{picture}(1,0.50625627)%
    \lineheight{1}%
    \setlength\tabcolsep{0pt}%
    \put(0,0){\includegraphics[width=\unitlength,page=1]{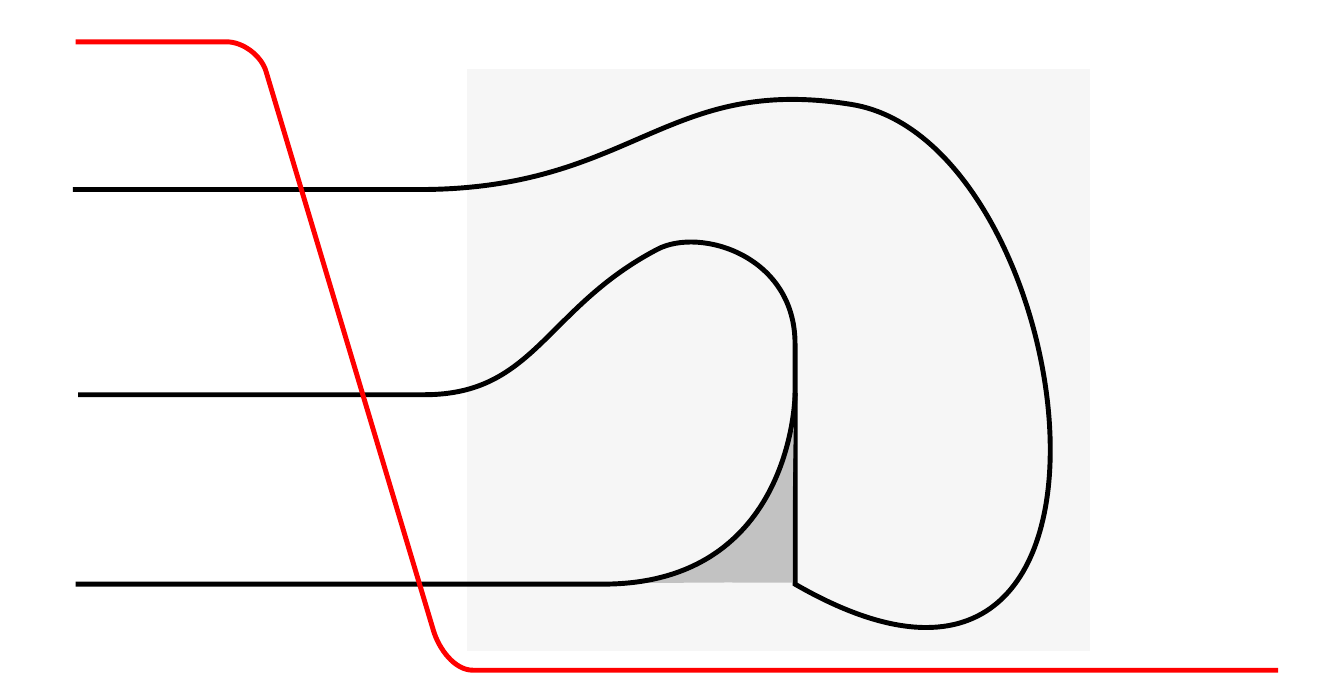}}%
    \put(-0.04499516,0.35338956){\color[rgb]{0,0,0}\makebox(0,0)[lt]{\lineheight{2.13000011}\smash{\begin{tabular}[t]{l}$L_1\#L_2$\end{tabular}}}}%
    \put(0.97265915,-0.00307992){\color[rgb]{0,0,0}\makebox(0,0)[lt]{\lineheight{2.13000011}\smash{\begin{tabular}[t]{l}$\gamma$\end{tabular}}}}%
    \put(-0.04499516,0.20051904){\color[rgb]{0,0,0}\makebox(0,0)[lt]{\lineheight{2.13000011}\smash{\begin{tabular}[t]{l}$L_2$\end{tabular}}}}%
    \put(-0.04499516,0.05548806){\color[rgb]{0,0,0}\makebox(0,0)[lt]{\lineheight{2.13000011}\smash{\begin{tabular}[t]{l}$L_1$\end{tabular}}}}%
    \put(0,0){\includegraphics[width=\unitlength,page=2]{diagram6.pdf}}%
    \put(0.7703128,0.46454243){\color[rgb]{0,0,0}\makebox(0,0)[lt]{\lineheight{2.13000011}\smash{\begin{tabular}[t]{l}$K$\end{tabular}}}}%
  \end{picture}%
\endgroup%

\caption[The cobordism associated to the Lagrangian surgery of $L_1$ and $L_2$]{\label{fig:diagram6}The cobordism associated to the Lagrangian surgery of $L_1$ and $L_2$, together with the path $\gamma$ and the compact set $K$}
\end{figure}

\subsubsection{The chain-level duality quasi-isomorphism}
We take a path $\gamma\subset \C$ as shown in Figure \ref{fig:diagram6}. For a Lagrangian $N\in\mathcal{L}^*_d(M)$, we will consider the Poincar\'{e} duality quasi-isomorphism for the Floer complex associated to the cobordisms $\widetilde{N}=\gamma\times N$ and $W$.
Considering the chain-level case allows for substantial simplification over the module-level case because it is possible to use Floer data of a simpler form than that described in Section \ref{SUBSECTION: Floer data and perturbation data}. In this situation where $\gamma$ intersects the projection of $W$ transversely, we can assume that outside of a compact set of the form $K\times M$ where $K$ is as in Figure \ref{fig:diagram6}, this data is given by $(G^+\circ \pi_M, i\oplus J_M^+)$. Here $\mathscr{D}^+_M=(G^+,J_M^+)$ is a Floer datum on $M$ which is regular for the pairs $(N,L_1)$, $(N,L_2)$, and $(N,L_1\# L_2)$. By abuse of notation we denote the datum $(G^+\circ \pi_M, i\oplus J_M^+)$ on $\widetilde{M}$ by $\mathscr{D}^+_{\widetilde{N},W}=(H^+_{\widetilde{N},W},J^+_{\widetilde{N},W})$. In other words, we are using a vanishing profile function in the description of Floer data in Section \ref{SUBSECTION: Floer data and perturbation data}. Likewise, we can assume that outside of $K\times M$, the Floer datum for the pair $(W,\widetilde{N})$ is of the form $(G^-\circ \pi_M,i\oplus J_M^-)$ where $\mathscr{D}^-_M:=(G^-,J_M^-)$ is a Floer datum on $M$ which is regular for the pairs $(L_1,N)$, $(L_2,N)$, and $(L_1\# L_2,N)$. We denote the datum $(G^-\circ \pi_M,i\oplus J_M^-)$ by $\mathscr{D}^-_{W,\widetilde{N}}=(H^-_{W,\widetilde{N}},J^-_{W,\widetilde{N}})$. 

The complexes $CF(\widetilde{N}, W; \mathscr{D}^+_{\widetilde{N},W})$ and $CF(W,\widetilde{N};\mathscr{D}^-_{W,\widetilde{N}})$ split as vector spaces as 
\begin{align}
CF(\widetilde{N}, W; \mathscr{D}^+_{\widetilde{N},W})=CF(N,L_1\#L_2;\mathscr{D}^+_M)\oplus CF(N,L_2;\mathscr{D}^+_M)\oplus CF(N,L_1;\mathscr{D}^+_M),\\
CF(W,\widetilde{N};\mathscr{D}^-_{W,\widetilde{N}})=CF(L_1\# L_2,N;\mathscr{D}^-_M)\oplus CF(L_2,N;\mathscr{D}^-_M)\oplus CF(L_1,N;\mathscr{D}^-_M).
\end{align}
Arguments based on the open mapping theorem and orientation considerations show that certain components of the differentials of these complexes vanish (see \cite{BC13}). These differentials are respectively of the form 
\begin{equation}\label{EQ: Cone differentials}
\partial^+_{\widetilde{N}, W}=\left(\begin{array}{ccc}\partial^+_{N,L_1\#L_2}&0&0\\ \rho_{32}&\partial^+_{N,L_2}&0\\\rho_{31}&\rho_{21}&\partial^+_{N,L_1}\end{array}\right),\quad \partial^-_{W,\widetilde{N}}=\left(\begin{array}{ccc}\partial^-_{L_1\#L_2,N}&\psi_{23}&\psi_{13}\\ 0&\partial^-_{L_2,N}&\psi_{12}\\0&0&\partial^-_{L_1,N}\end{array}\right).
\end{equation}
Here $\partial^+_{N,L_1\#L_2}$, $\partial^+_{N,L_2}$ and $\partial^+_{N,L_1}$ denote the differentials on the complexes $CF(N,L_1\#L_2;\mathscr{D}^+_M)$,  $CF(N,L_2;\mathscr{D}^+_M)$, and $CF(N,L_1;\mathscr{D}^+_M)$ respectively. Likewise, $\partial^-_{L_1\#L_2,N}$, $\partial^-_{L_2,N}$, and $\partial^-_{L_1,N}$ denote the differentials on the complexes $CF(L_1\# L_2,N;\mathscr{D}^-_M)$, $CF(L_2,N;\mathscr{D}^-_M)$, and $CF(L_1,N;\mathscr{D}^-_M)$. The maps $\rho_{32}$, $\rho_{21}$, $\psi_{23}$, and $\psi_{12}$ are maps of chain complexes, whereas the maps $\rho_{31}$ and $\psi_{13}$ do not commute with the respective differentials. Moreover, owing to the fact that the path $\gamma$ is horizontally isotopic to a path that does not intersect the projection of $W$, the complexes $CF(\widetilde{N}, W; \mathscr{D}^+_{\widetilde{N},W})$ and $CF(W,\widetilde{N};\mathscr{D}^-_{W,\widetilde{N}})$ are acyclic. Hence the maps 
\begin{align}
&(\rho_{32},\rho_{31}):CF(N,L_1\#L_2;\mathscr{D}^+_M)\to \mathrm{cone}(\rho_{21}),\\
&\psi_{23}\circ\mathrm{pr}_2+\psi_{13}\circ\mathrm{pr}_1:\mathrm{cone}(\psi_{12})\to CF(L_1\# L_2,N;\mathscr{D}^-_M),\label{EQ: right chain-level decomp quasi-iso}
\end{align}
are quasi-isomorphisms. Here $\mathrm{pr}_1$ and $\mathrm{pr}_2$ are the projections $\mathrm{cone}(\psi_{12})\to CF(L_1,N;\mathscr{D}^-_M)$ and $\mathrm{cone}(\psi_{12})\to CF(L_2,N;\mathscr{D}^-_M)$. 

The fact that the complex $CF(N,L_1\#L_2;\mathscr{D}^+_M)$ is quasi-isomorphic to the mapping cone of $\rho_{21}$ is a special case of Theorem 2.2.1 in \cite{BC13}. This theorem applies in general to cobordisms of the form $V:\emptyset\to(L_1,\ldots,L_s,L)$ in $\mathcal{CL}_d(\widetilde{M})$, where $L_1,\ldots,L_s,L\in \mathcal{L}^*_d(M)$ are arbitrary. The theorem associates to such a cobordism and a Lagrangian $N\in \mathcal{L}^*_d(M)$ a cone decomposition of $CF(N,L;\mathscr{D})$ in terms of $CF(N,L_1;\mathscr{D}),\ldots,CF(N,L_s;\mathscr{D})$. Here $\mathscr{D}$ is a Floer datum on $M$ which is regular for the pairs $(N,L_1),\ldots ,(N,L_s),(N,L)$.
By the same arguments used to prove this theorem, one can show that the cobordism $V$ also results in a cone decomposition of $CF(L,N;\mathscr{D}')$ in terms of $CF(L_1,N;\mathscr{D}'), \ldots, CF(L_s,N;\mathscr{D}')$, where $\mathscr{D}'$ is a regular Floer datum for the pairs $(L_1,N),\ldots,(L_s,N),(L,N)$. The existence of the quasi-isomorphism \eqref{EQ: right chain-level decomp quasi-iso} is a special case of this result.

We now consider the Poincar\'{e} duality quasi-isomorphism 
\begin{equation}
\phi_{\widetilde{N}, W}:CF(\widetilde{N}, W; \mathscr{D}^+_{\widetilde{N},W})\to CF(W,\widetilde{N};\mathscr{D}^-_{W,\widetilde{N}})^\vee.
\end{equation}
We assume that the perturbation datum $\mathscr{D}^\delta_{\widetilde{N},W}=(H^\delta_{\widetilde{N},W},J^\delta_{\widetilde{N},W})$ on the disc $\mathcal{S}^{1,1;1}$ used to define the map $\phi_{\widetilde{N}, W}$ is of the form 
\begin{equation}
H^\delta_{\widetilde{N},W}=G^\delta\circ\pi_M,\quad J^\delta_{\widetilde{N},W}=i\oplus J^\delta_M
\end{equation}
outside of the compact set $K\times M$. Here $(G^\delta,J^\delta_M)$ is a perturbation datum on $M$ which is regular for defining the Poincar\'{e} duality maps
\begin{align}
&\phi_{N,L_1\#L_2}:CF(N,L_1\#L_2;\mathscr{D}^+_M)\to CF(L_1\# L_2,N;\mathscr{D}^-_M)^\vee,\\ 
&\phi_{N,L_2}:CF(N,L_2;\mathscr{D}^+_M)\to CF(L_2,N;\mathscr{D}^-_M)^\vee,\\ 
&\phi_{N,L_1}:CF(N,L_1;\mathscr{D}^+_M)\to CF(L_1,N;\mathscr{D}^-_M)^\vee.
\end{align}

By similar arguments to those used to describe the form \eqref{EQ: Cone differentials} of the differentials on the complexes $CF(\widetilde{N}, W; \mathscr{D}^+_{\widetilde{N},W})$ and $CF(W,\widetilde{N};\mathscr{D}^-_{W,\widetilde{N}})$, the map $\phi_{\widetilde{N}, W}$ is of the form
\begin{equation}
\phi_{\widetilde{N}, W}=\left(\begin{array}{ccc}\phi_{N,L_1\#L_2}&0&0\\ \phi_{32}&\phi_{N,L_2}&0\\\phi_{31}&\phi_{21}&\phi_{N,L_1}\end{array}\right).
\end{equation}
The maps $\phi_{ij}$ for $(i,j)=(3,2),(3,1),(2,1)$ count discs with two inputs and satisfying boundary conditions along $W$ and $\widetilde{N}$. One of the inputs projects to an intersection point between $\gamma$ and the horizontal end of $W$ labelled $L_i$, and the other input projects to an intersection point between $\gamma$ and the horizontal end of $W$ labelled $L_j$. Here we use the convention $L_3=L_1\#L_2$. Figure \ref{fig:diagram7} depicts the projection to $\C$ of a curve in $\widetilde{M}$ counted by the map $\phi_{21}$. 

\begin{figure}
\centering
\def\svgwidth{135mm}
\begingroup%
  \makeatletter%
  \providecommand\color[2][]{%
    \errmessage{(Inkscape) Color is used for the text in Inkscape, but the package 'color.sty' is not loaded}%
    \renewcommand\color[2][]{}%
  }%
  \providecommand\transparent[1]{%
    \errmessage{(Inkscape) Transparency is used (non-zero) for the text in Inkscape, but the package 'transparent.sty' is not loaded}%
    \renewcommand\transparent[1]{}%
  }%
  \providecommand\rotatebox[2]{#2}%
  \newcommand*\fsize{\dimexpr\f@size pt\relax}%
  \newcommand*\lineheight[1]{\fontsize{\fsize}{#1\fsize}\selectfont}%
  \ifx\svgwidth\undefined%
    \setlength{\unitlength}{382.67716535bp}%
    \ifx\svgscale\undefined%
      \relax%
    \else%
      \setlength{\unitlength}{\unitlength * \real{\svgscale}}%
    \fi%
  \else%
    \setlength{\unitlength}{\svgwidth}%
  \fi%
  \global\let\svgwidth\undefined%
  \global\let\svgscale\undefined%
  \makeatother%
  \begin{picture}(1,0.47662664)%
    \lineheight{1}%
    \setlength\tabcolsep{0pt}%
    \put(0,0){\includegraphics[width=\unitlength,page=1]{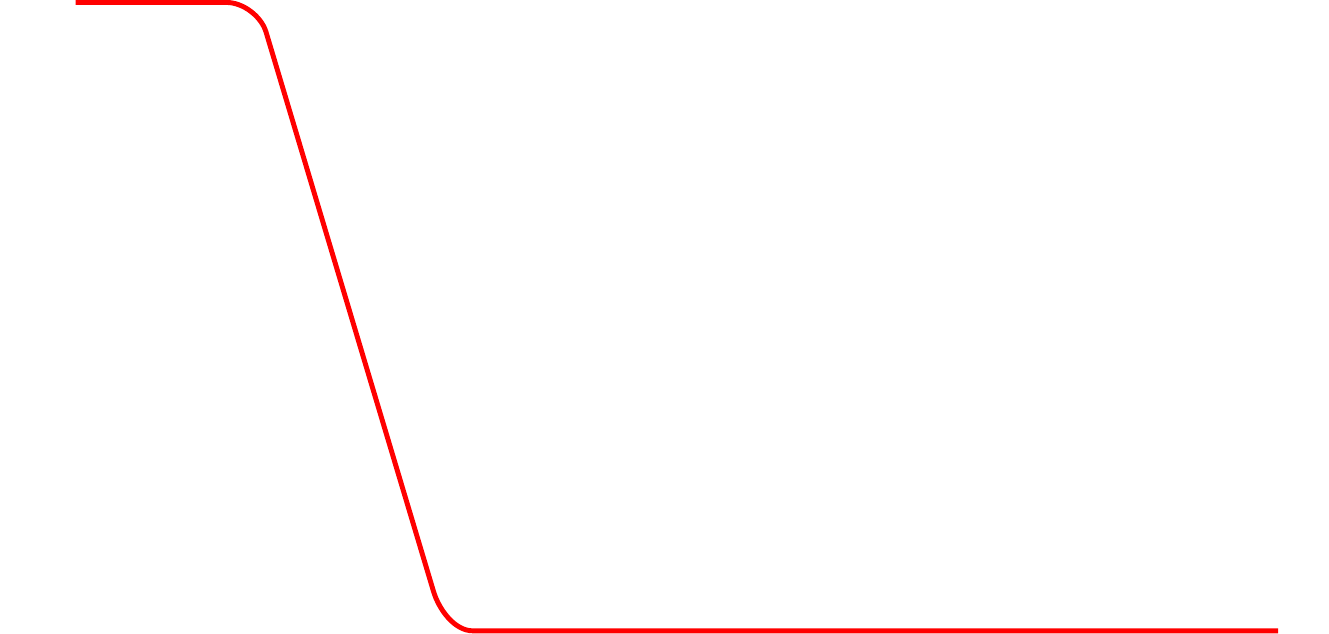}}%
    \put(-0.04499516,0.35338967){\color[rgb]{0,0,0}\makebox(0,0)[lt]{\lineheight{2.13000011}\smash{\begin{tabular}[t]{l}$L_1\#L_2$\end{tabular}}}}%
    \put(0.97265915,-0.0030798){\color[rgb]{0,0,0}\makebox(0,0)[lt]{\lineheight{2.13000011}\smash{\begin{tabular}[t]{l}$\gamma$\end{tabular}}}}%
    \put(-0.04499516,0.20051915){\color[rgb]{0,0,0}\makebox(0,0)[lt]{\lineheight{2.13000011}\smash{\begin{tabular}[t]{l}$L_2$\end{tabular}}}}%
    \put(-0.04499516,0.05548817){\color[rgb]{0,0,0}\makebox(0,0)[lt]{\lineheight{2.13000011}\smash{\begin{tabular}[t]{l}$L_1$\end{tabular}}}}%
    \put(0,0){\includegraphics[width=\unitlength,page=2]{diagram7.pdf}}%
    \put(0.42677111,0.14038666){\color[rgb]{0,0,1}\makebox(0,0)[lt]{\lineheight{2.13000011}\smash{\begin{tabular}[t]{l}$\phi_{21}$\end{tabular}}}}%
    \put(0,0){\includegraphics[width=\unitlength,page=3]{diagram7.pdf}}%
  \end{picture}%
\endgroup%

\caption[The projection of a curve contributing to the map $\phi_{21}$]{\label{fig:diagram7}The projection to $\C$ of a curve in $\widetilde{M}$ contributing to the map $\phi_{21}$}
\end{figure}

Since $\phi_{\widetilde{N}, W}$ is a chain map, the following diagram commutes up to homotopy
\begin{equation}\label{EQ: chain-level duality map surgery diagram}
\begindc{\commdiag}[70] 
\obj(0,10)[M_L_m^l]{$CF(N, L_1\# L_2; \mathscr{D}^+_{\widetilde{N},W})$}
\obj(20,10)[K_(m-1)^l]{$\mathrm{cone}(\rho_{21})$}
\obj(0,0)[M_L_m^rDual]{$CF(L_1\# L_2,N;\mathscr{D}^-_{W,\widetilde{N}})^\vee$}
\obj(20,0)[K_(m-1)^rDual]{$\mathrm{cone}(\psi_{12}^\vee)$}
\mor{M_L_m^l}{M_L_m^rDual}{$\phi_{N,L_1\#L_2}$}
\mor{M_L_m^l}{K_(m-1)^l}{$(\rho_{32},\rho_{31})$}
\mor{M_L_m^rDual}{K_(m-1)^rDual}{$(\psi_{23}^\vee,\psi_{13}^\vee)$}
\mor{K_(m-1)^l}{K_(m-1)^rDual}{$\left(\begin{array}{cc}\phi_{N,L_2}&0\\\phi_{21}&\phi_{N,L_1}\end{array}\right)$}
\enddc
\end{equation}
The chain homotopy is given by
\begin{equation}
(\phi_{32},\phi_{31}):CF(N, L_1\# L_2; \mathscr{D}^+_{\widetilde{N},W})\to \mathrm{cone}(\psi_{12}^\vee).
\end{equation}
From the diagram \eqref{EQ: chain-level duality map surgery diagram}, we see that the Poincar\'{e} duality map for the pair $N,L_1\#L_2$ is given on the level of homology in terms of the duality maps for the pairs $N,L_2$ and $N,L_1$ together with the map $\phi_{21}$ constructed from the cobordism $W$. 

\subsubsection{The module-level duality quasi-isomorphism}
Next we recall from \cite{BC14} the form of the module-level cone decomposition associated to the cobordism $W$ obtained from the Lagrangian surgery of $L_1$ and $L_2$. In this case, the decomposition is simply an exact triangle in the category $\mathcal{F}\mathit{\mbox{--}mod}$ of the form
\begin{equation}\label{EQ: Exact triangle for surgery}
\mathbf{Y}^l_{\mathcal{F}}(L_2)\to \mathbf{Y}^l_{\mathcal{F}}(L_1)\to \mathbf{Y}^l_{\mathcal{F}}(L_1\#L_2)\to \mathbf{Y}^l_{\mathcal{F}}(L_2).
\end{equation} 
This triangle results from considering properties of the left $\mathcal{F}$-modules $\mathcal{M}_{W,\gamma_j}^l:=\mathbf{Y}^l_{\mathcal{F}_c}(W)\circ\mathbf{I}_{\gamma_j}$ for an arrangement of paths $\gamma_1,\gamma_2,\gamma_3\subset\C$ as in Figure \ref{fig:diagram4} for $s=3$. At the heart of the proof of the existence of the triangle \eqref{EQ: Exact triangle for surgery} are arguments based on the open mapping theorem and orientation considerations, similar to those we have seen. 

Rather than considering the left $\mathcal{F}$-modules $\mathcal{M}_{W,\gamma_j}^l$, one could instead consider the right $\mathcal{F}$-modules $\mathcal{M}_{W,\gamma_j}^{r,\mathit{rel}}:=\mathbf{Y}^r_{\mathit{rel}}(W)\circ\mathbf{I}_{\gamma_j}$. The arguments in \cite{BC14} can be adapted to study the properties of these modules and to relate the $\mathcal{M}_{W,\gamma_j}^{r,\mathit{rel}}$ to the right $\mathcal{F}$-modules $\mathbf{Y}^r_{\mathcal{F}}(L_1)$, $\mathbf{Y}^r_{\mathcal{F}}(L_2)$, and $\mathbf{Y}^r_{\mathcal{F}}(L_1\#L_2)$. The result is an 
exact triangle in the category $\mathit{mod\mbox{--}}\mathcal{F}$
\begin{equation}\label{EQ: Exact triangle for surgery, right modules}
\mathbf{Y}^r_{\mathcal{F}}(L_1)\to \mathbf{Y}^r_{\mathcal{F}}(L_2)\to \mathbf{Y}^r_{\mathcal{F}}(L_1\#L_2)\to \mathbf{Y}^r_{\mathcal{F}}(L_1).
\end{equation}
Dualizing this triangle gives an exact triangle in $\mathcal{F}\mathit{\mbox{--}mod}$,
\begin{equation}\label{EQ: Exact triangle for surgery, right modules dual}
(\mathbf{Y}^\vee_{\mathcal{F}})^l(L_2)\to (\mathbf{Y}^\vee_{\mathcal{F}})^l(L_1)\to (\mathbf{Y}^\vee_{\mathcal{F}})^l(L_1\#L_2)\to (\mathbf{Y}^\vee_{\mathcal{F}})^l(L_2).
\end{equation}
It is then possible to compare the triangles \eqref{EQ: Exact triangle for surgery} and \eqref{EQ: Exact triangle for surgery, right modules dual}. These two triangles are related by the components $(\delta^\mathcal{F}_0)_{L_1}$, $(\delta^\mathcal{F}_0)_{L_2}$, and $(\delta^\mathcal{F}_0)_{L_1\# L_2}$ of the representative $\delta^\mathcal{F}\in \mathit{fun}(\mathcal{F},\mathcal{F}\mathit{\mbox{--}mod})(\mathbf{Y}_\mathcal{F}^l,(\mathbf{Y}_\mathcal{F}^\vee)^l)$ of the weak Calabi-Yau structure on the Fukaya category $\mathcal{F}$ of $M$. More precisely, the following diagram commutes in $\mathcal{F}\mathit{\mbox{--}mod}$ up to homotopy of module morphisms:
\begin{equation}\label{EQ: surgery commutative diagram of modules}
\begindc{\commdiag}[20] 
\obj(20,20)[M_L_j^l]{$\mathbf{Y}_\mathcal{F}^l(L_2)$} 
\obj(60,20)[K_(j-1)^l]{$\mathbf{Y}^l_{\mathcal{F}}(L_1)$} 
\obj(110,20)[K_j^l]{$\mathbf{Y}^l_{\mathcal{F}}(L_1\#L_2)$} 
\obj(160,20)[M_L_j^l-2]{$\mathbf{Y}_\mathcal{F}^l(L_2)$}
\obj(20,0)[M_L_j^rDual]{$ (\mathbf{Y}^\vee_{\mathcal{F}})^l(L_2)$} 
\obj(60,0)[K_(j-1)^rDual]{$(\mathbf{Y}^\vee_{\mathcal{F}})^l(L_1)$} 
\obj(110,0)[K_j^rDual]{$(\mathbf{Y}^\vee_{\mathcal{F}})^l(L_1\#L_2)$} 
\obj(160,0)[M_L_j^r-2Dual]{$(\mathbf{Y}^\vee_{\mathcal{F}})^l(L_2)$}
\mor{M_L_j^l}{K_(j-1)^l}{}
\mor{K_(j-1)^l}{K_j^l}{} 
\mor{K_j^l}{M_L_j^l-2}{}
\mor{M_L_j^rDual}{K_(j-1)^rDual}{}
\mor{K_(j-1)^rDual}{K_j^rDual}{}
\mor{K_j^rDual}{M_L_j^r-2Dual}{}
\mor{M_L_j^l}{M_L_j^rDual}{$(\delta^\mathcal{F}_0)_{L_2}$}
\mor{K_(j-1)^l}{K_(j-1)^rDual}{$(\delta^\mathcal{F}_0)_{L_1}$}
\mor{K_j^l}{K_j^rDual}{$(\delta^\mathcal{F}_0)_{L_1\# L_2}$}
\mor{M_L_j^l-2}{M_L_j^r-2Dual}{$(\delta^\mathcal{F}_0)_{L_2}$}
\enddc  
\end{equation}
In particular, the triangles \eqref{EQ: Exact triangle for surgery} and \eqref{EQ: Exact triangle for surgery, right modules dual} are isomorphic in the derived Fukaya category $D(\mathcal{F})$. We note that this result is not specific to the cobordism associated to Lagrangian surgery. Indeed, any three-ended cobordism in $\mathcal{CL}_d(\widetilde{M})$ leads to exact triangles in the categories $\mathcal{F}\mathit{\mbox{--}mod}$ and $\mathit{mod\mbox{--}}\mathcal{F}$, as in \eqref{EQ: Exact triangle for surgery} and \eqref{EQ: Exact triangle for surgery, right modules} in the case of surgery. The triangle in $\mathcal{F}\mathit{\mbox{--}mod}$ is related to the dual of the triangle in $\mathit{mod\mbox{--}}\mathcal{F}$ by a diagram like \eqref{EQ: surgery commutative diagram of modules}.

We will not undertake the proof of the commutativity of \eqref{EQ: surgery commutative diagram of modules} up to homotopy here. However, the basic method is to consider the properties of the duality quasi-isomorphisms
\begin{equation}
(\mathbf{I}_{\gamma_j})^*(\delta^{\mathcal{F}_c}_0)_W:\mathcal{M}^l_{W,\gamma_j}\to (\mathcal{M}^{r,\mathit{rel}}_{W,\gamma_j})^\vee, 
\end{equation}
and to relate these morphisms to the duality morphisms $(\delta^\mathcal{F}_0)_{L_1}$, $(\delta^\mathcal{F}_0)_{L_2}$, and $(\delta^\mathcal{F}_0)_{L_1\# L_2}$. The proof has at its core considerations about the moduli spaces figuring in the definition of $(\mathbf{I}_{\gamma_j})^*(\delta^{\mathcal{F}_c}_0)_W$. Specifically, these are the moduli spaces $\mathcal{R}_{\delta}^{m,0;1}(\gamma_1,\ldots,\gamma_m,\bm{\xi};\bm{\xi}')$ of solutions of \eqref{EQ: Perturbed CR equation} satisfying boundary conditions along the cobordisms $W$ and $\gamma_j\times N_i$ for $N_i\in\mathcal{L}^*_d(M)$, $i=0,\ldots,m$.

\subsection{General cone decompositions}\label{SECTION: General cone decompositions}

In this section we state a conjecture for cone decompositions associated to cobordisms which extends the result of the previous section for the cobordism obtained from Lagrangian surgery at a point. We take $W:\emptyset \to (L_1,\ldots,L_s)$ to be a cobordism in $\mathcal{CL}_d(\widetilde{M})$, where $L_1,\ldots,L_s\in \mathcal{L}^*_d(M)$ are arbitrary. Fix paths $\gamma_1,\ldots,\gamma_s$ in $\C$ as shown in Figure \ref{fig:diagram4}. 
We recall Theorem A of \cite{BC14}, which we state at the underived level.

\begin{thm}\label{THM: Cobordism cone decomp left modules}
The left $\mathcal{F}$-modules $\mathcal{M}_{W,\gamma_j}^l:=\mathbf{I}_{\gamma_j}^*(\mathbf{Y}^l_{\mathcal{F}_c}(W))$ satisfy:
\begin{enumerate}
\item\label{THM:B-C Main, left modules, part 1} $\mathcal{M}_{W,\gamma_1}^l=\mathbf{Y}_{\mathcal{F}}^l(L_1)$.
\item\label{THM:B-C Main, left modules, part 2} For $j=2,\ldots, s-1$, there exist module morphisms
\begin{equation}
\nu_j:\mathbf{Y}_{\mathcal{F}}^l(L_j)\to \mathcal{M}_{W,\gamma_{j-1}}^l 
\end{equation}
which fit into exact triangles in $\mathcal{F}\mathit{\mbox{--}mod}$ of the form
\begin{equation}\label{EQ: General exact triangles of left modules}
\mathbf{Y}_{\mathcal{F}}^l(L_j)\xrightarrow{\nu_j} \mathcal{M}_{W,\gamma_{j-1}}^l \xrightarrow{}\mathcal{M}_{W,\gamma_j}^l\xrightarrow{}\mathbf{Y}_{\mathcal{F}}^l(L_j).
\end{equation}
Moreover, there is a module quasi-isomorphism
\begin{equation}
\nu_s:\mathbf{Y}^l_\mathcal{F}(L_s)\to\mathcal{M}_{W,\gamma_{s-1}}^l.
\end{equation}
\end{enumerate}
\end{thm}

\begin{figure}
\centering
\def\svgwidth{0.7\linewidth}
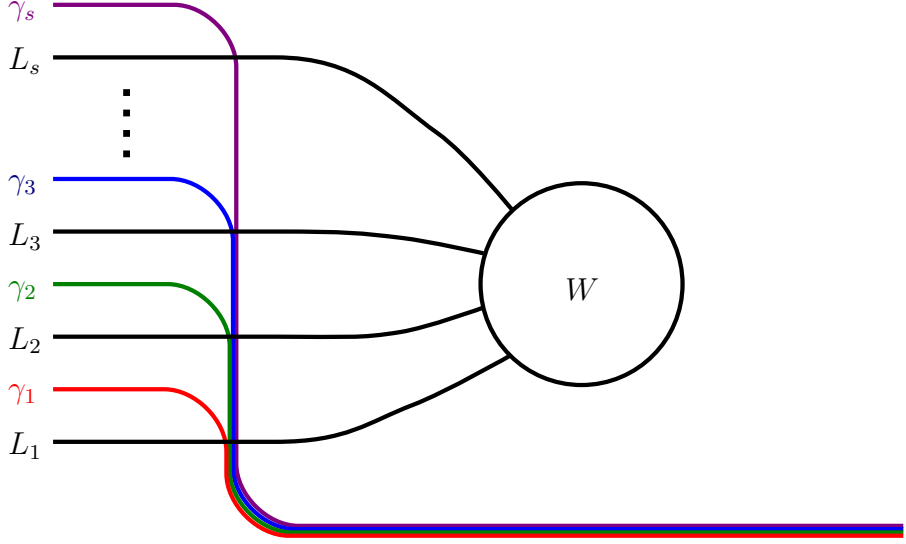
\caption{\label{fig:diagram4}The cobordism $W$ together with the paths $\gamma_1,\ldots,\gamma_s$}
\end{figure}

Using the methods of \cite{BC14}, it is possible to prove the following analogue for right $\mathcal{F}$-modules of the preceding theorem.

\begin{thm}\label{THM: Cobordism cone decomp right modules}
The right $\mathcal{F}$-modules $\mathcal{M}_{W,\gamma_j}^{r,\mathit{rel}}:=\mathbf{I}_{\gamma_j}^*(\mathbf{Y}^r_{\mathit{rel}}(W))$ satisfy:
\begin{enumerate}
\item\label{THM:B-C Main, right modules, part 1} $\mathcal{M}_{W,\gamma_1}^{r,\mathit{rel}}=\mathbf{Y}_{\mathcal{F}}^r(L_1)$.
\item\label{THM:B-C Main, right modules, part 2} For $j=2,\ldots, s-1$, there exist module morphisms
\begin{equation}\label{EQ: General exact triangles of right modules}
\xi_j:\mathcal{M}_{W,\gamma_{j-1}}^{r,\mathit{rel}}\to \mathbf{Y}_{\mathcal{F}}^r(L_j)
\end{equation}
which fit into exact triangles in $\mathit{mod\mbox{--}}\mathcal{F}$ of the form
\begin{equation}
\mathcal{M}_{W,\gamma_{j-1}}^{r,\mathit{rel}}\xrightarrow{\xi_j} \mathbf{Y}_{\mathcal{F}}^{r}(L_j)\xrightarrow{}\mathcal{M}_{W,\gamma_j}^{r,\mathit{rel}}\xrightarrow{}\mathcal{M}_{W,\gamma_{j-1}}^{r,\mathit{rel}}.
\end{equation}
Moreover, there is a module quasi-isomorphism
\begin{equation}
\xi_s:\mathcal{M}_{W,\gamma_{s-1}}^{r,\mathit{rel}}\to \mathbf{Y}^{r}_\mathcal{F}(L_s).
\end{equation}
\end{enumerate}
\end{thm}

Finally, it is possible to compare the cone decompositions of Theorems \ref{THM: Cobordism cone decomp left modules} and \ref{THM: Cobordism cone decomp right modules} using the relative weak Calabi-Yau pairing on $\mathcal{F}_c$. Although we do not carry out the details of this comparison here, we state the expected result in the following conjecture. 

\begin{conj}Consider the exact triangles of left $\mathcal{F}$-modules \eqref{EQ: General exact triangles of left modules} associated to the cobordism $W$, as well as the exact triangles of left $\mathcal{F}$-modules dual to the triangles \eqref{EQ: General exact triangles of right modules},
\begin{equation}
(\mathbf{Y}^{r}_{\mathcal{F}}(L_j))^\vee \xrightarrow{\xi_j^\vee} (\mathcal{M}_{W,\gamma_{j-1}}^{r,\mathit{rel}})^\vee\xrightarrow{} (\mathcal{M}_{W,\gamma_j}^{r,\mathit{rel}})^\vee\xrightarrow{} (\mathbf{Y}_\mathcal{F}^{r}(L_j))^\vee,\; 2\le j \le s-1.
\end{equation}
Then for $j=1,\ldots, s-1$ the quasi-isomorphisms 
\begin{equation}
\delta_{W,j}:=(\mathbf{I}_{\gamma_j})^*(\delta_0^{\mathcal{F}_c})_W:\mathcal{M}_{W,\gamma_j}^l\to (\mathcal{M}_{W,\gamma_{j-1}}^{r,\mathit{rel}})^\vee
\end{equation}
have the following properties: 
\begin{enumerate}
\item $\delta_{W,1}=(\delta^\mathcal{F}_0)_{L_1}$. 
\item For $j=2,\ldots, s-1$, the diagrams
$$
\begindc{\commdiag}[20] 
\obj(20,28)[M_L_j^l]{$\mathbf{Y}_\mathcal{F}^l(L_j)$} 
\obj(60,28)[K_(j-1)^l]{$\mathcal{M}_{W,\gamma_{j-1}}^l$} 
\obj(110,28)[K_j^l]{$\mathcal{M}_{W,\gamma_j}^l$} 
\obj(160,28)[M_L_j^l-2]{$\mathbf{Y}_\mathcal{F}^l(L_j)$}
\obj(20,0)[M_L_j^rDual]{$(\mathbf{Y}_\mathcal{F}^r(L_j))^\vee$} 
\obj(60,0)[K_(j-1)^rDual]{$(\mathcal{M}_{W,\gamma_{j-1}}^{r,\mathit{rel}})^\vee$} 
\obj(110,0)[K_j^rDual]{$(\mathcal{M}_{W,\gamma_j}^{r,\mathit{rel}})^\vee$} 
\obj(160,0)[M_L_j^r-2Dual]{$(\mathbf{Y}_\mathcal{F}^r(L_j))^\vee$}
\mor{M_L_j^l}{K_(j-1)^l}{$\nu_j$}
\mor{K_(j-1)^l}{K_j^l}{} 
\mor{K_j^l}{M_L_j^l-2}{}
\mor{M_L_j^rDual}{K_(j-1)^rDual}{$\xi_j^\vee$}
\mor{K_(j-1)^rDual}{K_j^rDual}{}
\mor{K_j^rDual}{M_L_j^r-2Dual}{}
\mor{M_L_j^l}{M_L_j^rDual}{$(\delta^\mathcal{F}_0)_{L_j}$}
\mor{K_(j-1)^l}{K_(j-1)^rDual}{$\delta_{W,j-1}$}
\mor{K_j^l}{K_j^rDual}{$\delta_{W,j}$}
\mor{M_L_j^l-2}{M_L_j^r-2Dual}{$(\delta^{\mathcal{F}}_0)_{L_j}$}
\enddc  $$
commute in $\mathcal{F}\mathit{\mbox{--}mod}$ up to homotopy of module morphisms.
\item The diagram $$
\begindc{\commdiag}[40] 
\obj(0,15)[M_L_m^l]{$\mathbf{Y}_\mathcal{F}^l(L_s)$}
\obj(25,15)[K_(m-1)^l]{$\mathcal{M}_{W,\gamma_{s-1}}^l$}
\obj(0,0)[M_L_m^rDual]{$(\mathbf{Y}_\mathcal{F}^r(L_s))^\vee$}
\obj(25,0)[K_(m-1)^rDual]{$(\mathcal{M}_{W,\gamma_{m-1}}^{r,\mathit{rel}})^\vee$}
\mor{M_L_m^l}{M_L_m^rDual}{$(\delta^\mathcal{F}_0)_{L_s}$}
\mor{M_L_m^l}{K_(m-1)^l}{$\nu_s$}
\mor{M_L_m^rDual}{K_(m-1)^rDual}{$\xi_s^\vee$}
\mor{K_(m-1)^l}{K_(m-1)^rDual}{$\delta_{W,s}$}
\enddc$$
commutes in $\mathcal{F}\mathit{\mbox{--}mod}$ up to homotopy of module morphisms. 
\end{enumerate} 
In particular, the induced diagrams in the derived category $D(\mathcal{F})$ commute.
\end{conj}

The expected proof of this conjecture extends the method described at the end of the previous section for the cobordism associated to Lagrangian surgery in a point.

\bibliographystyle{plainnat}
\bibliography{MasterReferenceList} 

\end{document}